\documentclass[10pt]{article}

\title{A random Hall-Paige conjecture}
\author{Alp M\"uyesser\thanks{New College, University of Oxford, UK. email: alp.muyesser@new.ox.ac.uk} \and Alexey Pokrovskiy\thanks{University College London, UK. email: dralexeypokrovskiy@gmail.com}}
\usepackage{graphicx, caption}
\usepackage{amssymb}
\usepackage{amsthm}
\usepackage{amsmath}
\usepackage{enumitem}
\usepackage{framed}
\usepackage{soul}
\usepackage{changepage}
\usepackage{comment}
\usepackage{float}
\usepackage{hyperref}
\usepackage[dvipsnames, table]{xcolor}
\usepackage{lscape}

\usepackage{mathtools}

\theoremstyle{plain}
\newtheorem{theorem}{Theorem}[section]
\newtheorem{lemma}[theorem]{Lemma}
\newtheorem{problem}[theorem]{Problem}
\newtheorem{proposition}[theorem]{Proposition}
\newtheorem{claim}{Claim}[theorem]
\newtheorem{conjecture}[theorem]{Conjecture}
\newtheorem{observation}[theorem]{Observation}
\newtheorem{corollary}[theorem]{Corollary}
\theoremstyle{definition}
\newtheorem{definition}[theorem]{Definition}

\newcommand\eps{\varepsilon}

\renewcommand\P{\mathbb{P}}

\usepackage{graphicx}% 

\newcommand\polysmall{\stackrel{\mathclap{\normalfont\mbox{\scalebox{.5}{POLY}}}}{\ll}}

\newcommand{\id}{{e}}%{\mathrm{id}}

\begin{document}
\maketitle

\begin{abstract} A complete mapping of a group $G$ is a bijection $\phi\colon G\to G$ such that $x\mapsto x\phi(x)$ is also bijective. Hall and Paige conjectured in 1955 that a finite group $G$ has a complete mapping whenever $\prod_{x\in G} x$ is the identity in the abelianization of $G$. This was confirmed in 2009 by Wilcox, Evans, and Bray with a proof using the classification of finite simple groups. 
\par In this paper, we give a combinatorial proof of a far-reaching generalisation of the Hall-Paige conjecture for large groups. We show that for random-like and equal-sized subsets $A,B,C$ of a group $G$, there exists a bijection $\phi\colon A\to B$ such that $x\mapsto x\phi(x)$ is a bijection from $A$ to $C$ whenever $\prod_{a\in A} a \prod_{b\in B} b=\prod_{c\in C} c$ in the abelianization of $G$. We use this statement as a black-box to settle the following old problems in combinatorial group theory for large groups.
\begin{enumerate}
    \item We characterise sequenceable groups, that is, groups which admit a permutation $\pi$ of their elements such that the partial products $\pi_1$, $\pi_1\pi_2$, $\pi_1\pi_2\cdots \pi_n$ are all distinct. This resolves a problem of Gordon from 1961 and confirms conjectures made by several authors, including Keedwell's 1981 conjecture that all large non-abelian groups are sequenceable. We also characterise the related $R$-sequenceable groups, addressing a problem of Ringel from 1974.
    \item We confirm in a strong form a conjecture of Snevily from 1999 by characterising large subsquares of multiplication tables of finite groups that admit transversals. Previously, this characterisation was known only for abelian groups of odd order (by a combination of papers by Alon and Dasgupta-K\'arolyi-Serra-Szegedy and Arsovski).
    \item We characterise the abelian groups that can be partitioned into zero-sum sets of specific sizes, solving a problem of Tannenbaum from 1981. This also confirms a recent conjecture of Cichacz.  
    \item We characterise harmonious groups, that is, groups with an ordering in which the product of each consecutive pair of elements is distinct, solving a problem of Evans from 2015.
\end{enumerate}

\end{abstract}

\section{Introduction}
\par This paper is about finding large-scale structures in finite groups using techniques from probabilistic combinatorics. The prototypical example of a ``large-scale structure'' in a group is a complete mapping, or equivalently, a transversal in the Latin square corresponding to the multiplication table of the group. We now define each of these terms. A \textbf{complete mapping} of a group $G$ is a bijection $\phi\colon G\to G$ such that the function $x\mapsto x \phi(x)$ is also a bijection. A \textbf{Latin array} is an $n\times n$ array filled in with arbitrary symbols such that each symbol appears at most once in each row and column. A \textbf{Latin square} is an $n\times n$ Latin array with exactly $n$ symbols. A \textbf{transversal} of a Latin array is a collection of $n$ cells which do not share a row, a column, or a symbol. Denote by $M(G)$ the \textbf{multiplication table} of the group $G$ whose rows and columns are labelled by the elements of the group, and the entry in row $g_i$ and column $g_j$ is the group element $g_i\cdot g_j$. Observe that $M(G)$ is a Latin square, and that $M(G)$ has a transversal if and only if $G$ has a complete mapping.

\par The study of transversals in Latin squares began more than two hundred years ago, when Euler posed a problem equivalent to determining for which $n$ there exists a $n\times n$ Latin square whose entries can be partitioned into transversals. This motivates the study of transversals in multiplication tables because if the multiplication table of a group has a transversal, then it can be partitioned into transversals. To see this, we can translate the columns of a transversal by a non-identity group element to produce another transversal, entirely disjoint from the first. However, some multiplication tables do not contain any transversals at all, let alone partitions into transversals. Indeed, if we suppose that the multiplication table of a group has a transversal, equivalently, we know that the group has a complete mapping $\phi$. Denote by $\pi$ the permutation $x\mapsto x \phi(x)$, so we have that $\pi(x)=x \phi(x)$ for every $x\in G$. Multiplying these equations\footnote{One can fix an arbitrary ordering of $G$ for the product to be well-defined for non-abelian groups. All the statements about products like this that we write in the introduction are independent of the ordering picked.} together for each $x\in G$, we obtain the following.
\begin{equation}\label{111}\prod_{x\in G} \pi(x) = \prod_{x\in G}x \phi(x)\end{equation}
\par For abelian groups, (\ref{111}) rearranges into $\prod_{x\in G}x=\id$ (where $\id$ is the identity element of $G$), giving a necessary condition for having a complete mapping in abelian groups. An immediate consequence is that even-order cyclic groups do not admit complete mappings (for example in $G=\mathbb Z_{2n}$ we have $\prod_{x\in G}x=n\neq \id$).
For non-abelian groups, by taking the image of (\ref{111}) in the abelianization of $G$, we obtain that $\prod_{x\in G}x\in G'$, where $G'$ denotes the commutator subgroup of $G$. Thus  ``$\prod_{x\in G}x\in G'$'' is a necessary condition for the existence of a complete mapping in a general group. 
Letting $G^\mathrm{ab}$ denote the abelianization of $G$, that is $G^\mathrm{ab}=G/G'$, we see that this condition is equivalent to $\prod_{x\in G}\pi_{\mathrm{ab}}(x)$ being equal to the identity in $G^{\mathrm{ab}}$ (where $\pi_{\mathrm{ab}}:G\to G^{\mathrm{ab}}$ is the quotient homomorphism). Since $G^{\mathrm{ab}}$ is abelian, we write this as ``$\sum_{x\in G}x=0$  in  $G^{\mathrm{ab}}$''.
\par The condition that $\sum_{x\in G}x=0$ in $G^\mathrm{ab}$ (or equivalently that $\prod_{x\in G}x\in G'$) is known as the Hall-Paige condition \cite{hallpaige}. We remark that the Hall-Paige condition is sometimes written as ``all $2$-Sylow subgroups of $G$ are trivial or non-cyclic''. This is equivalent to $\prod_{x\in G}x$ being trivial in $G^\mathrm{ab}$, as shown by Hall and Paige themselves \cite{hallpaige}. Perhaps astonishingly, the Hall-Paige condition is not only necessary, but also sufficient for the existence of a complete mapping. This was first conjectured by Hall and Paige~\cite{hallpaige}. The Hall-Paige Conjecture has a rich history, we refer the reader to the book of Evans~\cite{evans2018orthogonal} (Chapters 3-7) for a description of various approaches taken for this problem. The conjecture was finally shown to be true in a combination of papers by  Wilcox~\cite{wilcox1}, Evans~\cite{evans}, and Bray~\cite{BRAY} in 2009. The original proof of Wilcox, Evans, and Bray uses an inductive argument which relies on the classification of finite simple groups. However, recently, a completely different proof for large groups was found by Eberhard, Manners, and Mrazovi\'c using tools from analytic number theory \cite{asymptotichallpaige}.
\par In this paper, with the goal of giving a unified approach to many related conjectures in the area, we study complete mappings between subsets of groups. For example, given equal sized subsets $A,B,C$ of a group $G$, is there a bijection $\phi\colon A\to B$ such that the map $\psi(a):=a\phi(a)$ defines a bijection $A\to C$? This corresponds to starting with the multiplication table of a group, and then deleting the rows corresponding to $G\setminus A$, columns corresponding to $G\setminus B$, and symbols corresponding to $G\setminus C$, and then searching for a transversal in the resulting structure, which is simply a Latin array with some missing entries. Generalising the Hall-Paige condition to this set-up, we see that 
\begin{equation}\label{generalhallpaige}
    \sum A +\sum B = \sum C \text{ (in $G^{\mathrm{ab}}$)}
\end{equation}
is a necessary condition for the existence of such a map $\phi$ (we use $\sum S$ to denote $\sum_{s\in S}s$, $\prod S$ is defined analogously). Of course, we cannot expect (\ref{generalhallpaige}) to be a sufficient condition for any triple of equal sized subsets\footnote{For example, consider $G=\mathbb{Z}_{99}$, $A=\{98,1\}$, $B=\{98,1\}$, $C=\{49,50\}$. }. However, our main theorem essentially states that for most subsets $A,B,C\subseteq G$, (\ref{generalhallpaige}) is the only obstruction for finding the desired map $\phi$. Recall that a $p$-\textbf{random subset} $X$ of a finite set $G$ is a subset sampled by including each element of $G$ in $X$ independently with probability $p$, and $\mathbin{\triangle}$ denotes symmetric difference. When we say that an event holds with high probability, we mean that the probability of that event approaches $1$ as $n$ tends to infinity. The letter $n$ throughout the paper always denotes the order of the ambient group $G$.
\begin{theorem}[Main result]\label{thm:mainintro} Let $G$ be a group of order $n$. Let $p\geq n^{-1/10^{105}}$. Let $R^1,R^2,R^3\subseteq G$ be $p$-random subsets, sampled independently. Then, with high probability, the following holds. 
\par Let $X,Y,Z\subseteq G$ be equal sized subsets satisfying the following properties.
\begin{itemize}
    \item $|X\mathbin{\triangle} R^1|+|Y\mathbin{\triangle} R^2|+|Z\mathbin{\triangle} R^3|\leq p^{10^{30}}n/\log(n)^{10^{30}}$
    \item $\sum X +\sum Y = \sum Z$ in $G^{\mathrm{ab}}$ (or equivalently $\prod X \prod Y (\prod Z)^{-1}\in G'$)
\end{itemize}
Then, there exists a bijection $\phi\colon X\to Y$ such that $x\mapsto x\phi(x)$ is a bijection from $X$ to $Z$. 
\end{theorem}

\par We remark that setting $p=1$ and $X=Y=Z=G$, we see that the Hall-Paige conjecture holds for sufficiently large groups. However, Theorem~\ref{thm:mainintro} extends far beyond the setting of the Hall-Paige conjecture. Indeed, we can settle several longstanding conjectures at the interface of group theory and combinatorics using the full strength of Theorem~\ref{thm:mainintro}. In the following section, we discuss several such problems. All of these problems are similar in spirit to the Hall-Paige conjecture in the sense that they concern finding large-scale structures in groups with certain desirable properties. However, all of these problems lack the level of symmetry present in the Hall-Paige problem. For example, some of these problems concern finding complete mappings in subsets of groups with limited structure, or they concern finding complete mappings permuting the group elements via a particular cycle type. Both of these constraints seem difficult to reason about using only algebraic techniques. 
\par This perhaps explains why many of the commonly used tools, such as Alon's combinatorial Nullstellensatz, were only able to go so far in addressing these questions. On the other hand, one clearly needs use some of the group theoretic structure. Indeed, the key obstruction for problems of this type ends up being some derivative of the Hall-Paige condition, which is inherently group theoretic. As surveyed by Gowers in \cite{gowers2017probabilistic}, there are many problems in combinatorics which are difficult for a similar reason. These are problems where ``there is too much choice for constructions to be easy to discover, and too little choice for simple probabilistic arguments to work'' \cite{gowers2017probabilistic}. Our proof, similar in spirit to Keevash's celebrated construction of designs \cite{keevash2014existence}, uses probabilistic tools but also exploits the algebraic structure of the problem. We give a detailed overview of our strategy in Section~\ref{sec:proofoutline}.

\par From now on, additive notation will always imply that the corresponding operation is taking place in the abelianization of the ambient group (e.g. when for $A\subseteq G$ we write ``$\sum A=\id$''  we mean that $\prod_{a\in A}\pi_{\mathrm{ab}}(a)=\id$ where $\pi_{\mathrm{ab}}:G\to G^{\mathrm{ab}}$ is the quotient map to the abelianization of $G$). Otherwise, operations within non-abelian groups will be denoted multiplicatively.
The quantity $n$ always denotes the size of the ambient group. We use $\id$ to denote the identity element of a group, and we sometimes use $0$ to denote $\id$ in the special case of abelian groups.

\subsection{Applications}\label{sec:introapplications}
As we already noted, the first application of our main result, Theorem~\ref{thm:mainintro}, is an alternative proof that the Hall-Paige conjecture holds for all sufficiently large groups. This uses only the $p=1$ case of Theorem~\ref{thm:mainintro}, where we set the subsets $X,Y,Z$ to be the entirety of the group $G$. We now mention another result we can recover easily, this time setting $X,Y,Z$ to be sets of size $|G|-1$. Goddyn and Halasz recently proved that multiplication tables contain near transversals, that is, a collection of $n-1$ cells which do not share a row, column, or a symbol \cite{goddyn2020all}. Perhaps surprisingly, there does not seem to be an easy way of deriving this from the Hall-Paige conjecture directly, and the proof in \cite{goddyn2020all} is somewhat involved. However, we can derive this result from Theorem~\ref{thm:mainintro} as follows.
\begin{proposition}
    Let $M$ be the multiplication table of a sufficiently large finite group $G$. Then, $M$ contains a near transversal. 
\end{proposition}
\begin{proof}
    Applying Theorem~\ref{thm:mainintro} with $p=1$, we derive that for $|G|$ large enough, $R^1=R^2=R^3$ satisfies the conclusion of Theorem~\ref{thm:mainintro} with positive probability, and therefore, with probability exactly $1$. Let $z\in G$ be some element equal to $\sum G$ in $G^{\mathrm{ab}}$. Setting $X=Y=G\setminus\{\id\}$ and $Z=G\setminus \{z\}$, we have that $\sum X + \sum Y = \sum Z$ in $G^{\mathrm{ab}}$. This implies that there is a bijection $\phi \colon X\to Y$ such that $x\to x\phi(x)$ is injective. $\phi$ then corresponds to a near transversal in $M$, as desired. 
\end{proof}

The $p=1$ case of Theorem~\ref{thm:mainintro} when applied with other subsets $X,Y,Z$ has novel applications as well. We discuss such an application in the next section. The other three applications we discuss use the full strength of Theorem~\ref{thm:mainintro}. In fact, for these applications, we rely on an appropriate generalisation of Theorem~\ref{thm:mainintro} with a more complicated distribution on the sets $R^1,R^2,R^3$. In Section~\ref{sec:statements}, we state this generalised version of Theorem~\ref{thm:mainintro}.

\subsubsection{Snevily's conjecture}
\par Deleting $k$ rows and $k$ columns of a multiplication table, we obtain a natural Latin array which we call a \textbf{subsquare} of the multiplication table. In analogy with complete mappings, it is natural to ask which subsquares contain transversals. One may suspect that deleting rows and columns should only make it easier to find transversals, and Snevily's conjecture states that this indeed should be the case for abelian groups of odd order \cite{snevily}. However, for even order abelian groups $G$, one may delete rows and columns so that the remaining Latin array is in fact the multiplication table of an even order subgroup $H\subseteq G$ (or a translate of such a multiplication table). Such subsquares cannot contain transversals as multiplication tables of even order cyclic groups do no admit transversals. In 1999, Snevily conjectured that this should be the only subtlety for cyclic even order groups. Below, we formally state both cases of Snevily's conjecture. Note that $A\times B$ denotes the subsquare of a group $G$ obtained by keeping only the rows corresponding to $A$ and columns corresponding to $B$ in the multiplication table of $G$.

\begin{conjecture}[Snevily, \cite{snevily}]\label{snevilyconjecture}
Let $S=A\times B$ be a subsquare of the multiplication table of an abelian group $G$ defined by two $n$-element sets $A,B\subseteq G$.
\begin{enumerate}
    \item If $|G|$ is odd, then $S$ has a transversal.
    \item If $G\cong \mathbb{Z}_{2k}$ for some $k\in\mathbb{N}$, then $S$ has a transversal unless there exists $g_1,g_2\in G$ such that $g_1 + A = g_2 + B = H$ for some even-order cyclic subgroup $H\subseteq G$.
\end{enumerate}
\end{conjecture}
\par The first case of the conjecture was verified by Alon in 2000 for prime order cyclic groups \cite{alonsnevily}. In 2001, Dasgupta, K\'arolyi, Serra, and Szegedy generalised Alon's result to arbitrary odd order cyclic groups \cite{dasgupta}. Both of these results use the celebrated Combinatorial Nullstellensatz \cite{alon1999combinatorial}. A decade later, Arsovski fully resolved the first case of Snevily's conjecture, using character theory \cite{arsovski}. On the other hand, as far as the authors are aware, no partial progress has been reported on the second case of the conjecture.
\par As observed by Wanless \cite{wanless_2011}, the second part of Snevily's conjecture does not generalise straightforwardly to all even abelian groups due to the following construction of Akbari and Alireza \cite{akbari}. Let $G=(\mathbb{Z}_2)^k$ for some $k\geq 1$, and let $a_1,a_2\in G$ be distinct and let $b_1,b_2\in G$ be distinct such that $a_1+a_2+b_1+b_2=0$. Then, setting $A=G\setminus\{a_1,a_2\}$, $B=G\setminus\{b_1,b_2\}$, it is a simple exercise to check that the subsquare $A\times B$ does not contain a transversal. We show, perhaps surprisingly, that this is the only other barrier for an abelian subsquare to contain a transversal.
\begin{theorem}\label{thm:characterisationabelian}
There exists an $n_0\in \mathbb{N}$ such that the following holds for all $n\geq n_0$. Let $G$ be an abelian group, and let $A,B\subseteq G$ with $|A|=|B|=n$. Then, $A\times B$ has a transversal, unless there exists some $k\geq 1$, $g_1,g_2\in G$ and a subgroup $H\subseteq G$ such that one of the following holds.
\begin{enumerate}
    \item $H\cong \mathbb{Z}_{2k}\times H_{odd}$ for some odd-order group $H_{odd}$, and $H= g_1 A = g_2 B$, i.e. $A$ and $B$ are cosets of $H$.
    \item $H\cong (\mathbb{Z}_{2})^k$, $g_1 A= H\setminus\{a_1,a_2\}$, $g_2 B= H\setminus\{b_1,b_2\}$ for some distinct $a_1,a_2\in H$ and distinct $b_1,b_2\in H$ such that $a_1+a_2+b_1+b_2=0$.
\end{enumerate}
\end{theorem}
Note that this confirms both cases of Snevily's conjecture for sufficiently large subsquares. We discuss proving a stronger characterisation valid for all $n$ in Section~\ref{sec:concluding}. 
\par We take Snevily's conjecture further by proving a far more general theorem characterising subsquares without transversals of all groups.
\begin{theorem}\label{thm:characterisation}
There exists an $n_0\in \mathbb{N}$ such that the following holds for all $n\geq n_0$. Let $G$ be a group, and let $A,B\subseteq G$ with $|A|=|B|=n$. Then, $A\times B$ has a transversal, unless there exists some $k\geq 1$, $g_1,g_2\in G$ and a subgroup $H\subseteq G$ such that one of the following holds.
\begin{enumerate}
    \item $H$ is a group that does not satisfy the Hall-Paige condition, and $A= g_1H$ and $B= Hg_2$. 
    \item $H\cong (\mathbb{Z}_{2})^k$, $g_1 A = H\setminus\{a_1,a_2\}$, $g_2 B= H\setminus\{b_1,b_2\}$ for some distinct $a_1,a_2\in H$ and distinct $b_1,b_2\in H$ such that $a_1+a_2+b_1+b_2=0$.
\end{enumerate}
\end{theorem} Roughly speaking, Theorem~\ref{thm:characterisation} states that deleting rows and columns from a multiplication table only makes it easier to find transversals (supposing we do not end up with a translate of a multiplication table of a subgroup), except for a very specific scenario where we delete $2$ rows and $2$ columns summing to zero from the multiplication table of an elementary abelian $2$-group.  Theorem~\ref{thm:characterisation} is proved in Section~\ref{sec:provingthecharacterisation}.

\subsubsection{Sequenceable and R-sequenceable groups}
Given a finite group $G$, a \textbf{sequencing} is an ordering of the elements of $G$ as $b_1,b_2,\ldots, b_n$ where the partial products $b_1$, $b_1b_2$, $\ldots$ , $b_1b_2\cdots b_n$ are all distinct. Observe that in a sequencing, $b_1=\id$. A group that admits a sequencing is called \textbf{sequenceable}. A similar notion is that of an R-sequencing. An \textbf{R-sequencing} is an ordering of the elements of $G$ as $b_1,b_2,\ldots, b_{n}$ where $b_1=\id$, the partial products $b_1$, $b_1b_2$, $\ldots$ , $b_1b_2\cdots b_{n-1}$ are all distinct and $b_1b_2\cdots b_{n}=e$. A group that admits an R-sequencing is called \textbf{R-sequenceable}. We will briefly discuss the rich history of the problems relating to these concepts, and we refer the reader to \cite{ollis2002sequenceable} and \cite{evans2018orthogonal} for a more comprehensive survey.
\par Gordon introduced the problem of determining which groups are sequenceable in 1961 \cite{gordon1961sequences}. His motivation was to construct complete Latin squares, a concept which we now define. A Latin square is called \textbf{row-complete} if every pair of distinct symbols appears exactly once in each order in adjacent horizontal cells. A Latin square is called \textbf{column-complete} if it has the same property with respect to adjacent vertical cells. A Latin square is \textbf{complete} if it is both row-complete and column-complete. Complete Latin squares possess an additional layer of symmetry, making them useful in various contexts. For example, complete Latin squares are useful in graph theory to give decompositions of complete directed graphs into Hamilton paths (see \cite{ollis2002sequenceable} and the references therein). Also, some applications to experimental design are given in \cite{bate2008review}. Gordon was motivated by the observation that given a sequenceable group $G$, we can construct a complete Latin square by considering the multiplication table of the group $G$.
\par Ringel had a completely different motivation for studying $R$-sequenceable groups. Such groups come up naturally in Ringel's celebrated proof of the Heawood map colouring conjecture \cite{ringel2012map}. Hence, Ringel asked for a classification of all such groups \cite{ringeloldproblem}.
\par In Section~\ref{sec:pathlike}, we solve both of these problems for large groups. This addresses problems reiterated by several authors \cite{evans2018orthogonal, ollis2002sequenceable}, and confirms a conjecture of Keedwell \cite{keedwell1981sequenceability} (see also Conjecture 7 in \cite{goddyn2020all}). In particular, we show that any large group with the Hall-Paige condition is R-sequenceable and we show that any large non-abelian group is sequenceable which may be surprising in view of the fact that, for example, the nonabelian groups of order $6$ and $8$ are not sequenceable \cite{gordon1961sequences}. Therefore, at least for this problem, some mild assumption on the size of the group is necessary for a clean characterisation. We also remark that several partial results towards this characterisation were obtained by other researchers, see \cite{ollis2002sequenceable} for a survey.

\subsubsection{Partitioning abelian groups into zero-sum sets} We call a subset $S$ of a group \textbf{zero-sum} if $\sum S=0$. Given a sequence $a_1, a_2,\ldots, a_k$ ($a_i\geq 2$) with $\sum a_i=n-1$, where $n=|G|$, when can we partition the non-identity elements of an abelian group into zero-sum sets of size $a_1, a_2,\ldots, a_k$? A variant of this natural problem seems to have been first considered in 1957 by Skolem \cite{skolem}, and in 1960 by Hanani \cite{Hanani_1960}. A complete solution for cyclic groups for sequences $a_1, a_2,\ldots, a_k$ with $a_i\geq 3$ was given by Friedlander, Gordon, and Tannenbaum in 1981 \cite{friedlander}.
\par Obviously, we need that $\sum G=0$. This already rules out even-order cyclic groups, for example. There is another very natural necessary condition. Let $\ell$ denote the number of $a_i$ such that $a_i=2$. We need that $G$ should contain at least $\ell$-many pairs $\{x,-x\}$ where $x\neq -x$ (i.e. at least $2\ell$-elements of order greater than $2$). It turns out that for odd order abelian groups, these two conditions are known to be sufficient as well as necessary \cite{tannenbaumold, Zeng2015OnZP}. However, for even order abelian groups, Tannenbaum observed that additional necessary conditions are required \cite{tannenbaumpartitions}. To see this, we invite the reader to consider the case when $G=\mathbb{Z}_4\times \mathbb{Z}_2 \times \mathbb{Z}_2$, and the sequence $a_1, a_2,\ldots, a_k$ is $(2,2,2,3,3,3)$. See \cite{tannenbaumpartitions} for a solution to why $G$ does not have the desired partition in this case. 
\par Motivated by the previous example, Tannenbaum asked to find a set of necessary and sufficient conditions guaranteeing a partition of an even order abelian group into zero-sum sets of prescribed sizes \cite{tannenbaumpartitions}. Despite the apparent lack of structure in the problem as evidenced by a rich family of counterexamples, in Section~\ref{sec:tannenbaum}, we give a complete characterisation of the (large) abelian groups and integer sequences for which the desired partition exists. Various special cases of this problem were investigated by several authors, for example see \cite{cichaczboolean} and all the references therein. In particular, our characterisation confirms the following conjecture of Cichacz \cite{cichaczconjecture}. \begin{conjecture}[Cichacz, \cite{cichaczconjecture}]\label{Conjecture_Cichacz}
Let $G$ be an abelian group with at least $3$ involutions and suppose we have numbers $r_1, \dots, r_t\geq 3$ with $r_1+\dots+r_t=|G|-1$. Then there is a partition of $G\setminus \id=Z_1\cup \dots \cup Z_t$ where  each $Z_i$ is a zero-sum set of size $r_i$.
\end{conjecture}

\subsubsection{Harmonious groups}
Given a group $G$, not necessarily abelian, a \textbf{harmonious ordering} is an ordering of the elements of $G$ as $a_1,a_2,\ldots, a_n$ such that $a_1a_2, a_2a_3, \ldots, a_{n-1}a_n, a_n a_{1}$ is also an ordering of the elements of $G$ (i.e. the latter sequence contains no repetitions). Groups which have harmonious orderings are called \textbf{harmonious}. Harmonious groups were first introduced in 1991 by Beals, Gallian, Headley, and Jungreis \cite{BEALS1991223}. They gave a characterisation of all abelian harmonious groups. Evans asked for a complete characterisation. 
\begin{problem}[Evans, \cite{evans2015applications}]
Which finite groups are harmonious?
\end{problem}
In Section~\ref{sec:pathlike}, we give a characterisation of all sufficiently large harmonious groups. It turns out that a (large) group is harmonious if and only if it satisfies the Hall-Paige condition, and is not an elementary abelian $2$-group.

\subsection{Organisation of the paper}
In Section~\ref{sec:proofoutline}, we give a extensive proof sketch. Section~\ref{sec:preliminaries} collects some standard concentration/nibble type results, makes precise some key definitions and notation, and records some group theoretic results we rely on. Section~\ref{sec:freeproducts} in particular is quite central, and is devoted to giving a sufficient set of conditions allowing us to find various gadgets throughout the paper. In Section~\ref{sec:statements}, we state a more general version of Theorem~\ref{thm:mainintro}, and we derive several variants including Theorem~\ref{thm:mainintro} itself. In Section~\ref{sec:mainproof} the generalised version of the main theorem is proved. Section~\ref{sec:applications} is devoted to deriving the applications we have listed in Section~\ref{sec:introapplications}. In Section~\ref{sec:concluding}, we discuss further avenues of research.

\section{Proof strategy for the main result}\label{sec:proofoutline}
In this section, we attempt to give an accessible outline of a special case of our main result, concerning cyclic groups. We conclude in Section~\ref{sec:additionaldifficulties} by outlining some key difficulties we omit in the simplified discussion.
\par Firstly, instead of using the language of complete mappings, we will re-frame Theorem~\ref{thm:mainintro} as a hypergraph matching problem. Given a group $G$, we define the $3$-uniform $3$-partite multiplication hypergraph $H_G$ as follows. Set $V(H_G):=G_A\sqcup G_B\sqcup G_C$ where $G_{*}$ is a copy of $G$ and $\sqcup$ indicates a disjoint union. Set $E(G):=\{(g_A,h_B,k_C)\in G_A\times G_B\times G_C \colon g_Ah_Bk_C=\id \}$. Given $X,Y,Z\subseteq G$, we denote by $H_G[X,Y,Z]$ the induced subgraph of $H_G$ given by the vertex subset $(G_A\cap X)\sqcup (G_B\cap Y)\sqcup (G_C\cap Z)$. Given equal sized subsets $X,Y,Z$, observe that finding a bijection $\phi\colon X\to Y$ such that $x\mapsto x\phi(x)$ is a bijection from $X$ to $Z$ is equivalent to finding a perfect matching in $H_G[X,Y,Z^{-1}]$ where $Z^{-1}=\{z^{-1}\colon z\in Z\}$. We will outline a proof for the following result. For the rest of the outline, fix $\eps=1/100$.
\begin{proposition}\label{prop:simplemaintheorem}
Let $X,Y,Z\subseteq \mathbb{Z}_n$ such that $|X|=|Y|=|Z|=n-O(n^{1-\eps})$, and $\sum X + \sum Y +\sum Z=0$. Then, $H_{\mathbb{Z}_n}[X,Y,Z]$ has a perfect matching. 
\end{proposition}
Proposition~\ref{prop:simplemaintheorem} is already novel, and would be sufficient, for example, to deduce Snevily's conjecture for large subsquares. We remark that Proposition~\ref{prop:simplemaintheorem} is a simple corollary of Theorem~\ref{thm:mainintro} which can be obtained by setting $p=1$. We remark that throughout the rest of the paper, when we say that a subset $S\subseteq G$ is \textbf{zero-sum}, we mean that the product of all elements of $S$ (in any order) is in $G'$, the commutator subgroup. For abelian groups, this corresponds to $\sum S=0$.

\subsection{Absorption}
Our main tool is the \textit{absorption method}, a technique codified by R\"odl, Ruci\'nski, Szemer\'edi \cite{RRSab} (see also the earlier work of Erd\H{o}s, Gy{\'a}rf{\'a}s, and Pyber \cite{erdHos1991vertex}), adapted to the setting of hypergraphs defined by groups. Absorption is a general method that reduces the task of finding spanning structures to finding \textit{almost} spanning structures. In most cases, the latter task is considerably simpler, as evidenced by the celebrated nibble method which roughly states that pseudorandom hypergraphs contain large matchings \cite{AS}. The main technical innovation in our paper is developing an absorption strategy for multiplication hypergraphs, which in the setting of Proposition~\ref{prop:simplemaintheorem}, culminates in the following lemma.
\begin{lemma}[Simpler version of Lemma~\ref{lem:zerosumabsorptionnonabelian}]\label{lem:absorptionsimple}
$H_{\mathbb{Z}_n}[X,Y,Z]$ contains a vertex subset $\mathcal{A}$ of size $o(n)$ such that for any $S\subseteq V(H_{\mathbb{Z}_n}[X,Y,Z])\setminus \mathcal{A}$ of size $O(n^{1-\eps})$ intersecting $X$, $Y$ and $Z$ in the same number of vertices, and satisfying $\sum S = 0$, $\mathcal{A}\cup S$ has a perfect matching.
\end{lemma}
\par In Lemma~\ref{lem:absorptionsimple}, $\mathcal{A}$ functions as our \textit{absorber}, in the sense that it can \textit{absorb} small enough subsets by forming perfect matchings when combined with them. The key premise of the absorption method is that once a suitable absorber is found and set aside, the only remaining task is to find an \textit{almost perfect matching} in the leftover set. Indeed, $\mathcal{A}$ can absorb whatever small subset $S$ we fail to cover with the almost perfect matching. 

\par We now explain in a bit more detail how Lemma~\ref{lem:absorptionsimple} reduces the task of proving Proposition~\ref{prop:simplemaintheorem} to finding a matching of size $n-O(n^{1-\eps})$ in $V(H_{\mathbb{Z}_n}[X,Y,Z])\setminus \mathcal{A}$. First, note that setting $S=\emptyset$ in Lemma~\ref{lem:absorptionsimple} implies that $\mathcal{A}$ has a perfect matching, and thus $\sum \mathcal{A}=0$ (if a subset contains a perfect matching, by definition the subset can be partitioned into zero-sum sets, and hence is zero-sum itself). Now, suppose that having fixed the set $\mathcal{A}$, we were able to find a matching $M_1$ in $V(H_{\mathbb{Z}_n}[X,Y,Z])\setminus \mathcal{A}$ covering all but $O(n^{1-\eps})$ vertices. Let $S$ denote the set of these leftover vertices. As $\mathcal{A}$ and $V(M_1)$ are disjoint zero-sum sets contained in $V(H_{\mathbb{Z}_n}[X,Y,Z])$, and $\sum X+\sum Y+\sum Z=0$ by assumption, we have that $\sum S=0$ as well. So by the property in Lemma~\ref{lem:absorptionsimple}, $\mathcal{A}\cup S$ spans another perfect matching, $M_2$ say. Then, $M_1\cup M_2$ is the desired perfect matching of $H_{\mathbb{Z}_n}[X,Y,Z]$. 
\par We remark that, in reality, deleting $\mathcal{A}$ from $H_{\mathbb{Z}_n}[X,Y,Z]$ would damage the pseudorandomness properties of the hypergraph too greatly to be able to find the desired $M_1$ using the Rödl nibble \cite{AS}. Therefore, here we actually need a slightly stronger version of Lemma~\ref{lem:absorptionsimple} which can find $\mathcal{A}$ inside small random sets. This way, deleting $\mathcal{A}$ only spoils the pseudorandomness of a set $R$ much smaller than the multiplication hypergraph itself. $R\setminus \mathcal{A}$ can then be dealt with using standard pseudorandomness arguments, see for example Lemma~\ref{Lemma_2_random_1_deterministic}.
\par For Lemma~\ref{lem:absorptionsimple} to hold, we remark that some condition on the value of $\sum S$ is necessary. Indeed, recall that if a subset admits a perfect matching, it has to be zero-sum. So, if $S_1$ and $S_2$ are two sets disjoint with $\mathcal{A}$ such that $\mathcal{A}\cup S_1$ and $\mathcal{A}\cup S_2$ both admit perfect matchings, it follows that 
$$0=\sum \mathcal{A}\cup S_1 =\sum \mathcal{A}\cup S_2 $$
hence $\sum S_1=\sum S_2$. Therefore, the set $\mathcal{A}$ can have the flexibility of combining with any member of a large family of sets $\mathcal{F}$ to produce perfect matchings only if $\sum S$ is fixed for all $S\in\mathcal{F}$. For convenience, we fix this sum to be $0$, but Lemma~\ref{lem:absorptionsimple} would remain true if we replaced $0$ with any other fixed element.
\subsubsection{Building the absorber from small subgraphs}
\par Our starting point for building the absorber set $\mathcal{A}$, similar in spirit to most applications of the absorption method, is the existence of small subgraphs (gadgets) which give \textit{local variability}. More precisely, we will rely on the existence of $O(1)$-sized gadgets, $Q$ say, that can combine with $2$ distinct sets, $F_1$ and $F_2$ say, each of size $O(1)$, such that $Q\cup F_1$ and $Q\cup F_2$ both induce perfect matchings in $H_{\mathbb{Z}_n}[X,Y,Z]$. We say that $Q$ can \textit{switch} between $F_1$ and $F_2$.
\par The power of the absorption method rests in the fact that small gadgets such as $Q$ displaying rather limited variability can be combined in a way to build an absorber displaying \textit{global variability}. By global variability, we mean the type of property that $\mathcal{A}$ has in the statement of Lemma~\ref{lem:absorptionsimple}. In particular, we are referring to how $\mathcal{A}$ can combine with essentially any subset of size $O(n^{1-\eps})$, as opposed to just a few of size $O(1)$. To achieve this in our case, we will use a variant of the absorption technique called \textit{distributive absorption}, initially developed by Montgomery \cite{randomspanningtree}. The method has since been applied in numerous settings, notably in the proof of Ringel's conjecture by Montgomery, Pokrovskiy, and Sudakov \cite{ringel}. For a detailed discussion of how gadgets such as $Q$ can be combined to build an absorber, we refer the reader to the discussion in \cite{ringel}.
\par For readers who are familiar with the absorption method, we add a quick remark that commonplace methods of building absorbers, such as those used in \cite{RRSab}, do not work in our context for the following simple reason. Say we have a subset $S$ and we are interested in subsets $A$ of size $k$ such that both $A$ and $A\cup S$ span a perfect matching. In the usual absorption strategy, we would require that the number of subsets $A$ with this property is $\Omega(n^k)$. However, if $A$ spans a perfect matching, then $\sum A = 0$. The number of zero-sum subsets $A$ of order $k$ is $O(n^{k-1})$ -- too little to appear in abundance when a positive fraction of $k$-subsets are randomly sampled. On the other hand, with the distributive absorption strategy, one may build absorbers with fewer gadgets, provided that one can show that the gadgets are well-distributed within the host structure.
\subsubsection{Absorption for pairs}\label{absforpairs}
\par Typically, in applications of the distributive absorption method, one works with gadgets $Q$ switching between $F_1$ and $F_2$ where $F_1$ and $F_2$ are both singletons. This would be impossible to implement in our context as if $Q\cup F_1$ and $Q\cup F_2$ both contain perfect matchings, then $\sum F_1=\sum F_2$. Thus, if $F_1$ and $F_2$ were singletons, $F_1$ and $F_2$ would consist of the exact same vertex. Hence, if we want $Q$ to be a gadget that actually gives us some flexibility, we can only hope to switch between sets of size at least $2$. This motivates us to search for disjoint sets $Q$, $F_1$ and $F_2$ such that $Q\cup F_1$ and $Q\cup F_2$ both span perfect matchings, $|F_1|=|F_2|=2$, and $\sum F_1 = \sum F_2=0$, where the final equality is chosen for convenience as in the statement of Lemma~\ref{lem:absorptionsimple}. 

\par Finding $Q$ with this property turns out to be rather easy. For example, fix disjoint sets $F_1:=\{a,-a\}$ and $F_2:=\{b,-b\}$ where neither $a$ nor $b$ is equal to its own inverse. Let us view $F_1$ and $F_2$ as subsets of $G_C$. Consider some $x\in G_A$. Set $y:=-x-a\in G_B$, set $z=x+a-b\in G_A$ and $w=-x+b\in G_B$. Suppose that $x\neq z$ and $y\neq w$. Observe that $M_1=\{(x,y,a), (z,t,-a)\}$ is a matching of $\{x,y,z,w\}\cup F_1$ and $M_2=\{(x,w,-b), (z, y, b)\}$ is a matching of  $\{x,y,z,w\}\cup F_2$. Hence, $Q_x:=\{x,y,z,w\}$ is a gadget with the desirable property of switching between $F_1$ and $F_2$. See Figure~\ref{fig:sketchsimple} for an illustration. Moreover, it is not hard to see that there are many choices of $x$ for which the corresponding sets $Q_x$ are all disjoint, which is critical for the distributive absorption strategy. 

\begin{figure}[h]
    \centering
    \includegraphics[width=0.8\textwidth]{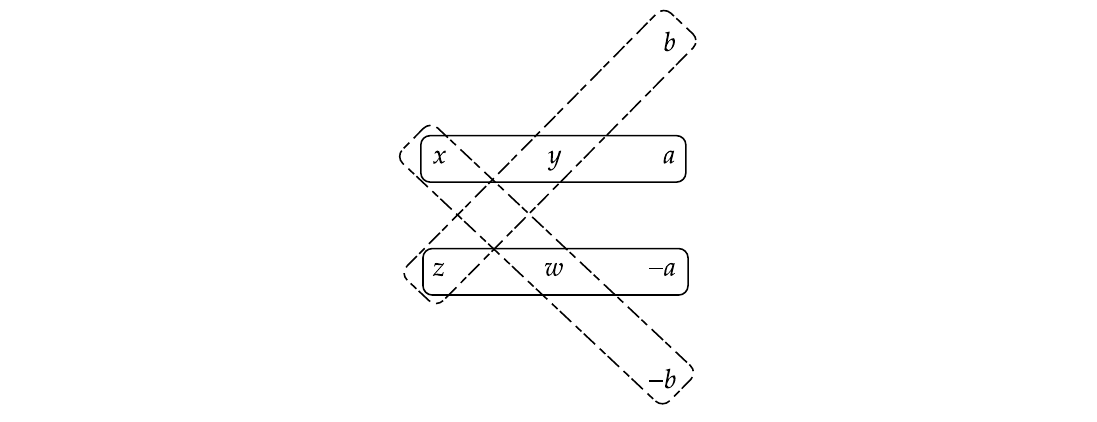}
    \caption{The gadget $Q_x$. Matchings $M_1$ and $M_2$ depicted in solid and dashed lines, respectively. }
    \label{fig:sketchsimple}
\end{figure}

\par Building on the idea detailed in the previous paragraph, we can show that there exists gadgets like $Q_x$ which can combine with any one of $100$ (as opposed to just $2$) pairs of inverses to produce a matching. This, combined with the usual distributive absorption strategy, is already sufficient to prove a version of Lemma~\ref{lem:absorptionsimple} with additional hypotheses on the set $S$. Namely, one can show the following.

\begin{lemma}[Simpler version of Lemma~\ref{lem:cosetpairedabsorber}]\label{lem:absorptionsimpler} Let $S_1$ be a $o(n)$-sized vertex subset of
$H_{\mathbb{Z}_n}[X,Y,Z]$. Then, $H_{\mathbb{Z}_n}[X,Y,Z]$ contains a vertex subset $\mathcal{A}'$ of size $o(n)$ such that for any $S_2\subseteq S_1$ of size $O(n^{1-\eps})$ intersecting $X$, $Y$ and $Z$ in the same number of vertices, closed under the function $x\to -x$, and containing no elements of order $\le 2$, we have that $\mathcal{A}'\cup (S_1\setminus S_2)$ has a perfect matching.
\end{lemma}
In the above lemma, it would arguably be more natural to insist that $\mathcal{A}'\cup S_2$ has a perfect matching, as opposed to $\mathcal{A}'\cup (S_1\setminus S_2)$. Such a version of the lemma would also be correct, and be even easier to prove. We formulate the lemma in the form above for a reason that will become clear shortly.

\subsubsection{From absorption for pairs to absorption for arbitrary zero-sum sets}
\par Now, we discuss how we can derive Lemma~\ref{lem:absorptionsimple} from Lemma~\ref{lem:absorptionsimpler}. The key idea is encapsulated in the following lemma.

\begin{lemma}[Simpler version of Lemma~\ref{lem:zerosumeliminate}]\label{lem:simplereliminate} $H_{\mathbb{Z}_n}[X,Y,Z]$ contains a vertex subset $\mathcal{T}$ of size $o(n)$ such that for any $S\subseteq V(H_{\mathbb{Z}_n}[X,Y,Z])\setminus \mathcal{T}$ of size $O(n^{1-\eps})$ intersecting $X$, $Y$ and $Z$ in the same number of vertices, and satisfying $\sum S = 0$, there exists a matching $M$ of order $O(n^{1-\eps})$ with $\mathcal{T}\cup S \supseteq V(M)\supseteq S$ satisfying also that $V(M)\setminus S$ is closed under $x\to -x$.
\end{lemma}
\par Deriving Lemma~\ref{lem:absorptionsimple} from Lemma~\ref{lem:absorptionsimpler} and Lemma~\ref{lem:simplereliminate} is a simple exercise. Indeed, let $\mathcal{T}$ be a vertex subset of $H_{\mathbb{Z}_n}[X,Y,Z])$ with the property in Lemma~\ref{lem:simplereliminate}. Apply Lemma~\ref{lem:absorptionsimpler} with $S_1=\mathcal{T}$ to obtain a vertex subset $\mathcal{A}'$. Set $\mathcal{A}:=\mathcal{A}'\cup \mathcal{T}$. We invite the reader to check that $\mathcal{A}$ then satisfies the property that Lemma~\ref{lem:absorptionsimple} requires.
\par To see how we prove a version of Lemma~\ref{lem:simplereliminate}, we invite the reader to see Section~\ref{sec:zerosumelimination}. Similar in spirit to the proof of Lemma~\ref{lem:absorptionsimpler}, our main trick here is to reduce (a version of) Lemma~\ref{lem:simplereliminate} to the existence of many small matchings with certain desirable properties (see Lemma~\ref{Lemma_cover_6_set}). 

\subsubsection{Additional difficulties}\label{sec:additionaldifficulties}
We now point out several complications we omitted in the previous discussion, along with some technicalities that arise in the level of generality of our main theorem. 
\par \textbf{Ensuring distinctness.} A fair portion of our arguments rely on the existence of constant sized matchings with specific properties, such as being closed under the map $x\to -x$. Often, the properties we require can be written as solutions to a particular system of linear equations. For example, consider $Q_x$ defined in Section~\ref{absforpairs}. The requirements from $\{x,y,z,w\}$ could be written as:
$$x+y+a=0 \hspace{1cm} z+w-a=0 \hspace{1cm} x+w+b=0 \hspace{1cm} z+y-b=0 $$
This is a system of equations with $4$ variables and $4$ constraints; however, any $3$ of these equations imply the fourth, allowing us to easily deduce that there are many $x,y,z,w$ with the desired properties. However, a solution to the system of equations is useful to us only when $x\neq z$ and $y\neq w$, as otherwise $\{(x,-x+b,-b), (x+a-b, -x-a, b)\}$ would not be a matching in $H_{\mathbb{Z}_n}$.
\par In the specific scenario outlined above, getting the relevant coordinates to be distinct is not particularly challenging. However, throughout the paper, we will require gadgets with properties significantly more complicated than those of $Q_x$. Furthermore, our main theorem works with subsets $X,Y,Z\subseteq G$ which are disjoint. Thus, we would have no chance of locating $Q_x$ within $H_G[X,Y,Z]$ unless all coordinates $x,y,z,w$ are distinct. Due to these technicalities, finding gadgets with the desired properties can become quite delicate. \par Section~\ref{sec:freeproducts} is entirely devoted to obtaining a sufficient set of conditions for a system of relations to yield solutions where each coordinate is distinct. The key result of that section, Lemma~\ref{Lemma_separated_set_random}, is used in abundance throughout the paper.    
\par \textbf{The elementary abelian $2$-group.} The strategy of working with pairs of inverses $\{x,-x\}$ with $x\ne -x$ fails for obvious reasons in the elementary abelian $2$-group. This turns out to be not a serious complication, as for general groups $G$, we will work with pairs $\{x, q_\phi x^{-1}\}$ for a carefully chosen $q_\phi$ so that $x\to q_\phi x^{-1}$ does not create too many fixed points. 
\par \textbf{Nonabelian groups.} Although it turns out that the case of general abelian groups is not significantly more complicated than cyclic groups, there are serious issues to overcome with nonabelian groups. As just one example, suppose that we wish to find a gadget similar to $Q_x$ from Section~\ref{absforpairs} in $H_G$, where $G$ is a non-abelian group. Suppose that $F_1=\{a,q_\phi a^{-1}\}$, $F_2=\{b,q_\phi b^{-1}\}$. Our goal is then to find (many) $Q$ such that $Q\cup F_1$ and $Q\cup F_2$ both can be perfectly matched. It is quite instructive to try to construct such $Q$, and we invite the reader to try to do so.
\par Some reflection shows that while it is difficult to construct $Q$ switching between $\{a,q_\phi a^{-1}\}$ and $\{b,q_\phi b^{-1}\}$, it is considerably simpler to find \textit{some} $b'$ such that $b'$ is in the same $G'$-coset as $q_\phi b^{-1}$ such that we can find a $Q$ switching between $\{a,q_\phi a^{-1}\}$ and $\{b, b'\}$. This motivates us to search for gadgets not only switching between pairs, but also switching between elements of the same $G'$-coset. Achieving this latter task is considerably more technical. We accomplish this by first devising a strategy for switching between commutator elements (as opposed to arbitrary elements in the commutator subgroup). To switch between arbitrary elements of $G'$, we have to rely on a non-trivial fact from representation theory (see Theorem~\ref{Theorem_write_commutators_as_short_products}). The statement we require is that arbitrary elements in the commutator subgroup can be written as products of just $O(\log n)$ commutators. Arguably, this is the only point in the proof of the main theorem where we use a group theoretic result beyond the undergraduate level. Due to bounds coming from Theorem~\ref{Theorem_write_commutators_as_short_products} (which are tight), we are obliged to use gadgets of logarithmic size to switch between elements of the same $G'$-coset. This inflates the error rate in our main theorem by a polylogarithmic factor. We refer the reader to Section~\ref{sec:absorbers} for more details.

\section{Preliminaries}\label{sec:preliminaries}

\subsection{Probabilistic tools}\label{sec:probtools}
\subsubsection{Concentration inequalities}
The below is a standard bound that can be found in many probability textbooks, for example see \cite{AS}. We will refer to it as Chernoff's bound. 
\begin{lemma}[Chernoff bound]\label{chernoff} Let $X:=\sum_{i=1}^m X_i$ where $(X_i)_{i\in[m]}$ is a sequence of independent indicator random variables with $\mathbb{P}(X_i=1)=p_i$. Let $\mathbb{E}[X]=\mu$. Then, for any $0<\gamma<1$, we have that $\mathbb{P}(|X-\mu|\geq \gamma \mu)\leq 2e^{-\mu \gamma^2/3}$.
\end{lemma}

We use the following corollary of Chernoff's bound often: that if $R$ is a $p$-random subset of an $n$-element set, then with high probability we have that $|pn-|R||\leq \log n\sqrt n$.

In almost all instances, the Chernoff bound will be all we need. Otherwise, we will make use of Azuma's inequality, which we now state. Given a product probability space $\Omega = \prod_{i\in[n]} \Omega_i$, a random variable $X\colon \Omega\to \mathbb{R}$ is called $C$-Lipschitz if $|X(\omega)-X(\omega')|\leq C$ whenever $\omega$ and $\omega'$ differ in at most $1$-coordinate. We will refer to the following standard bound as Azuma's inequality.
\begin{lemma}[Azuma's inequality]
Let $X$ be $C$-Lipschitz random variable on a product probability space with $n$ coordinates. Then, for any $t>0$, 
$$\mathbb{P}(|X-\mathbb{E}(X)|> t)\leq 2e^{\frac{-t^2}{nC^2}}.$$
\end{lemma}

\subsubsection{Pseudorandom graphs and the R\"odl nibble}
\par Here, we give some tools to find matchings in pseudorandom hypergraphs covering all but a few vertices. Our approach here is complicated by the fact that we need to find large matchings in subsets of hypergraphs, and some subsets could be arbitrary and we may as well suppose they were chosen adversarially. This comes from the first step of the proof where we set aside an absorber, whose complement could potentially have poor pseudorandomness properties. The most important result from this section is Lemma~\ref{Lemma_2_random_1_deterministic}, which tells us that subsets of multiplication tables contain large matchings whenever the subset is obtained by taking unions of two random sets, and one (potentially adversarially chosen) deterministic set. We now give the details.
\par For a $3$-uniform, $3$-partite hypergraph $H$, vertices $u,v$ and a subset $U\subseteq V(H)$, we define the \textbf{pair degree} of $(u,v)$ into $U$ as the number of vertices in $U$ which are in the neighbourhood of both $u$ and $v$, i.e. the number of vertices $z$ in $U$ such that there exists $v,w\in V(H)$ such that $\{u,z,v\}$ and $\{v,z,w\}$ are both edges of $H$. For a $2$-uniform, bipartite graph $H$, vertices $u,v$, and a subset $U\subseteq V(H)$, we define the \textbf{pair degree} of $(u,v)$ into $U$ as the number of mutual neighbours of $u$ and $v$ in $U$.  
\par We say that a $r$-partite $r$-uniform hypergraph $H$ (where $r$ will be either $2$ or $3$) is $(\gamma, p, n)$-\textbf{regular} if every part has $(1\pm \gamma)n$ vertices and every vertex has degree  $(1\pm \gamma)pn$. We say that $H$ is $(\gamma, p, n)$-\textbf{typical} if, additionally, every pair of vertices $x,y$ in the same part of $H$ have pair degree  $(1\pm \gamma)p^2n$ into every other part of $H$. We say that a hypergraph is \textbf{linear} if through every pair of vertices, there is at most one edge. Multiplication hypergraphs have all these properties.
\begin{observation}\label{obs:01ntypical}
For a group $G$ of order $n$, the multiplication hypergraph $H_G$ is $(0,1,n)$-typical and linear.
\end{observation}
\begin{proof}
It is immediate that all parts have size $n$.
For any two vertices $u,v$ in different parts, there is a unique edge through $u$ and $v$. If $u\in A, v\in B$, then this edge is $(u,v, v^{-1}u^{-1})$. If $u\in B, v\in C$, then this edge is $(v^{-1}u^{-1},  u,v)$. If $u\in A, v\in C$, then this edge is $(u,u^{-1}v^{-1}, v)$, hence $H_G$ is linear. This shows that all vertices have degree exactly $n$ and pair degree exactly $n$ to each part i.e. that the hypergraph is $(0,1,n)$-typical.
\end{proof}
Frankl and R\"odl \cite{frankl1985near} (also Pippenger, unpublished) showed that for all $\varepsilon, p\gg \gamma\gg n^{-1}$ every $(\gamma, p, n)$-regular hypergraph  has a matching of size $(1-\epsilon)n$. We need a well known variant of this where $\varepsilon, p, \gamma$ have polynomial dependencies on $n$.

\begin{lemma}\label{lem:delicate_nibble} Let $n$ be sufficiently large. Every $(\gamma, \delta, n)$-regular linear tripartite hypergraph has a matching covering  all but at most $n^{1-1/500}+3\gamma n$ vertices. 
\end{lemma}
\begin{proof}
There are various ways of proving this. We will deduce it from a result of Molloy-Reed --- Theorem~1 from~\cite{molloy2000near}. Applying that theorem with $k=3$, $\Delta=(1+\gamma) \delta n$ gives us a decomposition of our $(\gamma, \delta, n)$-regular hypergraph into $\Delta+c_k\Delta^{1-\frac{1}{k}}\log^4\Delta\leq \delta n+ \gamma \delta n+n^{6/7}$ matchings (for some constant $c_k$). By the pigeonhole principle one of these has at least $\frac{e(H)}{\delta n+ \gamma \delta n+n^{6/7}}\geq \frac{(1-\gamma)^2\delta n^2/3}{\delta n+ \gamma \delta n+n^{6/7}}\geq n- (n^{1-1/500}+3\gamma n)$ edges as required. 
\end{proof}

Typical graphs have the following well-known pseudorandomness condition, which dates back to work of Thomason~\cite{thomason1989dense}. Note that for the rest of the section, we assume that bipartite graphs come with a partition of their vertex set as $(A,B)$ and similarly tripartite hypergraphs come with a partition $(A,B,C)$. 
\begin{lemma}\label{Theorem_Thomassen}
Let $G$ be a $(\gamma,\delta,n)$-typical bipartite graph.
Then for every $A'\subseteq A, B'\subseteq B$, we have $e(A',B')=\delta|A'||B'|\pm 5\sqrt{(\delta + \gamma) n^3} + \gamma n^2$.
\end{lemma}
\begin{proof}
This will be a consequence of Theorem 2 of \cite{thomason1989dense}. First, delete at most $\gamma n$ vertices from one side of the graph to obtain a balanced bipartite graph. Now with parameters $p:=\delta-\gamma$ and $\mu:=5\gamma n$ it is easy to see that the hypothesis of Theorem 2 are satisfied. From the conclusion of Theorem 2, for every $A'\subseteq A, B'\subseteq B$, we have $e(A',B')=p|A'||B'|\pm 5\sqrt{(\delta + \gamma) n^3}$. With another $\gamma n^2$ term, we can account for the deleted vertices in the beginning, implying the desired bound.
\end{proof}
Typicality is preserved by taking random subsets, in the following sense.
\begin{lemma}\label{lem:typicality}
    Let $H=(A,B,C)$ be a tripartite linear hypergraph that is $(0, 1,n)$-typical. Let $p\geq n^{-1/600}$ and let $A'\subseteq A$ be $p$-random. Then, with probability at least $1-1/n^3$, the bipartite graph between $B$ and $C$ consisting of edges passing through $A'$ is  $(n^{-1/5},p, n)$-typical.
\end{lemma}
\begin{proof}
    For some $c,c'\in C$, let  $d_{A'}(c)=e(A',B,c)$ be the degree of $c$ into $A'$ and let $d_{A'}(c,c')$  denote the pair degree of $(c,c')$ into $A'$. We have $d_A(c)=d_A(c,c')=n$ (by $(0, 1,n)$-typicality of $H$).
\par Note that $\mathbb E(d_{A'}(c))=pd_{A}(c)= p n$ and $\mathbb E(d_{A'}(c,c'))=p^2 d_A(c,c')= p^2n$ (using linearity of $H$) for all $c,c'$. Set $\gamma:=n^{-1/5}$. By Chernoff's bound and a union bound, with probability at least $1-1/n^4$, for all $c$ we have $d_{A'}(c)=\mathbb E(d_{A'}(c))\pm \gamma n= pn\pm\gamma n$. Note that $d_{A'}(c,c')$ is $2$-Lipschitz. This is because for each $a$, there is exactly one $b$ with $abc$ an edge, and one $b$ with $abc'$ an edge (using linearity of $H$). Hence by Azuma's inequality and a union bound, we have that with probability at least $1-1/n^4$, for each pair $c,c'\in C$, $d_{A'}(c,c')= \mathbb E(d_{A'}(c,c'))\pm \gamma n= p^2n\pm\gamma n$. Corresponding bounds hold for $b,b'\in B$. With probability at least $1-3/n^4$ all these properties hold simultaneously. Whenever these properties all hold, we have that the bipartite graph $(B,C)$ consisting of edges through $A'$ is $(\gamma,p, n)$-typical as desired.
\end{proof}
Using the previous pseudorandomness property, we can derive the following lemma which states most vertices send approximately the expected number of edges through a random set and a deterministic set.

\begin{lemma}\label{Lemma_one_random_set_nearly_regular}
Let $H=(A,B,C)$ be a tripartite linear hypergraph that is $(0, 1,n)$-typical. Let $p\geq n^{-1/600}$ and let $A'\subseteq A$ be $p$-random.
Then, with probability at least $1-1/n^3$,
the following holds. For any $B'\subseteq B$, there are at most $n^{9/10}$ vertices $c\in C$ with $e_{H}(A',B', c)\neq p |B'|\pm  n^{9/10}$.
\end{lemma}
\begin{proof} 
\par 
Set $\gamma=n^{-1/5}$. By Lemma~\ref{lem:typicality}, with probability at least $1-1/n^3$, we have that the bipartite graph between $B$ and $C$ consisting of edges passing through $A'$ is  $(n^{-1/5},p, n)$-typical. Supposing this property holds, by Lemma~\ref{Theorem_Thomassen}, for any $B'\subseteq B, C'\subseteq C$, we have $e_{H}(A',B', C')= p|B'||C'|\pm 5\gamma^{1/2} n^2$. Let $C^-$ be the set of vertices with $e_{H}(A',B', c)< p|B'|- {\gamma}^{1/4} n$. We have that $e_{H}(A',B', C^-)< p'|B'||C^-|- {\gamma}^{1/4} n|C^-|$ and $e_{H}(A',B', C^-)= p'|B'||C^-|\pm 5\gamma^{1/2} n^2$ implying $|C^-| \leq 10{\gamma}^{1/4} n$. Similarly letting $C^+$ be the set of vertices with $e_{H}(A',B', c)> p|B'|+ \gamma^{1/2} n$, we get $|C^+|\leq 10{\gamma}^{1/4} n$. Plugging in the value of $\gamma$, this implies the lemma.
\end{proof}
The following lemma will allow us to find a large matching whenever we are given two random subsets and a deterministic subset. 
\begin{lemma}\label{Lemma_2_random_1_deterministic}
Let $H=(A,B,C)$ be a tripartite linear hypergraph that is $(0,1,n)$-typical. Let $p\geq n^{-1/600}$ and let $A'\subseteq A$ be $p$-random, and let $B'$ a $p$-random subset of $B$, where $A'$ and $B'$ are not necessarily independent. Then, with probability at least $1-10n^{-3}$, the following holds. 
\par For any $C'\subseteq C$ of size $(1\pm n^{-0.2})pn$, there is a matching covering all but $2n^{1-1/500}$ vertices in $A'\cup B'\cup C'$.
\end{lemma}
\begin{proof} 
With probability at least $1-10n^{-3}$, $A'$ and  $B'$ satisfy the conclusion of Lemma~\ref{Lemma_one_random_set_nearly_regular} and have size $(1\pm n^{-0.2})pn$ (by Chernoff's bound).
This means that $A'\cup B'\cup C'$ has $\leq n^{9/10}$ vertices with degree $\neq  p^2 n \pm n^{19/20}$ in $H[A', B', C']$. 
\par Deleting all such vertices gives a hypergraph satisfying the hypothesis of Lemma~\ref{lem:delicate_nibble} with $\gamma:=n^{-0.01}$, hence the desired matching exists.
\end{proof}

In some applications the following formulation which allows for three deterministic subsets as opposed to just one will be more convenient. The result follows simply by applying the previous result four times.

\begin{lemma}\label{lem:mainnibble}
Let $H=H_G$ be a multiplication hypergraph. Let $p\geq n^{-1/650}$. Let $A',B',C'$ be $p$-random subsets of $A$, $B$, $C$ respectively, not necessarily independent. Set $R:=A'\cup B'\cup C'$. With high probability the following holds. Let $q\leq 5p$. For any $X\subseteq V(H)\setminus R$ with $|X\cap A|, |X\cap B|, |X\cap C|=(1\pm n^{-0.25})qn$, there is a matching in $R\cup X$ covering all but at most $n^{1-10^{-4}}$ vertices of $R\cup X$.
\end{lemma}
\begin{proof} 
\par Let $q$ be a fixed rational number between $0$ and $1$ and denominator at most $n$. Suppose first that $q\leq n^{-1/600}$. We have that  Lemma~\ref{Lemma_2_random_1_deterministic} holds for $A'$ and $B'$ with probability $\geq 1-n^{1.5}$ and by Chernoff's bound $C'=(1\pm n^{-0.2})pn$, with probability $\geq 1-n^{-2}$. Both properties hold simultaneously with probability $\geq 1-n^{-1.49}$. Then, using Lemma~\ref{Lemma_2_random_1_deterministic}, $A'\cup B'\cup C'$ has a matching covering all but $n^{1-10^{-3}}$ vertices. Together with $X$, this gives $n^{1-10^{-3}} + 2n^{599/600}\leq n^{1-10^{-4}}$ vertices. 
\par Suppose now that $q\geq n^{-1/600}$. For each $\diamond\in\{A,B,C\}$, partition $\diamond'$ into a $q$-random set $\diamond_1$, and a $q$-random set $\diamond_2$, and a $(p-2q)$-random set $\diamond_3$. As $q,p-2q\geq n^{-1/600}$ by assumption, with probability $\geq 1-n^{-1.49}$ the pairs $(A_1,B_1), (A_2,C_1), (B_2, C_2)$ and $(A_3,B_3)$ satisfy the property of Lemma~\ref{Lemma_2_random_1_deterministic}, and $|\diamond_i|=(1+n^{-0.22})\mathbb{E}[|\diamond_i|]$ for each $\diamond_i$.

Let $X_A=X\cap A$, $X_B=X\cap B$, $X_C=X\cap C$ to get sets of size $(1\pm n^{-0.2})qn$. Using Lemma~\ref{Lemma_2_random_1_deterministic}, we have matchings $M_1, M_2, M_3, M_4$ covering all, but at most $n^{1-10^{-3}}$ vertices of $A_1\cup B_1\cup X_C$, $A_2\cup X_B\cup C_1$
, $X_A\cup B_2\cup C_2$, and $A_3\cup B_3\cup C_3$ respectively. In total, the number of uncovered vertices does not exceed  $4n^{1-10^{-3}}\leq n^{1-10^{-4}}/10$.
\par Taking a union bound over all rational $q$ with denominator at most $n$, we have shown that with high probability, there exists a matching covering all but $n^{1-10^{-4}}/10$ vertices. The statement for real values of $q$ follows simply by using the property for the closest rational value $q'$ to $q$ with denominator at most $n$. Indeed, leaving out or deleting few elements, we can ensure that the set $X$ has $|X\cap A|, |X\cap B|, |X\cap C|=(1\pm n^{-0.25})q'n$, thereby obtaining a matching that covers all but $n^{1-10^{-4}}/10 + n^{1-10^{-4}}/10\leq n^{1-10^{-4}}$ elements of the original sets $X\cup R$.  
\end{proof}

\subsection{The multiplication hypergraph}\label{sec:multiplicationhypergraph}

We make some clarifications regarding our notation with the multiplication hypergraph $H_G$ of a group (recall that this was defined in Section~\ref{sec:proofoutline}). While referring to the parts of the multiplication hypergraph, we often omit the $A/B/C$ subscripts and think of $A_G,B_G, C_G$ simply as copies of $G$. For example, whenever we have some vertices $v_1, v_2, \dots, v_k\in V(H_G)$, we write $v_1v_2\dots v_k$ to mean the product of the corresponding group elements in $G$ (as opposed to in any of the copies $A_G, B_G, C_G$). Similarly, if $v\in V(H_G)$ and $U$ is a subset of $G$, then we use $v\in U$ to mean that $v$ is an element of $G$ after dropping the $A/B/C$ subscript. For any subset $S\subseteq G$, we use $S_A/S_B/S_C$ to denote the corresponding subsets of $G_A/G_B/G_C$ respectively.
\par A subset $S$ of $V(H_G)$ is called \textbf{balanced}, if $|S\cap G_\diamond|$ does not depend on the value of $\diamond\in \{A,B,C\}$.

\subsection{Basic group theory definitions and results}
Throughout the paper we use $\id$ to denote the identity element of a group $G$. An involution is an element of order exactly  $2$.
Recall that we use $G'$ to denote the commutator subgroup of a group $G$. That is, $G'$ is the subgroup generated by elements of the form $[g,h]:=ghg^{-1}h^{-1}$ where $g,h\in G$. We denote the abelianization of $G$ (quotient of $G$ by $G'$) as $G^\mathrm{ab}$. We sometimes call the elements of $G^\mathrm{ab}$ $G'$-\textbf{cosets}. For $g\in G$, we  use $[g]$ to denote the unique $G'$-coset that $g$ is a member of. When $g\in V(H_G)$, we think of $[g]$ as the $G'$-coset that resides in the same part $G_A/G_B/G_C$ that $g$ resides in. As mentioned in the introduction, whenever we use additive notation together with elements of $G$, all operations take place in $G^\mathrm{ab}$. We do this so that we don't have to use the $[g]$ notation excessively. 

\begin{lemma}\label{Lemma_commutator_characterization}
$g\in G'$ if, and only if, $g$ can be written as $g=g_1\dots g_t$ such that there is a permutation $\sigma$ of $[t]$ with $g_{\sigma(1)}\dots g_{\sigma(t)}=\id$.
\end{lemma}
\begin{proof}
To see the ``only if'' direction, write $g$ as a product of commutators as $$g=[a_1,b_1]\dots [a_t,b_t]=a_1b_1a_1^{-1}b_1^{-1}\dots a_tb_ta_t^{-1}b_t^{-1}.$$ Clearly, the latter product can be permuted as $a_1a_1^{-1}b_1b_1^{-1}\dots a_ta_t^{-1}b_tb_t^{-1}=\id$, as required.
For the ``if'' direction, consider some $g_1\dots g_t$ which rearranges into $g_{\sigma(1)}\dots g_{\sigma(t)}=\id$. Consider the quotient homomorphism $\phi:G\to G/G'$. Then since $G/G'$ is abelian, we have $\phi(g_1\dots g_t)=\phi(g_1)\dots \phi(g_t)=\phi(g_{\sigma(1)})\dots \phi(g_{\sigma(t)})=\phi(g_{\sigma(1)}\dots g_{\sigma(t)})=\phi(\id)=\id$ i.e. $g_1\dots g_t\in ker(\phi)=G'$ as required. 
\end{proof}

\par The following is a well-known property of finite abelian groups. 
\begin{theorem}[Fundamental theorem of finite abelian groups]\label{fundamental} Let $G$ be an abelian group. Then, $G$ is isomorphic to a product of cyclic prime-power order groups.
\end{theorem}

 Given $g\in G$, $s(g)$ denotes the size of the set $\{x\in G\colon x^2=g\}$.
\begin{proposition}[\cite{410789}]\label{prop:overflow}
Let $G$ be group such that there exists some $g\in G$ with $s(g)>(3/4)|G|$. Then, $G$ is an elementary abelian $2$-group and $g=0$. 
\end{proposition}
\begin{proof}
We express our gratitude for all participants of the active discussion that took place on Math Overflow including Emil Jeřábek, Derek Holt, Saúl Rodríguez Martín, and Terry Tao. Here, we reproduce the argument of GH from MO \cite{410789}. Let $g\in G$, and suppose that $|s(g)|>(3/4)|G|$. Fix some $y\in G$, and define $S=\{x\in G\colon x^2=g\}$ and $T=\{x\in S\colon xy\in S\}$. Observe that $|G\setminus T|\leq 2|G\setminus S|$, since $x\notin T$ only if $x\notin S$ or $xy\notin S$, and there are at most $|G\setminus S|$ many $x$ of either type. By assumption $|G\setminus S|<|G|/4$, so $|G\setminus T|< |G|/2$. Consequently, $|T|>|G|/2$. 
\par Observe that for any $x\in T$, we have that $(xy)^2=g=x^2$, and so we have 
$$xyx^{-1}=(xy)(xy)^{-2}(xy)^2x^{-1}=(xy)^{-1}x^2x^{-1}=y^{-1}.$$
\par It is an easy exercise to check that $C=\{x\in G\colon xyx^{-1}=y^{-1}\}$ is a coset of $C(y)$, the centralizer subgroup of $y$. %To see this, take $x,w\in C$ and so $xyx^{-1}=y^{-1}=wyw^{-1}$, and rearranging we have that $(w^{-1}x)y(x^{-1}w)=y$, which implies that $w^{-1}x$ belongs to $C(y)$. Hence, $w$ and $x$ belong to the same coset of $C(y)$. Then, we can fix $C'$ to be a coset of $C(y)$ with $C\subseteq C'$. Now, take some $w\in C'$. $|C|>|G|/2$ since $T\subseteq C$, so in particular, $C$ is non-empty, so we can fix some $x\in C$. As $w,x\in C'$, $w^{-1}x\in C(y)$, so $(w^{-1}x)y(x^{-1}w)=y$, hence $xyx^{-1}=wyw^{-1}$. The left-hand side of the last equation is $y^{-1}$ by definition of $x\in C$, so $wyw^{-1}=y^{-1}$, implying $w\in C$. So $C'=C$, and $C$ is a coset of the desired form. 
\par As $|C|>|G|/2$, and $C$ is a coset, $C=G$ by Lagrange's theorem. This implies that $y=y^{-1}$ for each $y$, hence $G$ is an elementary abelian $2$-group. It follows that $g=0$, as $s(g)=0$ for all $g\neq 0$ in an elementary abelian $2$-group. 
\end{proof}

%\begin{lemma}
%For every $g, h\in G$, the number of solutions to $x^2=h$ and $x^2=ghg^{-1}$ is the same.
%\end{lemma}
%\begin{proof}
%Recall that the conjugation function $\phi:x\to gxg^{-1}$ is a bijection. 
%Notice that $x^2=ghg^{-1}$ rearranges into $(g^{-1}xg)^2=h$. 
%Therefore if $S$ is the set of solutions $x^2=h$ and $T$ is the set of solutions to $x^2=ghg^{-1}$, then we have $T=gSg^{-1}=\phi(S)$, giving that $|T|=|\phi(S)|=|S|$. 
%\end{proof}
%In particular this lemma shows that if $h$ is generic or $\phi$-generic, then so are any conjugates of $h$.

\subsection{Generic elements, and choice of of $a_{\phi}, b_{\phi}, c_{\phi}$} \label{sec:choosingaphi}
The following definition is critical.
\begin{definition}
A group element $g\in G$ is \textbf{generic} if $g\neq \id$ and there are at most $n/10^{9000}$ solutions to $x^2=g$ in $G$. \end{definition}
Let $N(G)$ denote the set of non-generic elements and note that $|N(G)|\leq 10^{9000}$, simply as in a group with $n$ elements, the number of $g$ with more than $k$ square-roots is at most $n/k$. Similarly, we call vertices of $H_G$ generic if the corresponding group element is generic.

As described in Section~\ref{sec:proofoutline}, it is critical to our method to pair up group elements with a fixed sum (which can be viewed as defining a suitable involution $\phi$), where the fixed sum has some desirable properties. The following lemma serves to show that the pairing we desire exists.
 
\begin{lemma}\label{lem:pairingsexist}
For every group $G$, there exist $a_{\phi}, b_{\phi}, c_{\phi}\in G$ such that the following all hold. 
\begin{enumerate}[label=(\alph*)]
\item $a_\phi b_\phi c_\phi=\id$.
\item  There are at most $30$ values of $x\in G$ such that $x^2\in \{a_{\phi},b_\phi,c_\phi\}$. In particular, $a_{\phi}, b_{\phi}, c_{\phi}$ are generic. 
\item There are at most $30|G'|$ values of $x\in G$ such that $x^2\in [a_\phi]\cup [b_\phi]\cup [c_\phi]$. 
\item If $|G'|\leq 10^{-9}n$, then $a_{\phi}, b_{\phi}, c_{\phi}\not\in G'$.%\Alexey{Add proof of this}
\end{enumerate}
\end{lemma}
\begin{proof}
Choose $a_\phi$, and $b_\phi$ uniformly at random and set $c_\phi:=(a_\phi b_\phi)^{-1}$, noting $c_\phi$ is then also sampled uniformly at random. For a random $g\in G$, the probability that $g$ has more than $30$ square-roots is at most $1/30$. By the same argument applied to $G^{\mathrm{ab}}$, the probability that $[g]$ has more than $30$ square-roots in $G^{\mathrm{ab}}$ is also at most $1/30$. The probability that $g\in G'$ is at most $10^{-9}$ if $|G'|\leq 10^{-9}n$. Then, with positive probability, all the conditions are satisfied for $a_\phi$, $b_\phi$, and $c_\phi$ by a union bound.\end{proof}
%Note that Lemma~\ref{lem:pairingsexist}(c) implies in particular that there are at most $30$ self-paired cosets, where a \textbf{self-paired} cosets are those cosets $[x]$ satisfying $[x^2]\in \{[a_\phi],[b_\phi],[c_\phi]\}$.
\begin{definition}
We say that an element $x$ is $\phi$-\textbf{generic} if $x,$  $a_\phi^{\pm 1} x,$ $b_\phi^{\pm 1} x,$ $c_\phi^{\pm 1} x,$  $a_\phi^{\pm 1} b_\phi^{\pm 1} x,$ $b_\phi^{\pm 1} c_{\phi}^{\pm 1} x,$ $c_\phi ^{\pm 1}a_{\phi}^{\pm 1} x$ are all generic. 
\end{definition}
 Notice that the number of elements which aren't $\phi$-generic is  $\leq 30|N(G)|\leq 10^{9010}$. We call a subset $\phi$-generic if all elements of that subset are $\phi$-generic.

For each group $G$, we fix a triple $a_\phi, b_\phi, c_\phi$ with the properties as in Lemma~\ref{lem:pairingsexist}. We call two vertices $v$ and $w$ of $H_G$ coming from the same part a \textbf{pair} if $v\cdot w\in [x_\phi]$ where $x=a,b,c$ depending on whether $v,w\in A,B,C$. 
\par We call a subset $S$ of $V(H_G)$ \textbf{coset-paired} if $S$ can be partitioned into a disjoint union of pairs. We call a coset $[x]$ of $G$, viewed as a subset of $G_A, G_B$ or $G_C$, \textbf{self-paired} if $[x^2]$ is equal to $[a_\phi]$, $[b_\phi]$, or $[c_\phi]$, respectively. Note that Lemma~\ref{lem:pairingsexist}(c) implies in particular that there are at most $30$ self-paired cosets. A subset $S\subseteq V(H_G)$ being coset-paired is equivalent to the following statement: $|S\cap [g]|=|S\cap [x_\phi g^{-1}]|$ for every non-self-paired coset $[g]$ and $|S\cap [g]|$ is even for every self-paired coset $[g]$.

\subsection{Symmetric sets}
Recall that a $p$-\textbf{random} subset of set $S$ is one obtained by sampling each element of $S$ independently with probability $p$. Similarly, we say a collection of random sets $R_1,\ldots, R_k\subseteq S$ is \textbf{disjoint} $p$-\textbf{random} if each element of $S$ belongs to each $R_i$ with probability $p$, and to none of the $R_i$ with probability $1-pk$, and these decisions are made independently for each element of $S$. Considering such disjoint distributions complicates our approach as it makes various gadgets significantly more difficult to find. The reason we are interested in such distributions is the applications we give later on in the paper. Indeed, all applications we give other than the alternative proof of Hall-Paige conjecture and Snevily's conjecture require that we work with such disjoint distributions.
\par In fact, we need to generalise the concept of a disjoint distribution even further so that we can work with random sets $X,Y,Z$ where $X,Y$ and $Z^{-1}=\{z^{-1}\colon z\in Z\}$ are sampled disjointly. This need comes from the applications to sequenceability and $R$-sequenceability. Thankfully, this generalisation does not create many additional combinatorial difficulties. However, we still need the following definitions to state a single theorem that covers all of the applications we want to give.
\par For $g\in G$, define $\hat g:=\{g, g^{-1}\}$, noting that $\hat g$ has size 1 or 2 (depending on whether $g$ has order $\le 2$ or not). For a subset $T\subseteq G$, let $\hat T=\{\hat t: t\in T\}$ and $\bigcup \hat T= \bigcup_{t\in T} \hat t=T\cup T^{-1}$. 
We say that a subset $T\subseteq G$ is \textbf{symmetric} if $T^{-1}=T$ (or equivalently if $T=\bigcup \hat T$). We call a subset $S\subseteq V(H_G)$ symmetric if $S\cap A, S\cap B, S\cap C$ are all symmetric. 
We say that $R$ is a \textbf{symmetric $p$-random subset} of $G$ if $R$ is always symmetric and $\hat R$ is a $p$-random subset of $\hat G$ (or equivalently if $R$ is formed by flipping an independent coin for each $\hat g\in \hat G$ and taking the union of all group elements for which heads comes up).
We say that $R^1, R^2, R^3$ are \textbf{disjoint symmetric $p$-random subsets} of $G$ if additionally the joint distribution of $\hat R^1, \hat R^2, \hat R^3$ is that of disjoint $p$-random subsets of $\hat G$. The following two lemmas are useful as they allow us to jump between these definitions.
\begin{lemma}\label{Lemma_symmetric_sets_inside_random_sets}
Let $R_1, R_2, R_3$ be disjoint $p$-random sets.
Then there are $S_1\subseteq R_1$, $S_2\subseteq R_2$, $S_3\subseteq  R_3$ so that the joint distribution on $S_1, S_2, S_3$ is that of disjoint symmetric $p^2$-random sets. 
\end{lemma}
\begin{proof}
Let $T$ be a $p$-random subset of $G$, independent of $R_1, R_2, R_3$.
For non-involutions, non-identity elements $g$, place $g,g^{-1}$ into $S_i$ whenever $\{g,g^{-1}\}\subseteq R_i$. For  $g$ an involution or the identity, place $g$ into $S_i$ whenever $g\in R_i\cap T$. Notice that for each $g\in G$, this gives $\P(\hat g\subseteq S)=p^2$. Also, these events are mutually independent for different $\hat g,\hat h$ since they depend on different coordinates.
\end{proof}

\begin{lemma}\label{Lemma_intersect_symmetric_sets_with_random}
Let $Q_1, Q_2, Q_3, R$ be random sets with $Q_1, Q_2, Q_3$ being disjoint symmetric $q$-random and $R$ being $p$-random and independent of $Q_1, Q_2, Q_3$.
Then, there are $S_1\subseteq Q_1\cap R$, $S_2\subseteq Q_2\cap R$, $S_3\subseteq Q_3\cap R$ so that the joint distribution on $S_1, S_2, S_3$ is that of disjoint symmetric $p^2q$-random sets. 
\end{lemma}
\begin{proof}
Let $T$ be a $p$-random subset of $G$, independent of $R,Q_1, Q_2, Q_3$.
For $g$ which is not an involution or the identity, place $g,g^{-1}$ into $S_i$ whenever $\{g,g^{-1}\}\subseteq Q_i\cap R$. For  $g$ an involution or the identity, place $g$ into $S_i$ whenever $g\in Q_i\cap R\cap T$. Notice that for each $\hat g\in \hat G$, this gives $\P(\hat g \subseteq S_i)=p^2q$. Also, these events are mutually independent for different $\hat g,\hat h$ since they depend on different coordinates.
\end{proof}

\subsection{Free products}\label{sec:freeproducts}
\par The goal of this section is to prove Lemma~\ref{Lemma_separated_set_random}, which is our main tool to find vertex-disjoint copies of constant sized substructures in multiplication hypergraphs. We first remind the reader of some standard group theoretic terminology. We use $F_k$ to denote the free group on $k$ generators $v_1, \dots, v_k$. 

For a group $G$, we use $G\ast F_k$ to denote the free product of $G$ and $F_k$. We call $v_1, \dots, v_k$ the \textbf{free variables} of the free product $G\ast F_k$. A \textbf{word} is simply an element of $G\ast F_k$.
For $w\in G\ast F_k$, we define the length of $w$ to be the minimum number $\ell$ needed to write $w=x_1x_2\dots x_{\ell}$ for $x_i\in \{v_i, v_{i}^{-1}: i=1, \dots, k\}\cup G$.
 We call a presentation of a word $g_0x_1g_1\dots x_tg_t=w\in G\ast F_k$ \textbf{reduced} if it cannot be made shorter using the group operations, i.e. if it doesn't contain consecutive elements of $G$, doesn't contain $\id$, and doesn't contain consecutive $v_i,v_i^{-1}$ or $v_i^{-1},v_i$ for any $i$. It is a standard property of free products that every $g\in G\ast F_k$ is uniquely expressible as a reduced word.

\begin{lemma}\label{Lemma_free_product_normal_form}
For $w\in G\ast F_k$, there's a unique way of writing $w=g_0x_1g_1\dots x_tg_t$ with $x_i\in \{v_1, \dots, v_k, v_1^{-1}, \dots, v_k^{-1}\}$ and $g_i\in G$ such that we don't have ``$x_i=x_{i+1}^{-1} \text{ and } g_i=\id$'' for any $i\in\{0,1,\ldots,t-1\}$.
\end{lemma}
\begin{proof}
Let $W$ be the set of words which can be written of the form $g_0x_1g_1\dots x_tg_t$ with $x_i\in \{v_1, \dots, v_k, v_1^{-1}, \dots, v_k^{-1}\}$ and $g_i\in G$ such that we don't have ``$x_i=x_{i+1}^{-1} \text{ and } g_i=\id$'' for any $i\in\{0,1,\ldots,t-1\}$. Let $R$ be the set of reduced words of $G\ast F_k$.
Let $f:W\to R$ be defined by mapping $w=g_0x_1g_1\dots x_tg_t$ to the word formed by removing all copies of $\id$ from $w$. Let $g:R\to W$ be defined by mapping $w'=y_1y_2\dots y_k$ to the word formed by inserting $\id$ between any $y_i, y_{i+1}$ which are both in $\{v_1, \dots, v_k, v_1^{-1}, \dots, v_k^{-1}\}$. It is easy to see that $f(g(w'))=w'$ and $g(f(w))=w$ for any $w\in W$ and $w'\in R$ i.e. both functions are bijections. Thus, since every $g\in G\ast F_k$ is uniquely expressible as a reduced word, it is also uniquely expressible as a word in $W$.
\end{proof}

From the above, we get that every $w\in G\ast F_k$ can be written as $w=g_0x_1g_1\dots x_tg_t$ with $x_i\in \{v_1, \dots, v_k, v_1^{-1}, \dots, v_k^{-1}\}$ and $g_i\in G$.
We say that  $w\in G\ast F_k$ is \textbf{linear in} $v_i$ there's a way of writing $w$ like this with precisely one occurrence of $v_i$ (meaning one occurrence $v_i$ or $v_i^{-1}$, but not both). We say that $w$ is \textbf{linear} if all variables occur at most once in $w$, and some variable occurs exactly once in $w$. A useful fact is that for a linear $w$, there's a unique way of writing it as $w=g_0x_1g_1\dots x_tg_t$ (with $x_i\in \{v_1, \dots, v_k, v_1^{-1}, \dots, v_k^{-1}\}$ and $g_i\in G$) such that there's at most one occurrence of each variable. This comes from Lemma~\ref{Lemma_free_product_normal_form}, because in a linear $w=g_0x_1g_1\dots x_tg_t$ it is impossible to have $x_i=x_{i+1}^{-1}, g_i=\id$ (since this would create two occurrences of the free variable $x_i$).

A homomorphism $\pi:G\ast F_k\to G$ is a \textbf{projection} if $\pi(g)=g$ for all $g\in G$.
\begin{lemma}\label{Lemma_free_extension_universal_property}
For each function $f:\{v_1, \dots, v_k\}\to G$, there is precisely one projection $\pi_f:G\ast F_k\to G$ which agrees with $f$ on $\{v_1, \dots, v_k\}$. In particular, there are precisely $n^k$ projections $G\ast F_k\to G$. 
\end{lemma}
\begin{proof}
Let   $f:\{v_1, \dots, v_k\}\to G$ be a function. By the universal property of free groups, there is a unique homomorphism $\phi_f:F_k\to G$ which agrees with $f$ on $\{v_1, \dots, v_k\}$. Let $\id_G:G\to G$ be the identity homomorphism. By the universal property of free products there is a unique homomorphism $\pi_f$ which agrees with $\phi_f$ on $F_k$ and agrees with $\id_G$ on $G$. Such a homomorphism is exactly a projection from $G\ast F_k$ to $G$.

For the ``in particular'' part, note that the number of functions $\{v_1, \dots, v_k\}\to G$ is exactly $n^k$, and so there are this many projections $G\ast F_k\to G$ by the first part. 
\end{proof}
We use $\pi_0:G\ast F_k\to G$ to denote the projection which maps all $w\in F_k$ to $\id$ (i.e. the map coming from Lemma~\ref{Lemma_free_extension_universal_property} via the function $f$ mapping all $v_i$ to $\id$). 
\begin{observation}\label{Observation_number_of_solutions_multiply_square}
For all $g, h\in G$ there are the same number of solutions to $x^2\in [g]$ and $x^2\in [h^2g]$.

In particular, for any $w,w'\in G\ast F_k$, there are the same number of solutions to    $x^2\in [\pi_0(w^{-1}w')]$ an $x^2\in [\pi_0(ww')]$ in $G$).
\end{observation}
\begin{proof}
We have that $x^2\in [g]$ if, and only if $(hx)^2\in [h^2g]$. Therefore, if $S$ is the set of solutions to $x^2\in [g]$, then $hS$ is the set of solutions to $(hx)^2\in [h^2g]$. Since $S$ and $hS$ always have the same size in a group, this gives what we want. 

The ``in particular'' part follows from the above by taking $g=\pi_0(w^{-1}w'), h=\pi_0(w)$, and noting that, since $\pi_0$ is a homomorphism, we have $\pi_0(ww')=h^2g$
\end{proof}

\begin{definition}\label{Definition_separable}
 Let $w,w'\in G\ast F_k$. We say that $w$ and $w'$ are \textbf{strongly separable} if any of the following hold.
\begin{enumerate}[label=(\alph*)]
\item A free variable $v_i$ appears once in one of $w/w'$, and never in the other. %\AM{28 Nov: wouldn't it be better to say the free variable $v_i$ appears \textbf{once} in one of... since we don't like the scenario where $v_i$ appears many times on one side but never on the other?}
\item $w,w'$ are linear and there is a $g\in G$ with $g$ generic so that $w'\in\{gw, g^{-1}w, gw^{-1}, g^{-1}w^{-1}, wg, wg^{-1}, w^{-1}g, w^{-1}g^{-1}\}$.
%The equation $w=w'$ rearranges into $(u)^2=g$ for some $u$ which is linear in a free variable $v_i$, and a generic group element $g\in G$.
%\item The equation $w=w'$ rearranges into $g=\id$ for some generic group element $g\in G$.
\item All of the following hold.\begin{itemize}
    \item $|G'|\leq 10^{-9}n$
    \item  $w$ and $w'$ are linear and have the same free variables (potentially with different signs).
  %  \item  the free variables appear in $w$/$w'$ with opposite signs (meaning $w$ contains $v_i\iff w'$ contains $v_i^{-1}$ and $w$ contains $v_i^{-1}\iff w'$ contains $v_i$)
  \item We either have $\pi_0(ww')\not\in G'$ or some free variable occurs with the same sign in $w,w'$.
    \item We either have $\pi_0(w^{-1}w')\not\in G'$ or some free variable occurs with the opposite sign in $w,w'$.
   \item There are $\leq 90|G'|$ solutions to $x^2\in [\pi_0(ww')]$ (or equivalently by Observation~\ref{Observation_number_of_solutions_multiply_square}, there are $\leq 90|G'|$ solutions to $x^2\in [\pi_0(w^{-1}w')]$ in $G$).
\end{itemize}
\end{enumerate} 
\end{definition}
We remark that strong separability is a symmetric relation. 
\begin{definition}We say that two sets $S,T\subseteq G$ are \textbf{strongly separable} if  every pair of elements $s\in S, t\in T$ are strongly separable.
\end{definition}
The following observation is quite critical, and justifies why considering symmetric disjoint random sets doesn't create additional complications when compared to disjoint random sets.
\begin{observation}\label{Observation_separable_symmetric}
$w,w'$ are strongly separable $\iff w^{-1}, w'$ are strongly separable $\iff$ $\hat w, \hat w'$ are strongly separable.
\end{observation}
\begin{proof} Within this proof, we say separable to mean strongly separable.
For ``$w,w'$ are separable $\iff w^{-1}, w'$ are separable'': Consider the possible cases of Definition~\ref{Definition_separable}. If $w,w'$ are separable by (a), then $w^{-1}, w'$ are also separable by (a) (this is immediate when one considers the ``a free variable $v_i$ appears once in one of $w/w'$, but not both'' version of (a)). If  $w,w'$ are separable by (b), then $w^{-1}, w'$ are also separable by (b) --- this is true because  the set $\{gw, g^{-1}w, gw^{-1}, g^{-1}w^{-1}, wg, wg^{-1}, w^{-1}g, w^{-1}g^{-1}\}$ doesn't change if you replace each ``$w$'' with ``$w^{-1}$''. If  $w,w'$ are separable by (c), then $w^{-1}, w'$ are also separable by (c). (The first bullet point doesn't involve $w,w'$. The second bullet point doesn't change by replacing $w$ with $w^{-1}$ because $w,w^{-1}$ are always linear in the same free variables. The 3rd and 4th bullet points get exchanged when replacing $w$ by $w^{-1}$. The 5th bullet point doesn't change when replacing w by $w^{-1}$ since there is the same number of solutions to $x^{2}\in [\pi_0(ww')]$ and $x^{2}\in [\pi_0(w^{-1}w')]$ by Observation~\ref{Observation_number_of_solutions_multiply_square}).

The direction ``$\hat w, \hat w'$ are separable $\implies$ $w,w'$ are separable'' is immediate from the definition of ``$S,T$ are separable''. For `` $w,  w'$ are separable $\implies$ $\hat w,\hat w'$ are separable'', note that once we know that $w,w'$ are separable, we also know that the pairs $(w^{-1}, w')$, $(w, (w')^{-1})$, $(w^{-1}, (w')^{-1})$ are separable (all coming from ``$w,w'$ are separable $\iff w^{-1}, w'$ are separable''). This gives that $\hat w, \hat w'$ are separable.
\end{proof}

\begin{observation}\label{Observation_partition_S_by_free_variables}
Let $S\subseteq G\ast F_k$ be a set of linear elements. For each $T\subseteq \{v_1, \dots, v_k\}$, let $S_T$ be the set of $w\in W$ such that the free variables in $w$ are exactly the set $T$. Then the sets $S_T, S_{T'}$ are strongly separable for distinct $T,T'$.
\end{observation}
\begin{proof}
If $T, T'$ are distinct, then there's some  $v_i\in T\Delta T'$ say $v_i\in T\setminus T'$. For any $w\in S_T, w'\in S_{T'}$, we have that $v_i$ appears in $w$ (just once by linearity) but not $w'$, and so part (a) of the definition of ``strongly separable'' applies. 
\end{proof}
\begin{definition}
Let $w,w'\in G\ast F_k$. We say that $w$ and $w'$ are \textbf{weakly separable} if either they are strongly separable, or they satisfy the following property.
\begin{itemize}
\item[$(b')$] For some non-identity element $g$, the equation $w=w'$ rearranges into $\id=g$, meaning that $w^{-1}w'$ is conjugate to an element of $G$.
\end{itemize}
\end{definition}

We remark that the key difference between strong and weak separability comes from the property $(b')$ not necessarily holding when $w$ is replaced with $w^{-1}$, i.e. Observation~\ref{Observation_separable_symmetric} fails for weak separability. This will not be an issue while we search for gadgets, as we rely on weak separation only to separate elements coming from the same set $A_G,B_G,C_G$. 

\begin{definition}
Let $S\subseteq G\ast F_k$. We say that a homomorphism $\phi:G\ast F_k\to G$, \textbf{separates} $S$ if for every weakly-separable $w,w'\in S$ we have $\phi(w)\neq \phi(w')$.
\end{definition}
 
Recall that $(G\ast F_k)'$ denotes the commutator subgroup of $(G\ast F_k)$.
\begin{lemma}\label{Lemma_separable_partiv_equivalent_form}
For a group $G$, let $w,w'\in G\ast F_k$ satisfy part (c) of the definition of ``strongly separable''. Then $w=w'$   rearranges to $u^2=\pi_0(w^{-1}w')y$ for some  $y\in (G\ast F_k)'$ and some  $u\in F_k$ which is either linear or equals $\id$. Additionally, $u=\id$ only if the free variables occur with the same signs in $w,w'$.
\end{lemma}
\begin{proof}
Let $w=g_0x_1g_1\dots x_tg_t$ and $w'=g_0'x_1'g_1'\dots x_{t'}'g_{t'}'$ (with $x_i, x_i'\in \{v_1, \dots, v_k, v_1^{-1}, \dots, v_k^{-1}\}$ and $g_i, g_i'\in G$). Note that the assumption ``$w$ and $w'$ are linear and have the same free variables'' implies that $t=t'$, and that there is a permutation $\sigma$ of $[t]$ so that $x_i'\in \{x_{\sigma(i)}, x_{\sigma(i)}^{-1}\}$ for $i=1, \dots, t$. Partition $[t]=I^+\cup I^-$ with $I^+=\{i: x_i'=x_{\sigma(i)}\}$ and $I^-=\{i: x_i'=x_{\sigma(i)}^{-1}\}$. Set $u=\prod_{i\in I^-}(x_i')^{-1}$, noting that $u$ is in $F_k$, and that $u=\id \iff I^-=\emptyset\iff$ the free variables occur with the same signs in $w,w'$. This also implies that if $u\neq e$, then $u$ is linear.
Note that $w=w'$ rearranges to $u^2=\pi_0(w^{-1}w')y$ where $y=\pi_0(w^{-1}w')^{-1}u^2w'w^{-1}$. Notice that 
\begin{align*}
\pi_0(&w^{-1}w')^{-1}u^2w'w^{-1}=\pi_0(w')^{-1}\pi_0(w)u^2w'w^{-1}
\\
&=
(g_t'^{-1}\dots g_1'^{-1}g_0'^{-1})
(g_0\dots g_{t-1}g_{t})
(\prod_{i\in I^-} (x_i')^{-1})
(\prod_{i\in I^-} (x_i')^{-1})
(g_0'x_1'g_1'\dots x_{t}'g_{t}')
(g_t^{-1}x_t^{-1}g_{t-1}^{-1}\dots x_1^{-1}g_0^{-1})
\end{align*}
%$$\pi_0(ww')^{-1}w'w=\pi_0(w')^{-1}\pi_0(w)^{-1}w'w =(g_t'^{-1}\dots g_1'^{-1}g_0'^{-1})(g_t^{-1}\dots g_1^{-1}g_0^{-1})(g_0'x_1'g_1'\dots x_{t}'g_{t}')(g_0x_1g_1\dots x_tg_t).$$ 
Notice that the above product can be permuted into the identity. Indeed each $g_i'^{-1}$ in the 1st bracket cancels with $g_i'$ in the 5th bracket, each $g_i$ in the 2nd bracket cancels with  $g_i^{-1}$  6th bracket, for $i\in I^-$ each $(x_i')^{-1}$ in the 3rd bracket cancels with $x_i'$ in the 5th bracket, each $(x_{i}')^{-1}$ in the 4th bracket cancels with $(x_{\sigma(i)})^{-1}=x_{i}'$ in the 6th bracket, and for $i\in I^+$ each $x_i'$ in the 5th bracket cancels with $x_{\sigma(i)}^{-1}=(x_i')^{-1}$ in the 6th bracket. Thus,  by Lemma~\ref{Lemma_commutator_characterization}, $t\in (G\ast F_k)'$.
\end{proof}

\begin{lemma}\label{Lemma_count_projections_fixing_one_image}
Let $w\in G\ast F_k$ be linear in some free variable $v_i$ and let $g\in G$. Then there are exactly $n^{k-1}$ projections $\pi:G\ast F_k\to G$ having $\pi(w)=g$.
\end{lemma}
\begin{proof}
Without loss of generality, suppose $i=k$. By Lemma~\ref{Lemma_free_extension_universal_property} there are exactly $n^{k-1}$ projections $\pi:G\ast F_{k-1}\to G$. We will show that for every such projection, there is a unique projection $\pi':G\ast F_k\to G$ that agrees with $\pi$ on $G\ast F_{k-1}$ and additionally has $\pi'(w)=g$. To see this, note that the equation $w=g$ in the group $G\ast F_k$ rearranges into $v_k=h$ where $h\in G\ast F_k$ such that the free variable $v_k$ doesn't occur in $h$ (this is possible because $w$ is linear in $v_k$). This shows that for any projection $\pi'$, the equation $\pi'(w)=g$ is equivalent to $\pi'(v_k)=\pi'(h)$ (using that for any $g\in G$ we have $\pi'(g)=g$ for any projection $\pi'$).
Since $h\in G\ast F_{k-1}$, the image $\pi(h)\in G$ is defined. Thus a projection $\pi'$ agrees with $\pi$ on $F_{k-1}\ast G$ and also has $\pi'(w)=g$ if and only if $\pi'(v_1)=\pi(v_i), \dots, \pi'(v_{k-1})=\pi(v_{k-1})$, and also $\pi'(v_k)=\pi(h)$. By Lemma~\ref{Lemma_free_extension_universal_property}, there is a unique projection satisfying this, as required. 
\end{proof}

\begin{lemma}\label{Lemma_upper_bound_set_of_linear_words}
Let $S\subseteq G\ast F_k$ be a set of elements which are each linear in at least one variable, and let $U\subseteq G$. Then the number of projections $\pi:G\ast F_k \to G$  for which $\pi(S)$ intersects $U$ is $\leq |S||U|n^{k-1}$.
\end{lemma}
\begin{proof}
For each $w\in S$ and $u\in U$, by Lemma~\ref{Lemma_count_projections_fixing_one_image}, there are $n^{k-1}$ projections $\pi:G\ast F_k \to G$ with $\pi(w)=u$. Thus, summing over all $w,u$, there are at most  $\leq |S||U|n^{k-1}$ projections with $\pi(S)$ intersecting $U$. 
\end{proof}

\begin{lemma}\label{Lemma_lower_bound_separated_set}
Let $n$ be sufficiently large.
Let $S\subseteq G\ast F_k$ be a  set of size $\leq 1000$. Then there are at most $0.1n^k$ projections $\pi:G\ast F_k \to G$ which do not separate $S$. If $|G'|>10^{-9}n$, then this can be improved to ``at most $10^{-8900}n^{k}$ projections''. 
%least $0.9n^{k}$  projections $\pi:G\ast F_k \to G$ which separate $S$. If $|G'|> n/\log^{10^{20}}n$, then this can be improved to ``at least $(1-\log^{10^{29} n)n^{k}$ projections''. 
\end{lemma}
\begin{proof}
The total number of projections $\pi:G\ast F_k \to G$  is $n^{k}$ by Lemma~\ref{Lemma_free_extension_universal_property}. Consider two weakly separable words $w, w'\in S$. We will count the number of projections for which $\pi(w)=\pi(w')$. There are four cases, depending on which part of the definition of weakly separable applies to $w,w'$. 
\begin{enumerate}[label=(\alph*)]
\item[$(a)$]  Note that the equation $\pi(w)=\pi(w')$ can be rearranged into $\pi(w^{-1}w')=\id$. Since in (a), $w^{-1}w'$ is linear in some variable,  Lemma~\ref{Lemma_count_projections_fixing_one_image} tells us that there are $n^{k-1}\leq n^k/10^{9000}$ projections $\pi:G\ast F_k \to G$ for which $\pi(w^{-1}w')=\id$. 

\item[$(b')$] Since $w=w'$ rearranges into $e=g$, the equation $\pi(w)=\pi(w')$ rearranges into $\pi(e)=\pi(g)$. But this is impossible for a projection $\pi$ (since the definition of ``projection'' gives that $\pi(\id)=\id$ and $\pi(g)=g$). Thus there are zero projections with $\pi(w)=\pi(w')$ in this case. 

\item[$(b)$] If we are not in some case covered by the previous bullet point, then the equation $w=w'$ rearranges into $(w)^2=g$ for a generic group element $g\in G$. Let $T_g$ be the set of solutions to $x^2=g$ to get a set of size $\leq n/10^{9000}$, using the definition of a generic group element. For a projection $\pi$ to have $\pi(w)=\pi(w')$, it must have (using that $\pi$ is a homomorphism) $\pi(w)^2=\pi(g)=g$ and so $\pi(w)\in T_g$. By Lemma~\ref{Lemma_upper_bound_set_of_linear_words}, there are $|T_g|n^{k-1}\leq n^{k}/10^{9000}$ projections with $\pi(w)\in T_g$.

\item[$(c)$] Using Lemma~\ref{Lemma_separable_partiv_equivalent_form}, $w=w'$ rearranges into $u^2=\pi_0(w^{-1}w')k$ where $u\in F_k$ is either linear or $u=\id$,  and $k\in (G\ast F_k)'$.  
If $u=\id$, then we know that free variables occur with the same signs in $w,w'$, which tells us  that $\pi_0(w^{-1}w')\not\in G'$. This gives that $[\pi_0(w^{-1}w')k]\neq G'$, and so $\id\neq \pi_0(w^{-1}w')k=u^2=e$, a contradiction. Thus in this case, there are no projections with $\pi(w)=\pi(w')$. So we can assume that $u\neq \id$ which implies by Lemma~\ref{Lemma_separable_partiv_equivalent_form} that $u$ is linear.

Let $T$ be set of solutions to $x^2\in [\pi_0(w^{-1}w')]$ in $G$. Since we are in (c), we have $|T|\leq 90|G'|\leq 90\cdot 10^{-9}n \leq  10^{-7}n$. For $\pi(w)=\pi(w')$ to hold, we must have (using that $\pi$ is a homomorphism)
$\pi(u^2)=\pi(\pi_0(w^{-1}w')k)$, which implies (using that $\pi$ is a projection)
$\pi(u)^2=\pi_0(w^{-1}w')\pi(k)$. Since $\pi_0(w^{-1}w')\pi(k)\in [\pi_0(w^{-1}w')]$ for any projection $\pi$ (we have $\phi(H')\subseteq G'$ for any group homomorphism $\phi:H\to G$, since if  $x$ is a commutator in $H$, then  $\phi(x)$ is a commutator in $G$), this would imply that $\pi(w)\in T$. By Lemma~\ref{Lemma_upper_bound_set_of_linear_words}, there are  $|T|n^{k-1}\leq 10^{-7} n^{k}$ projections with $\pi(w)\in T$ as required. 
\end{enumerate}
There are at most $\binom {|S|}2\leq \binom{1000}2\leq 10^6$  pairs of weakly separable $w,w'\in S$, and for each of them there are $\leq n^k/10^{7}$  projections with $w=w'$. Thus in total there are $\leq 10^6n^k/10^7= 0.1n^k$ projections which don't separate $S$. This implies the lemma when $|G'|\leq  10^{-9}n$. When $|G'|> 10^{-9}n$, note that case (c) can't occur, so we actually have $\leq n^k/10^{9000}$  projections with $w=w'$ for each separable $w,w'$. This gives a total of $\leq 2\cdot10^6n^k/10^{9000}\leq n^k/10^{8900}$ projections which don't separate $S$.
%0.1n^k/(\log^{10^{20}}n)^{1000}\leq  0.1(n/\log^{10^{20}}n)^k\leq 0.1|G'|^k=0.1(n')^k$ projections which don't separate $S$.
%Thus there remain $n^k-10^6n^k/10^{7}\geq 0.9n^k$ projections which separate $S$. If $|G'|>n/\log^{10^{20}}n$, then case (d) can't occur, so we actually have $\leq n^k/\log^{10^{30}}n$  projections with $w=w'$ for each separable $w,w'$. Thus the bound improves to $n^k-10^6n^k/\log^{10^{20}}n\geq (1-\log^{10^{29}} n)n^k$.
\end{proof}

\begin{lemma}\label{Lemma_disjoint_separating_projections}
Let $n$ be sufficiently large, $k\leq 200$,  and $S\subseteq G\ast F_k$  a  set of  $\leq 1000$  elements which are linear in at least one variable.
There are projections  $\pi_1, \dots, \pi_{10^{-3000}n}$  which separate $S$ and have $\pi_1(S), \dots, \pi_{10^{-3000}n}(S)$  disjoint.
If $|G'|>10^{-9}n$, then we can additionally ensure that $\pi_j(v_i)\in G'$ for all free variables $v_i$ and projections $\pi_j$.
\end{lemma}
\begin{proof}
For $|G'|\leq 10^{-9}n$ say that a projection $\pi$ is good if it separates $S$. For $|G'|> 10^{-9}n$ say that a projection $\pi$ is good if it separates $S$ and has $\pi(v_i)\in G'$ for all $v_i\in F_k$. Our task is to find $10^{-3000}n$ good projections $\pi_i(S)$, which have $\pi_i(S)$ disjoint for different $i$. 
Consider a maximal family $\pi_1, \dots, \pi_{t}$ of good projections. which have  $\pi_1(S), \dots, \pi_t(S)$  disjoint. Let $T=\pi_1(S)\cup \dots\cup \pi_t(S)$. By maximality, we have that all good projections $\pi$ have $\pi(S)\cap T\neq \emptyset$. Lemma~\ref{Lemma_upper_bound_set_of_linear_words} tells us that the number of  projections with  $\pi(S)\cap T\neq \emptyset$ is $\leq|T|n^{k-1}\leq t|S|n^{k-1}\leq 1000tn^{k-1}$. Thus we established that there are $\leq  1000tn^{k-1}$ good projections.

Suppose $|G'|\leq  10^{-9}n$.
Lemma~\ref{Lemma_lower_bound_separated_set} tells us that there are $\leq 0.1n^k$ bad projections, and hence $\geq n^k-0.1n^k=0.9n^k$ good projections. Thus $t\geq 0.9n/1000\geq 10^{-3000}n$.

Now suppose $|G'|> 10^{-9}n$:
There are $|G'|^k$ projections with $\pi(v_i)\in G'$ for all $v_i\in F_k$ (using Lemma~\ref{Lemma_free_extension_universal_property}). By Lemma~\ref{Lemma_lower_bound_separated_set}, there are $\leq n^k/10^{8900}\leq 0.1(10^{-9}n)^k\leq 0.1|G'|^k$ projections that don't separate $S$. Hence there remain $\geq 0.9|G'|^k\geq 0.9\cdot 10^{-9k}n^k$ projections that both separate $S$ and have all $\pi(v_i)\in G'$ i.e. there are $\geq 0.9\cdot 10^{-9k}n^k$ good projections. Combining with ``there are $\leq  1000tn^{k-1}$ good projections'', this gives $t\geq 0.9\cdot 10^{-9k}n/1000\geq 10^{-3000}n.$
\end{proof}

\begin{lemma}\label{Lemma_separated_set_random}
Let $p \geq n^{-1/700}$. 
Let  $R_A, R_B, R_C$ be disjoint $p$-random symmetric subsets of $G$ and set $R=R_A\cup R_B\cup R_C$. With  high probability, the following holds:

Let $k\leq 200$,  $S\subseteq G\ast F_k$ a set of $\leq 600$  elements of length  $\leq 200$ which are each linear in at least one variable, and  $U\subseteq G$ with $|U|\leq p^{800}n/10^{4000}$. Then there is a projection $\pi: G\ast F_k \to G$ which separates $S$ and has $\pi(S)\subseteq R\setminus U$. Moreover:
\begin{itemize}
\item For any  $S_A, S_B, S_C\subseteq S$, with $S_A$, $S_B$, $S_C$ being pairwise strongly separable sets,
we can ensure  $\pi(S_A)\subseteq R_A, \pi(S_B)\subseteq R_B, \pi(S_C)\subseteq R_C$.
\item  If $|G'|> 10^{-9}n$, then we can additionally ensure that $\pi(v_i)\in G'$ for all free variables $v_i$.
\end{itemize}
\end{lemma}
\begin{proof}
We can assume that $n$ is sufficiently large (otherwise the lemma is vacuous since ``with high probability'' wouldn't mean anything). %For $g\in G$, define $\hat g:=\{g, g^{-1}\}$, noting that $\hat g$ has size 1 or 2 (depending on whether $g$ is an involution or not). For a subset $T\subseteq G$, let $\hat T=\{\hat t: t\in T\}$ and $\bigcup \hat T= \bigcup_{t\in T} \hat T=T\cup T^{-1}$. 
First fix $k\leq 200$, and  some  set $S\subseteq G$, and $S_A, S_B, S_C\subseteq S$ as in the lemma. Note $|\bigcup \hat S|\leq 2|S|\leq 1200$.
Fixing $m:=n/10^{3000}$, apply Lemma~\ref{Lemma_disjoint_separating_projections} to get projections $\pi_1, \dots, \pi_{m}$  which separate $\bigcup \hat S$ and have $\pi_1(\bigcup \hat S), \dots, \pi_m(\bigcup \hat S)$  disjoint (and additionally, when $|G'|> 10^{-9}n$, having  that $\pi_j(v_i)\in G'$ for all free variables $v_i$ and projections $\pi_j$).

Note that for each $i$ we have $\pi_i(\hat S_A), \pi_i(\hat S_B), \pi_i(\hat S_C)$ pairwise disjoint. Indeed if say $\pi_i(\hat S_A)\cap \pi_i(\hat S_B)\neq \emptyset$, then there would be some $a\in S_A, b\in S_B$ with $\pi_i(\hat a)=\pi_i(\hat b)$. We know that $a,b$ separable, which  implies that $\hat a, \hat b$ are separable (by Observation~\ref{Observation_separable_symmetric}). But $\pi_i$ separates $\bigcup \hat S$, which shows that $\pi(\hat a), \pi(\hat b)$ are disjoint. 

For each $i$, $\hat g\in \pi_i(\hat S)$, let $$ABC(\hat g)=\begin{cases}
A \text{ if $\hat g\in \pi_i(\hat S_A)$}\\
B \text{ if $\hat g\in \pi_i(\hat S_B)$}\\
C \text{ if $\hat g\in \pi_i(\hat S_C)$}\\
C \text{ otherwise}
\end{cases}.$$
Note that this is well defined by the previous paragraph. 
%Also define: 
%\begin{align*}
%R_A'=R_A\setminus (R_B\cup R_B^{-1}\cup R_C\cup R_C^{-1}), 
%R_B'=R_B\setminus (R_A\cup R_A^{-1}\cup R_C\cup R_C^{-1}), 
%R_C'=R_C\setminus (R_B\cup R_B^{-1}\cup R_A\cup R_A^{-1})
%\end{align*} 
%Note $R=R_A'\cup R_B'\cup R_C'$.
For each $i, \hat g\in \pi_i(S)$, let $E^{\hat g}_i$ be the event``$\hat g\subseteq R_{ABC(\hat g)}$''.  
%If $|\hat g|=2$, we have  $P(E^g_i)=p^2(1-p)^4$, while if $|\hat g|=1$,
 Note that we  have $P(E^{\hat g}_i)=p$. 
Note that $E^{\hat g}_i, E^{\hat h}_j$ are independent for $\hat g\neq \hat h$ ($E^{\hat g}_i$ depends only on the coordinate $\hat g$, $E^{\hat h}_j$ depends only on $\hat h$). 
Let $E_i=\bigcap_{\hat g\in \pi_i(\hat S)}E_{i}^{\hat g}$. We have $\P(E_i)=\prod _{\hat g\in \pi_i(\hat S)}\P(E_i^{\hat g})\geq p^{|S|}$. For all $i=1, \dots, m$, the events $E_i$ are independent (since $E_i, E_j$ depend on the coordinates in $\pi_i(\bigcup \hat S), \pi_j(\bigcup \hat S)$ respectively. These are disjoint for $i\neq j$). 
By linearity of expectation, the expected number of indices $i$ for which $E_i$ occurs is at least $p^{|S|}m$.
By Chernoff's Bound, there are at least $p^{|S|}m/2$ indices for which $E_i$ occurs with probability 
$$\geq 1-2e^{\frac16p^{|S|}m}\geq 1-2e^{\frac16(n^{-1/700})^{600}n/10^{3000}}\geq   1-2e^{-n^{1/9}/10^{6000}}.$$  
Since $|U|\leq p^{800}n/10^{4000}\leq p^{600}n/(2\cdot 10^{3000})=  p^{|S|}m/2$, there is at least one such index with $\pi_i(S)\cap U=\emptyset$. This projection $\pi_i$ satisfies the lemma.

To get the lemma for all possible families $\{k,S,S_A, S_B, S_C\}$, notice that there are  $o(e^{n^{1/7}/36000})$ such families. Indeed, there are $200$ choices for $k$ and for each $k$, there are $\leq 201(400n)^{201}$ length $\leq 200$ elements $w\in G\ast F_{k}$. There are $\leq (201(400n)^{201})^{600}$ sets of $\leq 600$ such words. Hence, there are $\leq ((201(400n)^{201})^{600})^4=o(e^{n^{1/37}/36000})$ families of $4$-tuples of such subsets. So we can take a union bound over all such families.
\end{proof}

\begin{corollary}\label{Corollary_separated_set_random}
Let $p \geq 3n^{-1/1400}$. 
Let  $R$ be random set which is either $p$-random or symmetric $p$-random. With  high probability, the following holds:

Let $k\leq 200$,  $S\subseteq G\ast F_k$ a set of $\leq 600$  elements of length  $\leq 200$ which are each linear in at least one variable, and  $U\subseteq G$ with $|U|\leq p^{1600}n/10^{4002}$. Then there is a projection $\pi: G\ast F_k \to G$ which separates $S$ and has $\pi(S)\subseteq R\setminus U$.
\end{corollary}
\begin{proof}
Using Lemma~\ref{Lemma_symmetric_sets_inside_random_sets}, we can choose random sets $R_A, R_B, R_C\subseteq R$ such that the joint distribution on $R_A, R_B, R_C$ is that of disjoint symmetric $p^2/9$-random sets. Now, the lemma follows from Lemma~\ref{Lemma_separated_set_random}.
\end{proof}

We end with a simple application of the above lemma for later use.
\begin{lemma}\label{Lemma_find_paired_vertices}
Let $p\geq n^{-1/700}$. 
Let $G$ be a group $R$ a symmetric $p$-random subset of $G$. With high probability, the following holds.

  For any generic $x_{\phi}\in G$   and $U\subseteq V(H_G)$ with $|U|\leq p^{800}n/10^{4001}$, there are distinct and $\phi$-generic $x,x'\in R\setminus U$ with $xx'=x_\phi$. 
\end{lemma}
\begin{proof}
With high probability, Lemma~\ref{Lemma_separated_set_random} applies. Let $x_{\phi}$, $U$ be as in the lemma.  Add all non-$\phi$-generic elements to $U$ in order to get a set $U'$ with $|U'|\leq p^{800}n/10^{4000}$.
Let $x_{\phi}\in G$ and consider the set $\{v_1, x_\phi v_1^{-1}\}\subseteq G\ast F_1$. 
Note that $v_1, x_\phi v_1^{-1}$ are separable (by part (b) of the definition), and so
Lemma~\ref{Lemma_separated_set_random} gives a projection $\pi:G\ast F_1\to G$ with $\pi(v_1), \pi(x_\phi v_1^{-1})$ distinct and contained in $R\setminus U'$ setting $x=\pi(x_\phi v_1^{-1}), x'=\pi(v_1)$ gives the lemma.
 \end{proof}

\section{The main theorem and its variants}\label{sec:statements}
This section is devoted to stating the main technical result of the paper, and collecting various consequences thereof (including Theorem~\ref{thm:mainintro}) which will be more convenient to use for various applications we give. For the applications, we need variants of Theorem~\ref{thm:mainintro}, where the probability distributions on $R^1, R^2, R^3$ are different from the one given (e.g. it is sometimes useful to sample $R^1, R^2, R^3$ disjointly rather than independently).
We begin by stating a more technical version of Theorem~\ref{thm:mainintro} which covers all the different distributions of these sets that we might need.  The following definition is the most general case.
\begin{definition}
Let $G$ be a group and let $R^1, R^2, R^3$ be random subsets of $G$. We say that $R^1, R^2, R^3$ are $q$-\textbf{slightly-independent}, if there are random subsets
 $Q^1\subseteq R^1$, $Q^2\subseteq R^2$, $Q^3\subseteq R^3$ such that the joint distribution on $Q^1, Q^2, Q^3$ is that of disjoint symmetric, $q$-random subsets of $G$.
\end{definition}

Note that $q$-slightly-independent sets $R^1$ and $R^2$ do not necessarily have the same size (even in expectation). The following observation about this definition is useful.
\begin{observation}\label{obs:slightly_independent_symmetry}
If $R^1, R^2, R^3$ are $q$-slightly-independent, then so are $R^1, (R^2)^{-1}, R^3$.
\end{observation}
\begin{proof}
We have  $Q^1\subseteq R^1$, $Q^2\subseteq R^2$, $Q^3\subseteq R^3$ such that the joint distribution on $Q^1, Q^2, Q^3$ is that of disjoint, symmetric, $q$-random subsets of $G$. Note that since $Q^2$ is symmetric, $Q^2=(Q^2)^{-1}$. Also, $(Q^2)^{-1}\subseteq (R^2)^{-1}$. Thus, $Q^1, Q^2, Q^3$ witness $R^1, (R^2)^{-1}, R^3$ being $q$-slightly-independent.
\end{proof}

The following is the strongest version of the main theorem that we prove in this paper. It is proved in Section~\ref{sec:mainproof}.
\begin{theorem}\label{thm:main_strongest} Let $q\geq n^{-1/10^{101}}$. Let $G$ be a group of order $n$. Let $R^1,R^2,R^3\subseteq G$ be $p$-random, $q$-slightly-independent  subsets. Then, with high probability, the following holds. 
\par Let $X,Y,Z$ be equal-sized subsets of $G_A$, $G_B$, and $G_C$ respectively, satisfying the following properties.
\begin{itemize}
    \item $|(R^1_A\cup R^2_B\cup R^3_C) \mathbin{\triangle} (X\cup Y\cup Z) |\leq q^{10^{17}}n/\log(n)^{10^{17}}$
    \item $\sum X+\sum Y + \sum Z = 0$  (in $G^{\mathrm{ab}}$)
    \item $\id_G\notin X\cup Y\cup Z$	
\end{itemize}
 Then, $H_G[X,Y,Z]$ contains a perfect matching. 
\end{theorem}
We now state and prove a number of consequences of this theorem, where the distributions on  $R^1,R^2,R^3$ are more natural. 
Firstly, the following theorem easily implies Theorem~\ref{thm:mainintro} (we will prove this formally later on in the section).
\begin{theorem}\label{thm:maintheoremnondisjoint} Let $p\geq n^{-1/10^{102}}$. Let $G$ be a group of order $n$. Let $R^1,R^2,R^3\subseteq G$ be $p$-random subsets, sampled independently. Then, with high probability, the following holds. 
\par Let $X,Y,Z$ be equal-sized subsets of $G_A$, $G_B$, and $G_C$ respectively, satisfying the following properties.
\begin{itemize}
    \item $|(R^1_A\cup R^2_B\cup R^3_C) \mathbin{\triangle} (X\cup Y\cup Z) |\leq p^{10^{18}}n/\log(n)^{10^{18}}$
    \item $\sum X+\sum Y + \sum Z = 0$  (in $G^{\mathrm{ab}}$)
\end{itemize}
Then, $H_G[X,Y,Z]$ contains a perfect matching. 
\end{theorem}

When  $R^1,R^2,R^3$ are sampled disjointly, then the statement of the theorem needs to change slightly. In this case, when the group is $\mathbb{Z}_2^k$, then it is impossible to cover $\id$ with a hyperedge contained in $R^1\cup R^2\cup R^3$. Thus to get a matching, we need to additionally have the condition ``$\id_G\notin X\cup Y\cup Z$ if $G\cong \mathbb Z_2^{k}$'' (Proposition~\ref{prop:overflow} shows that this is the only local obstruction of this type).
\begin{theorem}\label{thm:maintheorem_disjoint} Let $p\geq n^{-1/10^{102}}$. Let $G$ be a group of order $n$. Let $R^1,R^2,R^3\subseteq G$ be $p$-random disjoint subsets. Then, with high probability, the following holds. 
\par Let $X,Y,Z$ be equal-sized subsets of $G_A$, $G_B$, and $G_C$ respectively, satisfying the following properties.
\begin{itemize}
    \item $|(R^1_A\cup R^2_B\cup R^3_C) \mathbin{\triangle} (X\cup Y\cup Z) |\leq p^{10^{18}}n/\log(n)^{10^{18}}$
    \item $\sum X+\sum Y + \sum Z = 0$  (in $G^{\mathrm{ab}}$)
    \item $\id_G\notin X\cup Y\cup Z$ if $G\cong \mathbb Z_2^{k}$.	
\end{itemize}
 Then, $H_G[X,Y,Z]$ contains a perfect matching. 
\end{theorem}

Finally, we have two versions of the theorem which are intermediate between the previous two. First one which asks $R^1,R^2$ to be sampled disjointly, and $R^3$ to be sampled independently.
\begin{theorem}\label{thm:maintheoremsemidisjoint} Let $p\geq n^{-1/10^{102}}$. Let $G$ be a group of order $n$. Let $R^1,R^2\subseteq G$ be disjoint $p$-random subsets, and let $R^3\subseteq G$ be a $p$-random subset, sampled independently with $R^1$ and $R^2$. Then, with high probability, the following holds. 
\par Let $X,Y,Z$ be subsets of $G_A$, $G_B$, and $G_C$ be equal sized subsets satisfying the following properties.
\begin{itemize}
    \item $|(R^1_A\cup R^2_B\cup R^3_C) \triangle (X\cup Y\cup Z) |\leq p^{10^{18}}n/\log(n)^{10^{18}}$
    \item $\sum X+\sum Y + \sum Z = 0$ (in the abelianization of $G$)
    \item If $G\cong(\mathbb{Z}_2)^k$ for some $k$, suppose that $\id \notin Z$.
\end{itemize}
Then, $H_G[X,Y,Z]$ contains a perfect matching. 
\end{theorem}

The next theorem is almost the same as the previous one, with the difference that we sample  $(R^1)^{-1},R^2$ disjointly (as opposed to $R^1,R^2$).
\begin{theorem}\label{thm:maintheorem_semidisjointdivision} Let $p\geq n^{-1/10^{102}}$. Let $G$ be a group of order $n$. Let $(R^1)^{-1},R^2\subseteq G$ be disjoint $p$-random subsets, and let $R^3\subseteq G$ be a $p$-random subset, sampled independently with $R^1$ and $R^2$. Then, with high probability, the following holds. 
\par Let $X,Y,Z$ be subsets of $G_A$, $G_B$, and $G_C$ be equal sized subsets satisfying the following properties.
\begin{itemize}
    \item $|(R^1_A\cup R^2_B\cup R^3_C) \triangle (X\cup Y\cup Z) |\leq p^{10^{18}}n/\log(n)^{10^{18}}$
    \item $\sum X+\sum Y + \sum Z=0$ (in the abelianization of $G$)
   \item If $G\cong(\mathbb{Z}_2)^k$ for some $k$, suppose that $\id \notin Z$.
\end{itemize}
Then, $H_G[X,Y,Z]$ contains a perfect matching. 
\end{theorem}

These theorems all have almost the same proof. The idea is to first show that the distribution of $R_1, R_2, R_3$ is $q$-slightly-independent (and so Theorem~\ref{thm:main_strongest}) applies. When the sets $X,Y,Z$ don't contain the identity, then this already gives what we want. When $X,Y,Z$ do  contain the identity, then we first find a small matching covering all copies of the identity in  $X,Y,Z$, and then find another matching covering all remaining vertices using Theorem~\ref{thm:main_strongest}.
\begin{proof}[Proof of Theorem~\ref{thm:maintheorem_disjoint}]
Note that $R^1, R^2, R^3$ are $q$-slightly-independent for $q=p^2/9$. To see this, consider disjoint $1/3$-random sets $T^1, T^2, T^3\subseteq G$, chosen independently of $R^1, R^2, R^3$. We have that $R^1\cap T^1, R^2\cap T^2, R^3\cap T^3$ are disjoint $p/3$-random subsets of $G$. Indeed for every $g\in G$, we have $\P(g\in R^i\cap T^i)=\P(g\in R^i)\P(g\in T^i)= p/3$, and $\P(g\in R^i\cap T^i\cap R^j\cap T^j)\leq \P(g\in T^i\cap T^j)=0$. These imply $\P(g\not\in \bigcup_{i=1}^3R^i\cap T^i)=1-3\alpha p$. Also for different $g,h$, their locations are independent of each other (since there was no correlations between distinct $g,h$, in any of $R^1, R^2, R^3, T^1, T^2, T^3$), giving that  $R^1\cap T^1, R^2\cap T^2, R^3\cap T^3$ are disjoint $p/3$-random subsets of $G$. Use Lemma~\ref{Lemma_symmetric_sets_inside_random_sets} to pick disjoint, symmetric, $q$-random subsets  $Q^1\subseteq R^1\cap T^1, Q^2\subseteq R^2\cap T^2, Q^3\subseteq R^3\cap T^3$. Now $Q^1, Q^2, Q^3$ demonstrate $R^1, R^2, R^3$ being $q$-slightly-independent.

So Theorem~\ref{thm:main_strongest} applies to $R^1, R^2, R^3$. %Moreover, due to Observation~\ref{??}, Theorem~\ref{??} applies to $R^1, R^2^{-1}, R^3$.
We also have the following property if $G\neq \mathbb{Z}_2^k$:
\begin{itemize}
\item[P:] For each $i,j\in\{1,2,3\}$, $i\neq j$, $R^i\times R^j$ contains $\geq p^{2}n/1000$ pairs of the form $(x,x^{-1})$ where for each such pair of pairs $\{x,x^{-1}\}\cap \{y,y^{-1}\}=\emptyset$ (this follows by Chernoff's bound).
\end{itemize}
 Now consider  sets $X, Y, Z$ as in the theorems. 
 Use $(P)$ to pick distinct elements $g_A, g_B, g_C$ with $g_A\in R^2$, $g_{A}^{-1}\in R^3$, $g_B\in R^1$, $g_{B}^{-1}\in R^3$, $g_C\in R^1$, $g_{C}^{-1}\in R^2$. Noting that we have $p^2n/1000$ choices for each  $g_A, g_B, g_C$, we can choose them to have $\{g_A, g_A^{-1}\}, \{g_B, g_B^{-1}\},\{g_C, g_C^{-1}\}$ disjoint from each other and from $R^1\setminus X, R^2\setminus Y, R^3\setminus Z$.  The result is that the three edges $f_A:=(\id_A, g_A, g_A^{-1}), f_B:=(g_B, \id_B, g_B^{-1}), f_C:=(g_C, g_C^{-1}, \id_C)$  form a matching with all vertices, other than possibly $\id_B, \id_B, \id_C$ contained in $X\cup Y\cup Z$.  Let $N=\{f_i: \id_i\in X\cup Y\cup Z\}$  to get a matching contained in $X\cup Y\cup Z$ covering all copies of the identity in $X\cup Y\cup Z$.

Let $X'=X\setminus N, Y'=Y\setminus N, Z'=Z\setminus N$, noting that these have the same size and have $\sum X'+\sum Y'+\sum Z'=0$ in $G^{ab}$ (due to $N$ being a matching). Thus the property of Theorem~\ref{thm:main_strongest} applies to give a perfect matching $M$ in $H_G[X',Y',Z']$. Now $M\cup N$ satisfies the theorem.
\end{proof}

Next we show how to modify the above proof to obtain Theorems~\ref{thm:maintheoremnondisjoint},~\ref{thm:maintheoremsemidisjoint},~\ref{thm:maintheorem_semidisjointdivision}.
\begin{proof}[Proof of Theorems~\ref{thm:maintheoremnondisjoint},~\ref{thm:maintheoremsemidisjoint},~\ref{thm:maintheorem_semidisjointdivision}]
First notice that in all three theorems, we have that $R^1, R^2, R^3$ are $q$-slightly-independent. In Theorems~\ref{thm:maintheoremnondisjoint},~\ref{thm:maintheoremsemidisjoint} this is exactly the first paragraph of the proof of Theorem~\ref{thm:maintheorem_disjoint}. For Theorem~\ref{thm:maintheorem_semidisjointdivision}, that paragraph shows that  $(R^1)^{-1}, R^2, R^3$ are $q$-slightly-independent. But, then by Observation~\ref{obs:slightly_independent_symmetry}, we have that $R^1, R^2, R^3$ are $q$-slightly-independent.

Next note that property (P) holds in the following cases:
\begin{itemize}
\item Theorem~\ref{thm:maintheoremnondisjoint}: here property (P) always holds.
\item Theorem~\ref{thm:maintheoremsemidisjoint}: here property (P) always holds for $(i,j)=(2,3)$ and $(i,j)=(1,3)$. For $(i,j)=(1,2)$ property (P) holds when $G\neq \mathbb{Z}_2^k$.
\item Theorem~\ref{thm:maintheorem_semidisjointdivision}: here property (P) always holds for $(i,j)=(2,3)$ and $(i,j)=(1,3)$. For $(i,j)=(1,2)$ property (P) holds when $G\neq \mathbb{Z}_2^k$.
\end{itemize}
The rest of the proofs are the same as in Theorem~\ref{thm:maintheorem_disjoint} --- property (P) produces a matching of size $\leq 3$ covering all copies of the identity in $X\cup Y\cup Z$, and then the property of Theorem~\ref{thm:main_strongest} gives a matching covering the rest of $X\cup Y\cup Z$.
\end{proof}

Finally we show how to derive Theorem~\ref{thm:mainintro}  as stated in the introduction.
\begin{proof}[Proof of Theorem~\ref{thm:mainintro} via Theorem~\ref{thm:maintheoremnondisjoint}] Let $R^1,R^2,R^3\subseteq G$ be $p$-random subsets, independently sampled. Observe that $R^1,R^2,(R^3)^{-1}\subseteq G$ are also $p$-random subsets, independently sampled, so with high probability Theorem~\ref{thm:maintheoremnondisjoint} applies. Let $X,Y,Z\subseteq G$ be subsets with the properties as in the statement of Theorem~\ref{thm:mainintro}. Then, the sets  $X,Y,Z^{-1}\subseteq G$ clearly satisfy the two properties required by Theorem~\ref{thm:maintheoremnondisjoint} with respect to $R^1,R^2,(R^3)^{-1}\subseteq G$. Thus, $H_G[X,Y,Z^{-1}]$ contains a perfect matching, say $M$. Define the bijection $\phi\colon X\to Y$ so that $x$ maps to the unique element $y$ of $Y$ such that $x$ and $y$ are contained in an edge together in $M$. As for each edge $(x,\phi(x),z)$ of $M$, $x\phi(x)z=\id$, $x\mapsto x\phi(x)$ is a bijection $X\to (Z^{-1})^{-1}=Z$, as desired.
\end{proof}
\subsection{Complete mappings and orthomorphisms}
We conclude this section with a version of the main theorem that fits better with the results proved in Section~\ref{sec:pathlike}. Given a triple of subsets of a group $G$ as $(X,Y,Z)$, a \textbf{complete mapping} is a bijection $\phi\colon X\to Y$ such that the induced map from $X$ to $Z$ via $x\to x\phi(x)$ is also a bijection, whereas an \textbf{orthomorphism} is a bijection $\phi\colon X\to Y$ such that the induced map from $X$ to $Z$ via $x\to x^{-1}\phi(x)$ is also a bijection.

\begin{observation}\label{obs:completeortho}
Let $X,Y,Z$ be subsets of a group $G$. Then,
\begin{itemize}
    \item $(X,Y,Z)$ admits a complete mapping if and only if $H_G[X,Y,Z^{-1}]$ has a perfect matching.
    \item $(X,Y,Z)$ admits an orthomorphism if and only if $H_G[X^{-1}, Y, Z^{-1}]$ has a perfect matching.
\end{itemize}
\end{observation}
\begin{proof} The proof is routine and we refer the reader to an earlier version of the paper available at  arXiv:2204.09666v2 for a proof. %Note that $x\phi(x)=z$ if and only if $x\phi(x)z^{-1}=\id$. As $z\to z^{-1}$ is a bijection between $Z$ and $Z^{-1}$, we have that whenever $\phi$ is a complete mapping of $X,Y,Z$, we have a perfect matching of $H_G[X,Y,Z^{-1}]$ given by the edges of the form $(x,\phi(x), \phi(x)^{-1}x^{-1})$. Given a perfect matching of $H_G[X,Y,Z^{-1}]$, we can similarly define $\phi$ so that $\phi$ maps the $G_A$ coordinate of each matched edge to the $G_B$ coordinate to get a complete mapping of $(X,Y,(Z^{-1})^{-1})=(X,Y,Z)$. This proves the first bullet point.
%\par Similarly, $x^{-1}\phi(x)=z$ if and only if $x^{-1}\phi(x)z^{-1}=\id$. Since $x\to x^{-1}$ is also a bijection between $X$ and $X^{-1}$, if $\phi$ is an orthomorphism of $X,Y,Z$, we have a perfect matching of $H_G[X^{-1}, Y, Z^{-1}]$ given by the edges of the form $(x^{-1}, \phi(x), z^{-1})$. And given a perfect matching of $H_G[X^{-1}, Y, Z^{-1}]$, taking $\phi$ to map the inverse of the $G_A$ coordinate of each matched edge into the $G_B$ coordinate of that same matched edge, we obtain an orthomorphism of $((X^{-1})^{-1}, Y, (Z^{-1})^{-1})=(X,Y,Z)$. This proves the second bullet point.
\end{proof}

\begin{theorem}\label{thm:completeortho} Let $p\geq n^{-1/10^{102}}$. Let $G$ be a group of order $n$. Let $R^1,R^2\subseteq G$ be disjoint $p$-random subsets, and let $R^3\subseteq G$ be a $p$-random subset, sampled independently with $R^1$ and $R^2$. Then, with high probability, the following holds. 
\par Let $X,Y,Z$ be equal-sized subsets of $G_A$, $G_B$, and $G_C$ satisfying the following properties.
\begin{itemize}
    \item $|(R^1_A\cup R^2_B\cup R^3_C) \triangle (X\cup Y\cup Z) |\leq p^{10^{18}}n/\log(n)^{10^{18}}$
    \item One of the following identities holds in the abelianization of $G$.
    \begin{itemize}
        \item[\textbf{C.}] $\sum X+\sum Y = \sum Z$ 
        \item[\textbf{O.}] $\sum Y-\sum X = \sum Z$ 
    \end{itemize}
   \item If $G=(\mathbb{Z}_2)^k$ for some $k$, then $\id \notin Z$.
\end{itemize}
Then, if \textbf{C} holds, $(X,Y,Z)$ admits a complete mapping, and if \textbf{O} holds, $(X,Y,Z)$ admits an orthomorphism. 
\end{theorem}
\begin{proof}
We refer the reader to an earlier version of the paper available at arXiv:2204.09666v2 for a short and routine proof.\end{proof}

\section{Proof of the main theorem}\label{sec:mainproof}
In this section, we prove the main result of the paper, Theorem~\ref{thm:main_strongest}. We begin by clarifying that for every (sufficiently large) group $G$, we fix $a_\phi,b_\phi,c_\phi\in G$ with properties as in Section~\ref{sec:choosingaphi}. Definitions such as $\phi$-generic, pair, and coset-paired are with respect to these fixed choices of $a_\phi,b_\phi,c_\phi\in G$ given by Lemma~\ref{lem:pairingsexist}.

\subsection{Absorbers}\label{sec:absorbers}
In this section we give constructions of absorbers i.e. subsets  $R\subseteq V(H_G)$ which can be extended into a matching in several different ways. Lemma~\ref{lem:mainabsorptionlemma} is the main result of this section, and the only result we need for the rest of the paper.  The following definition precisely describes what we will be looking for.
\begin{definition}
Let $\mathcal F= \{S_1, \dots, S_t\}$ be a family of subsets of $V(H)$ for a hypergraph $H$. We say that a set of vertices $R$ $m$\textbf{-absorbs} $\mathcal F$ if for every subfamily   $\mathcal F'\subseteq \mathcal F$  of size $m$, there is a hypergraph matching whose vertex set is exactly $R\cup \bigcup_{S_i\in \mathcal F'} S_i$.
\end{definition}
It will be convenient to note that when $t=2$ and $m=1$, then the above definition is equivalent to
``Let $X,Y$ be sets of vertices in a hypergraph $H$. We say that a set of vertices $R$ $1$\textbf{-absorbs} $\{X,Y\}$ if there are hypergraph matchings $R^-, R^+$ whose vertex sets are exactly $V(R^-)=R\cup X$ and $V(R^+)=R\cup Y$.''
\par The following lemma shows how the parameter $h$ changes when we pass to a subfamily of $\mathcal F$.
\begin{lemma}\label{Lemma_absorption_subsets}
Let $\mathcal F$ be a family of disjoint subsets of $V(H)$ and $\mathcal F'\subseteq F$ a subfamily with $|\mathcal F\setminus \mathcal F'|=t$. If $R$ $h$-absorbs $\mathcal F$ then
\begin{enumerate}
\item $R$ $h$-absorbs $\mathcal F'$.
\item $R\cup \bigcup_{S\in \mathcal F\setminus \mathcal F'} S$ $(h-t)$-absorbs $\mathcal F'$.
\end{enumerate}
\end{lemma}
\begin{proof}
Part (1) is immediate from the definition of absorbing. For part (2), notice that for any subfamily $\mathcal F''\subseteq \mathcal F'$ of size $h-t$, the subfamily $\mathcal F''\cup (\mathcal F\setminus \mathcal F')$ has size $h$. Therefore there is a matching with vertex set $R\cup \bigcup_{S\in \mathcal F''\cup (\mathcal F\setminus \mathcal F')} S= (R\cup \bigcup_{S\in \mathcal F\setminus \mathcal F'} S)\cup \bigcup_{S\in \mathcal F''} S$.
\end{proof}

We build larger absorbers from smaller ones. The following lemma allows us to take unions of $1$-absorbers to get another $1$-absorber.
\begin{lemma}\label{Lemma_absorbers_unions}
Suppose that $\{R_1, \dots, R_t\}, \{X_1, \dots, X_t\}, \{Y_1, \dots, Y_t\}$ are three families of disjoint sets with $\bigcup R_i$ disjoint from $\bigcup (X_i\cup Y_i)$.  
Suppose that $R_i$ 1-absorbs $\{X_i, Y_i\}$ for $i=1, \dots, t$. Set $Z=(\bigcup X_i)\cap (\bigcup Y_i)$.
Then, $\bigcup R_i\cup Z$ 1-absorbs $\{\bigcup X_i\setminus Z, \bigcup Y_i\setminus Z\}$.
\end{lemma}
\begin{proof}
By the definition of $1$-absorbs we have matchings $R_i^-$ and $R_i^+$ with vertices $V(R_i^-)=R_i\cup X_i$ and $V(R_i^+)=R_i\cup Y_i$. Note that $\bigcup R_i^-$ and $\bigcup R_i^+$ are matchings (since $\{R_1, \dots, R_t\}, \{X_1, \dots, X_t\}, \{Y_1, \dots, Y_t\}$ are  families of disjoint sets with $\bigcup R_i$ disjoint from $\bigcup (X_i\cup Y_i)$).
Also $V(\bigcup R_i^-)= \bigcup R_i\cup \bigcup X_i=(\bigcup R_i\cup Z) \cup (\bigcup X_i\setminus Z)$ and $V(\bigcup R_i^+)= \bigcup R_i\cup \bigcup Y_i=(\bigcup R_i\cup Z) \cup (\bigcup Y_i\setminus Z)$. Thus $\bigcup R_i^-$ and $\bigcup R_i^+$ are matchings satisfying the definition of ``$\bigcup R_i\cup Z$ 1-absorbs $\{\bigcup X_i\setminus Z, \bigcup Y_i\setminus Z\}$''.
\end{proof}

We state a technical consequence of the above lemma for later use.
\begin{lemma}\label{Lemma_absorbers_unions_3_to_2}
Suppose that we have distinct vertices $a,b,c,d\in V(H_G)$,  vertices $w_1, \dots, w_k\in V(H_G)\setminus\{a,b,c,d\}$ and disjoint sets $R_0, R_1, \dots, R_{k-1}\subseteq V(H_G)\setminus \{w_1, \dots, w_k,a,b,c,d\}$ with $R_0$ 1-absorbing $\{\{a,w_1, b\}, \{a,w_k, b\} \}$ and $R_i$ 1-absorbing $\{w_i, w_{i+1}\}$ for $i=1, \dots, k-1$. Then, there is a subset $R'\subseteq \bigcup_{i=0}^{k-1}R_i\cup\{w_1, \dots, w_k\}$ which 1-absorbs $\{\{a, b\}, \{c, d\}\}$.
\end{lemma}
\begin{proof}

Without loss of generality, we can assume that $w_1, \dots, w_k$ are distinct (by passing to a subset of $\{w_1, \dots, w_k\}$ of distinct vertices, and a corresponding subfamily of $\{R_1, \dots, R_{k-1}\}$). Now, the lemma follows from Lemma~\ref{Lemma_absorbers_unions} with $\{R_0, \dots, R_k\}$,
$$\{X_0, \dots, X_k\}:=\{\{a,w_1, b\}, \{w_2\}, \{w_3\}, \dots, \{w_k\}\}, $$ and $$\{Y_1, \dots, Y_t\}:=\{\{c,w_k,d\}, \{w_1\}, \{w_2\},  \dots, \{w_{k-1}\}\}.$$
\end{proof}

\subsubsection{Constructing $1$-absorbing sets}
The following lemma shows that there are many edges through generic vertices in $H_G$. 
\begin{lemma}\label{Lemma_edge_through_generic_vertex}
Let $p\geq n^{-1/700}$. 
Let $H_G$ be a multiplication hypergraph, $R^1, R^2, R^3$ disjoint, symmetric $p$-random subsets of $G$  and set $R=R^1_A\cup R^2_B\cup R^3_C$. With high probability, the following holds: 

For any generic $v\in V(H_G)$  and $U\subseteq V(H_G)$ with $|U|\leq p^{800}n/10^{4000}$, 
 there is an edge $e$ of $H_G$ passing through $v$ and having the other two vertices in $R\setminus U$.
\end{lemma}
\begin{proof}
With high probability, Lemma~\ref{Lemma_separated_set_random} applies. 
We'll just look at the case $v\in G_A$, the other two cases are symmetric. Fix some generic $v\in G_A $. Thinking of $b$ as a free variable, consider the set  $S=\{b, b^{-1}v^{-1}\}\subseteq G\ast F_1$. Note that this is a set of linear elements with $w=b, w'=b^{-1}v^{-1}$ separable (by  (b) since $v^{-1}$ is generic). Lemma~\ref{Lemma_separated_set_random} now implies what we want.
\end{proof}

The following is the basic building block for the absorbers that we construct. It shows that we can $1$-absorb sets of the form $\{[a,b]c,c\}$.
\begin{lemma}\label{Lemma_absorber_one_commutator}
Let $p\geq n^{-1/700}$. 
Let $H_G$ be a multiplication hypergraph, $R^1, R^2, R^3$ disjoint, symmetric $p$-random subsets of $G$  and set $R=R^1_A\cup R^2_B\cup R^3_C$. With high probability, the following holds:

  For any $a,b,c\in G_C$ with $c, c^{-1}bab^{-1}a^{-1}, c^{-1}bab^{-1}, c^{-1}ba, c^{-1}b$ generic and $U\subseteq V(H_G)$ with $|U|\leq p^{800}n/10^{4010}$, 
 there is a set $R'\subseteq R\setminus U$ of size $\leq 14$ which $1$-absorbs $\{[a,b]c, c\}$.
\end{lemma}
\begin{proof}
With high probability, the properties of Lemmas~\ref{Lemma_separated_set_random} and~\ref{Lemma_edge_through_generic_vertex} hold. 
Fix some $a,b,c\in G_C$ with  $$c, c^{-1}bab^{-1}a^{-1}, c^{-1}bab^{-1}, c^{-1}ba, c^{-1}b$$  all being generic and $U\subseteq V(H_G)$ with $|U|\leq p^{800}n/10^{4010}$. First suppose that $[a,b]=\id$ i.e. that $[a,b]c=c$. From Lemma~\ref{Lemma_edge_through_generic_vertex} we know that there is an edge $f=\{s,t,c\}$ of $H_G$ with $s,t\in R\setminus U$. Then $R'=\{s,t\}$ satisfies the lemma (with both $M^-$ and $M^+$ for the definition of ``$1$-absorbs'' equal to the single edge $f$).
 
\begin{figure}[h]
  \centering
    \includegraphics[width=\textwidth]{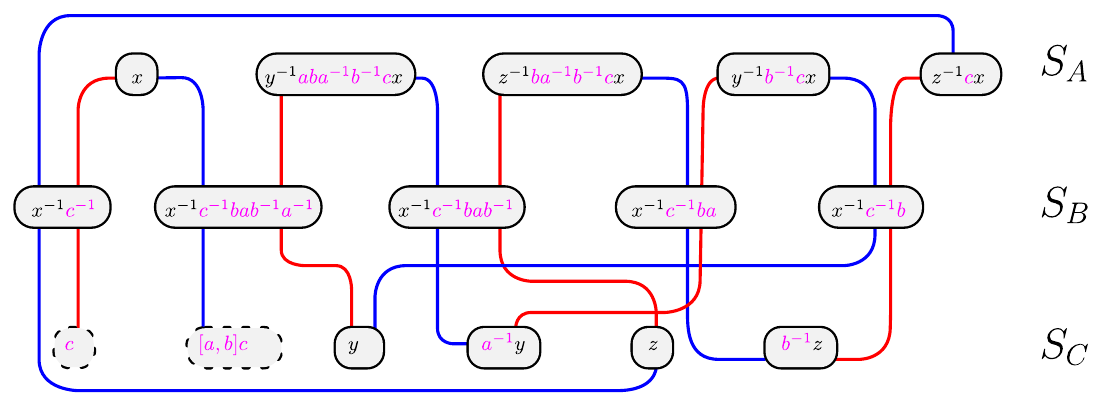}
  \caption{The set $S\subseteq G\ast F_3$ in Lemma~\ref{Lemma_absorber_one_commutator}. Black letters $x,y,z$ are free variables, while pink letters are elements of $G$. The two elements $[a,b]c$, $c$ are not part of $S$ (and are just pictured to show how $S$ $1$-absorbs $\{[a,b]c,c\}$).}
\label{Figure_commutator}
\end{figure}
 
\begin{figure}[h]
  \centering
\footnotesize
\hspace{-0.9cm}$S_{x}:$\begin{tabular}{|l|l|l|l|l|l|l|}
\hline
 & $x$ & $x^{-1}\color{magenta}c^{-1}$ & $x^{-1}{\color{magenta}   c^{-1}bab^{-1}a^{-1}}$ & $x^{-1}{\color{magenta} c^{-1}bab^{-1}}$ & $x^{-1}{\color{magenta} c^{-1}ba}$ & $x^{-1}{\color{magenta} c^{-1}b}$ \\ \hline
$x$ & \cellcolor[HTML]{D9D9D9} & \cellcolor[HTML]{DDEBF7}$\color{magenta}c^{-1}$ & \cellcolor[HTML]{DDEBF7}${\color{magenta} c^{-1}bab^{-1}a^{-1}}$ & \cellcolor[HTML]{DDEBF7}${\color{magenta} c^{-1}bab^{-1}}$ & \cellcolor[HTML]{DDEBF7}${\color{magenta} c^{-1}ba}$ & \cellcolor[HTML]{DDEBF7}${\color{magenta} c^{-1}b}$ \\ \hline
$x^{-1}\color{magenta}c^{-1}$ & \cellcolor[HTML]{DDEBF7}$\color{magenta}c^{-1}$ & \cellcolor[HTML]{D9D9D9} & \cellcolor[HTML]{E2EFDA}${\color{magenta} bab^{-1}a^{-1}}$ & \cellcolor[HTML]{E2EFDA}${\color{magenta} bab^{-1}}$ & \cellcolor[HTML]{E2EFDA}${\color{magenta} ba}$ & \cellcolor[HTML]{E2EFDA}${\color{magenta} b}$ \\ \hline
$x^{-1}{\color{magenta}   c^{-1}bab^{-1}a^{-1}}$ & \cellcolor[HTML]{DDEBF7}${\color{magenta} c^{-1}bab^{-1}a^{-1}}$ & \cellcolor[HTML]{E2EFDA}${\color{magenta} bab^{-1}a^{-1}}$ & \cellcolor[HTML]{D9D9D9} & \cellcolor[HTML]{E2EFDA}${\color{magenta} a^{-1}}$ & \cellcolor[HTML]{E2EFDA}${\color{magenta} b^{-1}a^{-1}}$ & \cellcolor[HTML]{E2EFDA}${\color{magenta} ab^{-1}a^{-1}}$ \\ \hline
$x^{-1}{\color{magenta} c^{-1}bab^{-1}}$ & \cellcolor[HTML]{DDEBF7}${\color{magenta} c^{-1}bab^{-1}}$ & \cellcolor[HTML]{E2EFDA}${\color{magenta} bab^{-1}}$ & \cellcolor[HTML]{E2EFDA}${\color{magenta} a^{-1}}$ & \cellcolor[HTML]{D9D9D9} & \cellcolor[HTML]{E2EFDA}${\color{magenta} b^{-1}}$ & \cellcolor[HTML]{E2EFDA}${\color{magenta} ab^{-1}}$ \\ \hline
$x^{-1}{\color{magenta} c^{-1}ba}$ & \cellcolor[HTML]{DDEBF7}${\color{magenta} c^{-1}ba}$ & \cellcolor[HTML]{E2EFDA}${\color{magenta} ba}$ & \cellcolor[HTML]{E2EFDA}${\color{magenta} b^{-1}a^{-1}}$ & \cellcolor[HTML]{E2EFDA}${\color{magenta} b^{-1}}$ & \cellcolor[HTML]{D9D9D9} & \cellcolor[HTML]{E2EFDA}${\color{magenta} a}$ \\ \hline
$x^{-1}{\color{magenta} c^{-1}b}$ & \cellcolor[HTML]{DDEBF7}${\color{magenta} c^{-1}b}$ & \cellcolor[HTML]{E2EFDA}${\color{magenta} b}$ & \cellcolor[HTML]{E2EFDA}${\color{magenta} ab^{-1}a^{-1}}$ & \cellcolor[HTML]{E2EFDA}${\color{magenta} ab^{-1}}$ & \cellcolor[HTML]{E2EFDA}${\color{magenta} a}$ & \cellcolor[HTML]{D9D9D9} \\ \hline
\end{tabular}

\footnotesize
$S_{x,y}:$\begin{tabular}{|l|l|l|}
\hline
 & $y^{-1}{\color{magenta}   aba^{-1}b^{-1}c}x$ & $y^{-1}{\color{magenta} b^{-1}c}x$ \\ \hline
$y^{-1}{\color{magenta}   aba^{-1}b^{-1}c}x$ & \cellcolor[HTML]{D9D9D9} & \cellcolor[HTML]{E2EFDA}${\color{magenta} aba^{-1}}$ \\ \hline
$y^{-1}{\color{magenta} b^{-1}c}x$ & \cellcolor[HTML]{E2EFDA}${\color{magenta} aba^{-1}}$ & \cellcolor[HTML]{D9D9D9} \\ \hline
\end{tabular}
$S_{y}:$\begin{tabular}{|l|l|l|}
\hline
 & $y$ & ${\color{magenta} a^{-1}}y$ \\ \hline
$y$ & \cellcolor[HTML]{D9D9D9} & \cellcolor[HTML]{E2EFDA}${\color{magenta} a^{-1}}$ \\ \hline
${\color{magenta} a^{-1}}y$ & \cellcolor[HTML]{E2EFDA}${\color{magenta} a^{-1}}$ & \cellcolor[HTML]{D9D9D9} \\ \hline
\end{tabular}
$S_{x,z}:$\begin{tabular}{|l|l|l|}
\hline
 & $z^{-1}{\color{magenta}   ba^{-1}b^{-1}c}x$ & $z^{-1}{\color{magenta}c}x$ \\ \hline
$z^{-1}{\color{magenta}   ba^{-1}b^{-1}c}x$ & \cellcolor[HTML]{D9D9D9} & \cellcolor[HTML]{E2EFDA}${\color{magenta} ba^{-1}b^{-1}}$ \\ \hline
$z^{-1}{\color{magenta}c}x$ & \cellcolor[HTML]{E2EFDA}${\color{magenta} ba^{-1}b^{-1}}$ & \cellcolor[HTML]{D9D9D9} \\ \hline
\end{tabular}
$S_{z}:$\begin{tabular}{|l|l|l|}
\hline
 & $z$ & ${\color{magenta} b^{-1}}z$ \\ \hline
$z$ & \cellcolor[HTML]{D9D9D9} & \cellcolor[HTML]{E2EFDA}${\color{magenta} b^{-1}}$ \\ \hline
${\color{magenta} b^{-1}}z$ & \cellcolor[HTML]{E2EFDA}${\color{magenta} b^{-1}}$ & \cellcolor[HTML]{D9D9D9} \\ \hline
\end{tabular}
  \caption{Proofs for weak/strong separability and linearity of all pairs $w,w'\in S$ in Lemma~\ref{Lemma_absorber_one_commutator} to justify the application of Lemma~\ref{Lemma_separated_set_random}. 
  For strong separability, first we have partitioned $S$ into five subsets $S=S_{x}\cup S_{x,y}\cup  S_{x,z}\cup S_{z}\cup S_{y}$ based on which free variables appear in each $w\in S$ (as in Observation~\ref{Observation_partition_S_by_free_variables}). By Observation~\ref{Observation_partition_S_by_free_variables}, any $w,w'$ in different subsets are strongly separable by part (a) of the definition. 
  For each of the sets $S_{x}, S_{x,y},  S_{x,z}, S_{z}, S_{y}$ we give a table explaining why the $w,w'$ in that set are strongly/weakly-separable. Note that for words coming from different $S_A/S_B/S_C$ we need to show strong separability, but for words coming from the same part, weak separability suffices.  Blue cells represent $w,w'$ being strongly separable via part (b) of the definition, green cells represent  $w,w'$ being weakly separable via part (b'), and grey cells represent $w,w'$ not being separable/weakly-separable. The group element inside each blue cell is a generic element $g$ so that $w'\in\{gw, g^{-1}w, gw^{-1}, g^{-1}w^{-1}, wg, wg^{-1}, w^{-1}g, w^{-1}g^{-1}\}$ (thus checking (b) for $w,w'$). The group element inside the green cells is a non-identity element $g$ so that $w=w'$ rearranges into $\id=g$ (thus checking (b') for $w,w'$). Observe that green cells are used only between pairs of words coming from the same part of $S_A/S_B/S_C$, meaning that we have strong separation for pairs of words coming from different parts $S_A/S_B/S_C$, as needed. To see that every $s\in S$ is linear notice that every word pictured has no repetitions of black letters. }
\label{Figure_justification_commutator}
\end{figure}

 Now suppose $[a,b]\neq\id$. Notice that this implies that  $a,b,bab^{-1}a^{-1}, bab^{-1}, ba,  b^{-1}a^{-1}, ab^{-1}a^{-1}, ab^{-1}\neq \id$ also.
Thinking of $x,y,z$ as free variables in $G\ast F_3$, consider the set $S$  given in Figure~\ref{Figure_commutator}, with  partition $S=S_A\cup S_B\cup S_C$. Notice that all words in $S$ are linear, all pairs of words weakly-separable, and $S_A, S_B, S_C$ are pairwise strongly separable (see Figure~\ref{Figure_justification_commutator} for justification).
Using Lemma~\ref{Lemma_separated_set_random}, there is some projection $\pi:G\ast F_3 \to G$ which separates $S$ and has $\pi(S)\subseteq R\setminus U$.
Any such $\pi(S)$ $1$-absorbs $\{[a,b]c, c\}$ (using the red/blue matchings in Figure~\ref{Figure_commutator}).
\end{proof}

The following lemma is identical to the previous one, except that it weakens the assumption on what elements are generic. 
\begin{lemma}\label{Lemma_absorber_one_commutator_clean}
Let $p\geq n^{-1/700}$.
Let $H_G$ be a multiplication hypergraph, $R^1, R^2, R^3$ disjoint, symmetric $p$-random subsets of $G$  and set $R=R^1_A\cup R^2_B\cup R^3_C$. With high probability, the following holds:

  For any $a,b,x\in G_B$ with $x, x[a,b]$ generic and $U\subseteq V(H_G)$ with $|U|\leq p^{800}n/10^{4020}$, 
 there is a set $R'\subseteq R\setminus U$ of size $\leq 16$ which $1$-absorbs $\{x[a,b], x\}$.
\end{lemma}
\begin{proof}
With high probability, the properties of Lemmas~\ref{Lemma_separated_set_random},~\ref{Lemma_edge_through_generic_vertex} and~\ref{Lemma_absorber_one_commutator} hold. 
Fix some $a,b,x\in G_B$ with $x, x[a,b]$ generic and $U\subseteq V(H_G)$ with $|U|\leq p^{800}n/10^{4020}$. 
As in Lemma~\ref{Lemma_absorber_one_commutator}, the conclusion  trivially follows from Lemma~\ref{Lemma_edge_through_generic_vertex} if $[a,b]=\id$, so assume that this doesn't happen.  
With $v$ the free variable in $G\ast F_1$, consider the sets of words $T:=\{v, [b,a]x^{-1}v^{-1}, x^{-1}v^{-1}\}$ and $S=\{x^{-1}v^{-1}, vxaba^{-1}b^{-1}, vxaba^{-1}, vxab, vxa\}$. Define $T_A=\{v\}, T_B=\emptyset, T_C=\{[b,a]x^{-1}v^{-1}, x^{-1}v^{-1}\}$ to get a partition of $T$. It is easy to check that all  $w\in T\cup S$ are linear (since $v$ appears precisely once in each $w\in T\cup S$), that $T_A, T_C$ are strongly separable (by part (b) of the definition, using that  $x, x[a,b]$ are generic), and that $[b,a]x^{-1}v^{-1}, x^{-1}v^{-1}$ are weakly separable (since $[b,a]x^{-1}v^{-1}=x^{-1}v^{-1}$ rearranges into $[b,a]=\id$). By Lemma~\ref{Lemma_separated_set_random}, there is a projection $\pi$  which separates $T\cup S$ and has $\pi(T\cup S)\subseteq R\setminus (U\cup N(G))$.  Since $T_A, T_B, T_C$ are pairwise strongly separable, we additionally get $\pi(T_A)\subseteq R_A^1, \pi(T_B)\subseteq R_B^2, \pi(T_C)\subseteq R_C^3$. 
Combining this with the fact that all vertices in $\pi(T)$ are distinct  (which comes from all pairs of words in $T$ being weakly separable, and so $\pi(t)\neq \pi (t')$ for distingt $t,t'\in T$), we have that $e^-=(\pi(v), x[a,b], [b,a]x^{-1}\pi(v)^{-1})$ and $e^+=(\pi(v), x, x^{-1}\pi(v)^{-1})$ are edges of $H_G$  contained in $R\cup \{x,x[a,b]\}$.

Using Lemma~\ref{Lemma_absorber_one_commutator} with $a'=b, b'=a, c=x^{-1}\pi(v^{-1})$ we find a set $Q$ disjoint from $U\cup \pi(T)$ which 1-absorbs  $\{[b,a]x^{-1}\pi(v)^{-1}, x^{-1}\pi(v)^{-1}\}$ (all the required elements are generic for that lemma as a consequence of $\pi(S)\cap N(G)=\emptyset$). 
Now $Q\cup \pi(T)$ 1-absorbs $\{x[a,b], x\}$. This can be seen directly or by first noticing that  $\{\pi(v)\}$ 1-absorbs $\{\{x[a,b], [b,a]x^{-1}\pi(v)^{-1}\}, \{x, x^{-1}\pi(v)^{-1}\}\}$ (the single-edge matchings $e^-, e^+$ witness this). Then Lemma~\ref{Lemma_absorbers_unions} shows that $Q\cup \pi(T)$ 1-absorbs $\{x[a,b], x\}$.
\end{proof}

For two sets $X,Y\subseteq G$ and $s\in G$, use $XsY$ to denote $\{xsy:x\in X, y\in Y\}$. The following lemma lets us $1$-absorb a pair of size $3$ sets.
\begin{lemma}\label{Lemma_absorber_3set}
Let $p\geq n^{-1/700}$. 
Let $H_G$ be a multiplication hypergraph, $R^1, R^2, R^3$ disjoint, symmetric $p$-random subsets of $G$  and set $R=R^1_A\cup R^2_B\cup R^3_C$. With high probability, the following holds: 

For any distinct, generic  $a,b,c,d\in G_B$ and  $X,Y,U\subseteq V(H_G)$ with $|U|\leq p^{800}n/10^{4000}, |X|,|Y|\leq 5$, 
 there is a $R'\subseteq R\setminus U$ of size $6$ and $s$ with $XsY\subseteq R\setminus (U\cup R')$ such that $R'$ $1$-absorbs $\{\{a, dsc, b\}, \{c,bsa,d\}\}$. 
\end{lemma}
\begin{proof}
With high probability, the property of Lemma~\ref{Lemma_separated_set_random} holds. Suppose we have generic  $a,b,c,d\in G_B$ and  $X,Y,U\subseteq V(H_G)$ with $|U|\leq p^{800}n/10^{4000}, |X|,|Y|\leq 5$. Let $X'=X\cup\{b,d\}, Y'=Y\cup \{a,c\}$.
For free variables $u,w$, consider the sets of words $S=\{u, { a^{-1}}u^{-1}, { c^{-1}}u^{-1}, w, { b^{-1}}w^{-1}, {  d^{-1}}w^{-1}\}$ with $S_A=\{u, { b^{-1}}w^{-1}, {  d^{-1}}w^{-1}\}, S_B=\emptyset, S_C=\{w, { a^{-1}}u^{-1}, { c^{-1}}u^{-1}\}$.  Notice that all words in $S\cup (X'wuY')$ are linear,  all pairs of words in $S$ are weakly separable (see Figure~\ref{Figure_justofication_3_absorber}), and $S_A, S_B, S_C$ are pairwise strongly separable. Also the pair of sets $(S, X'wuY')$ is strongly separable (using part (a) of the definition of separable). By Lemma~\ref{Lemma_separated_set_random}, there is some projection $\pi$ separating $S\cup (X'wuY')$ and having $S\cup (X'wuY')\subseteq R\setminus U$. For any such projection,  $R'=\pi(S)$ and $s:=\pi(wu)$ satisfy the lemma (with the red/blue matchings in Figure~\ref{Figure_3absorber} satisfying the definition of $R'$ 1-absorbing $\{\{a, dsc, b\}, \{c,bsa,d\}\}$. Note that this works regardless of whether $dwuc=bwua$ or not).
\begin{figure}[h]
  \centering
    \includegraphics[width=0.6\textwidth]{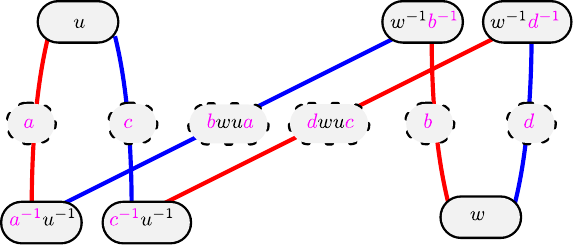}
  \caption{The set of words $S\subseteq G\ast F_2$ for Lemma~\ref{Lemma_absorber_3set}. First, second, and third row of words represent $S_A$, $S_B$, and $S_C$, respectively. Black letters $u,w$ are free variables of $F_2$, while pink letters are elements of $G$. The elements $a,b,c,d, bwua, dwuc$ are not part of $S$ (and are pictured just to demonstrate how the absorption works).}
\label{Figure_3absorber}
\end{figure}
\end{proof}

\begin{figure}[h]
  \centering 
{  \footnotesize
$S_u:$\begin{tabular}{|l|l|l|l|}
\hline
 & $u$ & ${\color{magenta} a^{-1}}u^{-1}$ & ${\color{magenta} c^{-1}}u^{-1}$ \\ \hline
$u$ & \cellcolor[HTML]{D9D9D9} & \cellcolor[HTML]{DDEBF7}${\color{magenta} a^{-1}}$ & \cellcolor[HTML]{DDEBF7}${\color{magenta} c^{-1}}$ \\ \hline
${\color{magenta} a^{-1}}u^{-1}$ & \cellcolor[HTML]{DDEBF7}${\color{magenta} a^{-1}}$ & \cellcolor[HTML]{D9D9D9} & \cellcolor[HTML]{E2EFDA}${\color{magenta} ac^{-1}}$ \\ \hline
${\color{magenta} c^{-1}}u^{-1}$ & \cellcolor[HTML]{DDEBF7}${\color{magenta} c^{-1}}$ & \cellcolor[HTML]{E2EFDA}${\color{magenta} ac^{-1}}$ & \cellcolor[HTML]{D9D9D9} \\ \hline
\end{tabular}
$S_w:$\begin{tabular}{|l|l|l|l|}
\hline
 & $w$ & $w^{-1}\color{magenta} b^{-1}$ & $w^{-1}\color{magenta} d^{-1}$ \\ \hline
$w$ & \cellcolor[HTML]{D9D9D9} & \cellcolor[HTML]{DDEBF7}${\color{magenta} b^{-1}}$ & \cellcolor[HTML]{DDEBF7}${\color{magenta} d^{-1}}$ \\ \hline
$w^{-1}\color{magenta} b^{-1}$ & \cellcolor[HTML]{DDEBF7}${\color{magenta} b^{-1}}$ & \cellcolor[HTML]{D9D9D9} & \cellcolor[HTML]{E2EFDA}${\color{magenta} bd^{-1}}$ \\ \hline
$w^{-1}\color{magenta} d^{-1}$ & \cellcolor[HTML]{DDEBF7}${\color{magenta} d^{-1}}$ & \cellcolor[HTML]{E2EFDA}${\color{magenta} bd^{-1}}$ & \cellcolor[HTML]{D9D9D9} \\ \hline
\end{tabular}}
  \caption{Justification for linearity and separability of  pairs $(w,w')$ in $S$ in Lemma~\ref{Lemma_absorber_3set}. For separability, first split $S$ into  $S_u=\{u, { a^{-1}}u^{-1}, { c^{-1}}u^{-1}\}$, $S_w=\{w, { b^{-1}}w^{-1}, {  d^{-1}}w^{-1}\}$ based on which free variables appear in the elements (as  in Observation~\ref{Observation_partition_S_by_free_variables}). By Observation~\ref{Observation_partition_S_by_free_variables}, pairs $(w,w')$ with $w,w'$ in different sets $S_u/S_w$ fall under part (a) of the definition of strongly separable, so it remains to check pairs inside $S_u$ and $S_w$. Justification for this is given in the two tables above (with the same conventions as in Figure~\ref{Figure_justification_commutator}, in particular, green cells are used only between pairs of vertices coming from the same part $S_A/S_B/S_C$).
 Relevant elements are generic/non-identity as a consequence of $a,b,c,d$ being distinct and generic). To see that each $w\in S$ is linear, note that there are no repetitions of black letters in each $w$ (and each $w\in S$ contains at least one black letter).}
\label{Figure_justofication_3_absorber}
\end{figure}
 
 The following technical lemma allows us to $1$-absorb certain pairs of sets of size $2$. 
\begin{lemma}\label{Lemma_absorber_pair_multiply_by_commutator}
Let $p\geq n^{-1/700}$. 
Let $H_G$ be a multiplication hypergraph, $R^1, R^2, R^3$ disjoint, symmetric $p$-random subsets of $G$  and set $R=R^1_A\cup R^2_B\cup R^3_C$. With high probability, the following holds:

  For any $g,x,y, a,b\in  G_B$ with $yg,x,y,x[a,b]g$ distinct, generic and  $U\subseteq V(H_G)$ with $|U|\leq p^{800}n/10^{4030}$, 
 there is a set $R'\subseteq R\setminus U$ of size $\leq 60$ which $1$-absorbs $\{\{yg, x\}, \{y, x[a,b]g\}\}$.
\end{lemma}
\begin{proof}
With high probability, the properties of Lemmas~\ref{Lemma_absorber_one_commutator_clean} and~\ref{Lemma_absorber_3set} hold. Let  $g,x,y, a,b$ and $U$ be as in the lemma. 
Fix $a'=yg$, $b'=x$, $c'=y$, $d'=x[a,b]g$ and note these are distinct and generic by assumption. Define $X=\{x,x[a,b]g, x[a,b], x[a,b]yg, \id\}$, $Y=\{yg, y, g[a,b], yg[a,b],\id\}$.
By Lemma~\ref{Lemma_absorber_3set}, we get  $R'$ and $s$ with $R_0\cup (XsY)\subseteq V(H_G)\setminus(U\cup N(G)\cup \{a',b',c',d'\})$   with $R_0$ 1-absorbing $\{\{a',d'sc', b'\}, \{c',b'sa',d'\}\}$. 

Define $w_1=d'sc'=x[a,b]gsy=x[a,b]ygs[y,gs]$, $w_2=x[a,b]ygs=x[a,b]syg[s,yg]$,  $w_3=x[a,b]syg=xsyg[a,b][syg,[a,b]]$, $w_4=xsyg[a,b]$, and $w_4=b'sa'=xsyg$, noting that all of these are in $XsW$ and so are generic and disjoint from $R_0\cup\{a',b',c',d'\}$.
So we can use Lemma~\ref{Lemma_absorber_one_commutator_clean} to find disjoint sets $R_1, R_2, R_3, R_4\subseteq R\setminus (R'\cup U)$ with $R_i$ $1$-absorbing $\{w_i, w_{i+1}\}$. We also have $R_0$ 1-absorbing $\{\{a',w_1, b'\}, \{c',w_5,d'\}\}$. By Lemma~\ref{Lemma_absorbers_unions_3_to_2}, there is a subset $R'\subseteq \bigcup_{i=0}^4 R_i\cup \{w_1, w_2, w_3, w_4, w_5\}$ which $1$-absorbs $\{\{a', b'\}, \{c',d'\}\}$.
\end{proof}

The following lemma gives us a collection of distinct, generic elements.
\begin{lemma}\label{Lemma_commutator_absorption_choose_yi}
Let $p\geq n^{-1/700}$ and $t\leq p^{800} n/10^{4020}$. 
Let $G$ be a group and $R$ a symmetric $p$-random subset. With high probability, the following holds: 

Let $g_1, \dots, g_t\neq \id$ be distinct and  $U\subseteq G$ with $|U|\leq  p^{800} n/10^{4020}$. There are $y_1, \dots, y_{t-1}\in G$ such that 
the elements $y_1,y_2, \dots, y_{t-1},  y_{1}g_{2}, y_{2}g_{3}, \dots, y_{t-1}g_{t}$ are distinct, generic elements in  $R\setminus U$.
\end{lemma} 
\begin{proof}
Lemma~\ref{Lemma_separated_set_random} applies with $R_A=R$, and $R_B, R_C$ arbitrary disjoint symmetric $p$-random subsets (which won't be used in the proof).
Let $g_1, \dots, g_t\neq \id$ be distinct and  $U\subseteq G$ with $|U|\leq  p^{800} n/10^{4020}$. 
Thinking of $y$ as the free variable in $G\ast F_1$, let $S_i=\{y, yg_{i}\}$. Note all $w, w'\in S_i$  are linear and strongly separable (by  part (b) since $g_i$ is generic). 
For $i=1, \dots, t$, use Lemma~\ref{Lemma_separated_set_random} to pick projections $\pi_1, \dots, \pi_i$ such that  $\pi_i(S_i)$ is separated  and disjoint from $U':=U\cup N(G)\cup\bigcup_{j<i} \pi_j(S_j)$ (noting that $|U'|\leq |U|+ |N(G)|+2t\leq 3p^{800} n/10^{4020}+10^{9000}\leq n/10^{4010}$). Now $y_1=\pi_1(y), \dots,  y_{t-1}=\pi_{t-1}(y)$ satisfy the lemma. 
\end{proof} 

The following theorem of Gallagher is key to our approach and shows that elements of the commutator subgroup can be written as a product of a small number of commutators. Its proof uses character theory.
\begin{theorem}[Gallagher, \cite{gallagher1962group}]\label{Theorem_write_commutators_as_short_products}
Let $G$ be a group. Any $g\in G'$ can be written as $g=\prod_{i=1}^{t} [a_i, b_i]$ for some $a_i, b_i\in G$ and $t\leq \log_4 |G'|\leq 10\log n$.
\end{theorem}

We now prove one of the main lemmas in this section. It strengthens several earlier lemmas and shows that we can 1-absorb any pair of elements as long as they are in the same coset of $G'$.
\begin{lemma}\label{Lemma_absorber_singleton_main}
Let $p\geq n^{-1/700}$. 
Let $H_G$ be a multiplication hypergraph, $R^1, R^2, R^3$ disjoint, symmetric $p$-random subsets of $G$  and set $R=R^1_A\cup R^2_B\cup R^3_C$. With high probability, the following holds: 

For any generic $h,k\in G_A, G_B,$ or $G_C$  with  $[h]=[k]$ and $U\subseteq V(H_G)$ with $|U|\leq p^{800}n/10^{4040}$,
 there is a set $R'\subseteq R\setminus U$ of size $\leq 400\log n$ which $1$-absorbs $\{h,k\}$.
\end{lemma}
\begin{proof}
With high probability, the properties of Lemmas~\ref{Lemma_edge_through_generic_vertex},~\ref{Lemma_absorber_one_commutator_clean}, \ref{Lemma_absorber_pair_multiply_by_commutator}, and~\ref{Lemma_commutator_absorption_choose_yi} hold. 
We just prove the lemma for $h,k\in G_B$, the other two cases follow by symmetry. Let $h,k\in G_B$ be generic with  $[h]=[k]$ and $U\subseteq V(H_G)$ with $|U|\leq p^{800}n/10^{4040}$. As in Lemma~\ref{Lemma_absorber_one_commutator},    the conclusion  trivially follows from Lemma~\ref{Lemma_edge_through_generic_vertex} if $h=k$, so assume $h\neq k$.  
Write $h=kg$ for some $g\in G'$. 
By Theorem~\ref{Theorem_write_commutators_as_short_products}, we have $g=\prod_{i=1}^{t} [a_i, b_i]$ for some $a_i, b_i\in G$ and $t\leq 10\log n$. Assume this product is as short as possible i.e. that $t$ is minimal.  For each $s=0, \dots, t-1$, define $g_s=\prod_{i=s}^{t} [a_i, b_i]$. By minimality of $t$, we have $[a_i, b_i]\neq \id$ and  $g_i\neq \id$ for all $i$. 
 Set $y_0=k$ and note that $y_0, y_0g_1$ are distinct, generic (since  $k,h$ are distinct, generic). 

Note that $t\leq 10\log n\leq n^{1/8}/10^{4010}\leq  p^{800}n/10^{4010}$ (using that $n$ is large which follows from ``with high probability'').
Use Lemma~\ref{Lemma_commutator_absorption_choose_yi} with $U'=U\cup \{y_0, y_0g_1\}$ to get elements $y_1, \dots, y_{t-1}$, noting that now $$y_0,y_1, \dots, y_{t-1},  y_{0}g_{1}, y_{1}g_{2}, \dots, y_{t-1}g_{t}$$ are distinct, generic. $U'=U\cup \{y_i, y_{i}g_{i+1}: i=1, \dots, t-1\}$.
Using Lemma~\ref{Lemma_absorber_pair_multiply_by_commutator} with $y=y_i$, $x=y_{i-1}$, $g=g_{i+1}, a=a_i, b=b_i$ for all $i$ (the conditions ``$y_ig_{i+1},y_{i-1},y_i,y_{i-1}[a_i,b_i]g_{i+1}$ distinct, generic'' coming from Lemma~\ref{Lemma_commutator_absorption_choose_yi}),  gives a set $R_i$ which $1$-absorbs 
$\{\{y_{i}g_{i+1}, y_{i-1}\}, \{y_i, y_{i-1}[a_{i},b_i]g_{i+1}\}\}=\{\{y_{i}g_{i+1}, y_{i-1}\}, \{y_i, y_{i-1}g_{i}\}\}$.
Using Lemma~\ref{Lemma_absorber_one_commutator_clean} with $x=y_{t-1}$, $a=a_{t}, b=b_{t}$ (condition ``$y_{t-1}, y_{t-1}[a_t,b_t]$ generic'' coming from Lemma~\ref{Lemma_commutator_absorption_choose_yi}),  gives a set $R_t$ which $1$-absorbs 
$\{y_{t-1}[a_{t}, b_t], y_{t-1}\}=\{y_{t-1}g_t, y_{t-1}\}$.
By enlarging the set $U$ during the application of these lemmas, we can assume that $R_1,\dots, R_{t-1}, R_t, \{y_i, y_{i}g_{i+1}: i=1, \dots, t-1\}$ are all disjoint (there's space to do this because  $|R_1\cup \dots \cup R_{i-1}|\leq 400\log n<10^{800}n^{1/8}/10^{4020}\leq  p^{800}n/10^{4030}$). 

Set $R=R_1\cup \dots\cup R_{t-1}\cup R_t\cup \{y_i, y_{i}g_{i+1}: i=1, \dots, t-1\}$. 
Set $X_i=\{y_{i}g_{i+1}, y_{i-1}\}$, $Y_i=\{y_i, y_{i-1}g_{i}\}$ for $i=1, \dots, t-1$, and $X_t=\{y_{t-1}\}$, $Y_t=\{y_{t-1}g_t\}$. Notice that the sets in the families $\{R_1, \dots, R_{t-1}\}$, $\{X_1, \dots, X_{t-1}\}$, $\{Y_1, \dots, Y_{t-1}\}$ are disjoint, that  $\bigcup R_i$ is disjoint from $\bigcup X_i, \bigcup Y_i$ and that $(\bigcup X_i)\cap (\bigcup Y_i)=\{y_i, y_{i}g_{i+1}: i=1, \dots, t-1\}$. 
By Lemma~\ref{Lemma_absorbers_unions} $R$ $1$-absorbs $\{y_0g_1,y_0\}=\{kg, k\}=\{h,k\}$.
\end{proof}

The following lemma is similar to the previous one, except it allows us to 1-absorb pairs of sets of size $2$.
\begin{lemma}\label{Lemma_absorber_pair_main}
 Let $p\geq n^{-1/700}$. 
Let $H_G$ be a multiplication hypergraph, $R^1, R^2, R^3$ disjoint, symmetric $p$-random subsets of $G$  and set $R=R^1_A\cup R^2_B\cup R^3_C$. With high probability, the following holds: 
 
  For any distinct, generic $a,b,c,d\in G_A, G_B,$ or $G_C$ with $[ab]=[cd]$ and $U\subseteq V(H_G)$ with $|U|\leq p^{800}n/10^{4050}$,
 there is a set $R'\subseteq R\setminus U$ of size $\leq 500\log n$ which $1$-absorbs $\{\{a,b\},\{c,d\}\}$.
\end{lemma}
\begin{proof}
With high probability, the properties of Lemmas~\ref{Lemma_absorber_3set} and~\ref{Lemma_absorber_singleton_main} hold. 
We just prove the lemma for $a,b,c,d\in   G_B$. The other two cases follow by symmetry. 
Suppose  that we have distinct, generic $a,b,c,d\in B(H_G)$ with $[ab]=[cd]$, and $U\subseteq G$ with $|U|\leq  p^{800}n/10^{21}$. Let $X=\{d,b\}, Y=\{c,a\}$. 
Using Lemma~\ref{Lemma_absorber_3set}, pick some $R', s$  with $R'\cup XsY\subseteq R\setminus (U\cup N(G))$ and  $R'$ 1-absorbing $\{\{a,dsc,b\}, \{c,bsa, d\}\}$.
When $dsc=bsa$, $R'\cup dsc$ 1-absorbs $\{\{a,b\}, \{c, d\}\}$ and so satisfies the lemma. So suppose $dsc\neq bsa$.
Notice that $[ab]=[cd]$ implies $[bsa]=[dsc]$ and so using Lemma~\ref{Lemma_absorber_singleton_main} we can choose a $Q\subseteq V(H_G)\setminus (U\cup R'\cup \{dsc, bsa\})$ which 1-absorbs $\{bsa, dsc\}$ ($bsa, dsc$ are generic because they are contained in $XsC$). Set $R''=R'\cup Q\cup\{bsa, dsc\}$ to get a set which 1-absorbs $\{\{a,b\},\{c,d\}\}$ by Lemma~\ref{Lemma_absorbers_unions}.
\end{proof}

\subsubsection{Distributive absorption for pairs}
In this section we prove a variety of lemmas about $h$-absorbing sets of the form $\{\{a_1,b_1\}, \dots, \{a_k,b_k\}\}$ where $[a_1b_1]=\dots=[a_kb_k]$. The following does this for $h=k-1$.
 \begin{lemma}\label{Lemma_absorber_k-1_pair}
 Let $p\geq n^{-1/700}$.
Let $H_G$ be a multiplication hypergraph, $R^1, R^2, R^3$ disjoint, symmetric $p$-random subsets of $G$  and set $R=R^1_A\cup R^2_B\cup R^3_C$. With high probability, the following holds:

Let  $k\leq 200$ and $x_1, y_1, \dots, x_k, y_k\in G_A, G_B$ or $G_C$ be distinct, generic with $[x_1y_1]=\dots=[x_ky_k]$ and let  $U\subseteq G$ with $|U|\leq p^{800}n/10^{4060}$. 
Then there is a set  $R'\subseteq R\setminus U$  of size $\leq 600k\log n$ which $(k-1)$-absorbs $\{\{x_1, y_1\}, \dots, \{x_k, y_k\}\}$.
\end{lemma}
 \begin{proof}
With high probability, the property of Lemma~\ref{Lemma_absorber_pair_main} holds. 
 Let  $x_1, y_1, \dots, x_k, y_k$ be distinct, generic with $[x_1y_1]=\dots=[x_ky_k]$ and let  $U\subseteq G$ with $|U|\leq p^{800}n/10^{4060}$. Apply the property of Lemma~\ref{Lemma_absorber_singleton_main} to get disjoint sets $R_1, \dots, R_{k-1}\subseteq R\setminus U$ which $1$-absorb $\{\{x_i,y_{i}\}, \{x_{i+1},y_{i+1}\}\}$ for each $i\in [k-1]$ (for disjointness use $U'=U\cup\bigcup_{j=1}^{i-1}R_j$ at each application. This set has size $\leq |U|+k500\log n\leq  p^{800}n/10^{4060}$). Set $R'=R_1\cup \dots\cup R_{k-1}$. For each $i$, we have matchings $R_i^-$ and $R_i^+$ with vertex sets $R_i\cup\{x_i,y_i\}$ and $R_i\cup \{x_{i+1},y_{i+1}\}$. Now the matchings $M_i=(\bigcup_{j=1}^{i-1}R_i^-)\cup (\bigcup_{j=i}^{|Y|-1}R_i^+)$ have vertex sets exactly $V(M_i)=R'\cup Y\setminus\{x_i,y_i\}$, and so they satisfy the definition of $R'$ $(k-1)$-absorbing $Y$.
\end{proof}

We'll need the following lemma which finds a common neighbour of a set of vertices in $H_k$. 
\begin{lemma}\label{Lemma_common_neighbour_generic}
 Let   $p\geq n^{-1/700}$ and 
let $H_G$ be a multiplication hypergraph, $R^1, R^2, R^3$ disjoint, symmetric $p$-random subsets of $G$  and set $R=R^1_A\cup R^2_B\cup R^3_C$. With high probability, the following holds:

Let $k\leq 200$ and $a_1, \dots, a_k\in G_A$ be distinct and generic and let  $U\subseteq V(H_G)$ with $|U|\leq p^{800}n/10^{4001}$. Then there are distinct, generic $b, c_1, \dots, c_k\in G\setminus U$ such that for each $i$, $\{a_i,b,c_i\}$ is an edge of $H_G$.
\end{lemma}
\begin{proof}
With $b$ the free variable in $G\ast F_1$, consider the set    $S=\{b, b^{-1}a_1^{-1}, \dots, b^{-1}a_k^{-1}\}$, $S_A=\emptyset$, $S_B=\{b\}$, $S_C=\{b^{-1}a_1^{-1}, \dots, b^{-1}a_k^{-1}\}$. Notice that all $w\in S$ are linear, that $S_A, S_B, S_C$ are pairwise separable (by part (b) of the definition since $a_1, \dots, a_k$ are generic), and all $w,w'\in S_C$ are weakly separable (by (b') since equalities between $w,w'$ rearrange to  $\id=a_ia_j^{-1}$ and we know $a_ia_j^{-1}\neq \id$ by distinctness). Thus the lemma follows from Lemma~\ref{Lemma_separated_set_random} applied with $U'=U\cup N(G)$.
\end{proof}

The following lemma 1-absorbs a set of $k$ pairs.
\begin{lemma}\label{Lemma_absorber_k_pair}
Let $p\geq n^{-1/700}$.
Let $H_G$ be a multiplication hypergraph, $R^1, R^2, R^3$ disjoint, symmetric $p$-random subsets of $G$  and set $R=R^1_A\cup R^2_B\cup R^3_C$. With high probability, the following holds: 
 
Let  $k\leq 200$ and $x_1, y_1, \dots, x_k, y_k\in G_A$ be distinct, generic with $[x_1y_1]=\dots=[x_ky_k]$ and let  $U\subseteq G$ with $|U|\leq p^{800}n/10^{4070}$.
Then there is a set  $R'\subseteq R$  of size $\leq 600k\log n$ which $1$-absorbs $\{\{x_1, y_1\}, \dots, \{x_k, y_k\}\}$.
\end{lemma}
\begin{proof}
With high probability, the conclusions of Lemmas~\ref{Lemma_common_neighbour_generic} and~\ref{Lemma_absorber_k-1_pair} apply.
Let $x_1, y_1, \dots, x_k, y_k\in G_A$ be distinct, generic with $[x_1y_1]=\dots=[x_ky_k]$ and let  $U\subseteq G$ with $|U|\leq p^{800}n/10^{4070}$. Use the conclusion of Lemma~\ref{Lemma_common_neighbour_generic} twice to get elements $b^x, b^y$ and $c_1^x, c_1^y, \dots, c_k^x, c_k^y$ outside $U$, such that $x_1, y_1, \dots, x_k, y_k$, $b^x, b^y$, $c_1^x, c_1^y, \dots, c_k^x, c_k^y$ are all distinct, generic and also $x_ib^xc_i^x, y_ib^yc^y_i$ are edges for all $i$ (to get distinctness of all the vertices, enlarge $U$ to include previously found vertices between the two applications). Notice that this implies that $[c_i^xc_i^y]=[x_i^{-1}b_x^{-1}y_i^{-1}b_y^{-1}]=[x_1^{-1}b_x^{-1}y_1^{-1}b_y^{-1}]$ for all $i$ (using $[x_1y_1]=\dots=[x_ky_k]$). Thus using the conclusion of Lemma~\ref{Lemma_absorber_k-1_pair}, we get a set $R$ which $(k-1)$-absorbs $\{\{c_1^x, c_1^y\}, \dots, \{c_k^x, c_k^y\}\}$. In other words we have matchings $M_1, \dots, M_k$ with  $V(M_i)=R\cup \{\{c_j^x, c_j^y\}: j\neq i\}$. 
Set  $R''=R'\cup\{b^x, b^y, c_1^x,c_1^y \dots, c_{k}^x, c_k^y\}$. Now for each $i$, $M_i\cup \{x_ib^xc_i^x, y_ib^yc_i^y\}$  is a matching with vertex set exactly $R''\cup\{x_i, y_i\}$, verifying the definition of ``$1$-absorbs''.
\end{proof}

The below was shown by Montgomery in \cite{randomspanningtree}, and is the essence of the distributive absorption approach. 
\begin{lemma}[Montgomery, \cite{randomspanningtree}]\label{lem:robustbipartite}
There is a constant $h_0$ such that for every $h\geq h_0$ there exists a bipartite graph $K$ with maximum degree at most $100$ and vertex classes $X$ and $Y\cup Y'$ with $|X|=3h$, $|Y|=|Y'|=2h$, so that the following holds. For any $Y_0\subseteq Y'$ with $|Y_0|=h$, there is a perfect matching between $X$ and $Y\cup Y_0$.
\end{lemma}
Graphs produced by this lemma are called \textbf{robustly matchable bipartite graphs}. Combining this lemma with the previous one, we can $h$-absorb sets of pairs.
\begin{lemma}\label{Lemma_absorber_half_size_pair}
Let $p\geq n^{-1/700}$. 
Let $H_G$ be a multiplication hypergraph, $R^1, R^2, R^3$ disjoint, symmetric $p$-random subsets of $G$  and set $R=R^1_A\cup R^2_B\cup R^3_C$. With high probability, the following holds: 
 
Consider sets $Y, Y'$ of disjoint pairs of generic  elements $(a_1,a_2)\in( A\setminus \id)\times (A\setminus \id)$  having $[a_1a_2]=[a_1'a_2']$ for $(a_1,a_2), (a_1',a_2')\in Y, Y'$ and also  $2h=|Y|=|Y|'\leq \frac{p^{800}n}{10^{4080}\log n}$. Let $U\subseteq G$ with $|U|\leq p^{800}n/10^{4080}$. Then there is a subset $R'\subseteq R\setminus U$ of size $\leq 800|Y|\log n$  such that $R'\cup Y'$ $h$-absorbs $Y$.
\end{lemma}
\begin{proof} 
With high probability the property of Lemma~\ref{Lemma_absorber_k_pair} holds. 

Let $Y, Y'$ be  disjoint sets of pairs of elements $(a_1,a_2)\in A\times A$  having $[a_1a_2]=[a_1'a_2']$ for $(a_1,a_2), (a_1',a_2')\in Y$ and also  $2h:=|Y|=|Y|'\leq \frac{p^{800} n}{10^{4080}\log n}$. Let $U\subseteq G$ with $|U|\leq p^{800}n/10^{4080}$.
Consider  a robustly matchable bipartite graph $D$ with $\Delta(D)\leq 100$ whose sets are $Y,Y'$ and $X$ as in Lemma~\ref{lem:robustbipartite} (here $X$ is just an abstract set of size $3h$ unrelated to the hypergraph we have). 
For all $x\in X$ use Lemma~\ref{Lemma_absorber_k_pair} to pick a set $R_x\subseteq R \setminus (U\cup Y\cup Y')$ which 1-absorbs $N(x)$. We can choose all these sets to be disjoint by enlarging $U$ to contain  the union of previously picked sets at each application (whose total size is at most $100|X|\times 700\log n\leq p^{800}n/10^{4070}$).

Letting $R'\bigcup_{x\in X} R_x$ we claim that $R'\cup Y'$ will $h$-absorb $Y$, and so satisfy the lemma. Let $Y_0\subseteq Y$ be a set of $|Y|/2=h$ pairs. Since $D$ is robustly matchable, there is a matching $M$ with vertex set $Y_0\cup Y'\cup X$. For each $x\in X$, let $xy_x$ be the matching edge of $M$ through $x$, and let $N_x$ be a matching with vertex set $R_x\cup y_x$ (which exists because $R_x$ 1-absorbs $N(x)$). Now $\bigcup_{x\in X} N_x$ is a matching with vertex set $Y_0\cup R'\cup Y'$.
\end{proof}

The following is a version of the previous lemma with more versatility with choosing $h$.
\begin{lemma}\label{Lemma_absorber_variable_size_pair}
Let $p\geq n^{-1/700}$. 
Let $H_G$ be a multiplication hypergraph, $R^1, R^2, R^3$ disjoint, symmetric $p$-random subsets of $G$  and set $R=R^1_A\cup R^2_B\cup R^3_C$. With high probability, the following holds: 
 
Consider sets $Y, Z$ of disjoint pairs of generic  elements $(a_1,a_2)\in A\times A$  having $[a_1a_2]=[a_1'a_2']$ for $(a_1,a_2), (a_1',a_2')\in Y, Z$ and $4|Y|\leq |Z|= \frac{p^{800} n}{10^{4090}\log n}$. Let $h\in \mathbb N$. Let $U\subseteq G$ with $|U|\leq p^{800}n/10^{4090}$. Then there is a subset $R'\subseteq R\setminus U, Y'\subseteq Z$ of size $\leq 800|Y|\log n$  such that $R'\cup Y'$ $h$-absorbs $Y$.
\end{lemma}
\begin{proof}
With high probability Lemma~\ref{Lemma_absorber_half_size_pair} applies. If $h> |Y|$, then the definition of ``$h$-absorbs $Y$'' is vacuous, so we can suppose that $h\leq |Y|$.
If $|Y|/2\leq h\leq |Y|$, pick subsets $\hat Y, Y'\subseteq Z$ with $|\hat Y|=2h-|Y|$ and $|Y'|=2h$. Apply Lemma~\ref{Lemma_absorber_half_size_pair} to $g$, $Y\cup \hat Y$, $Y'$, and $h$ to get a set $R'$ such that $R'\cup Y'$ $h$-absorbs $Y\cup \hat Y$. By Lemma~\ref{Lemma_absorption_subsets} (1), $R'\cup Y'$ $h$-absorbs $Y$ also.

If $h\leq |Y|/2$, pick subsets $\hat Y, Y'\subseteq Z$ with $|\hat Y|=|Y|-2h$ and $|Y'|=2|Y|-2h$. Apply Lemma~\ref{Lemma_absorber_half_size_pair} to $g$, $Y\cup \hat Y$, $Y'$, and $h'=|Y|-h$ to get a set $R'$ such that $R'\cup Y'$ $(|Y|-h)$-absorbs $Y\cup \hat Y$. By Lemma~\ref{Lemma_absorption_subsets} (2) with $t=|\hat Y|=|Y|-2h$, $R'\cup Y'\cup \hat Y$ $(|Y|-h-|\hat Y|)$-absorbs $Y$. This implies the lemma since $|Y|-h-|\hat Y|=h$.
\end{proof}

\subsubsection{Distributive absorption for singletons}
Everything in this section is almost identical to the previous one (though often a bit easier). We give constructions of $h$-absorbers of sets of vertices $Y$ contained in a coset of $G'$. The following lemma does this with $h=|Y|-1$.
\begin{lemma}\label{Lemma_absorber_k-1_singleton}
Let $p\geq n^{-1/700}$. 
Let $H_G$ be a multiplication hypergraph, $R^1, R^2, R^3$ disjoint, symmetric $p$-random subsets of $G$  and set $R=R^1_A\cup R^2_B\cup R^3_C$. With high probability, the following holds:

 For $g\in G$, let $Y\subseteq G_A, G_B$, or $G_C$ with $Y\subseteq [g]$ be a set of generic elements with $|Y|\leq \frac{p^{800} n}{10^{4100}\log n}$ and let $U\subseteq G$ with $|U|\leq p^{800}n/10^{4100}$.  Then there is a set  $R'\subseteq R\setminus U$  of size $\leq 700|Y|\log n$ which $(|Y|-1)$-absorbs $Y$.
\end{lemma}
\begin{proof}
With high probability, the property of Lemma~\ref{Lemma_absorber_singleton_main} holds. 

 Let $g\in G$ and $Y=\{y_1, \dots, y_{|Y|}\}\subseteq [g]$ with  $|Y|\leq \frac{p^{800} n}{10^{4100}\log n}$  and $U$ with $|U|\leq p^{800}n/10^{17}$. Apply the property of Lemma~\ref{Lemma_absorber_singleton_main} to get disjoint sets $R_1, \dots, R_{|Y|-1}\subseteq R\setminus U$ which $1$-absorb $\{y_{i}, y_{i+1}\}$ (for disjointness use $U'=U\cup\bigcup_{j=1}^{i-1}R_j$ at each application. This set has size $\leq |U|+|Y|500\log n\leq  p^{800}n/10^{4090}$). Set $R'=R_1\cup \dots\cup R_{|Y|-1}$. For each $i$, we have matchings $R_i^-$ and $R_i^+$ with vertex sets $R_i\cup\{y_i\}$ and $R_i\cup \{y_{i+1}\}$. Now the matchings $M_i=(\bigcup_{j=1}^{i-1}R_i^-)\cup (\bigcup_{j=i}^{|Y|-1}R_i^+)$ have vertex sets exactly $V(M_i)=R'\cup Y\setminus\{y_i\}$, and so they satisfy the definition of $R'$ $(|Y|-1)$-absorbing $Y$.
\end{proof}

The next lemma builds 1-absorbers of small sets $Y$.
\begin{lemma}\label{Lemma_absorber_k_singleton}
 Let $p\geq n^{-1/700}$.
Let $H_G$ be a multiplication hypergraph, $R^1, R^2, R^3$ disjoint, symmetric $p$-random subsets of $G$  and set $R=R^1_A\cup R^2_B\cup R^3_C$. With high probability, the following holds:

  Let  $k\leq 200$ and $s_1, \dots, s_{k}\in G_A$ be distinct, generic  with $[s_1]=\dots=[s_k]$ and let  $U\subseteq G$ with $|U|\leq p^{800}n/10^{4110}$. Then there is a set  $R'\subseteq R\setminus U$  of size $\leq 700k\log n$ which $1$-absorbs $\{s_1, \dots, s_k\}$.
\end{lemma}
\begin{proof}
With high probability, the conclusions of Lemmas~\ref{Lemma_common_neighbour_generic} and~\ref{Lemma_absorber_k-1_singleton} apply. 
Let $s_1, \dots, s_{k}\in G$ be distinct, generic  with $[s_1]=\dots=[s_k]$ and let  $U\subseteq G$ with $|U|\leq p^{800}n/10^{4110}$. Use the conclusion of Lemma~\ref{Lemma_common_neighbour_generic}  to get distinct, generic elements $b$ and $c_1, \dots, c_k\in G\setminus U$, such that also $s_ibc_i$ are edges for all $i$.  Thus using the conclusion of Lemma~\ref{Lemma_absorber_k-1_singleton}, we get a set $R$ which $(k-1)$-absorbs $\{c_1, \dots, c_k\}$. From the definition of $(k-1)$-absorbs, we have matchings $M_1, \dots, M_k$ with $V(M_i)=R\cup \{c_1, \dots, c_{i-1}, c_{i+1}, \dots, c_k\}$. Set  $R''=R'\cup\{b,c_1, \dots, c_{k}\}$. Now for each $i$, $M_i\cup s_ibc_i$  is a matching with vertex set exactly $R''\cup\{s_i\}$, verifying the definition of ``$1$-absorbs''.
\end{proof}

The next lemma uses distributive absorption to build $h$-absorbers.
\begin{lemma}\label{Lemma_absorber_half_size_singleton}
Let $p\geq n^{-1/700}$. 
Let $H_G$ be a multiplication hypergraph, $R^1, R^2, R^3$ disjoint, symmetric $p$-random subsets of $G$  and set $R=R^1_A\cup R^2_B\cup R^3_C$. With high probability, the following holds: 
 
Let $g\in G$, and consider disjoint sets $Y,Y'\subseteq [g]$ of generic elements of $G_A$ of size $2h:=|Y|=|Y'|\leq  \frac{p^{800} n}{10^{4120}\log n}$. Let  $U\subseteq G$ with $|U|\leq p^{800}n/10^{4120}$. Then there is are subset $R'\subseteq R\setminus U$ of size $\leq 10^5|Y|\log n$  such that $R'\cup Y'$ $h$-absorbs $Y$.
\end{lemma}
\begin{proof}
With high probability the property of Lemma~\ref{Lemma_absorber_k_singleton} holds. 
Let $g\in G$, every disjoint sets $Y,Y'\subseteq [g]$ of size $2h=|Y|=|Y'|\leq \frac{p^{800} n}{10^{4120}\log n}$. Let $U\subseteq G$ with $|U|\leq p^{800}n/10^{4120}$.
Consider  a robustly matchable bipartite graph $D$ with $\Delta(D)\leq 100$ whose sets are $Y,Y'$ and $X$ as in Lemma~\ref{lem:robustbipartite} (here $X$ is just an abstract set of size $3h$ unrelated to the hypergraph we have). 
For all $x\in X$ fix a set $R_x\subseteq R$ which 1-absorbs $N(x)$. We can choose all these sets to be disjoint by letting $U$ be the union of previously picked sets at each application (whose total size is at most $100|X|\times 700\log n\leq p^{800}n/10^{4110}$).

We claim that for $R':=\bigcup_{x\in X} R_x$, we have $R'\cup Y'$ $h$-absorbing $Y$, and so satisfying the lemma. Let $Y_0\subseteq Y$ be a set of $|Y|/2=h$ pairs. Since $D$ is robustly matchable, there is a matching $M$ with vertex set $Y_0\cup Y'\cup X$. For each $x\in X$, let $xy_x$ be the matching edge of $M$ through $x$, and let $N_x$ be a matching with vertex set $R_x\cup y_x$ (which exists because $R_x$ 1-absorbs $N(x)$). Now $\bigcup_{x\in X} N_x$ is a matching with vertex set $Y_0\cup R'\cup Y'$.
\end{proof}

The next lemma is a version of the previous one which allows for more flexibility in the value of $h$.
\begin{lemma}\label{Lemma_absorber_variable_size_singleton}
Let $p\geq n^{-1/700}$. 
Let $H_G$ be a multiplication hypergraph, $R^1, R^2, R^3$ disjoint, symmetric $p$-random subsets of $G$  and set $R=R^1_A\cup R^2_B\cup R^3_C$. With high probability, the following holds: 
 
Let $g\in G$, and consider disjoint sets $Y,Z\subseteq [g]$ of generic elements of $G_A$ with  $4|Y|\leq |Z|=\frac{p^{800} n}{10^{4130}\log n}$. Let $h\in \mathbb N$. Let  $U\subseteq G$ with $|U|\leq p^{800}n/10^{4130}$. Then there is are subset $R'\subseteq R\setminus U, Y'\subseteq Z$ of size $\leq 10^5|Y|\log n$  such that $R'\cup Y'$ $h$-absorbs $Y$.
\end{lemma}
\begin{proof}
With high probability Lemma~\ref{Lemma_absorber_half_size_singleton} applies.  If $h> |Y|$, then the definition of ``$h$-absorbs $Y$'' is vacuous, so we can suppose that $h\leq |Y|$. 
If $|Y|/2\leq h\leq |Y|$, pick subsets $\hat Y, Y'\subseteq Z$ with $|\hat Y|=2h-|Y|$ and $|Y'|=2h$. Apply Lemma~\ref{Lemma_absorber_half_size_singleton} to $g$, $Y\cup \hat Y$, $Y'$, and $h$ to get a set $R'$ such that $R'\cup Y'$ $h$-absorbs $Y\cup \hat Y$. By Lemma~\ref{Lemma_absorption_subsets} (1), $R'\cup Y'$ $h$-absorbs $Y$ also.

If $h\leq |Y|/2$, pick subsets $\hat Y, Y'\subseteq Z$ with $|\hat Y|=|Y|-2h$ and $|Y'|=2|Y|-2h$. Apply Lemma~\ref{Lemma_absorber_half_size_singleton} to $g$, $Y\cup \hat Y$, $Y'$, and $h'=|Y|-h$ to get a set $R'$ such that $R'\cup Y'$ $(|Y|-h)$-absorbs $Y\cup \hat Y$. By Lemma~\ref{Lemma_absorption_subsets} (2) with $t=|\hat Y|=|Y|-2h$, $R'\cup Y'\cup \hat Y$ $(|Y|-h-|\hat Y|)$-absorbs $Y$. This implies the lemma since $|Y|-h-|\hat Y|=h$.
\end{proof}

\subsubsection{Main absorption lemma}
The following lemma summarises everything from Section~\ref{sec:absorbers} that we use in other sections.
\begin{lemma}\label{lem:mainabsorptionlemma}
Let $p \geq n^{-1/700}$. 
Let $H_G$ be a multiplication hypergraph, $R^1, R^2, R^3$ disjoint, symmetric $p$-random subsets of $G$  and set $R=R^1_A\cup R^2_B\cup R^3_C$. With high probability, the following holds:

   Let $U\subseteq G$ with $|U|\leq p^{800}n/10^{4140}$. Let $\ast\in \{A,B,C\}$. Suppose that $Y$ and $h$ are one of the following:
\begin{enumerate}[label=(\arabic*)]
    \item For some $g\in G$, $Y\subseteq [g]\cap \ast\setminus N(G)$  with $|Y|\leq \frac{p^{800}|G'|}{10^{4140}\log n}$, and $h\leq |Y|$.
    
    \item  For some $g\in G$, $Y\subseteq [g]\cap \ast\setminus N(G)$  with $|Y|\leq \frac{p^{800}n}{10^{4140}\log n}$, and  $h=|Y|-1$.
    
    \item For some generic $a_{\phi}$, we have $Y\subseteq G\setminus G$  a set  of $|Y|\leq \frac{p^{800}n}{10^{4140}\log n}$ disjoint pairs of generic elements, with $[a_1a_2]=[a_\phi]$ for $(a_1,a_2)\in Y$, and   $h\leq |Y|$. 
\end{enumerate}
Then there  is a subset $R'\subseteq R\setminus U$ of size $\leq 10^6|Y|\log n$  which $h$-absorbs $Y$.
\end{lemma}

\begin{proof}
With high probability, Lemmas~\ref{Lemma_find_paired_vertices},~\ref{Lemma_absorber_k-1_singleton}, ~\ref{Lemma_absorber_variable_size_pair}, and~\ref{Lemma_absorber_variable_size_singleton} apply. Additionally when $|G'|\geq 10^{4140}p^{-{800}}\log n$, we can assume that $|R_i\cap [g]|\geq p|G'|/2\geq 4\times p^{800} |G'|/10^{4140}\log n$ for all cosets $[g]$ and $i=1,2,3$ (using Chernoff's bound).
\begin{enumerate}[label=(\arabic*)]
\item We can suppose that $|G'|\geq 10^{4140}p^{-{800}}\log n$, since otherwise $Y$ would have to be empty and the conclusion is vacuous.  In this case, we have that for every $g$, there is a subset $Z_g\subseteq [g]$ of size $4\times p^{800} |G'|/10^{4140}\log n$. 
Let $g\in G$, $Y\subseteq [g]\cap \ast$ of size $|Y|\leq \frac{p^{800}|G'|}{10^{4140}\log n}$, and $h\leq |Y|$. The result follows from Lemma~\ref{Lemma_absorber_variable_size_singleton} applied to $Y,Z_g,U$.

\item This is strictly weaker than Lemma~\ref{Lemma_absorber_k-1_singleton}.

\item   We'll just prove the lemma when $Y\subseteq G_A\times G_A$, the other cases are symmetric.  Use Lemma~\ref{Lemma_find_paired_vertices} in order to  find a set $Z\subseteq (R_A\setminus U)\times (R_A\setminus U)$ of $4\times \frac{p^{800}n}{10^{4140}\log n}$ pairs $(a,a')$ of generic elements having $aa'=a_\phi$.
Using Lemma~\ref{Lemma_absorber_variable_size_pair}, there is a $R'\subseteq R\setminus (U\cup Z)$ and  $Z'\subseteq Z$ with $R'\cup Z'$  $h$-absorbing  $Y$.

\end{enumerate}
\end{proof}

\subsection{Absorbing coset-paired sets}
In this section we prove results about absorbing coset-paired sets. The main lemmas in this section are all along the lines of ``given a set $Q$, there exists a set $R'$ with the property that for every coset-paired $S\subseteq Q$, there is a matching with vertex set $R'\cup (Q\setminus S)$''.

For all $g\in G$, define $i_{[g]}=2$ if $[g]$ is self-paired and $1$ otherwise. The following lemma allows us to absorb the complement of a coset-paired set.
\begin{lemma}\label{Lemma_absorb_coset_paired_one_side}
Let $p \geq n^{-1/700}$. 
Let $H_G$ be a multiplication hypergraph, $R^1, R^2, R^3$ disjoint, symmetric $p$-random subsets of $G$  and set $R=R^1_A\cup R^2_B\cup R^3_C$. With high probability, the following holds:

Let $s\in (0,1]$. Let $m\in \mathbb N$. Let $U\subseteq G$ and $Q\subseteq G_\diamond$ (for some $\diamond\in \{A,B,C\}$) be disjoint, and suppose $|U|\leq p^{800}n/10^{4150}$.
   Suppose also that $Q$ is generic, $|Q|\leq \frac{p^{800}n}{10^{4150}\log n}$, and that $Q$ satisfies:
\begin{enumerate}[label=(\roman*)]
\item For all pairs $g,h$ either $Q\cap [g]=Q\cap [h]=\emptyset$ or $|Q\cap [g]|,|Q\cap [h]| \geq i_{[g]}\lceil 12s |G'| \rceil$.
\item If $|G'|> s^{-1}/12$ then for all pairs $g,h$  $|Q\cap [g]|,|Q\cap [h]|\leq \frac{p^{800}|G'|}{10^{4150}\log n}$
\item If $|G'|\leq s^{-1}/12$, then $Q$ doesn't intersect any self-paired cosets. 
\end{enumerate}
Then, there exists an $R'\subseteq R\setminus U$ such that for all coset-paired, $\lceil s |G'|\rceil$-coset-bounded $S\subseteq Q$ of size $m$, there is a matching with vertex set $R'\cup (Q\setminus S)$.
\end{lemma}
\begin{proof}
With high probability Lemma~\ref{lem:mainabsorptionlemma} applies to $R$. Let $s, m,U,Q$ be as in the lemma. Let $K$ be the set of cosets $[g]$ with  $Q\cap [g]$ nonempty.  %noting that none of these are self-paired.
Partition $Q=Q_1\cup Q_2$ where, for each coset $[g]\in K$, we have $|Q_1\cap [g]|=i_{[g]}\lceil 12s |G'| \rceil$ (and so $|Q_2\cap [g]|=i_{[g]}|Q\cap[g]|- \lceil 12s |G'| \rceil$). Note that $Q_1$ is coset-paired (since it has even intersection with self-paired cosets and intersection $\lceil 12s |G'| \rceil$ with other pairs of cosets) and so we can partition it into a set of $|Q_1|/2\leq|Q|\leq  \frac{p^{800}n}{10^{4150}\log n}$ pairs $P_1$.  
 For each $g$, let $h_g=|Q\cap[g]|- i_{[g]}\lceil s |G'| \rceil$, noting that this is nonnegative by (i). Use Lemma~\ref{lem:mainabsorptionlemma} (1) or (2) to pick a subset $R_g\subseteq R$ which $h_g$-absorbs $Q\cap [g]$ (if $|G'|\geq s^{-1}/12$, then, using condition (ii),  part (1) of that lemma gives this. Otherwise, we have that $Q$ intersects no self-paired cosets and $\lceil s |G'| \rceil=1$ and so part (2) gives it). Set $h_P=\sum_{[g]\in K} i_{[g]}\lceil s |G'|\rceil$, noting that this is even. Also note that we can assume that $m$ is even and $\leq h_g$ (otherwise there could be no  coset-paired,  $\lceil s |G'|\rceil$-coset-bounded $S$ of size $m$ contained in $Q$).
 So, we can use Lemma~\ref{lem:mainabsorptionlemma} (3) to pick a subset $R_P$ which $\frac12(h_P-m)$-absorbs $P_1$. We can ensure that these are all disjoint by enlarging $U$ as we choose them (the maximum total size of sets we need to avoid is $|U|+\sum_{[g]\in K}|R_g|\leq |U|+\sum_{[g]\in K}10^6\log n|Q\cap[g]|\leq |U|+10^6\log n|Q|  \leq p^{800}n/10^{4150}+ 10^6\log n(\frac{p^{800}n}{10^{4150}\log n})\leq p^{800}n/10^{4140}$).
Let $R'=R_P\cup \bigcup P_1 \cup \bigcup_{g\in K}R_g$.

Consider some coset-paired, $\lceil s |G'|\rceil$-coset-bounded $S\subseteq Q$ of size $m$. Since $S$ is coset-paired, $|S\cap [g]|/i_{[g]}$ is an integer for all $[g]$.
\begin{claim}
There is a subset  $P_1'\subseteq P_1$ of pairs disjoint from $S$ which have  $\lceil s |G'| \rceil-|S\cap [g]|/i_{[g]}$ pairs intersecting each $[g]\in K$. 
\end{claim}
\begin{proof}
First suppose that $\lceil s |G'| \rceil=1$. Note that for each $[g]\in K$ this means that we need to select at most one pair intersecting $[g]$, and we only need to do so when $S\cap [g], S\cap [\phi(g)]=\emptyset$. In this case we have $|Q_1\cap [g]|=|Q_1\cap [\phi(g)]|= i_{[g]}|\lceil 12s|G'|\rceil$ and so there is at least one pair in $P_1$ contained $[g]\cup [\phi(g)]$ (and choosing that pair works).

Next suppose that $\lceil s |G'| \rceil>1$, which implies that $s |G'|\geq \lceil s |G'| \rceil-1\geq \lceil s |G'| \rceil/2$. Then $|Q_1\cap [g]|= i_{[g]}\lceil 12s|G'|\rceil\geq 12s|G'|\geq 6\lceil s |G'| \rceil$. Thus there are at least $3\lceil s |G'| \rceil$ pairs contained in $[g]\cup [\phi(g)]$. At most $2\lceil s |G'| \rceil$ of them can intersect $S\cap ([g]\cup [\phi(g)])$ and so we can choose $\lceil s |G'| \rceil-|S\cap [g]|/i_{[g]}$ of them disjointly from $S$.
\end{proof} 

Note that $|P_1'|=\frac12|\bigcup P_1'|=\frac12\sum_{[g]\in K}i_{[g]}(\lceil s |G'| \rceil-|S\cap [g]|/i_{[g]}) =\frac12(h_P-m)$. 
By the property of $R_P$, there is a matching $M_P$ with vertex set $R_P\cup \bigcup P_1'$. For each $g\in K$, note that $|[g]\cap (S\cup \bigcup P_1')|=i_{[g]}\lceil s |G'| \rceil$ and so $|[g]\cap (Q\setminus(S\cup \bigcup P_1'))|=h_g$. By the property of $R_g$, there is a matching $M_g$ with vertex set $R_g\cup ([g]\cap (Q\setminus(S\cup \bigcup P_1')))$. The union of these matchings has vertex set $(R_P\cup \bigcup P_1') \cup \bigcup_{[g]\in K}R_g\cup ([g]\cap (Q\setminus(S\cup \bigcup P_1'))) =R'\cup \bigcup P_1'\cup (Q\setminus(S\cup \bigcup P_1'))  = R'\cup (Q\setminus S)$ as required. 
\end{proof}

The following simple lemma covers generic sets by matchings. 
\begin{lemma}\label{Lemma_matching_through_generic_vertices}
Let $p\geq n^{-1/700}$. 
Let $H_G$ be a multiplication hypergraph, $R^1, R^2, R^3$ disjoint, symmetric $p$-random subsets of $G$  and set $R=R^1_A\cup R^2_B\cup R^3_C$. With high probability, the following holds: 

For any $U\subseteq V(H_G)$ and any generic $X\subseteq V(H_G)$   with $|X|,|U|\leq p^{800}n/10^{4010}$,  
 there is a matching $M$ of size $|X|$ in $H_G$ covering $X$ and having all other vertices in $R\setminus U$.
\end{lemma}
\begin{proof}
With high probability, the property of Lemma~\ref{Lemma_edge_through_generic_vertex} applies i.e. for any generic $v\in V(H_G)$  and $U\subseteq V(H_G)$ with $|U|\leq p^{800}n/10^{4000}$, 
 there is an edge $e$ of $H_G$ passing through $v$ and having the other two vertices in $R\setminus U$. Now let $U\subseteq V(H_G)$ and any generic $X\subseteq V(H_G)$   with $|X|,|U|\leq p^{800}n/10^{4010}$. Applying the property of Lemma~\ref{Lemma_edge_through_generic_vertex}  to each $v\in X$ we get edges $e_v$ passing through $v$ and having other vertices in  $R\setminus U$. By enlarging $U$ as we choose these to include previously found vertices, we can ensure that all $e_v$ are disjoint i.e. they give us a matching like the lemma asks for. 
\end{proof}

 The following is a variant of Lemma~\ref{Lemma_absorb_coset_paired_one_side} where $Q$ is a random set. 
 \begin{lemma}\label{Lemma_absorb_coset_paired_one_side_random}
 Let $p \geq n^{-1/701}$ and $q\leq  \frac{p^{800}}{10^{4160}\log n}$. 
Let $H_G$ be a multiplication hypergraph, $R^1, R^2, R^3$ disjoint, symmetric $p$-random subsets of $G$ and $Q$ a disjoint $q$-random subset. Set $R=R^1_A\cup R^2_B\cup R^3_C$. With high probability, the following holds: 
 
 Let $s\leq \frac{q^2p^{800}}{10^{4160}\log^{3} n}$.
For every $m \in \mathbb{N}$ and $U\subseteq G$ with $|U|\leq q^2p^{800}n/10^{4160}$, there is a $U'$ with $U\subseteq U'\subseteq U\cup Q$, $|U'|\leq  5q^{-1}|U|+50s^{-1}$, and $R'\subseteq R\setminus U$ such that for all coset-paired, $\lceil s |G'|\rceil$-coset-bounded $S\subseteq Q\setminus U'$ of size $m$, there is a matching with vertex set $R'\cup (Q\setminus (S\cup U'))$.
\end{lemma}
\begin{proof}
Randomly split each $R^i$ into two disjoint, symmetric $p/2$-random sets $R^i_-$, $R^i_+$. Set $R^-=R^1_{-A}\cup R^2_{-B}\cup R^3_{-C}$ and $R^+=R^1_{+A}\cup R^2_{+B}\cup R^3_{+C}$. Note that the following hold with high probability.
\begin{enumerate}[label=(A\arabic*)]
\item $R^-$ satisfies Lemmas~\ref{Lemma_absorb_coset_paired_one_side} while $R^+$ satisfies Lemma~\ref{Lemma_matching_through_generic_vertices}.
\item When $|G'|\geq q^{-2}\log^2 n$, then  $|Q\cap [g]|\in [q|G'|/2, 2q|G'|]$ for all $g\in G$ (from Chernoff's bound).
\item $|Q|\leq qn+\sqrt n \log n \leq \frac{p^{800}n}{10^{4150}\log n}$ (from Chernoff's bound).
\end{enumerate}
Now let $s\leq \frac{q^2p^{800}}{10^{4160}\log^{3} n}\leq q^{2}\log^{-2} n/12$, $m \in \mathbb{N}$, and $U\subseteq G$ with  $|U|\leq q^2p^{800}n/10^{4160}$. When $|G'|\leq  s^{-1}/12$, let $U_1$ be the union of self-paired cosets, noting that $|U_1|\leq 30|G'|\leq 30s^{-1}$. When $|G'|>  s^{-1}/12$, let $U_1$  be the union of cosets $[g]$ and their pairs for which $|(U\cup N(G))\cap [g]|\geq q|G'|/10$, noting that $|U_1|\leq 2|G'|\frac{|U|+|N(G)|}{q|G'|/10}$. Set $U'=U\cup N(G)\cup U_1$, noting that $|U'|\leq 5q^{-1}|U|+ 6q^{-1}|N(G)|+30s^{-1}\leq 5q^{-1}|U|+ 6q^{-1}10^{9000}+30s^{-1}\leq 5q^{-1}|U|+50s^{-1}$. Set $Q_1=Q\setminus U'$. 
Let $Q_2\subseteq Q_1$ be the subset formed by deleting all pairs of cosets for which $Q_1\cap [g]$ or $Q_1\cap [\phi(g)]$ is empty.
Note that $Q_2$ satisfies all of the following properties.
\begin{enumerate}
\item $Q_2$ is generic (since it's disjoint from $N(G)$).
\item $| Q_2|\leq \frac{p^{800}n}{10^{4150}\log n}$ (by (A3) and $Q_2\subseteq Q$).
\item When $|G'|\leq  s^{-1}/12$, $ Q_2$ doesn't intersect self-paired cosets (since it's disjoint from $U_1$).
\item When $|G'|\leq   s^{-1}/12$, for all $g\in G$ we have that either $Q_2\cap [g]=Q_2\cap [h]=\emptyset$ or $i_{[g]}\lceil12s|G'|\rceil=1 \leq |Q_2\cap [g]|, |Q_2\cap [h]|$ (by construction of $Q_2$ from $Q_1$).
\item When $|G'|> s^{-1}/12$, for all $[g]$ either  $Q_2\cap [g]=Q_2\cap [h]=\emptyset$ or we have $i_{[g]}12s|G'|\leq 24s|G'| \leq q|G'|/4\leq q|G'|/2-q|G'|/10 \leq |Q_2\cap [g]|, |Q_2\cap [h]|\leq 2q|G'|\leq \frac{p^{800}|G'|}{10^{4150}\log n}$ (by (A2), since $Q_2$ is disjoint from $U_1$, and since $s\leq \frac{q^2p^{800}}{10^{4160}\log^{3} n}$, $q\leq  \frac{p^{800}}{10^{4160}\log n}$).
\end{enumerate}
Thus Lemma~\ref{Lemma_absorb_coset_paired_one_side} applies to $Q_2$. This gives us a set $R'^-\subseteq R^-\setminus U$.
By Lemma~\ref{Lemma_matching_through_generic_vertices}, there is a matching $M_1$ covering $Q_1\setminus Q_2$ with $V(M_1)\setminus (Q_1\setminus Q_2)\subseteq R^+\setminus U$. Set $R'=R'^-\cup (M_1\setminus Q)\subseteq R\setminus U$. This is possible as $Q_1\setminus Q_2\subseteq Q$ and (A3).
\par Now consider a  coset-paired, $\lceil s |G'|\rceil$-coset-bounded $S\subseteq Q\setminus U'=Q_1$ of size $m$. Note that we must have $S\subseteq Q_2$ because $S$ is coset paired (and for every $g\in Q_1\setminus Q_2$ we have $[\phi(g)]\cap Q_1$ empty). Therefore, Lemma~\ref{Lemma_absorb_coset_paired_one_side} gives us a matching $M_2$ with vertex set $R'^-\cup (Q_2\setminus S)$. Combining this with $M_1$ gives a matching with vertex set $R'^-\cup (Q_2\setminus S)\cup (M_1\setminus Q)\cup (Q_1\setminus Q_2)= R'\cup (Q_1\setminus S)=R'\cup (Q\setminus (S\cup U'))$ as required.
\end{proof}

The following shows that large coset-paired matchings exist inside random sets. 
\begin{lemma} \label{Lemma_large_coset_paired_matching}
Let $p \geq n^{-1/700}$. 
Let $H_G$ be a multiplication hypergraph, $R^1, R^2, R^3$ disjoint, symmetric $p$-random subsets of $G$  and set $R=R^1_A\cup R^2_B\cup R^3_C$. With high probability, the following holds:

Let  $U\subseteq G$  with  $|U|\leq p^{800}n/10^{4140}$. For any $k\in [1, |G'|]$,
there is a coset-paired, $k$-coset-bounded matching of size $\frac{p^{800}nk}{10^{4140}|G'|}$ in $R\setminus U$.
\end{lemma}
\begin{proof}
With high probability, Lemma~\ref{Lemma_separated_set_random} applies. 
When  $|G'|\leq n/\log^{10^{20}}n$, consider the set $$S=\{x, x^{-1}a_{\phi}, y, b_{\phi}y^{-1}, y^{-1}x^{-1}, yc_{\phi}x\}\subseteq G\ast F_2.$$ 

When $|G'|> n/\log^{10^{20}}n$, consider the set $$S=\{x_1, x_2a_{\phi}, y_1, y_2b_{\phi}, y_1^{-1}x_1^{-1}, y_2c_{\phi}x_2\}\subseteq G\ast F_4.$$ 

Note that in either case all words in $S$ are linear and all pairs of words are separable (using part (a) or (c) of the definition of ``separable''. For checking (c), we use that for any $(w,w')=(x, x^{-1}a_{\phi}), (y, b_{\phi}y^{-1}), (y^{-1}x^{-1}, yc_{\phi}x)$ we have $\pi_0(ww'), \pi_0(w^{-1}w')\in \{a_\phi, b_\phi, c_\phi, a_\phi^{-1}, b_\phi^{-1}, c_\phi^{-1}\}$, which all satisfy (c) from Lemma~\ref{lem:pairingsexist}).

Let  $U\subseteq G$  with  $|U|\leq p^{800}n/10^{4120}$. If $|G'|\leq p^{800}n/10^{4120}$ add all self-paired cosets to $U$ in order to get a set $U'$ with $|U'|\leq p^{800}n/10^{4110}$ (otherwise set $U'=U$). 
Use Lemma~\ref{Lemma_separated_set_random} to get a projection $\pi$ which separates $S$ and has $\pi(S)\subseteq R\setminus U'$. This gives us a coset-paired matching of size $2$ as in the lemma (whose edges are $(\pi(x), \pi(y), \pi(y^{-1}x^{-1}))$ and $(\pi(x^{-1}a_{\phi}), \pi(b_{\phi}y^{-1}), \pi(yc_{\phi}x))$). To get one of size $\frac{p^{800}nk}{10^{4140}|G'|}$, keep selecting multiple matchings like this, enlarging $U$ at every step in order to keep them disjoint (we can do this as long as $3|G'||M|k^{-1}\leq p^{800}n/10^{4130}$. Indeed for a matching $M$, letting $U_M$ be the union of cosets $[g]$ having $|V(M)\cap[g]|\geq  k$, note that $|U_M|/|G'|\leq |V(M)|/k$ which gives $|U_M|\leq 3|G'||M|/k$. For any pair of edges outside $U\cup U_M$, adding them to $M$ gives a bigger matching like we want). 
\end{proof}

Now we arrive at the main lemma of this section.
The following again shows that we can absorb the complement of a coset-paired set. Unlike Lemmas~\ref{Lemma_absorb_coset_paired_one_side} and~\ref{Lemma_absorb_coset_paired_one_side_random}, this one allows the coset-paired set to have variable size. 
 \begin{lemma}\label{lem:cosetpairedabsorber}
 Let $q,p \geq n^{-1/701}$ with $q\leq \frac{p^{801}}{10^{4170}\log n}$. 
Let $H_G$ be a multiplication hypergraph, let $R^1, R^2, R^3, Q_1, Q_2, Q_3$ be disjoint, symmetric subsets of $G$ with  $R^1, R^2, R^3$ $p$-random  and $Q_1, Q_2, Q_3$ $q$-random subsets. Set $R=R^1_A\cup R^2_B\cup R^3_C$ and $Q=Q^1_A\cup Q^2_B\cup Q^3_C$. With high probability, the following holds: 
 
 Let $s\leq\frac{q^{1600}p^{10^{7}}}{10^{10^{7}}\log^{2400} n}$, $U\subseteq G$ with $|U|\leq q^{9}p^{800}n/10^{4170}$. There is a $U'\subseteq U$ with $|U'|\leq 500q^{-3}|U|+10^{10}s^{-3}$, and $R'\subseteq R\setminus U$ such that for all coset-paired, $\lceil s |G'|\rceil$-coset-bounded, balanced $S\subseteq Q\setminus U'$, with $|S|\leq s^2n$ there is a matching with vertex set $R'\cup (Q\setminus (S\cup U'))$.
\end{lemma}
\begin{proof}
Set $m:= 8\lceil sn\rceil$ and $s':=s^{1/800}10^{4140}$.
Split $R$ into three disjoint, symmetric $p/3$-random subsets  $R_1, R_2, R_3$  of $V(H_G)$ (by placing each $\hat g$ in $R_1/R_2/R_3$ with probability $1/3$, and making these choices independently of those for the random sets $R^1, R^2, R^3$ in the lemma. So now for all $i$ $R_i\cap R^1_A$, $R_i\cap R^2_B$, $R_i\cap R^3_C$ are three disjoint, symmetric $p/3$-random subsets of $G$ partitioning $R_i$). Lemma~\ref{Lemma_absorb_coset_paired_one_side_random} applies to each pair $R'=R_i, Q'=Q_i$ for $i=1,2,3$ (with $R'^1=R_i\cap R^1_A$, $R'^2=R_i\cap R^2_B$, $R'^3=R_i\cap R^3_C$). Lemma~\ref{Lemma_large_coset_paired_matching} applies to $Q$. Let $U_0:=U$ be as in the lemma. 
For $i=1,2,3$ build $U_1, U_2, U_3, R_1, R_2, R_3$ by applying Lemma~\ref{Lemma_absorb_coset_paired_one_side_random} to $R_i, Q_i$ with $U=U_{i-1}$, $m=m, s=4s'$. The end result is sets $R_1'\subseteq R_1, R_2'\subseteq R_2, R_3'\subseteq R_3$ and a set $U'$ with   $|U'|\leq  500q^{-3}|U|+10^{10}s^{-3}$ such that  for all coset-paired, $\lceil 4s' |G'|\rceil$-coset-bounded $S_i\subseteq Q_i\setminus U'$ of size $m$, there  are matchings with vertex sets $R_i'\cup (Q_i\setminus (S_i\cup U'))$. Set $R'=R_1'\cup R_2'\cup R_3'$.
Use Lemma~\ref{Lemma_large_coset_paired_matching} with $k=\lceil s' |G'|\rceil$ to find a coset-paired, $\lceil s' |G'|\rceil$-coset-bounded matching $M_0$ of size $\frac{p^{800}nk}{10^{4140}|G'|}\geq q^{800}s'n/10^{4140}\geq 8m$ in $Q_1\cup Q_2\cup Q_3\setminus U'$. 

Now, consider a  coset-paired, $\lceil s|G'|\rceil$-coset-bounded, balanced $S\subseteq Q\setminus U'$ with $|S|\leq s^2 n$. Let $m_0=|S\cap G_A| =|S\cap G_B|=|S\cap G_C|\leq s^2n$, noting that these are even by coset-pairedness.
\begin{claim}
There is a coset-paired submatching $M_0'\subseteq M_0$ 
of size $m-m_0$ disjoint from $S$. Additionally when $|G'|\leq s^{-1}$, then $M_0'$ and  $S$ never intersect the same cosets.
\end{claim}
\begin{proof}
When $|G'|\leq s^{-1}$ let $S'$ be the union of cosets intersecting $S$, otherwise let $S'=S$. Note that in either case  $|S'|\leq |S|s^{-1}\leq sn\leq m$. Since $|M_0|\geq 8m$, there is a paired submatching  $M_0'\subseteq M_0$  of size $m-m_0$ disjoint from $S'$.
\end{proof}

Now set $S_{1}':=(S\cup M_0')\cap G_{A}$, $S_{2}':=(S\cup M_0')\cap G_{B}$, $S_{3}':=(S\cup M_0')\cap G_{C}$ and note that these are coset-paired, $\lceil 4s' |G'|\rceil$-coset-bounded sets of size $m$. By the properties of $R_1', R_2', R_3'$, we get matchings $M_1, M_2, M_3$ with vertex sets $V(M_i)=R_i'\cup (Q_i\setminus (S'_i\cup U'))$. Now $M=M_0'\cup M_1\cup M_2\cup M_3$ is a matching with vertex set $V(M_0')\cup \bigcup_{i=1}^3R_i'\cup (Q_i\setminus (S'_i\cup U'))=R'\cup V(M_0')\cup (Q\setminus(S\cup M_0'\cup U')) =R'\cup (Q\setminus (S\cup U'))$ as required. 
\end{proof}

\subsection{Absorbing zero-sum sets}
The goal of this section is to prove a lemma which can absorb arbitrary balanced zero-sum sets (Lemma~\ref{lem:zerosumabsorptionnonabelian}). To do this we first prove results about covering zero-sum sets using coset-paired matchings. 

\subsubsection{Covering zero-sum sets}\label{sec:zerosumelimination}
Here we prove Lemma~\ref{lem:zerosumeliminate}, which roughly states that given a random vertex subset $R$ of $H_G$ and a small zero-sum set $S$, there exists a coset-paired set $R'\subseteq R$ such that $S\cup R'$ contains a perfect matching. As explained in Section~\ref{sec:proofoutline}, this will allow us to reduce the task of absorbing arbitrary zero-sum sets to absorbing coset-paired sets. 
\par The following lemma does the majority of the work for this section. It allows us to find the desired set $R'$ iteratively by reducing the size of the set of vertices we need to cover by $3$ at every step. 
\begin{lemma}\label{Lemma_cover_6_set}
Let $p\geq n^{-1/700}$. 
Let $H_G$ be a multiplication hypergraph, $R^1, R^2, R^3$ disjoint, symmetric $p$-random subsets of $G$  and set $R=R^1_A\cup R^2_B\cup R^3_C$. With high probability, the following holds.

Let $F\subseteq V(H_G)$ be a balanced $\phi$-generic set of size $6$, $h\in [\prod F]$, and $U\subseteq V(H_G)$ with $|U|\leq p^{800}n/10^{4001}$. Then there is a matching $M$ in $H_G$ of size $15$ and a disjoint set $\{a,b,c\}$ such that $F\subseteq V(M)$, $\{a,b,c\}\cup V(M)\setminus F\subseteq R\setminus U$, $\{a,b,c\}\cup V(M)\setminus F$ is coset-paired and $abc=h$. Furthermore, if $|G'|\leq \log^5n/p^{2000}$, then for each $\diamond\in \{A,B,C\}$, and $g\in G$ we have that $|(\{a,b,c\}\cup V(M)\setminus F)\cap G_\diamond \cap [g]|\leq 1$.
\end{lemma}
\begin{proof}
With high probability, the property of Lemma~\ref{Lemma_separated_set_random} holds.  
Let $F\subseteq V(H_G)$ be a balanced $\phi$-generic set of size $6$, $h\in [\prod F]$, and $U\subseteq V(H_G)$ with $|U|\leq p^{800}n/10^{4001}$. 
Let $F=\{a_1, a_2, b_1, b_2, c_1, c_2\}$ with $a_i\in G_A, b_i\in G_B, c_i \in G_C$. Let $k=ha_2^{-1}c_1^{-1}c_2^{-1}b_2^{-1}b_1^{-1}a_1^{-1}$ and note that by the assumption that $h\in [\prod F]$, we know that $k\in G'$.
Depending on whether $|G'|\geq \log^5 n/p^{2000}$ or not, consider the set of words $S$ given in Figure~\ref{Figure_justification_small_commutator} or~\ref{Figure_justification_large_commutator}, noting that blue/green/red vertices represent a partition as $S_A/S_B/S_C$. Note that in either case, all the words in $S$ are linear and all pairs of words in $S$ are weakly separable and words in $S$ coming from different $S_A/S_B/S_C$ are strongly separable (see the figure captions for justification). In fact, it will be the case that every pair of words are strongly separable, meaning we never use condition $(b')$, so this distinction will not be essential in the justification.
\par For $\ast=A,B,C$, let $T_{\ast}=\{w^{-1}w': w,w'\in S_{\ast} \text{ and $w,w'$ do not have the same free variables}\}$, noting that $|T_{\ast}|\leq \binom{|S_{\ast}|}2=\binom{14}2=91$ and that elements of $T_{\ast}$ are all linear in at least one variable.
If $|G'|> \log^5 n/p^{2000}$, let $U'=U\cup F$. If $|G'|\leq \log^5 n/p^{2000}$, then let $U'$ be $U$ together with all the self-paired cosets and all the cosets intersecting $F$, and $G'$. Note that in both cases, $|U'|\leq |U|+36\log^5 n/p^{2000}\leq p^{800}n/10^{4000}$.

 Use Lemma~\ref{Lemma_separated_set_random} to get a projection $\pi$ with $\pi(S\cup T_A\cup T_B\cup T_C)\subseteq  R\setminus (U\cup F)$ that separates $S\cup T_A\cup T_B\cup T_C$ and ensures that $S_A\subseteq R_1$, $S_B\subseteq R_2$, $S_C\subseteq R_3$ (so the matching is contained in $R$). 
 \par First, we prove the lemma without the ``furthermore'' part.
 Since all $w,w'\in S$ are separable, this means that all the vertices of $\pi(S)$ are distinct.  Recall that $S_A/S_B/S_C$ are the blue/green/red vertices in the figures. Note that  $F\cup \pi(S)$ has a partition into a matching $M$ (given by the black triangles in Figures~\ref{Figure_covering_small_commutator} and \ref{Figure_covering_large_commutator}) and a set $\{a,b,c\}$ having $abc=h$ (given by the pink triangle in the same figures where $a$ is the top left vertex, $b$ is the bottom vertex, and $c$ is the top right vertex). We claim that this matching $M$ together with the set $\{a,b,c\}$ satisfy the lemma. The things that need to be verified by inspecting the figure are as follows.  
\begin{itemize}
    \item Each black triangle has that if its vertices are multiplied in the order blue/green/red ($G_A/G_B/G_C$), we obtain $\id$. This shows that the black triangles are in fact edges of $H_G$.
    \item The product of the top left, bottom, and top right vertices of the yellow (central) triangle is $h$ (here, use the definition of $k$). This gives that $abc=h$.
    \item The product of the blue/green/red edges belongs to $[a_\phi]/[b_\phi]/[c_\phi]$. In the case where $|G'|\geq 10^{-9}n$, note that we select the black (free) variables from $G'$, hence they can be ignored while performing this check. This shows that $\{a,b,c\}\cup V(M)\setminus F$ is coset-paired.
\end{itemize}
For the ``furthermore'' part, we have that $|G'|\leq \log^5 n/p^{2000}$. For $\ast=A,B,C$, notice that for all $w,w'\in S_{\ast}$ we either have that $(\pi(w), \pi(w'))$ is a pair, or there is a free variable which appears in one of $w/w'$, but not both (to check this, note that for any vertices $w,w'$ of the same colour in Figure~\ref{Figure_covering_small_commutator}, either $w,w'$ are joined by an edge or $w$ and $w'$ have different combinations of black letters). In the first case we have that  $[\pi(w)]\neq [\pi(w')]$ since $\pi(S)$ is disjoint from all self-paired cosets. In the second case we have $w^{-1}w'\in T_{\ast}$, which implies $\pi(w^{-1}w')\not \in G'$ (since $G'\subseteq U'$), or equivalently  $[\pi(w)]\neq [\pi(w')]$. 
\end{proof}

\begin{figure}
  \centering
    \includegraphics[width=\textwidth]{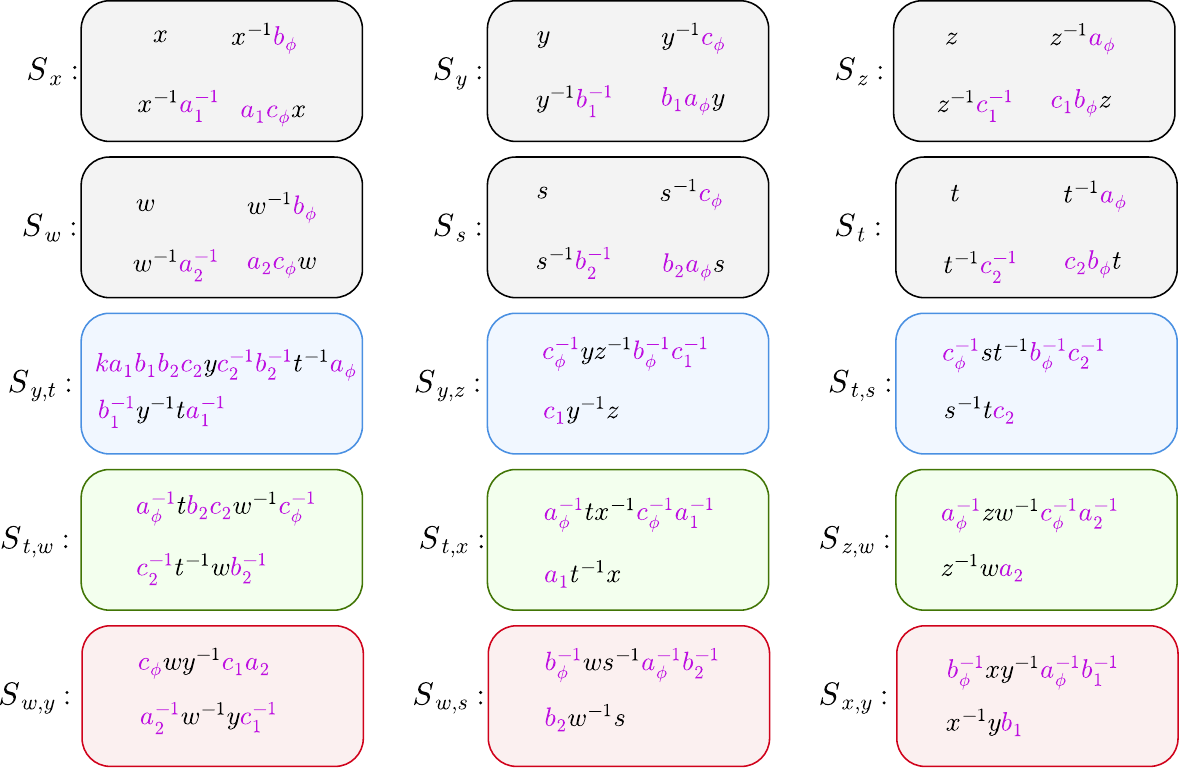}
    \captionsetup{singlelinecheck=off}
  \caption[foo]{The set  $S$ when $|G'|\leq 10^{-9}n$. 
Black letters represent free variables, while pink ones represent elements of $G$.  The words are grouped into rectangles $S_T$ based on which free variables occur where, as in  Observation~\ref{Observation_partition_S_by_free_variables}.
    To see that all words in $S$ are linear, check that there are no repetitions of black letters in any word (and every word has at least one black letter).  To see that any pair  $w,w'$ is strongly separable, note that by Observation~\ref{Observation_partition_S_by_free_variables}, any $w,w'$ coming from different rectangles fall under part (a) of the definition of separable. 
    
    In the coloured rectangles, there are always two words $w,w'$ which fall under part (c) of the definition of strongly separable (to check this first notice that $w,w'$ have the same free variables, verifying the 2nd bullet point. The free variables occur with opposite signs in $w,w'$ verifying the 4th bullet point.
    We have $\pi_0(ww')\in  [a_\phi]\cup[b_\phi]\cup[c_\phi]$ so in particular $\pi_0(ww')\not\in G'$, verifying the 3rd bullet point. Finally,  since $\pi_0(ww')\in [a_\phi]\cup[b_\phi]\cup[c_\phi]$, there are $\leq 90|G'|$ solutions to $x^2\in [a_\phi]\cup[b_\phi]\cup[c_\phi]$ verifying the last bullet point).
     
   This leaves the grey rectangles. In each such rectangle the four elements are $v, v^{-1}d_i^{-1},  v^{-1}e_\phi, d_if_\phi v$ for  a free variable $v$, $i\in \{1,2\}$, and $(d,e,f)$ some permutation of $(a,b,c)$. 
{  \begin{itemize} 
\item The pair  $w=v^{-1}d_i^{-1}, w'=d_if_\phi v$ falls under (c) because $\pi_0(ww')=\pi_0(v^{-1}d_i^{-1}d_if_\phi v)=f_\phi$. 
\item All other pairs fall under (b), as witnessed by the following equations  $v=v^{-1}d_i^{-1}$, $v=d_if_\phi v$, $v=v^{-1}e_\phi$, $v^{-1}d_i^{-1}=(v^{-1}e_\phi)(e_{\phi}^{-1}d_i^{-1})$, $v^{-1}e_\phi=(d_if_\phi v)^{-1}(d_if_{\phi}e_{\phi})$.  The  elements $g$ in these equations are $e_{\phi}, d_i$, $d_if_\phi, e_{\phi}^{-1}d_i^{-1}, d_if_{\phi}e_{\phi}$, which are all generic (since $d_i$ is $\phi$-generic).
%\item   rearrange into  $d_if_\phi=\id$ and $e_\phi d_i=\id$ which fall under (c) since $d_i$ is $\phi$-generic. 
% \item  $v^{-1}e_\phi=vd_if_\phi$ rearranges into $v^2=e_\phi d_i^{-1}f_\phi^{-1}$ which falls under (b) since $d_i$ is $\phi$-generic.  
  \end{itemize}}
  }
\label{Figure_justification_small_commutator}
\end{figure}
\begin{figure}
  \centering
    \includegraphics[width=\textwidth]{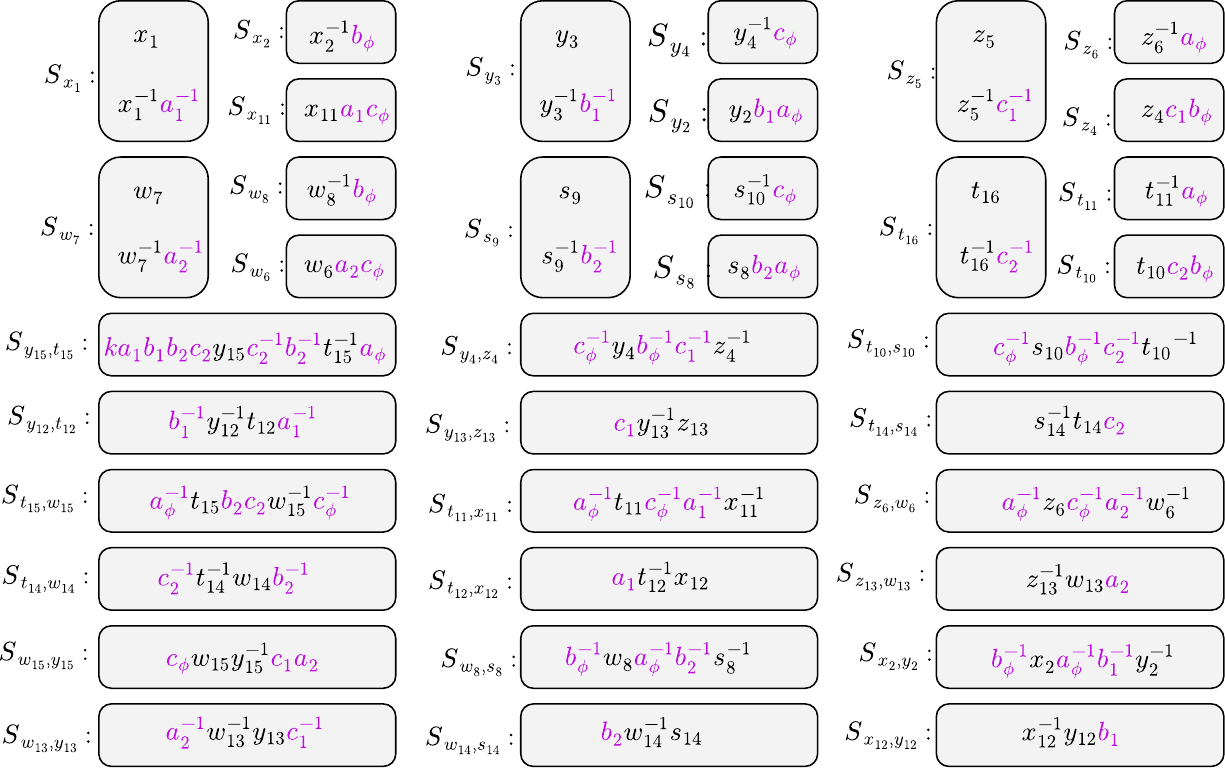}
  \caption{The set  $S$ when $|G'|\geq n/10^{9}$. Black letters represent free variables, while pink ones represent elements of $G$. The words here are exactly the same as in Figure~\ref{Figure_justification_small_commutator}, except that there are more free variables e.g. instead of the free variable $x$, we have free variables $x_1, x_2, x_{11}, x_{12}$ with $x$ replaced by one of $x_1, x_2, x_{11}, x_{12}$ wherever it occurred. The words are grouped into rectangles $S_T$ based on which free variables occur where, as in  Observation~\ref{Observation_partition_S_by_free_variables}.
  To see that all words in $S$ are linear, check that there are no repetitions of black letters in every word (and every word has at least one black letter). To see that all $w,w'\in S$ are strongly separable, note that from Observation~\ref{Observation_partition_S_by_free_variables} this holds when $w,w'$ come from different rectangles. 
  This leaves only the pairs inside $S_{x_1}, S_{y_3}, S_{z_5}, S_{w_{7}}, S_{s_{9}}, S_{t_{16}}$, which are separable by (b) since $a_1, a_2, b_1, b_2, c_1, c_2$ are generic.  }
\label{Figure_justification_large_commutator}
\end{figure}

\begin{figure}
 \includegraphics[width=\textwidth]{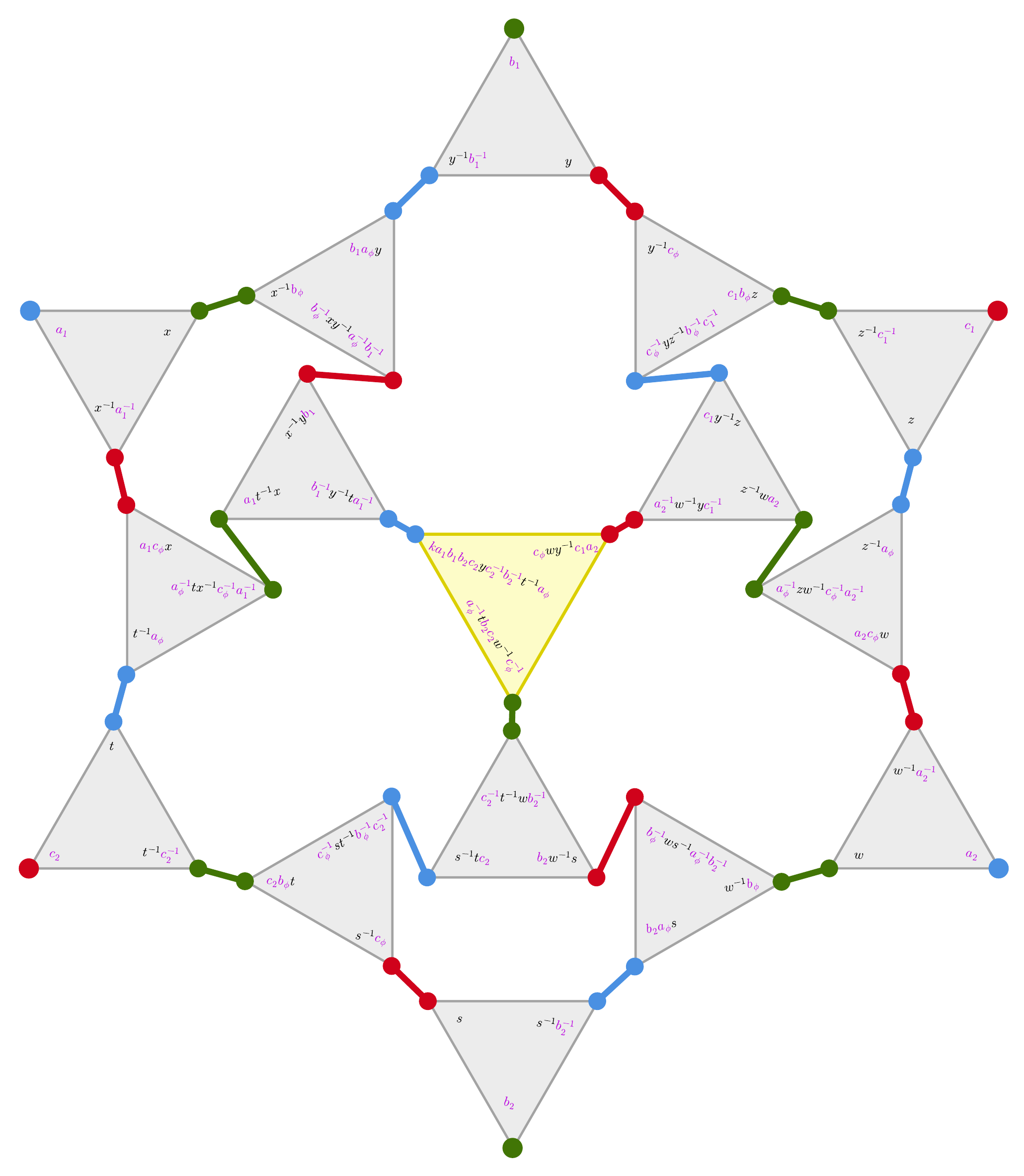}
  \caption{The set $S$ when $|G'|\leq 10^{-9} n$.  Blue, green, and red vertices give a partition of $S$ into sets $S_A/S_B/S_C$ which are pairwise strongly separable. Black letters represent free variables, while pink ones represent elements of $G$. The elements $a_1, a_2, b_1, b_2, c_1, c_2$ are not part of $S$, they are depicted only to illustrate the matching. The grey triangles are matching edges, while the yellow triangle is the set $\{a,b,c\}$. Note that the yellow triangle is an edge if and only if $k=\id$.}
\label{Figure_covering_small_commutator}
\end{figure}

\begin{figure}
  \includegraphics[width=\textwidth]{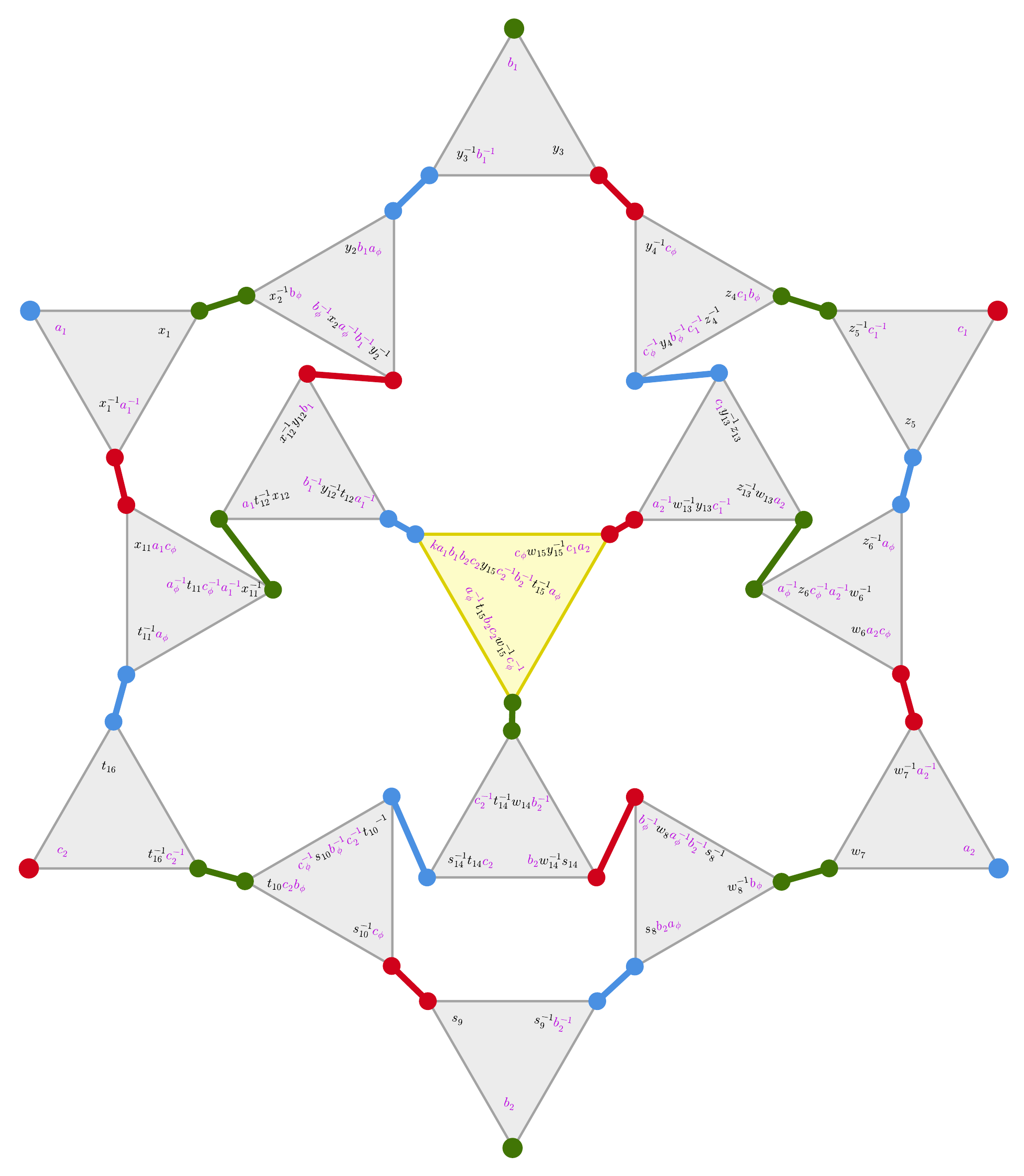}
  \caption{The  set $S$ when $|G'|> 10^{-9}n$. Blue, green, and red vertices give a partition of $S$ into sets $S_A/S_B/S_C$ which are pairwise strongly separable. Black letters represent free variables, while pink ones represent elements of $G$. The elements $a_1, a_2, b_1, b_2, c_1, c_2$ are not part of $S$. This set is exactly the same as Figure~\ref{Figure_covering_small_commutator}, except that there are more free variables. More precisely, for every grey/yellow triangle we introduce new variables and use them only in words occurring inside that triangle.  The grey triangles are matching edges, while the yellow triangle is the set $\{a,b,c\}$ where $a$ is the top left vertex, $b$ is the bottom vertex, and $c$ is the top right vertex.  }
  
\label{Figure_covering_large_commutator}
\end{figure}

Now, we may prove the main result of this section.
\begin{lemma}\label{lem:zerosumeliminate}
Let $p\geq n^{-1/10^{20}}$. 
Let $H_G$ be a multiplication hypergraph, $R^1, R^2, R^3$ disjoint, symmetric $p$-random subsets of $G$  and set $R=R^1_A\cup R^2_B\cup R^3_C$. With high probability, the following holds:

Let $S\subseteq V(H_{G})$ be a balanced and $\phi$-generic subset with $\prod S\in G'$ and $|S|\leq \frac{p^{10^{13}}n}{10^{10^6}\log(n)^{10^8}}$. Let $U\subseteq V(H_G)$ with $|U|\leq p^{800}n/10^{4100}$. Then, there exists a matching $M$ in $H_G$ with the following properties.
\begin{enumerate}[label=\textbf{Q\arabic*},ref=X\arabic*] 
\item $S\subseteq V(M)$
\item $V(M)\setminus S \subseteq R\setminus U$
\item $V(M)\setminus S$ is coset-paired
\item If $|G'|\leq \log(n)^{8000}/p^{10^{10}}$, then for each $\diamond\in \{A,B,C\}$, and $g\in G$ we have that $|(V(M)\setminus S)\cap G_\diamond \cap [g]|\leq 1$.
\end{enumerate}
\end{lemma} 
\begin{proof}
With high probability, the property of Lemma~\ref{Lemma_cover_6_set} holds and Lemma~\ref{Lemma_find_paired_vertices} applies to each $R^1, R^2, R^3$. 
Let $S\subseteq V(H_{G})$ be a balanced and $\phi$-generic subset with $\prod S\in G'$ and $|S|\leq \frac{p^{10^{13}}n}{10^{10^6}\log(n)^{10^8}}$ and $U\subseteq V(H_G)$ with $|U|\leq p^{800}n/10^{4010}$. Without loss of generality, we can assume that $|S|\geq 6$ (if this doesn't hold, then pick $\phi$-generic vertices $a,a',b,b',c,c'\in R\setminus (S\cup U)$ with $aa'=a_\phi, bb'=b_\phi, cc'=c_\phi$ using Lemma~\ref{Lemma_find_paired_vertices}, and define $S'=S\cup\{a,a',b,b',c,c'\}$. Now $S'$ is a set with $|S'|\geq 6$, so we can continue the proof with $S'$ rather than $S$. (In the case when $|G'|\leq \log(n)^{8000}/p^{10^{10}}$ we can include in $U$ the $30|G'|\ll \sqrt{n}$ many elements of $G$ in self-paired cosets before applying Lemma~\ref{Lemma_find_paired_vertices} to guarantee that the pairs we are adding come from distinct cosets. This way, \textbf{Q4} is not violated.)
\par Set $t=|S|/3$. Since $S$ is balanced we can write $S=\{a_1, \dots, a_t, b_1, \dots, b_t, c_1, \dots, c_t\}$ with $a_i\in G_A, b_i\in G_B, c_i \in G_C$. 

We build matchings $M_2, \dots, M_{t}$ and $\phi$-generic vertices $\{a_2', \dots, a_t',  b_2', \dots, b_t',  c_2', \dots, c_t'\}$ as follows.
Define $a'_2=a_1, b'_2=b_1, c'_2=c_1$. Given a subset $S\subseteq G$, denote by $\psi(S)$ the set of all $g$ in $G$ such that there exists some $s\in S$ with $[s]=[g]$. Observe that $|\psi(S)|\leq |S|\cdot |G'|$.
For $i=2, \dots, t-1$ apply Lemma~\ref{Lemma_cover_6_set} to $F_i=\{a_{i}, b_{i}, c_{i}, a'_{i}, b'_{i}, c'_{i}\}$ and $$U=U\cup N(G)\cup \bigcup_{j<i}\psi(V(M_i))$$
in the case when $|G'|\leq \log(n)^{8000}/p^{10^{10}}$ and 
$$U=U\cup N(G)\cup \bigcup_{j<i}V(M_i)$$
otherwise. This way, we obtain a matching $M_i$ and a set $\{a_{i+1}', b_{i+1}', c_{i+1}'\}$ (use an arbitrary choice of $h$ from the corresponding $G'$-coset for these applications). Further, \textbf{Q4} is maintained for $\bigcup_{j\leq i}M_i$ in the case when $|G'|\leq \log(n)^{8000}/p^{10^{10}}$ by our inclusion in $U$ of all previously used $G'$-cosets. To see that the necessary upper bound on $U$ holds for this application, note that when $|G'|\leq \log(n)^{8000}/p^{10^{10}}$, we have that $|\bigcup_{j<i}\psi(V(M_i))|\leq 100t|G'|\leq 100 \left(\frac{p^{10^{13}}n}{10^{10^6}\log(n)^{10^8}}\right)(\log(n)^{8000}/p^{10^{10}})\ll p^{800}n/10^{4001}$. Otherwise, when $|G'|$ is large, we have that $|\bigcup_{j<i}V(M_i)|\leq 100t \ll p^{800}n/10^{4001}$ as well.
\par Notice that $a'_{t}b'_tc'_ta_tb_tc_t\in G'$ (to see this, show that for all $i$ we have $a'_{i}b'_ic'_ia_ib_ic_ia_{i+1}b_{i+1}c_{i+1}\dots a_tb_tc_t\in G'$ by induction. The initial case is just the assumption $\prod S\in G'$. For the induction step use that Lemma~\ref{Lemma_cover_6_set} gives $[a'_{i}b'_ic'_ia_ib_ic_i]=[a'_{i+1}b'_{i+1}c'_{i+1}]$). 
Now apply Lemma~\ref{Lemma_cover_6_set}  to $F_t=\{a_t, b_t, c_t, a'_{t}, b'_t, c'_t\}$ with $h=\id\in G'$ and $U=U\cup N(G)\cup \bigcup_{j<t}V(M_i)$ in order to obtain a matching $M_t$ and a set $\{a, b, c\}$ with $abc=\id$.
Let $M=\bigcup_{i=2}^{t}M_i\cup\{abc\}$ in order to get a matching satisfying the lemma. Checking \textbf{Q1}-\textbf{Q2} is routine. 
To verify \textbf{Q3}, note that Lemma~\ref{Lemma_cover_6_set} tells us that $N_i:=\{a_{i+1}', b_{i+1}', c_{i+1}'\}\cup V(M_i)\setminus F_i$ is coset-paired for $i=2, \dots, t-1$ and also that $N_t:=\{a,b,c\}\cup M_t\setminus F_t$ is coset-paired. Note that $\bigcup_{j<i}N_i\subseteq F_i\cup\bigcup_{j<i}V(M_i)$ for each $i$, which shows that $N_1, \dots, N_t$ are disjoint ($N_i$ is trivially disjoint from $F_i$, and is disjoint from $\bigcup_{j<i}V(M_i)$ by choice of $U$ when applying Lemma~\ref{Lemma_cover_6_set}). Also,  $\bigcup_{i=1}^t N_i=V(M)\setminus S$, which shows that this set is coset-paired as well.

\end{proof}

\subsubsection{The zero-sum absorption lemma}
Now we arrive at the main lemma of this section.
\begin{lemma}\label{lem:zerosumabsorptionnonabelian} Let $n^{-1/10^{100}}\leq p$. Let $R^1,R^2,R^3\subseteq G$ be disjoint, symmetric $p$-random subsets and set $R=R^1_A\cup R^2_B\cup R^3_C$. With high probability, the following holds.
\par Let $U\subseteq G$ with $|U|\leq p^{10^{14}}n/\log (n)^{10^{14}}$. Then, there exists a subset $R'\subseteq R\setminus U$ such that for all balanced and $\phi$-generic subsets $S\subseteq V(H_G)\setminus R'$ with $|S|\leq \frac{p^{10^{13}}}{10^{10^8}\log(n)^{10^8}}$, $\sum S=0$, there exists a matching with vertex set $R'\cup S$.
\end{lemma}
We remark that the set $R'$ in this lemma is always balanced and zero-sum. To see this, notice that the conclusion of the lemma applies with $S=\emptyset$ (since the empty set is balanced and zero-sum). This gives a matching with vertex set $R'\cup \emptyset=R'$. Since matchings are balanced and zero-sum, we get that $R'$ is balanced and zero-sum.
\begin{proof}
Let $q=\frac{p^{801}}{10^{4180}\log n}$ and $s=\frac{q^{1600}p^{10^{7}}}{10^{10^{6}}\log^{2400} n}$. For each $i\in\{1,2,3\}$, let $Q^i,W^i$ be $q$-random and $s^4$-random, symmetric subsets of $G$ respectively, satisfying $W^i\subseteq Q^i\subseteq R^i$. Note that $Q^1,Q^2,Q^3$ and $W^1,W^2,W^3$ are disjoint, symmetric $q$-random and disjoint, symmetric $s$-random subsets of $G$, respectively. The following all hold simultaneously with high probability.
\begin{enumerate}
    \item Lemma~\ref{lem:cosetpairedabsorber} holds for $R^1,R^2,R^3$ and $Q^1,Q^2,Q^3$.
    \item Lemma~\ref{lem:zerosumeliminate} holds for $W^1,W^2,W^3$ (observe that $s^4\geq n^{-1/10^{20}}$).
    \item If $|G'|\geq s^{-5}$, each $W^i$ is $\lceil 2s^4|G'|\rceil$-coset-bounded. (This follows by Chernoff's bound, as $1/s\gg \log(n)^{10}$.) 
    \item $|\bigcup W_i|\leq 100s^4n$, by Chernoff's bound.
\end{enumerate}
\par Fix some $U\subseteq G$ such that $|U|\leq p^{10^{14}}n/\log (n)^{10^{14}}$ as in the statement of the lemma. Note that it also follows that $|U|\leq q^9p^{800}n/10^{4170}$. Then, Lemma~\ref{lem:cosetpairedabsorber} gives us a set $U'\supseteq U$ with $|U'|\leq 500q^{-3}|U|+10^{10}s^{-3}$ and $R'\subseteq R\setminus U$ with the property that for all coset-paired, $\lceil s |G'|\rceil$-coset-bounded, balanced, $S\subseteq Q\setminus U'$, with $|S|\leq s^2n$ there is a matching with vertex set $R'\cup (Q\setminus (S\cup U'))$. 
\par We claim that $R'\cup \bigcup_{i\in[3]} Q^i \setminus U'$ satisfies the property required of $R'$ in the statement of the lemma, and the rest of the proof will justify this. Fix some set $S'$ with the properties of $S$ as in the statement of the lemma, that is, $S'$ is $\phi$-generic, balanced, disjoint with $R'\cup \bigcup_{i\in[3]} W^i$, $|S'|\leq \frac{p^{10^{13}}}{10^{10^8}\log(n)^{10^8}}$, and $S'$ is zero-sum. Note it easily follows that $|S|\leq p^{4000}n/\log(n)^{10}$, and we also have that $|U'|\leq p^{800}/10^{4100}$, so we may invoke the property of $W^i$ coming from Lemma~\ref{lem:zerosumeliminate} with $S=S'$ and $U=U'$ to deduce the existence of a matching $M_1$ with properties \textbf{Q1}-\textbf{Q4}. Then, the set $S'':=\bigcup W^i\cap (V(M_1)\setminus S')$ is coset-paired by \textbf{Q3}, $\lceil s |G'|\rceil$-coset-bounded (if $|G'|\geq s^{-5}$, this follows by (3), and if $|G'|< s^{-5}\leq \log(n)^{8000}/p^{10^{10}}$ this follows by \textbf{Q4}, as we have $1$-coset-boundedness in this case), balanced, and have size at most $s^2n$ (as $S''\subseteq \bigcup W^i$ and (4)). By property of the set $R'$, we have that $R'\cup (Q\setminus (S\cup U'))$ has a perfect matching, $M_2$ say. Then, $M_1\cup M_2$ is a perfect matching of $(R'\cup \bigcup_{i\in[3]} Q^i \setminus U')\cup S'$, as desired.
\end{proof}

\subsection{Proof of Theorem~\ref{thm:main_strongest}}\label{sec:proofofmaintheorem}

Our goal in this section is to prove the main result of the paper, Theorem~\ref{thm:main_strongest}, as stated in Section~\ref{sec:statements}. We first prove the following slight variant which changes the 3rd bullet point from ``$\id \not\in X,Y,Z$'' to the more restrictive condition ``$X,Y,Z$ are $\phi$-generic''.

\begin{theorem}\label{thm:maintheorem_generic} Let $\alpha\geq n^{-1/10^{101}}$. Let $G$ be a group of order $n$.
 Let $R^1, R^2, R^3\subseteq G$ be $p$-random, $\alpha$-slightly-independent subsets. 
Then, with high probability, the following holds. 
\par Let $X,Y,Z$ be equal-sized subsets of $R^1_A$, $R^2_B$, and $R^3_C$ respectively, satisfying the following properties.
\begin{itemize}
    \item $|(R^1_A\cup R^2_B\cup R^3_C) \setminus (X\cup Y\cup Z) |\leq \alpha^{10^{15}}n/\log(n)^{10^{15}}$
    \item $\sum X+\sum Y + \sum Z = 0$ (in $G^{\mathrm{ab}}$)
    \item All elements of $X$, $Y$ and $Z$ are $\phi$-generic.
\end{itemize}
Then, $H_G[X,Y,Z]$ contains a perfect matching. 
\end{theorem}

\begin{proof} 
By the definition of $R^1, R^2, R^3$ being $\alpha$-slightly-independent, we have  $Q^1\subseteq R^1, Q^2\subseteq R^2, Q^3\subseteq R^3$ such that the joint distribution of $Q^1, Q^2, Q^3$ is that of disjoint $\alpha$-random  subsets of $G$.

Partition $G$ into a $(1/10^5)$-random subset $G_-$  and disjoint $(1-1/10^5)$-random subset $G_+$, making the choices independently of $R^1, R^2, R^3, Q^1, Q^2, Q^3$. 
For each $i=1,2,3$, let $R^i_+=R^i\cap G_+$ to get three $(1/10^5)p$-random subsets. Use Lemma~\ref{Lemma_intersect_symmetric_sets_with_random}  to pick $Q^1_-\subseteq Q^1\cap G_-, Q^2_-\subseteq Q^2\cap G_-, Q^3_-\subseteq Q^3\cap G_-$, so that their joint distribution is that of disjoint, symmetric $\alpha/10^{10}$-random subsets of $G$. 
Let $Q_-:=(Q^{1}_-)_A\cup (Q^{2}_-)_B\cup (Q^{3}_-)_C$, $R:=R^1_A\cup R^2_B\cup R^3_C$, and $R_+:=(R^{1}_+)_A\cup (R^{2}_+)_B\cup (R^{3}_+)_C$. With high probability, the property in Lemma~\ref{lem:zerosumabsorptionnonabelian} holds for $Q_-$. As the multiplication hypergraph $H_G$ is $(0,1,n)$-typical, with high probability, the property in Lemma~\ref{lem:mainnibble} holds with the random sets $(A',B',C')=(R^{1}_+, R^{2}_+, R^{3}_+$). Finally, as a consequence of Chernoff's bound, with high probability, $|Q_-|\leq |R_+|/100$, and for $i$ we have $|R^i_+|=(p(1-/10^5) \pm n^{-0.27})n$ and $|R^i|(p \pm n^{-0.27})n$

\par Now, given subsets $X,Y,Z$, define $U:=(R^1_A\cup R^2_B\cup R^3_C) \setminus (X\cup Y\cup Z)$ and note $|U|\leq \alpha^{10^{15}}n/\log(n)^{10^{15}}$. From the property in Lemma~\ref{lem:zerosumabsorptionnonabelian}, we find a set $Q'_-\subseteq Q_- \setminus U$ that can combine with balanced zero-sum sets to produce perfect matchings. By the remark after Lemma~\ref{lem:zerosumabsorptionnonabelian},   $Q_-'$ is balanced and zero-sum.

\par Let $L:=R\setminus (Q_-'\cup R_+)$. Note that as $R_+^1, R_+^2, R_+^3$ each have size $(p(1-/10^5) \pm n^{-0.27})n$, and $Q_-'$ is a balanced set contained in $R$, there exists a $q\leq p/10^5$ such that $L\cap X$, $L\cap Y$, $L\cap Z$ each have size $(q \pm n^{-0.26})n$. Note that $q\leq (1-10^{-5})p/100$, so we may apply Lemma~\ref{lem:mainnibble}, to conclude that there exists a matching $M_1'$ with $V(M_1')\subseteq L\cup R_+$ covering all but $n^{1-10^{-4}}$ vertices of $L\cup R_+$. Of the edges forming this matching, at most $\alpha^{10^{15}}n/\log(n)^{10^{15}}$ many of them can meet $U$, hence we find a matching $M_1$ of $H_G[X,Y,Z]$ covering all but at most $\alpha^{10^{15}}n/\log(n)^{10^{15}}+n^{1-10^{-4}}\leq (\alpha/10^{10})^{10^{14}}/\log(n)^{10^{14}}$ vertices of $(X\cup Y\cup Z)\setminus Q'_-$. Call this uncovered set of vertices $S$.
\par We claim that $S$ is a balanced zero-sum set. This is as the sets $Q'_-$, $V(M_1)$, and $X\cup Y\cup Z$ are each balanced and zero-sum, and $S$ is obtained by removing all of the former two sets from the latter set, and the former two sets are disjoint. Further, $S$ is $\phi$-generic, as by assumption, $X\cup Y\cup Z$ is $\phi$-generic. Then, by property of the set $Q'_-$, $S\cup Q'_-$ spans a perfect matching, $M_2$ say. $M_1\cup M_2$ then is the desired perfect matching of $X\cup Y\cup Z$. 
\end{proof}

It remains to prove Theorem~\ref{thm:main_strongest}. We will need the following intermediary result.
\begin{lemma}\label{lem:preprocessing}
Let $p\geq n^{-1/10^{100}}$. Let $G$ be a group of order $n$. Let $Q^1,Q^2,Q^3\subseteq G$ be disjoint, symmetric $p$-random subsets and set $Q=Q^1_A\cup Q^2_B\cup Q^3_C$. With high probability, the following holds.
\par Let $S, U\subseteq V(H_G)$  with $|S|, |U|\leq p^{10^{10}}n/\log(n)^{10^{10}}$ and $\id \not \in S$. 
Then, there exists a matching $M$ in $H_G[Q]\setminus U$ of size at most $2p^{10^{10}}n/\log(n)^{10^{10}}$ saturating $S$, meaning $S\subseteq V(M)$.
%|(R^1_A\cup R^2_B\cup R^3_C) \Delta (X\cup Y\cup Z) |\leq p^{10^{10}}n/\log(n)^{10^{10}}$ and $\id \not \in X\cup Y \cup Z$. %Further suppose that if $G\cong (\mathbb{Z}_2)^{n}$, then $\id_G\notin X\cup Y\cup Z$. 
%Then, there exists a matching $M$ of $H_G[X,Y,Z]$ of size at most $2p^{10^{10}}n/\log(n)^{10^{10}}$ saturating all non-$\phi$-generic elements as well as all elements of $(X\cup Y\cup Z)\setminus (R^1_A\cup R^2_B\cup R^3_C)$.
\end{lemma}
\begin{proof}
Let $g\in G$ be an arbitrary element of the group such that $g\neq \id$. %in the case that $G=(\mathbb{Z}_2)^k$ for some $k$. 
Then, by Proposition~\ref{prop:overflow}, it follows that $g$ is contained in at least $n/5$ edges $\{g,x,y\}$ such that $x\not\in \{y,y^{-1}\}$, regardless of whether $g\in G_A$, $G_B$, or $G_C$. For any $i$, $j$, and any such $\{x,y\}$, $x\in Q^i$ and $y\in Q^j$ with probability $p^2$. For each of the $n/5$ such edges, the corresponding pairs $\{x,y\}$ are disjoint. Hence, we may apply Chernoff's bound to deduce that with probability at least $1-1/n^2$, $g$ has degree at least $p^2n/1000$ in $H_G[Q^1_A, Q^2_B, Q^3_C]$, as long as $g\in V(H_G[Q^1_A, Q^2_B, Q^3_C])$. By a union bound, this property holds for all $g\in G$ such that $g\neq \id$ in the case that $G=(\mathbb{Z}_2)^k$ for some $k$.
\par Now, let $S, U$ be given.  
By the minimum degree property from the previous paragraph, all elements of $S$ have degree at least $p^2n/1000 - 2p^{10^{10}}n/\log(n)^{10^{10}}\geq p^2n/2000$ in $H_G[Q]\setminus U$. As $p^2n\gg |S|$, we may greedily pick a matching saturating $S$ and disjoint from $U$ of size $|S|\leq 2p^{10^{10}}n/\log(n)^{10^{10}}$ as claimed.
\end{proof}

We can now give the final proof of the section.  

\begin{proof}[Proof of Theorem~\ref{thm:main_strongest} from  Theorem~\ref{thm:maintheorem_generic} and Lemma~\ref{lem:preprocessing}]  By the definition of $q$-slightly-independent, we have disjoint, symmetric, $q$-random sets $Q^1, Q^2, Q^3$ with $Q^1\subseteq R^1, Q^2\subseteq R^2, Q^3\subseteq R^3$. With high probability Theorem~\ref{thm:maintheorem_generic} holds  $R^1$, $R^2$ and $R^3$ and Lemma~\ref{lem:preprocessing} holds for $Q^1, Q^2, Q^3$. 
Given $X,Y,Z$, let $S=(X\cup Y\cup Z)\setminus (R^1_A\cup R^2_B\cup R^3_C)$ together with the non-$\phi$-generic elements of $X\cup Y\cup Z$ (of which there are at most $10^{9010}$ many). Let $U=(R^1_A\cup R^2_B\cup R^3_C)\setminus (X\cup Y\cup Z)$. Note that $|S|,|U|\leq  p^{10^{16}}n/\log(n)^{10^{16}}$ and $\id \not \in S$ in the case that $G=\mathbb{Z}_2^k$. 
Apply Lemma~\ref{lem:preprocessing} to find a matching $M$, and set $X',Y',Z'$ to be $X\setminus V(M),Y\setminus V(M),Z\setminus V(M)$ respectively.   Observe that $X',Y',Z'$ satisfy the hypothesis of $X,Y,Z$ in Theorem~\ref{thm:main_strongest}. Indeed, $X'/Y'/Z'$ is contained in $R^1/R^2/R^3$ and does not contain non-$\phi$-generic elements since all such elements are in $S\subseteq V(M)$. We have $\sum X' +\sum Y' +\sum Z'= 0$, as $\sum X +\sum Y +\sum Z= 0$, and $\sum V(M)=0$ (as $V(M)$ spans a perfect matching). We have $|X'|=|Y'|=|Z'|$ for the same reason. Hence, $H_G[X',Y',Z']$ spans a perfect matching by Theorem~\ref{thm:main_strongest}, and this together with $M$ gives a perfect matching of $H_G[X,Y,Z]$ as desired.
\end{proof}

\section{Applications}\label{sec:applications}

\subsection{Characterising subsquares with transversals}\label{sec:provingthecharacterisation}
The goal of this section is to prove a far-reaching generalisation of Snevily's conjecture as stated in the introduction (Theorem~\ref{thm:characterisation}).
\subsubsection{Preliminaries}
The below lemma allowing us to find zero-sum sets of prescribed size will also be useful in Section~\ref{sec:tannenbaum}.
\begin{lemma}\label{lem:greedyzerosum}
Let $k\in\{3,4,5\}$. Let $p\geq 3n^{-1/1400}$. Let $G$ be an abelian group, and let $R$ be a $p$-random subset of $G$. Then, the following holds with high probability.
\par Let $X\subseteq R $ with $|R\setminus X|\leq p^{8000}n/10^{6000}$. Then, there exists at least $p^{800}n/10^{4010}$ many disjoint zero-sum sets of size $k$ in $X$.
\end{lemma}%\AP{I changed this proof as suggested by comment (54b). Also I think that this lemma isn't used anymore until much later in the paper, so should probably be moved.}
\begin{proof}
With high probability, Corollary~\ref{Corollary_separated_set_random} applies to $R$. If it is the case that $k=3$, consider the set  $S=\{xy, x^{-1}, y^{-1}\}\subseteq G\ast F_2$ and note that all pairs of words in $S$ are linear and separable (by part (a) of the definition of separable). Thus, setting $U:=R\setminus X$ there is a projection $\pi:G\ast F_2\to G$ which has $\pi(S)\subseteq X$ and which separates $S$. Resetting $U$ to include the vertices in the projection, and invoking Corollary~\ref{Lemma_separated_set_random} iteratively, we can conclude that there exists at least $(p^{800}n/10^{4000}-p^{8000}n/10^{6000})/10\gg p^{800}n/10^{4010}$ disjoint zero-sum sets of size $3$ contained in $X$. If $k=4$ set $S=\{xyz, x^{-1}, y^{-1}, z^{-1}\}$, and if $k=5$ set $S=\{xyzw, x^{-1}, y^{-1}, z^{-1}, w^{-1}\}$ to obtain sets of words which are linear and pairwise strongly separable, and proceed in the same way as the case $k=3$.
\end{proof}

For this section, we will only need the following easy corollary of Lemma~\ref{lem:greedyzerosum}
\begin{lemma}\label{lem:finishingsum}
Let $G$ be a group, and let $3\leq t\leq n/3-1$, and let $g\in G$. Then, there exists distinct $c_1,\ldots, c_t\in G$ such that $\sum c_i=g$ in $G^{\mathrm{ab}}$.
\end{lemma}
\begin{proof} Let $c_4,\ldots, c_t$ be arbitrary distinct elements of $G$, and set $g'=g(c_4\cdots c_t)^{-1}$. Observe there exists at least $n^2-3n$ choices of $c_1,c_2$ and $c_3$ such that $c_1c_2c_3=g'$ and $c_1,c_2,c_3$ are distinct. Among those choices, at most $3tn$ of them have $\{c_1,c_2,c_3\}\cap \{c_4,\ldots, c_t\}\neq \emptyset$. This implies that if $n^2-3n-3tn>0$, i.e. if $t\leq n/3-1$, there exists a choice of $c_1,\ldots, c_t$ with the desired properties. 
\end{proof}

The below result has essentially appeared in \cite{spanningrainbow}. For our application here, we provide a more precise formulation as well as a full proof for the sake of completeness. We remark that the following proof was suggested to us by an anonymous referee. 

\begin{lemma}\label{lem:fournier} Let $n$ be sufficiently large and set $\gamma:=1/\log(n)^{10^{100}}$. 
Let $S=A\times B$ be a subsquare of a multiplication table of a group $G$ defined by two $n$-element sets $A,B\subseteq G$. Then, either $S$ has at most $(1-\gamma)n$ symbols occurring more than $(1-\gamma)n$ times or there is a subgroup $H\subseteq G$ and elements $g,g'\in G$ such that $|A\Delta gH|, |B\Delta Hg'|\leq \gamma^{1/10} n$. 
\end{lemma}

\begin{proof}
We will use a theorem of Fournier as stated as Theorem 1.3.3 in the lecture notes of Green \cite{green}. Towards this goal, we recall the definition of \textbf{multiplicative energy} of a subset $A$ of a group. It is defined to be the quantity $E(A):=|\{(a_1,a_2,b_1,b_2)\in A^4\colon a_1a_2^{-1}=b_1b_2^{-1}\}|$. We set $E(A,B):=|\{(a_1,a_2,b_1,b_2)\in A^2\times B^2\colon a_1a_2^{-1}=b_1b_2^{-1}\}|$, note that $E(A)=E(A,A)$. Denote by $r_A(x):=\{(a_1,a_2)\in A^2\colon a_1a_2^{-1}=x\}$, and $r_B(x)$ is defined analogously. Then, by the Cauchy-Schwarz inequality, we have that
$$E(A^{-1},B)=\sum_{x\in G}r_{A^{-1}}(x)r_B(x)\leq E(A^{-1})^{1/2}E(B)^{1/2}.$$
\par Note also that $E(A^{-1},B):=\{(a_1,a_2,b_1,b_2)\in A^2\times B^2\colon a_1b_1=a_2b_2\}$. Suppose now that $S$ has more than $(1-\gamma)n$ symbols occurring more that $(1-\gamma)n$ times. Thus, $E(A^{-1},B)\geq (1-2\gamma)^3 n^3$. As $E(A^{-1}), E(B)\leq n^3$, we can deduce using the above that $E(A^{-1}), E(B)\geq (1-2\gamma)^6 n^3$.  From Theorem 1.3.3 in \cite{green}, it follows that there exists a subgroup $H$ and $g\in G$ such that $|A\Delta gH|\leq \gamma^{1/9} n$. Consider an element $t$ of $S$ that repeats at least $(1-\gamma) n$ times. So we have $a_1b_1=\cdots=a_{(1-\gamma) n}b_{(1-\gamma) n}=t$. Let $i\in [(1-\gamma) n]$ so that $a_i\in gH$, and note that there are at least $ (1-2\gamma^{1/9}) n$ indices $i$ with this property. So we have that $t=(gh)b_i$ for some $h\in H$. Then, $b_i=h^{-1}g^{-1}t$, so $b_i\in H(g^{-1}t)$ for each such $i$. Then, the statement holds for $g'=g^{-1}t$.

\par

\end{proof}

\begin{observation}\label{obseasy}
Let $A\times B$ be a subsquare of a multiplication table of a group $G$ defined by two $n$-element sets $A,B\subseteq G$. Let $g,g'\in G$. Suppose that $A\times B$ contains a transversal. Then, $gA\times Bg'$ also contains a transversal.
\end{observation}
\begin{proof} Let $(a_i,\phi(a_i))$, $i\in [n]$ be a transversal of $A\times B$, where $\phi$ is a bijection $A\to B$. Then, $(ga_i,\phi(a_i)g')$ is a transversal of $gA\times Bg'$, since $a_i\phi(a_i)\neq a_j\phi(a_j)$ implies that $ga_i\phi(a_i)g'\neq ga_j\phi(a_j)g'$.
\end{proof}

\begin{lemma}\label{decompcor}
Let $\eps>0$ be sufficiently small and let $n$ be sufficiently large. Let $\gamma \geq n^{-\eps}$. Let $S$ be a subsquare of a group $G$ of order $n$ such that $S$ has at most $(1-\gamma)n$ symbols occuring more than $(1-\gamma)n$ times. Then, $S$ contains a transversal.
\end{lemma}
\begin{proof}
As subsquares of $G$ correspond to edge-coloured balanced complete graphs, this lemma is a direct corollary of Lemma 8.1.3 from \cite{spanningrainbow}. 
\end{proof}

\subsubsection{The characterisation}

We first prove the following weakening of the main theorem which characterises when subsquares which are close to the whole multiplication table have transversals. 
\begin{lemma}\label{Lemma_very_large_subsquare_characterization}
Let $t\leq n/\log(n)^{10^{30}}$.
Let $G$ be a sufficiently large group and $A,B\subseteq G$ with $|A|=|B|= n-t$. Then $A\times B$ has a transversal unless one of the following holds.
\begin{enumerate}
    \item $G$ is a group that does not satisfy the Hall-Paige condition, and $A=B=G$. 
    \item $G\cong (\mathbb{Z}_{2})^k$, $A=G\setminus \{a_1, a_2\}$, $B=G\setminus \{b_1, b_2\}$ for some distinct $a_1,a_2\in G$ and distinct $b_1,b_2\in G$ such that $a_1+a_2+b_1+b_2=0$.
\end{enumerate}
\end{lemma}
\begin{proof}
If $t=0$ then there is a transversal (by any proofs of the Hall-Paige Conjecture e.g. by Theorem~\ref{thm:maintheoremnondisjoint} with $p=1$). Thus, we can suppose that $1\leq t\leq n/\log(n)^{10^{30}}$. Let $\alpha:=\prod G (\prod  A \prod  B)^{-1}$ (where the terms in the products are multiplied in any fixed order). 
\begin{claim}
There are distinct elements $c_1,\ldots, c_t\in G$ such that $c_1+\cdots +c_t=\alpha$ in $G^{\mathrm{ab}}$.
\end{claim}
\begin{proof}
If $t=1$, one can trivially choose $c_1=\alpha$. If $t=2$, and $\alpha\neq \id$, then we could select $c_1=\id$ and $c_2=\alpha$. The case $t=2$, $\alpha=\id$,  $G=(\mathbb{Z}_2)^k$ is excluded by the hypothesis. Indeed, in this case we'd have $A=G\setminus \{a_1, a_2\}$, $B=G\setminus \{b_1, b_2\}$ for some distinct $a_1,a_2\in H$ and distinct $b_1,b_2\in G$. But then in $G^{ab}$ we'd have $0 =\alpha=\sum G - \sum A-  \sum B=a_1+a_2+b_1+b_2-\sum (\mathbb{Z}_2)^k=a_1+a_2+b_1+b_2$.
When $t=2$, $\alpha=\id$, $G\neq (\mathbb{Z}_2)^k$, then, there must exist distinct $c_1,c_2\in G$ such that $c_1c_2=\id$. Finally, when $t\geq 3$, then the claim follows from Lemma~\ref{lem:finishingsum}. 
\end{proof}
Let $C=G\setminus \{c_1, \dots, c_t\}$ to get a set with $|C|=|A|=|B|$. Now Theorem~\ref{thm:maintheoremnondisjoint} applies with $p=1$ to give a perfect matching $\{(a_1, b_1, c_1), \dots, (a_{n-1}, b_{n-1}, c_{n-1})\}$ in $H_G[A,B, C^{-1}]$. The entries $(a_1, b_1), \dots, (a_{n-t}, b_{n-t})$ give a transversal in $A\times B$.
\end{proof}

We now prove the main result of this section.
\begin{theorem}\label{thm:characterisationv2}
There exists a $n_0\in \mathbb{N}$ such that the following holds for all $n\geq n_0$. Let $G$ be a group, and let $A,B\subseteq G$ with $|A|=|B|=n$. Then, $A\times B$ has a transversal, unless there exists some $k\geq 1$, $g_1,g_2\in G$ and a subgroup $H\subseteq G$ such that one of the following holds.
\begin{enumerate}
    \item $H$ is a group that does not satisfy the Hall-Paige condition, and $A=g_1H$ and $B=Hg_2$. 
    \item $H\cong (\mathbb{Z}_{2})^k$, $g_1 A= H\setminus\{a_1,a_2\}$, $g_2 B= H\setminus\{b_1,b_2\}$ for some distinct $a_1,a_2\in H$ and distinct $b_1,b_2\in H$ such that $a_1+a_2+b_1+b_2=0$.
\end{enumerate}
\end{theorem}

\begin{proof}[Proof of Theorem~\ref{thm:characterisationv2}] Set $\gamma=1/\log(n)^{10^{100}}$. Suppose that $A\times B$ does not have a transversal. Then, by Lemma~\ref{decompcor}, we may assume that $A\times B$ contains more than $(1-\gamma)n$ symbols occuring more than $(1-\gamma)n$ times. By Lemma~\ref{lem:fournier}, it follows that there is a subgroup $H\subseteq G$ and elements $g,g'\in G$ such that $|A\Delta gH|, |B\Delta Hg'|\leq \gamma^{1/10} n$. As $A\times B$ does not contain a transversal, by Observation~\ref{obseasy}, $g^{-1}A\times Bg'^{-1}$ does not contain a transversal either. Set $A'\times B'=g^{-1}A\times Bg'^{-1}$ and observe that $|A'\Delta H|, |B'\Delta H|\leq \gamma^{1/10} n$. 

\par 
Set $A_1=A'\cap H$, $A_2=A'\setminus H$, $B_1=B'\cap H$, $B_2=B'\setminus H$, noting that $|A_2|, |B_2|\leq \gamma^{1/10} n$. If  $A_2= B_2=\emptyset$, then the theorem follows from Lemma~\ref{Lemma_very_large_subsquare_characterization} applied with $G=H$. Thus we can suppose that $A_2$ and/or $B_2$ are nonempty. Note that all the elements in the multiplication table in $A_1\times B_2$ and $A_2\times B_1$ are outside $H$. Let $A_2= \{a_1, \dots, a_{|A_2|}\}$,  $B_2=\{b_1, \dots, b_{|B_2|}\}$. We can greedily select a partial transversal $T_1=\{(a_1, b_1'), \dots, (a_{|A_2|}, b_{|A_2|}')$ 
, $(a_1', b_1), \dots, (a_{|B_2|}', b_{|B_2|})\}$ by selecting  elements $b_1', \dots, b_{|A_2|}'\in B_1, a_1', \dots, a_{|B_2|}'\in A_1$ in order (to see this note that there are at least $\min(|A_1|, |B_1|)\geq n/2$ choices for each element and so there's room to avoid the $|A_2|+|B_2|\leq \gamma^{1/10} n$ rows/columns/symbols previously used). Note that since there are at least $n/4$ choices for the last element  $a_{|B_2|}'$, we can additionally ensure that $\sum A_1\setminus\{a_1',\dots,a_{|B_2|}'\}+ \sum B_1\setminus\{b_1',\dots,b_{|A_2|}'\}\neq 0$ in $H^{ab}$ in the case where $H^{ab}$ has at least $100$ elements. Thus Lemma~\ref{Lemma_very_large_subsquare_characterization} applies to give a transversal $T_2$ in $(A_1\setminus\{a_1',\dots,a_{|B_2|}')\times (B_1\setminus\{b_1',\dots,b_{|A_2|}'\})$ (to apply Lemma~\ref{Lemma_very_large_subsquare_characterization}, we need to know that we are not in cases 1  and 2. We're not in case 1 because we're assuming $A_2\cup B_2\neq \emptyset$. We're not in case 2 because in this case we have $|H^{ab}|\geq 100$, and also that $\sum A+\sum B=0$, and we selected $a_i'$ and $b_i'$ to avoid this scenario). Now $T_1\cup T_2$ is a transversal in $A'\times B'$ as required.
\end{proof}

\subsection{Path-like structures in groups}\label{sec:pathlike}
\par The goal of this section is to give a characterisation of sequenceable, R-sequenceable, and harmonious groups which are sufficiently large. It will be convenient to rephrase all three of these problems as finding rainbow structures in edge-coloured digraphs. Towards this aim, we give some definitions.

\par Given a group $G$, by $K^+_G$ we denote the complete directed edge-coloured graph with vertex set $G$, edge set $\{\vec{ab}\colon a\neq b\text{ and } a,b\in G\}$, where the edge $\vec{ab}$ gets assigned the colour $ab\in G$. We call this the \textbf{multiplication digraph of} $G$. Similarly, by $K^-_G$ we denote the \textbf{division digraph of} $G$. In the division digraph, the edge $\vec{ab}$ gets assigned the colour $a^{-1}b\in G$, and all other properties of $K^-_G$ are same with those of $K^+_G$. We sometimes use the notation $K^\pm_G$ to make statements and definitions about $K^+_G$ and $K^-_G$ simultaneously. For subsets $R,R'\subseteq G$, we will use $K^\pm_G[R;R']$ to denote the subgraph of $K^\pm_G$ induced on vertex set $R$ consisting of all edges of colours in $R'$. For disjoint subsets $V_1,V_2\subseteq G$, by $K^\pm_G[V_1,V_2;R']$ we denote the bipartite subgraph of $K^\pm_G$ obtained by keeping only the edges between $V_1$ and $V_2$ with colour in $R'$. 

\par Recall that a subgraph of an edge-coloured graph is called \textbf{rainbow} if all edges have distinct colours. The definitions of sequenceable, $R$-sequenceable, and harmonious were given in Section~\ref{sec:introapplications}. The following is straightforward to derive from the definitions.
\begin{observation}
    A group $G$ is harmonious if and only if $K^+_G$ has a directed rainbow Hamilton cycle, sequenceable if and only if $K^-_G$ has a directed Hamilton path with colour set $G\setminus \id$, and $R$-sequenceable if and only if $K^-_G$ has a directed rainbow cycle with colour set $G\setminus \id$.
\end{observation}

\par The main trick in this section is to use our main theorem iteratively to build rainbow paths out of rainbow matchings. One key issue with this idea is that this does not give us to freedom to construct paths connecting specified end-points, which is critical for building Hamilton paths, instead of an arbitrary path/cycle-factor. To remedy this, in Section~\ref{sec:sortingnetworks}, based on ideas of Kühn, Lapinskas, Osthus, and Patel \cite{kuhnsorting}, we introduce a way of building path systems allowing us to construct path-factors with specified end-points. The key result of that section, combined with a variant of our main theorem, allows us to deduce the following theorem.

\begin{theorem}\label{thm:hamiltonpathsingroups} Let $G$ be a sufficiently large group on $n$ vertices. Let $V,C\subseteq G$ and $x,y\in G$ be such that $|V|+1=|C|\geq n-n^{1/2}$, $x\neq y$, $x,y\notin V$, and further suppose that $e\notin C$ if $G$ is an elementary abelian $2$-group. Then, $K^-_G[\{x,y\}\cup V; C]$ has a directed rainbow Hamilton path from $x$ to $y$ if $\sum C=y-x$ in $G^{\mathrm{ab}}$, and $K^+_G[\{x,y\}\cup V; C]$ has a directed rainbow Hamilton path from $x$ to $y$ if $\sum C=x+y+2\sum V$ in $G^{\mathrm{ab}}$. 
\end{theorem}

This implies the following characterisation of which rainbow Hamilton paths can be found in $K^+_G$.

\begin{corollary}\label{lem:hampathmain} Let $G$ be a sufficiently large group not isomorphic to an elementary abelian $2$-group. Let $c, x_1, x_2\in G$ such that $0 = \sum G + c - x_1 - x_2$ in $G^\mathrm{ab}$, and $x_1\neq x_2$. Then, $K^+_G$ contains a directed rainbow Hamilton path using the colour set $G-c$ and with endpoints $x_1$ and $x_2$.
\end{corollary}
\begin{proof}
    This is immediate by applying Theorem~\ref{thm:hamiltonpathsingroups} with $V:=G\setminus \{x_1,x_2\}$ and $(x,y)=(x_1,x_2)$ and $C=G\setminus\{c\}$.
\end{proof}
This gives a characterisation of groups where $K^+_G$ has a directed rainbow Hamilton cycle which is equivalent to $G$ being harmonious.
\begin{corollary}\label{thm:evans}
Let $G$ be a sufficiently large group satisfying the Hall-Paige condition, and suppose $G$ is not isomorphic to an elementary abelian $2$-group. Then, $K^+_G$ has a directed rainbow Hamilton cycle, i.e. $G$ is harmonious. 
\end{corollary}
\begin{proof}
Let $x_1$ and $x_2$ be such that $x_1x_2=\id$ and $x_1\neq x_2$. Such $x_1$ and $x_2$ exist as $G$ is not isomorphic to an elementary abelian $2$-group. By assumption, $\sum G=\id$ in $G^\mathrm{ab}$. Then, $0= \id -x_1-x_2=\sum G+\id -x_1-x_2$ in $G^\mathrm{ab}$, hence by Corollary~\ref{lem:hampathmain} we have a directed rainbow Hamilton path from $x_1$ to $x_2$ using all colours but $\id$. Combined with the edge $x_2\to x_1$, this gives the desired directed rainbow Hamilton cycle.
\end{proof}
\begin{comment}
\begin{theorem}\label{thm:patrias}
Let $G$ be a sufficiently large group not isomorphic to an elementary abelian $2$-group. Let $P$ be a finite path. Then, $P$ admits a cordial $G$-labelling.
\end{theorem}
\begin{proof} Denote the vertices of the path as $v_0\to v_1\to v_2 \to \cdots \to v_{n-1}$. If $G$ satisfies the Hall-Paige condition, by Corollary~\ref{thm:evans}, $K^+_G$ contains a directed rainbow Hamilton cycle, $w_0\to w_1 \to \cdots \to w_{m-1}\to w_0$ say. Consider the labelling $\phi(v_i)=w_i$ where the second subscript is to be interpreted modulo $m$. It is easy to see that $\phi$ is a cordial labelling of $P$.
\par Now, suppose $\sum G=c\neq \id$ in $G^\mathrm{ab}$. In this case we can fix a $x_1\neq \id$ such that $[x_1]=\sum G$ and $x_1^2=\id$. Indeed, letting $i_j$ ($j\in [k]$) be the sequence of order $2$ elements of $G$, we have that $[i_1i_2\cdots i_k]=\sum G$ holds for any group $G$. Further, the set $\{i_1,i_2,\ldots, i_k\}$ is closed under products, hence we can take $x_1=i_1i_2\cdots i_k$. By Lemma~\ref{lem:hampathmain}, $K^+_G$ contains a directed rainbow Hamilton path from $x_1=\id$ to $x_2$ missing the colour $x_2$, $w_0\to w_1 \to \cdots \to w_{m-1}$ say. Consider the infinite sequence $S$ obtained by concatenating copies of $(0,1,\ldots, m-2,m-1,m-1,m-2,m-3,\ldots, 2,1,0)$. To obtain a cordial labelling, define $\phi(v_i)=w_{S[i]}$, concluding the proof.
\end{proof}

\end{comment}
We now characterise large sequenceable groups. Note that for abelian groups, the below result was proved by Gordon \cite{gordon1961sequences} (with no assumption on the size of the group).
\begin{theorem}
    Let $G$ be a sufficiently large group. If $G$ is abelian, suppose that $\sum G\neq 0$, or equivalently, $G$ has a unique element of order $2$. Then, $G$ is sequenceable.
\end{theorem}
\begin{proof}
    First suppose $G$ is abelian. Let $k\neq 0$ be the unique element of order $2$ in $G$. Apply Theorem~\ref{thm:hamiltonpathsingroups} with $\ast=-$, $C=G\setminus \{0\}$, $(x,y)=(0,k)$, $V=G\setminus \{0,k\}$, noting this is possible as $y-x-k-0=k=\sum G=\sum C$. This gives us the desired directed rainbow Hamilton path with colour set $G\setminus \setminus \{0\}$. %path with vertex set $(0,x_1,x_2,\ldots, x_{n-2}, k)$, and a colour sequence $(x_1, -x_1+x_2, \ldots, -x_{n-2}+k)$. We claim that $(0,x_1, -x_1+x_2, \ldots, -x_{n-2}+k)$ is a sequencing of $G$. Indeed, the partial sums give the sequence $(0, x_1, x_2, x_3, \ldots ,x_{n-2}, k)$, and these are all distinct as this sequence also is a Hamilton path in $K^{-}_G$. 
    \par If $G$ is nonabelian, let $y$ be such that $y=\sum G$ in $G^{\mathrm{ab}}$ and $y\neq e$. Such a $y$ exists since each coset of the commutator subgroup has at least $2$ elements.  We can now invoke Theorem~\ref{thm:hamiltonpathsingroups} with $C=G-\id$, $(x,y)=(\id,y)$ and $V=G-\id-y$, this is possible as $y-\id=\sum C=\sum G$ in $G^{\mathrm{ab}}$ by choice of $y$. This again gives us the desired directed rainbow Hamilton path.
\end{proof}

Using similar ideas, we characterise R-sequenceable groups.
\begin{theorem}
    Let $G$ be a sufficiently large group satisfying the Hall-Paige condition, that is, $\sum G=0$ in $G^{\mathrm{ab}}$. Then, $G$ is $R$-sequenceable.
\end{theorem}
\begin{proof}
    Let $x,y$ be two distinct elements of the group $G$, and apply Theorem~\ref{thm:hamiltonpathsingroups} (the $K^-_G$ case) with $C=G\setminus \{e, y^{-1}x\}$, $(x,y)=(x,y)$ and $V=G\setminus \{x,y, \id\}$, this application is valid as $\sum G = 0 $. This gives a path from $x$ to $y$, and combined with the edge from $y\to x$, we obtain a directed rainbow cycle in $K^-_G$ using all colours but $e$, meaning that $G$ is $R$-sequenceable.
\end{proof}

In the rest of this section, we are focused on proving Theorem~\ref{thm:hamiltonpathsingroups}.

\subsubsection{Sorting networks}\label{sec:sortingnetworks}
In \cite{kuhnsorting} (see in particular Lemma 4.3), an ingenious method was introduced in order to construct path systems which can connect specified endpoints. The key idea is to use an appropriate \textit{sorting network} as a template while building the path system. In this section, we adapt the arguments from \cite{kuhnsorting} to our context. First, we introduce some terminology. For a more detailed treatment, we refer the reader to \cite{cormen2009}. 
\begin{definition}
A \textbf{comparison network} is a union of four types of objects: input nodes $x_1, \dots, x_m$, output nodes $y_1, \dots, y_m$, comparators $C_1, \dots, C_t$ and wires $w_1, \dots, w_s$.
\begin{itemize}
\item Comparators are sets of $4$ nodes $C_i=\{y_i^-,y_i^+,x_i^-, x_i^+\}$ (which are disjoint with the input and output nodes).
\item Each wire joins an $x$-node to a $y$-node. Additionally, each node is in precisely one wire, and the directed graph formed by contracting comparators into single nodes is acyclic.
\end{itemize}
\end{definition}

\begin{figure}[h]
  \centering
    \includegraphics[width=0.8\textwidth]{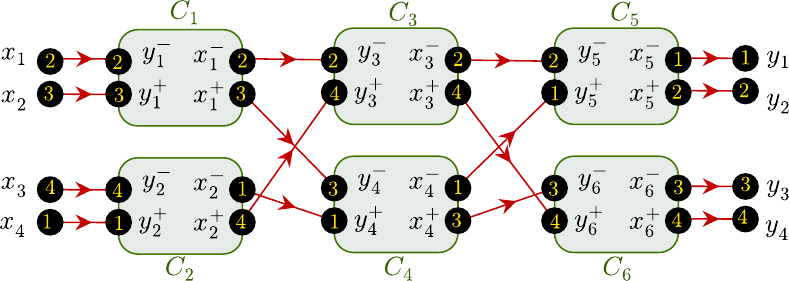}
  \caption{A comparison sorting network for sorting four numbers. The black circles represent nodes, the green rectangles are comparators, and the red arrows are wires. Here the network was given input values $v(x_1)=1, v(x_2)=3, v(x_3)=4, v(x_4)=1$ (represented by the yellow numbers inside those nodes). Then all other nodes get a value based on the rules in Definition~\ref{Definition_sorting_network} (represented by the yellow numbers in the other nodes). The network correctly sorted the numbers, which can be seen by the fact that each $y_i$ contains yellow number $i$.}
\label{Figure_sortingnetwork}
\end{figure}

\begin{definition}\label{Definition_sorting_network}
A \textbf{comparison sorting network} is a comparison network with the following additional property. Let $\sigma$ be any permutation of $[m]$. Assign each $x_i$ the value $v(x_i)=\sigma(i)$. Assign the values of the other nodes via the following rules.
\begin{enumerate}
\item If $xy$ is a wire then $v(y):=v(x)$.
\item If  $C_i=\{y_i^-,y_i^+,x_i^-, x_i^+\}$ is a comparator, then $v(x_i^-):=\min(v(y_i^-),v(y_i^+))$ and $v(x_i^+):=\max(v(y^-),v(y^+))$.
\end{enumerate}
Then, all nodes get assigned a value and moreover,  $v(y_i)=i$ for $i=1, \dots, m$.
\end{definition}
See Figure~\ref{Figure_sortingnetwork} for an example of these definitions.
A classical result due to Batcher \cite{batchersorting} states that for all $m\in \mathbb{N}$, there is a sorting network with $m$ input/outputs and $100m\log^2m$ comparators. In fact, there are sorting networks with $O(m\log m)$ comparators thanks to a celebrated result of Ajtai, Koml\'os, and Szemer\'edi \cite{ajtaisorting} but we will not need this sharper bound here. However, it will be convenient for us to have sorting networks with symmetry in the following sense.

\begin{lemma}[\cite{ajtaisorting}]\label{Lemma_sorting_network}
For all $m\in \mathbb{N}$, there is a sorting network such that the length of every path from $x_i$ to $y_j$ is exactly $\lceil 100\log^2 m \rceil$ (in the directed graph formed by contracting every comparator into a single node).
\end{lemma}
\par The above can be proved by inspecting any common method of constructing a sorting network, for example the method of Batcher \cite{batchersorting}. Indeed, the bound of Batcher is in terms of the \textit{depth} of the network as opposed to the total number of comparators, so we can simply add redundant comparators to ensure the conclusion of Lemma~\ref{Lemma_sorting_network}.

\par We now show how to simulate the task of a comparator in a sorting network via a collection of paths. 

\begin{lemma}\label{Lemma_find_comparator}
Let $p\geq n^{-1/700}$. Let $R^1, R^2$ be $p$-random subsets of $G$, sampled independently.

With high probability, for any $U\subseteq G$ with $|U|\leq p^{800}n/10^{4010}$,  there is a subgraph $C\subseteq K^\pm_G[R^1\setminus U; R^2\setminus U]$ consisting of  $12$ vertices and $10$ colours containing vertices $x^-,x^+,y^-,y^+$ and directed paths $Q_{x^-,y^-}, Q_{x^+,y^+}, Q_{x^-,y^+}, Q_{x^+,y^-}$ with each $Q_{x,y}$ having length 5 and going from $x$ to $y$. Additionally the vertices and colours of the path pairs $(Q_{x^-,y^-}, Q_{x^+,y^+})$ and $(Q_{x^-,y^+}, Q_{x^+,y^-})$ both partition the $12$ vertices and $10$ colours of $C$.
\end{lemma}
\begin{proof}
We prove the lemma when $K_G^\pm=K_G^+$. A slight change in variables proves the lemma also when $K_G^\pm=K_G^-$ (see Figure~\ref{Figure_comparator_division}). With high probability, Corollary~\ref{Corollary_separated_set_random} applies with $R=R^1\cap R^2$. Thinking of $x,y,a,b,c,d,f$, as free variables consider the set $$S=\{x,y,a,b,xd,yd,d^{-1}b, d^{-1}c,xf,yf,f^{-1}a, f^{-1}c,xa,yb,axd,byd,xb,ya,xc,yc,d^{-1}bxf, d^{-1}cyf\}$$ (see Figure~\ref{Figure_commutator}). Note that all pairs $w,w'$ are linear and separable (by (a),  since they're all linear in different combinations of free variables). Lemma~\ref{Lemma_separated_set_random} gives a projection $\pi$ which separates $S$ and has $\pi(S)\subseteq R$. This means that $\pi(w)$ are distinct for all $w\in S$. Now the graph given in Figure~\ref{Figure_comparator} satisfies the lemma with $x^-=\pi(x), x^+=\pi(y), y^-=\pi(f^{-1}a), y^+=\pi(f^{-1}c)$, with paths as shown.
\begin{figure}[h]
  \centering
    \includegraphics[width=0.5\textwidth]{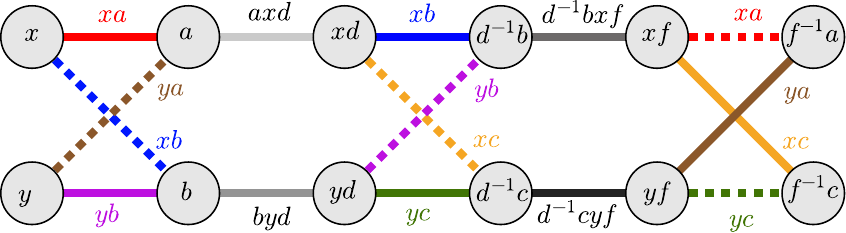}
  \caption{The coloured graph produced by Lemma~\ref{Lemma_find_comparator} for $K_G^+$. Each edge is directed towards the right. Notice that all vertices/edges are labelled by elements of $S$. Since all $w,w'\in S$ are separable, this means that the $\pi$-image of this graph has all vertices/colours distinct (and so in particular has 12 vertices and 10 colours as required). To see that it satisfies the lemma, we need to exhibit paths  $Q_{x^-,y^-}, Q_{x^+,y^+}, Q_{x^-,y^+}, Q_{x^+,y^-}$  between $x^-=x, x^+=y, y^-=f^{-1}a, y^+=f^{-1}c$. The solid lines in the picture give the two paths $Q_{x^-,y^+}, Q_{x^+,y^-}$. Replacing the coloured solid lines for the coloured dashed lines (and keeping all grey lines) gives the two paths $Q_{x^-,y^-}, Q_{x^+,y^+}$.}
\label{Figure_comparator}
\end{figure}

\begin{figure}[h]
  \centering
    \includegraphics[width=0.52\textwidth]{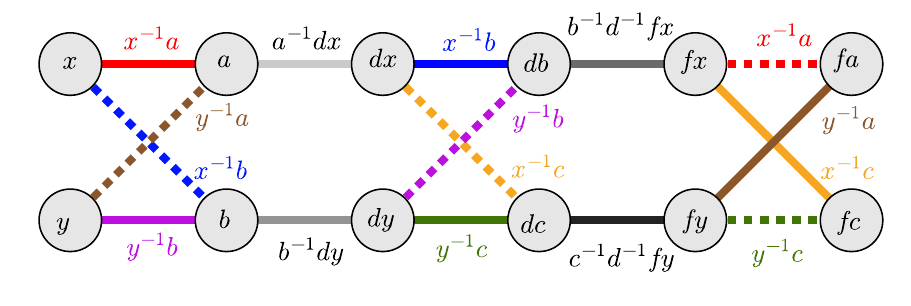}
  \caption{ The analogue of Figure~\ref{Figure_comparator} when proving Lemma~\ref{Lemma_find_comparator} for $K^-_G$. Aside from replacing $S$ with the set of elements given in this figure, the proof for $K^-$ is identical to the one given for $K^+$. }
\label{Figure_comparator_division}
\end{figure}
\end{proof}

We now show how to simulate the task of a wire in a sorting network via short paths. 
\begin{lemma}\label{Lemma_find_wire}
Let $p\geq n^{-1/700}$. 
Let $H_G$ be a multiplication hypergraph, $R^1, R^2$ $p$-random subsets of $G$, sampled independently. With high probability, for any $x,y\in V(K^\pm_G)$, $U\subseteq G$ with $|U|\leq p^{800}n/10^{4010}$, there is a length $3$ $x$ to $y$ path $xuvy$ in $K^\pm_G$ with $u,v\in R_1\setminus U, c(xv), c(vy)\in R_2\setminus U$.
\end{lemma}
\begin{proof} We prove the lemma when $K_G^\pm=K_G^+$. To prove it for $K_G^\pm=K^-_G$ replace the definition of  $S$ below by $S=\{u,v, u^{-1}v, x^{-1}u, v^{-1}y\}$.  
\par With high probability, Corollary~\ref{Corollary_separated_set_random} applies with $R=R^1\cap R^2$. Consider the set $S=\{u,v, uv, xu, vy\}\subseteq G\ast F_3$. Note that all pairs $w,w'$ are linear and the pairs $(u,v), (xu,uy), (xu,uv), (uv,vy)$ are separable (since they're all linear in different combinations of free variables). Lemma~\ref{Lemma_separated_set_random} gives a projection $\pi$ which separates $S$. Now the path $P=x\pi(u)\pi(v)y$ satisfies the lemma --- the vertices are distinct since $u,v$ are separable, while the colours are distinct since $uv, xu, uy$ are pairwise separable.
\end{proof}

We now prove the analogue of Lemma 4.3 from \cite{kuhnsorting} adapted to a setting where the host structure is $K^\pm_G$. 
 
\begin{lemma}\label{lem:sortingnetwork} 
Let $p\geq n^{-1/800}$. Let $t=8\lceil 100\log^2 n \rceil$. Let $R_V$ and $R_C$ be $p$-random subsets of $G$, sampled independently. Then, the following holds with high probability.

 Let $A,B\subseteq G$ be disjoint subsets with $|A|=|B|\leq \frac{p^{1000}n}{10^{5000}\log(n)^3}$, and let $U\subseteq G$ with $|U|\leq \frac{p^{1000}n}{10^{5000}}$. Then, there exists $V\subseteq R_V\setminus U$ and $C\subseteq R_C\setminus U$ such that for any bijection $\phi\colon A\to B$, there exists a system of paths using exactly the vertices/colours of $K^\pm_G[A\cup B\cup V; C]$ from $A$ to $B$, each of length $t$,  and connecting $a$ to $\phi(a)$ for each $a\in A$. 
\end{lemma} 
\begin{proof} With high probability, Lemmas~\ref{Lemma_find_comparator} and~\ref{Lemma_find_wire} apply. 
Let $N$ be a sorting network given by Lemma~\ref{Lemma_sorting_network}, with $m:=|A|=|B|$, noting this sorting network has $\leq 200m\log^2 m$ comparators. For each comparator $C_i=\{x_i^-, x_i^+, y_i^-, y_i^+\}$, use Lemma~\ref{Lemma_find_comparator} to find a subgraph $C_i'$ in $R\setminus (A\cup B)$. Identify the nodes $x_i^-, x_i^+, y_i^-, y_i^+$ of the comparator with the vertices  $x_i^-, x_i^+, y_i^-, y_i^+$  of $C_i'$.
Let $A=\{x_1, \dots, x_m\}, B=\{y_1, \dots, y_m\}$. For each wire $xy$  of the sorting network use Lemma~\ref{Lemma_find_wire} to find a rainbow length 3 path $P_{xy}$ joining corresponding vertices of $A\cup B \cup \bigcup C_i$. By enlarging the set $U$ at all these applications, we can ensure that the subgraphs $C_i'$ are all disjoint, and that the paths $P_{xy}$ are all internally disjoint and colour-disjoint from the subgraphs $C_i'$ and from each other.
We claim that $V=A\cup B \cup \bigcup V(C_i)\cup \bigcup V(P_w)$ and $C=\bigcup C(C_i)\cup \bigcup C(P_w)$ satisfy the lemma. 

Consider a bijection $\phi\colon A\to B$. This gives a permutation $\sigma$ of $[m]$ so that $\phi(x_i)=y_{\sigma(i)}$.  Assign value $v(a_i):=\sigma(i)$. This gives a value to each node and wire of the sorting network as in Definition~\ref{Definition_sorting_network}. We now translate this into values for the corresponding paths/vertices in $K^\pm_G$. The values we assign come from the set $\{1, \dots, m\}\cup \{0\}$. For each wire $xy$ of the sorting network, define $v(P_{xy})=v(x)=v(y)$ and give all vertices/edges of $P_{xy}$ this value. For each comparator $C=\{y_i^-,y_i^+,x_i^-, x_i^+\}$ we have either $v(y_i^-)=v(x_i^-), v(y_i^+)=v(x_i^+)$ or  $v(y_i^-)=v(x_i^+), v(y_i^+)=v(x_i^-)$. In the former case, define $v(Q_{y_i^-,x_i^-})=v(y_i^-)=v(x_i^-)$, $v(Q_{y_i^+,x_i^+})=v(y_i^+)=v(x_i^+)$, and give the vertices and edges of the corresponding paths the same value. Give the paths $Q_{y_i^-,x_i^+}$ and $v(Q_{y_i^+,x_i^-})$ as well as the unused edges of $C_i'$ (those edges of $C_i'$ not on $Q_{y_i^-,x_i^-}$ or $Q_{y_i^+,x_i^+}$) value $0$. In the latter case, define $v(Q_{y_i^-,x_i^+})=v(y_i^-)=v(x_i^+)$, $v(Q_{y_i^+,x_i^-})=v(y_i^+)=v(x_i^-)$, and give the vertices and edges of the corresponding paths the same value. As before, give the other two paths and the unused edges value $0$. Note that this way every vertex/edge of $U$ gets a value, and these values match those that corresponding nodes/wires have in $N$. For every $i=1, \dots, m$, let $$P_i=\bigcup_{v(P_{xy})=i} P_{xy}\cup \bigcup_{v(Q_{y_i^{\diamond_1},x_i^{\diamond_2}})=i} Q_{y_i^{\diamond_1},x_i^{\diamond_2}}.$$ 
\par We clarify that above $xy$ is quantified over the set of wires, and $\diamond_j$ is quantified over $\{+,-\}$. We claim that $P_1, \dots, P_m$ are each paths and have all the required properties. 

First note that every vertex in $V\setminus (A\cup B)$ has exactly one in-going edge of non-zero value and exactly one outgoing edge of non-zero value. The vertices $a_i$ have no in-going edges and one out-going edge (whose value is $\sigma(i)$). The vertices $b_i$ have no out-going edges and one in-going edge (whose value is $i$). Combined with the whole graph being acyclic (which holds due to the sorting network being acyclic), this shows that $\bigcup P_i$ is a union of paths.
\par Since $v(a_i)=\sigma(i)$ and $v(b_i)=i$, path $P_{\sigma(i)}$ goes from $a_{i}$ to $b_{\sigma(i)}$. In particular, each $P_i$ is a path. Also, this shows that the paths partition the vertices $V$ and have the correct endpoints. We now show that their union is rainbow using exactly the colours $C$. The fact that every colour of $\bigcup C(P_{xy})$ is used exactly once comes from the fact that every wire has a value, and so every edge of each $P_{xy}$ as a value. So such colours are used at least once (and hence exactly once because these colours occur once in the whole graph). The colours on the comparators $C_i'$ are used once as a consequence of Lemma~\ref{Lemma_find_comparator}. To see this, note each such colour comes up exactly twice --- once in the paths $Q_{x_i^-,y_i^-}\cup Q_{x_i^+,y_i^+}$  and once in the paths $Q_{x_i^-,y_i^+}\cup Q_{x_i^+,y_i^-}$. By the assignment of the values to the comparator one of these always has value $0$. 
\par Finally, to see that each $P_i$ has length exactly $t$, observe that by Lemma~\ref{Lemma_sorting_network} we know each $P_i$ has length $\lceil 100\log^2 n \rceil$ when viewed as a path in the sorting network. As each wire gadget corresponds to a path of length $3$ and each comparator corresponds to a path of length $5$, it follows that each $P_i$ has length $8\lceil 100\log^2 n \rceil$, as required. 
\end{proof}

\subsubsection{Deducing Lemma~\ref{lem:pathlikemain}}

We will need the following technical lemma, which allows us to use the nibble method to saturate large sets of non-random vertices. This is necessary, as after applying Lemma~\ref{lem:sortingnetwork} to a random subset, we will be left with a large set of non-random vertices which do not immediately fit into the setting of our main theorem.

\begin{lemma}\label{lem:extracolours} There exists $C=C_{\ref{lem:extracolours}}\geq 10$ sufficiently large so that the following holds. Let $1/n\ll \gamma$, and let $1\geq a,b,c\geq 1/\log^C n$. Set $m:=\max \{an,bn,cn\}$ and let $\zeta\in[0,1]$ be such that $1/m\leq \zeta^{C}/C$ and $\zeta\leq \min\{a,b,c\}/100$. Suppose $\ell \geq m-m^{1-\gamma}$ and setting $(x,y,z):=(\ell - an, \ell -bn, \ell -cn )$, suppose that $x+y\leq cn/2$, $x+z\leq bn/2$ and $y+z\leq an/2$. Let $A,B,C\subseteq G$ be $a,b,c$-random subsets of $G$ respectively, sampled with $A$ and $B$ disjoint, and $C$ independent of $A,B$. Then, with probability at least $1-1/n^{2.5}$ the following holds. 
\par Let $A', B', C'\subseteq G$ with $|B\setminus B'|, |A\setminus A'|, |C\setminus C'|\leq n^{1-\gamma}$, $(1-\zeta)|C'| =|A'|=|B'|=\ell$. Then, there is a perfect directed $C'$-matching in $K^\pm_G[A',B'; C']$.
\end{lemma}

To prove the above lemma, we will make use of the following result of Montgomery, Pokrovskiy, and Sudakov. We use the notation $x\polysmall y$ to mean that $x,y\in(0,1]$ and there is some absolute constant $C\geq 1$ such that the proof works with $x\leq y^C/C$. Recall that an edge-coloured graph is \textbf{globally} $K$-\textbf{bounded} if each colour occurs at most $K$ times in the colouring. When we say an edge-coloured bipartite graph is typical, we refer to the typicality of the underlying uncoloured bipartite graph. We remark that in \cite{spanningrainbow}, the following lemma is stated with the additional hypothesis that $n^{-1}\polysmall \gamma$, but this is easily seen to be redundant as $\gamma$ being smaller only makes the statement easier to prove (as a $(\gamma, \delta,n)$-typical graph is also $(2\gamma, \delta,n)$-typical).

\begin{lemma}[\cite{spanningrainbow}, Corollary 8.12]\label{lem:spanningrainbowlemma} Let $n,\delta, p,\gamma$ be such that $n^{-1}, \gamma \polysmall p,\delta \leq 1$. Every properly coloured, $(\gamma,\delta, n)$-typical, globally $(1-p)\delta n$-bounded, balanced bipartite graph $G$ of order $2n$ has $(1-p)\delta n$ edge-disjoint rainbow perfect matchings. 
\end{lemma}
\begin{proof}[Proof of Lemma~\ref{lem:extracolours}]
Recall $m:=\max \{an,bn,cn\}$ and note that $m\leq \ell + m^{1-\gamma}$. Set $p:=m/n$. Note that combining the given inequalities we obtain that $a+b+c\geq 12(\ell/n)/5\geq 11p/5$. We define the following random sets.
\begin{itemize}
    \item Let $A_1,A_2\subseteq A$ be disjoint $p-b$ and $(p-c+\zeta)$-random subsets of $A$. Let $A_3:=A\setminus A_1\setminus A_2$ noting that this set is $(a-2p+b+c-\zeta)=:\alpha$-random.
    \item Let $B_1,B_2\subseteq B$ be $(p-a)$ and $(p-c+\zeta)$-random subsets of $B$. Let $B_3:=B\setminus B_1\setminus B_2$ noting that this set is $(b-2p+a+c-\zeta)=\alpha$-random.
    \item Let $C_1,C_2\subseteq C$ be $(p-a)$ and $(p-b)$-random subsets. Let $C_3:=C\setminus C_1\setminus C_2$ noting that this set is $(c-2p+a+b)=:\beta$-random.
\end{itemize}
Note there is space to sample these sets disjointly due to the assumptions on the size of $\ell$ and $a,b,c$. In particular, using that $a+b+c\geq 11p/5$ and $\zeta\leq p/100$, we have that $\alpha,\beta\geq p/5-\zeta\geq p/10=1/10\log^C n$ and $p-a,p-c+\zeta,p-b\geq 0$ so the random sets with these parameters are well-defined.
For each pair of random sets $(A_2, B_2)$, $(A_1, C_2)$, and $(B_1,C_1)$ such that the corresponding randomness parameters are  both at least $n^{-1/600}$ (the randomness parameter for each pair is in fact equal), we have that with probability at least $1-100/n^{3}$, Lemma~\ref{Lemma_2_random_1_deterministic} holds (the linearity and typicality of the corresponding $3$-uniform hypergraph of $K^\pm_G$ follows by a straightforward modification of Observation~\ref{obs:01ntypical}, which we omit). By Chernoff's bound, the sizes of all these random sets are at most $n^{0.6}$ away from their expectations with probability at least $1-1/n^5$. 
\par  Set $\nu:=n^{-0.25}$. For any $v\in G$, the expected number of (out-)neighbours of $v$ in $K^\pm_G[\{v\}\cup B_3; C_3]$ is $\alpha \beta (n-1)$, as $C$ is sampled independently with $B$ and the analogous statement holds for expected number of neighbours of $v$ in $K^\pm_G[A_3\cup\{v\}; C_3]$. By Chernoff's bound and a union bound, both of these random variables are at most $\nu n$ away from their expectation with probability at least $1-1/n^3$. Let $d_A(a,a')$ denote the pair degree of $a$ and $a'$ in $K^\pm_G[\{a,a'\}\cup B_3; C_3]$ (viewed as an uncoloured bipartite graph), and let $d_B(b,b')$ denote the pair degree of $b$ and $b'$ in $K^\pm_G[A_3\cup \{b,b'\}; C_3]$. For each $a,a'$ and $b,b'$, we have that $\mathbb{E}(d_A(a,a'))=\mathbb{E}(d_B(b,b'))=\alpha \beta^2 n$. Further, observe that $d_A(a,a')$ and $d_B(b,b')$ are $2$-Lipschitz random variables. Hence, by Azuma's inequality and a union bound, with probability at least $1-1/n^3$, for each $a,a'$ and $b,b'$, $d_B(b,b')=d_A(a,a')=(\alpha \beta^2 \pm \nu) n$. This establishes in particular that $K^\pm_G[A_3,B_3;C_3]$ is $(\nu, \beta,\alpha n)$-typical as an uncoloured bipartite graph. Let $d(c)$ denote the number of times $c$ occurs in $K^\pm_G[A_3,B_3,\{c\}]$. Then $\mathbb{E}(d(c))\leq\alpha^2n$ for each $c$. By Chernoff's bound and a union bound, $d(c)\leq(1\pm \nu) \alpha^2n$ for each $c\in G$. In particular, this implies that $K^\pm_G[A_3,B_3,C_3]$ is globally $(\alpha^2+\nu)n$-bounded, which implies that it is globally $(1-\zeta/10)\beta (\alpha n)$-bounded (using bounds on $\alpha,\beta$ obtained earlier). 
\par With probability $\geq 1-1/n^{2.5}$, all of the previous properties hold. Let $A',B',C'$ be given with the indicated properties. By properties coming from Lemma~\ref{Lemma_2_random_1_deterministic}, $(A_2,B_2, C'\setminus C)$, $(A_1, B'\setminus B, C_2)$ and $(A'\setminus A, B_1, C_1)$ each contain matchings covering all but $n^{1-\gamma/2}$ vertices. Here, if the corresponding random sets have parameters smaller than $n^{-1/600}$ (so we cannot apply Lemma~\ref{Lemma_2_random_1_deterministic}),  we simply take an empty matching. This is sufficient as in this case by Chernoff's bound, the random sets themselves cannot contain more than $n^{1-1/603}$ elements. Accounting for differences $A\setminus A'$, $B\setminus B'$, $C\setminus C'$ (each of size $\leq n^{1-\gamma}$), we have a matching covering all but $10n^{1-\gamma/2}$ vertices of $(A_2\cap A',B_2\cap B', C'\setminus C)$, $(A_1\cap A', B'\setminus B, C_2\cap C')$ and $(A'\setminus A, B_1\cap B', C_1\cap C')$.
\par The set of leftover vertices of $A'$, $B'$ and $C'$ have a symmetric difference with $A_3,B_3$ and $C_3$ (respectively) of size at most $100n^{1-\gamma/2}$. Set $\gamma':=n^{-\gamma/100}$. By the pair-degree and vertex-degree bounds we obtained earlier, this implies that the associated properly coloured bipartite graph with the leftover vertices is $(\gamma', \beta, \alpha n)$-typical and globally $(1-\zeta/100)\beta (\alpha n)$-bounded. Then, by Lemma~\ref{lem:spanningrainbowlemma} we can find a matching saturating the remaining vertices of $A'$ and $B'$ using the leftover colours from $C'$, as desired. To see that the hypothesis of Lemma~\ref{lem:spanningrainbowlemma} are satisfied, note first that $(\alpha n)^{-1}\polysmall \beta$ holds as $\beta\geq p/10\geq 1/(10\log^C n)$ and $n$ is sufficiently large. Also, $(\alpha n)^{-1}\polysmall \zeta/100$ follows from $1/m\leq \zeta^{C}/C$ (supposing $C$ is sufficiently large) and $\alpha n\geq 10m$ (as $\alpha\geq p/10$).\end{proof}

We can now give the proof of the key lemma of this section.

\begin{lemma}\label{lem:pathlikemain}
    Let $1/n\ll \gamma, p\leq 1$, let $t$ be a positive integer between $\log^7(n)$ and $\log^8(n)$, and set $q:=p/(t-1)$. Let $G$ be a group of order $n$. Let $V_{str}, V_{mid}, V_{end}$ be disjoint random subsets with $V_{str}, V_{end}$ $q$-random and $V_{mid}$ $p$-random. Let $C$ be a $(q+p)$-random subset, sampled independently with the previous sets. Then, with high probability, the following holds. 
    \par Let $V_{str}'$, $V_{end}'$, $V_{mid}'$ be disjoint subsets of $G$, let $C'$ be a subset of $G$, and let $\ell=|V'_{mid}|/(t-1)$. Suppose all of the following hold.
    \begin{enumerate}
        \item For each random set $R\in\{V_{str}, V_{mid}, V_{end}, C\}$, we have that $|R\Delta R'|\leq n^{1-\gamma}$.
        \item Either \textbf{O.} $\sum V_{end}'-\sum V_{str}'=\sum C'$ or \textbf{C.} $\sum V_{end}'+\sum V_{str}'+ 2\sum V'_{mid}=\sum C'$ holds in the abelianization of $G$.
        \item $\id \notin C'$ if $G$ is an elementary abelian $2$-group.
        \item $\ell:=|V_{str}'|=|V_{end}'|=|V'_{mid}|/(t-1)=|C'|/t$
    \end{enumerate}
    Then, given any bijection $f\colon V_{str}'\to V_{end}'$, we have that $K^\ast_G[V_{str}'\cup V_{end}'\cup V_{mid}';C']$ has a rainbow $\vec{P}_t$-factor where each path starts on some $v\in V_{str}'$ and ends on $f(v)\in V_{end}'$ where $\ast=+$ if \textbf{C} holds and $\ast=-$ if \textbf{O} holds.
\end{lemma}
\begin{proof} Set $m=8\lceil 100 \log^2 n \rceil$, $r=10^{-100}$. Partition $V_{mid}$ into random subsets $R_V, V_1,\ldots, V_{t-m}$ where $R_V$ is $rp$-random, $V_1, V_2,V_3$ and  $V_{t-m}$ are $q$-random and the rest are $((1-r)p-4q)/(t-m-4)$-random. Note that $9q/10\leq ((1-r)p-4q)/(t-m-4)\leq (1-r^{100})q$.  Independently with the previous sets, partition $C$ into $R_C, C_1,\ldots, C_{t-m}$ where $R_C$ is $rp$-random, $C_1, C_2$ and $C_{t-m}$ are $q$-random, and the rest are $((1-r)p-3q)/(t-m-3)$-random. Note similarly that $9q/10\leq ((1-r)p-3q)/(t-m-3)\leq (1-r^{100})q$.
\par With high probability, Theorem~\ref{thm:completeortho} applies with the $q$-random sets $(V_2, V_3, C_2)$, Lemma~\ref{lem:sortingnetwork} applies with $R_V, R_C$ and Lemma~\ref{lem:extracolours} applies with all potential values of $\ell$ and all rational values of $\zeta$ with denominator at most $n$ for all triples of random sets \begin{equation}\label{triples}(V_{str}, V_1,  C_1), (V_3, V_4, C_3), (V_4, V_5, C_4), \ldots, (V_{t-m-1}, V_{t-m}, C_{t-m-1}), (V_{t-m}, V_{end}, C_{t-m}).\end{equation}  We remark that here we use a union bound over all of the $\leq n$ potential values of $\ell$ and $p$ as well as all of the  $\leq \log^5 n$ listed triples (note $a,b,c\geq 1 /\log^9(n)$ in each of these applications). This is possible as the failure probability of Lemma~\ref{lem:extracolours} is at most $1/n^{2.5}$. %An easy check shows that the hypothesis of Lemma~\ref{lem:extracolours} on the randomness parameters holds for each listed triple. Indeed, each random set listed above has parameter at most $q$, and we have $\ell\geq qn-n^{0.61}$ which checks the first inequality in Lemma~\ref{lem:extracolours}. Also, each random set listed has parameter at least $9q/10$ and $\ell\leq qn+n^{0.61}$ which implies the second inequality in Lemma~\ref{lem:extracolours}. 
Finally, with high probability, all the random sets have sizes within $n^{0.6}$ of their expectations. 
\par Fix all the random sets and the integer $\ell$ so they have all of the collected properties. Fix also the sets $V_{str}', V_{mid}', V_{end}', C'$ so that they satisfy properties (1-4). Note that as the random sets have size close to their expectations, and property $(1)$ holds for $V_{str}'/V_{mid}'/V_{end}'$, this implies in particular that $\ell=qn\pm n^{1-\gamma/2}$. Recall that the maximum and minimum values of the randomness parameters of the random sets listed in (\ref{triples}) are $q$ and $9q/10$ respectively. These imply that $\ell$ satisfies the requirements for Lemma~\ref{lem:extracolours} (with $\gamma/2$ playing the role of $\gamma$) to be applicable for each of the triples in (\ref{triples})). We define $\zeta$ to be a rational number with denominator at most $n$ as close as possible to $1/\log^{10^{10}}n$ as possible while ensuring that $\ell/(1-\zeta)$ is an integer, noting $\zeta$ also satisfies the constraints of Lemma~\ref{lem:extracolours} (for any triple in (\ref{triples})) if $n$ is sufficiently large.
\par We can find some $R_V'\subseteq R_V$ and $R_C'\subseteq R_C$ such that these sets have the property from Lemma~\ref{lem:sortingnetwork} with respect to $V_1'$ and $V_2'$. That is, for any bijection $\phi\colon V_1'\to V_2'$, there exists a system of rainbow paths using exactly the vertices/colours of $K^\pm_G[V_1'\cup V_2'\cup R_V'; C']$ from $V_1'$ to $V_2'$, each of length $m=8\lceil 100\log^2 n \rceil$, and connecting $v$ to $\phi(v)$ for each $v\in V_1'$. Note that the relevant inequalities hold while applying Lemma~\ref{lem:sortingnetwork} with $(A,B)=(V_1',V_2')$ and $R_V$ and $R_C$, since $|V_1'|=|V_2'|=\ell \leq n/\log^7 n \ll (rp)^{1000}n/\log^4 n$ since $1/n\ll rp$. 
\par Distributing the leftover vertices in $R_V\setminus R_V'$ into $V_i$ (for $i\notin\{1,2,3,t-m\}$), we can find disjoint subsets $V_1', V_2',\ldots, V_{t-m}'\subseteq G$ with the following properties. 
\begin{enumerate}
    \item $R_V', V_1', V_2',\ldots, V_{t-m}'$ partition $V_{mid}'$ 
    \item For each $i\in[t-m]$, $|V_i'|=\ell$ and $|V_i\setminus V_i'|\leq n^{1-\gamma/2}$ 
    \item For each $i\in \{1,2,3, t-m\}$, $|V_i\Delta V_i'|\leq n^{1-\gamma/4}$ 
\end{enumerate}
\par We can perform a similar distribution, adding to each $C_i$ with $i\neq 2$ $\zeta$-fraction more colours than necessary (to be able to apply Lemma~\ref{lem:extracolours}) borrowed from $C_2$, we can find disjoint subsets $C_1', C_2',\cdots, C_{t-m}'\subseteq G$ partitioning $C'\setminus R_C'$ with the following property: $|C'_i|=\ell/(1-\zeta)$ for each $i\neq 2$, $|C_i\setminus C_i'|\leq n^{1-\gamma/4}$ for each $i\neq 2$ and $|C_2\Delta C_2'|\leq \log^{9}n \zeta n$. 

\par Invoke Lemma~\ref{lem:extracolours} with sets $$(V_{str}', V_1',  C_1'), (V_3', V_4', C_3'), (V_4', V_5', C_4'), \ldots, (V_{t-m-1}', V_{t-m}', C_{t-m-1}'), (V_{t-m}', V_{end}', C_{t-m}')$$ to find perfect matchings (saturating the vertex sets, and using all but $\zeta$-fraction of the colours from each of the colour sets). For each $i\neq 2$, denote the colour subset of $C_i'$ used in the corresponding perfect matching by $C_i''$. Let $C_2''$ be the union of $C_2'$ and all of the unused colours from each $C_i'$, that is, $C_i'\setminus C_i''$ ($i\neq 2$), noting that there are $pn(t-m-1)\ll \log^{10}n(\zeta n)$ such unused colours, meaning that $C_2''$ is defined to retake the borrowed colours from earlier.
\begin{claim} $|C_0''|=\ell$ and if \textbf{O} holds, we have $\sum V_3'-\sum V_2'=\sum C_2''$, and if \textbf{C} holds, we have $\sum V_3'+\sum V_2'=\sum C_2''$.
\end{claim}
\begin{proof} To see this, it is convenient to apply the property of $R_V'$ and $R_C'$ coming from Lemma~\ref{lem:sortingnetwork} with an arbitrary choice of $\phi$. This gives a packing of rainbow paths in $K^\pm_G$ with the set of endpoints of the paths being $V_1', V_2'$ and $R_C'$ being the set of colours used on the paths. This, together with assumption $(4)$, and that the colours in $C'\setminus R_C'\setminus C_2''$ have been used to find perfect matchings in the specified sets implies the first part of the claim. In the case that we work with the division digraph, considering all the directed paths we have found so far, this implies that $\sum V_2'-\sum V_{str}=\sum C_1''+\sum R_C'$ and $\sum V_{end}- \sum V_3'=\sum C_3''\cup C_4''\cup \cdots\cup C_{t-m}''$. Combining these equalities with \textbf{O}, we obtain the second part of the claim. In the case that we work with the multiplication digraph instead, a similar argument shows the second part of the claim.
\end{proof}
The previous claim allows us to apply Theorem~\ref{thm:completeortho} with $(V_2',V_3',C_2'')$ to find a perfect matching directed from $V_2'$ towards $V_3'$ both in the case of division digraphs and multiplication digraphs (note $C_2''$ cannot contain $e$ if $G$ is Boolean by assumption). Now, we invoke the property of the sets $R_V'$ and $R_C'$ with a choice of $\phi$ so that we produce a $\vec{P}_t$-factor connecting each $v\in V_{str}$ to $f(v)\in V_{end}$. To define such a $\phi$, for each $v\in V_1$, let $v'$ be the matched neighbour of $v$ in $V_{str}$. Consider the vertex $v''$ which is obtained by starting with $f(v')\in V_{end}$, and following each matched edge until we reach a vertex of $V_2$. Set $\phi(v)=v''$. It is easy to see that this function $\phi$ has the desired behaviour, concluding the proof.

\end{proof}
\subsubsection{Deducing Theorem~\ref{thm:hamiltonpathsingroups}}
\begin{proof} Set $t=\lfloor 2 \log^7(n)\rfloor$. Let $w$ be the remainder when $n$ is divided by $t+1$, and set $\ell$ so that $n=(t+1)\ell + w$. Note that $w\leq t+1\leq 2\log^7 n$. Set $q:=(\ell-1)/n$ and set $p:=(t-1)q$, noting $2q+p=1\pm \frac{3t}n$. Take random disjoint sets $V_{str}$, $V_{end}$, and $V_{mid}$ where the former two are $q$-random and the latter is $p$-random (noting that this is essentially a partition of $G$). Independently, take disjoint random sets $C_1$ and $C_0$ where $C_1$ is $(q+p)$-random and $C_0$ is $q$-random (hence $C_0$ and $C_1$ almost partition $G$). With high probability, Lemma~\ref{lem:pathlikemain} applies with $V_{str}$, $V_{end}, V_{mid}$ and $C_1$ (set to be $C$ in the statement of Lemma~\ref{lem:pathlikemain}) and $t$. With high probability, Lemma~\ref{lem:extracolours} applies with $(V_{str}, V_{end}, C_0)$ and $\ell-1$ (in this application, $a=b=c=q$ and $\ell:= \ell -1$ so the relevant inequalities hold for Lemma~\ref{lem:extracolours} to apply). With high probability, all the random sets have size close to their expectations. Fix sets $V,C$ and vertices $x,y$ as in the statement of the theorem. 
\par As $w=o(n)$, we can greedily find a rainbow directed path using vertices from $V\setminus\{y\}$ and colours from $C$, say $P_0$, on $w+1$ vertices, starting on $x$ and terminating on some $x'$ where $x'$. Now, we may partition $(V\setminus V(P_0))\cup \{x'\}$ as $V_{str}'$, $V_{end}'$ and $V_{mid}'$ so that $x'\in V_{str}'$ and $y\in V_{end}'$, $\ell=|V_{str}'|=|V_{end}'|=|V_{mid}'|/(t-1)$ (possible due to the divisibility condition coming from the size of $P_0$) and for each random set $R$, $R$ and $R'$ have small symmetric difference. For this, it is important that  $2q+p=1\pm \frac{3t}n$. Set $k:=\lfloor n^{1-10^{-5}} \rfloor$. We may partition $C\setminus C(P_0)$ as $C'$ and $C_0'$ so that $C_0'=\ell + k -1$, and $C':=C\setminus C_0'$ where the pairs of sets $(C_0', C_0)$ and $(C_1,C')$ have small symmetric difference. 
\par Now, invoke Lemma~\ref{lem:extracolours} with the sets $(V'_{end}-y, V'_{str}-x', C_0')$ to find a perfect directed matching from $V'_{end}-y$ to $V'_{str}-x'$ using $\ell-1$ colours coming from $C_0'$, calling this set $C_0''$. Re-define (without relabelling) $C'$ to include the $k$ unused colours from $C_0'$, noting that $C'$ still has small symmetric difference with $C_1$. Note that $C':=C\setminus C(P_0)\setminus C_0''$.
\begin{claim}
    Hypothesis 2. of Lemma~\ref{lem:pathlikemain} holds with the sets $V_{str}', V_{end}', V_{mid}', C'$. 
\end{claim}
\begin{proof}
    Suppose first that $K_G^\pm=K_G^{-}$, that is, we work with the division digraph. Then, it must be that $\sum (V_{str}'-x')-\sum (V_{end}'-y)=\sum C_0''$ from the perfect matching we found. We also have $\sum C_0''=\sum C-\sum C' -\sum C(P_0)$, and that $\sum C(P_0)=x'-x$ (again using that we work with the division digraph). By assumption of the theorem, we also have $\sum C=y-x$. Adding up all of these equalities, we obtain $\sum V_{end}' - \sum V_{str}'=\sum C'$ as desired. An analogous argument works for multiplication digraphs and we omit the details. 
\end{proof}

 The other hypotheses are easy to verify and thus we can invoke Lemma~\ref{lem:pathlikemain} with an appropriate choice of a bijection $f\colon V_{str}'\to V_{end}'$ so that the union of the path system we obtain combined with the initial perfect matching creates a path directed from $x'$ to $y$. Combined with $P_0$, we obtain the desired path from $x$ to $y$, concluding the proof.
\end{proof}

\subsection{Zero sum partitions}\label{sec:tannenbaum}
In this section we apply our main theorem to solve a conjecture of Cichacz and a problem of Tannenbaum about partitions of abelian groups into zero-sum sets. 
\subsubsection{The Cichacz Conjecture}
The following lemma makes our main theorem easier to apply by slightly modifying a set into zero-sum sets of specified size. 
\begin{lemma}\label{Lemma_find_set_with_correct_sum}
Let $p\geq n^{-1/700}$ and let $R$ be a $p$-random subset of an abelian group $G$. With high probability the following holds. 

Let $\epsilon\in [2\log n/\sqrt n, p^{800}/10^{4010}]$. For any $m$ with $|m-pn|\leq  \epsilon n$, $g\in G$ and $Z$ with $|Z|\geq m+3$, $|R\setminus Z|\leq  \epsilon n$, there is a set $R'\subseteq Z$ with $|R'|=m$, $|R'\triangle R|\leq 6\epsilon n$, and $\sum R'=g$.
\end{lemma}  
\begin{proof}
By Chernoff's bound, with high probability $||R|-pn|\leq \sqrt n \log n\leq \epsilon n$. Also with high probability, Corollary~\ref{Corollary_separated_set_random} applies to $R$. Let $\epsilon\in [2\log n/\sqrt n, p^{800}/10^{4010}]$, $m$ with $|m-pn|\leq  \epsilon n$, $g\in G$ and  $Z$ with  $|Z|\geq m+3$, $|R\setminus Z|\leq  \epsilon n$.
Let $X$ be a subset of $Z$ of order $m+3$ which either contains or is contained in $R\cap Z$ (depending on whether $|R\cap Z|\leq m+3$ or not).
%Let $X$ equal $Z\cap R$ together with $m+3-|Z\cap R|$ elements of $Z\setminus R$ (this is possible since $m+3\leq |Z|$\AP{note to self: it looks like the case  when $|Z\cap R|> m+3$ is missing here}). 
Note that in both cases $|X\setminus R|\le m+3-|R\cap Z|$ and $|R\setminus X|\le |R\setminus Z|$. These give 
%Note that $X\subseteq Z$ and has $|X|=m+3$ and 
$|X\triangle R|=|R\setminus X|+|X\setminus R|\leq |R\setminus Z|+(m+3-|R\cap Z|)=2|R\setminus Z|+3+m-|R|\le 2|R\setminus Z|+3+|m-|R|| \leq 2|R\setminus Z|+3+||R|-pn| +|m-pn|\leq 5\epsilon n$. Set $U=R\setminus X$ to get a set of size $\leq 5\epsilon n\leq p^{800}n/10^{4005}$. Let $h$ be an element such that $h=\sum X-g$ in $G^{ab}$.
Thinking of $x$ and $y$ as free variables in $G\ast F_2$, consider the set  $S=\{hxy, x^{-1}, y^{-1}\}\subseteq G\ast F_2$ and note that all pairs of words in $S$ are linear and separable (by part (a) of the definition of separable). Thus there is a projection $\pi:G\ast F_2\to G$ which has  $\pi(S)\subseteq R\setminus U=R\cap X$ and which separates $S$ (i.e. has $\pi(hxy), \pi(x^{-1}), \pi(y^{-1})$ distinct). Thus setting $R'=X\setminus \{\pi(hxy), \pi(x^{-1}), \pi(y^{-1})\}$ gives a set of size $|X|-3=m$ with $\sum R'=\sum X-h=g$.
\end{proof}

The first result about zero-sum partitions that we prove gives a zero-sum partition of a subset of a group into sets of some small fixed size.  
\begin{lemma}\label{Lemma_zero_sum_equipartition}
Let $p\geq n^{-1/10^{100}}$ and $k=3,4$, or $5$. Let $R$ be a $p$-random subset of an abelian group $G$. With high probability the following holds. 

Let $X\subseteq G$ with $|X\triangle R|\leq  p^{10^{10}}n/\log(n)^{10^{18}}$, $0\notin X$, $\sum X=0$, and $|X|\equiv 0\pmod k$. Then, $X$ can be partitioned into zero-sum sets of size $k$.
\end{lemma} 
\begin{proof}
Partition $R$ into  $(p/k)$-random sets $R_1, \dots, R_k$ by placing each vertex of $R$ in each set independently with probability $1/k$. Let $S_1, S_2$ be $(p/k)$-random, and chosen independently of these and each other. Note that for $i=1,2$, the sets $S_i^{-1}=\{-s:s\in S_i\}$ are also $(p/k)$-random and independent of $R_1, \dots, R_k$. By Chernoff's bound, with high probability, $|R_i|=pn/k\pm \sqrt n \log n$.
With high probability Lemma~\ref{Lemma_find_set_with_correct_sum} applies to all these sets. Also with high probability, the conclusion of Theorems~\ref{thm:maintheorem_disjoint},~\ref{thm:maintheoremnondisjoint}, or~\ref{thm:maintheoremsemidisjoint} applies to any choice of 3 sets out of $R_1, \dots, R_k, S_1, S_2, S_1^{-1}, S_2^{-1}$  aside from choices which use both $S_i$ and $S_i^{-1}$ for $i=1,2$.

Let $X\subseteq G$  with $|X\triangle R|\leq  p^{10^{10}}n/\log(n)^{10^{18}}$, $\sum X=0$, and $|X|\equiv 0\pmod k$. Partition $X$ into sets $X_1, \dots, X_k$ satisfying  $|X_i|=|X|/k, |X_i\triangle R_i|\leq  p^{10^{10}}n/\log(n)^{10^{16}}, \sum X_i=0$ for $i=1, \dots, k$ (to do this, use Lemma~\ref{Lemma_find_set_with_correct_sum} $k-1$ times, applying it to the pairs $(R,Z)=(R_1, X), (R_2, X\setminus X_1), \dots, (R_{k-1}, X\setminus\bigcup_{i=1}^{k-2}X_i)$ with $p'=p/k, m=|X|/k, \epsilon = p^{10^{10}}n/\log(n)^{10^{17}}, g=0$. Afterwards setting $X_k=X\setminus\bigcup_{i=1}^{k-1}X_i$ gives a set with $\sum X_k=0$, $|X_k|=|X|/k$, and $|X_i\triangle R_i|\leq  p^{10^{16}}n/\log(n)^{10^{10}}$). 
Use Lemma~\ref{Lemma_find_set_with_correct_sum} to choose $Y_1,Y_2$ with  $|Y_1\triangle S_1|,|Y_2\triangle S_1|\leq p^{10^{10}}n/\log(n)^{10^{16}}$,  $|Y_1|=|Y_2|=|X|/k$ and $\sum Y_1, \sum Y_2=0$ (for this application, use $Z=G$). We proceed differently based on the value of $k$:
\begin{itemize}
\item If $k=3$, then we have $\sum X_1+\sum X_2+\sum X_3=0$ and so Theorem~\ref{thm:maintheoremsemidisjoint} applies to these sets giving a perfect matching (whose edges give a partition of $X$ into zero-sum sets of size $3$). Note that here we used that $0\notin X$ in the case that $G$ is an elementary $2$-group.
\item If $k=4$, then $\sum X_1+\sum X_2+\sum Y_1=0$ and $\sum X_3+\sum X_3+\sum Y_1^{-1}=0$. Thus Theorem~\ref{thm:maintheoremsemidisjoint} gives perfect matchings $M_1, M_2$ in $H_G[X_1,X_1,Y_1]$, $H_G[X_3,X_4,Y_1^{-1}]$ respectively. For each $y\in Y_1$, let  $(a_y, b_y, y)$ and $(c_y, d_y, -y)$ be the edges of $M_1/M_2$ through $y$/$-y$ respectively. Notice that for each $y$, we have $a_y+ b_y+ c_y+ d_y =(a_y+ b_y+ y)+(c_y+ d_y -y)=0$ and so $\{(a_y, b_y, c_y, d_y ): y\in Y_1\}$ gives a partition of $X$ into zero-sum sets of size $4$.
\item If $k=5$, then $\sum X_1+\sum X_2+\sum Y_1=0$, $\sum X_4+\sum X_5+\sum Y_2=0$, and $\sum X_3+\sum Y_1^{-1}+\sum Y_2^{-1}=0$. Thus Theorems~\ref{thm:maintheoremnondisjoint},~\ref{thm:maintheoremsemidisjoint} give perfect matchings $M_1, M_2,M_3$ in $H_G[X_1,X_1,Y_1]$, $H_G[X_4,X_5,Y_2]$,  $H_G[X_3,Y_1^{-1},Y_2^{-1}]$ respectively. For each $x\in X_3$, let  $(x, y_x, z_x)\in M_3$ be the edge of $M_3$ through $x$ and let $(a_x, b_x, -y_x)$, $(c_x, d_x, -z_x)$ be the edges of $M_1/M_2$ through $-y_x$/$-z_x$ respectively. Notice that for each $x$, we have $a_x+ b_x+ x+c_x+ d_x =(a_x+ b_x- y_x)+(x+y_x+z_x)+(c_x+ d_x -z_x)=0$ and so $\{(a_y, b_y,x, c_y, d_y ): x\in X_3\}$ gives a partition of $X$ into zero-sum sets of size $5$.
\end{itemize}
\end{proof}

We now prove a version of the previous result where the sizes of the zero-sum sets are not fixed. We remark that the $p=1$, $X=G\setminus \id$ case of the following result proves  Conjecture~\ref{Conjecture_Cichacz} for sufficiently large groups. We'll work with multisets in this section and will use $m_i(M)$ to denote the number of occurrences of element $i$ in a multiset $M$. 

\begin{lemma}\label{Lemma_zero_sum_345_partition}
Let $p\geq 1/\log(n)^{10^{99}}$ and $M\subseteq \{3,4,5, \dots\}$ be a multiset.
Let $R$ be a $p$-random subset of an abelian group $G$. With high probability the following holds. 

Let $X\subseteq G$  with $|X\triangle R|\leq  p^{10^{10}}n/\log(n)^{10^{24}}$, $0\notin X$, $\sum X=0$, and $|X|=\sum M$. Then, $X$ has a zero-sum  $M$-partition.
\end{lemma} 
\begin{proof}
Without loss of generality, we can suppose that $M\subseteq \{3,4,5\}$. Otherwise consider a set $M'$ formed by replacing each $y\in M$ with $y>5$ by $y_1, \dots, y_t$ with $y_1+\dots+y_t=y$ and $y_i\in\{3,4, 5\}$. It is easy to see that $M'$ still satisfies the hypotheses of the lemma and a $M'$-partition of $G$ gives a $M$-partition by combining $y_1,\dots,y_t$-sets into a single $y$-set for each $y$.
\par For $i=3,4,5$, fix $m_i=m_i(M)$.
Partition $R=R_3\cup R_4\cup R_5$ into disjoint sets with $R_i$ being $q_i$-random for $q_i:=\frac{im_i}{3m_3+4m_4+5m_5}$. For each $i\in\{3,4,5\}$, either $q_i\leq n^{-1/10^{100}}$, or with high probability, $R_i$ satisfies the conclusions of Lemmas~\ref{Lemma_find_set_with_correct_sum} and~\ref{Lemma_zero_sum_equipartition}.  We'll suppose that  $q_3\geq q_4\geq q_5$ --- the proof is identical if these are ordered differently (switching the roles of $R_3$/$R_4$/$R_5$ correspondingly). In particular, we will suppose that $R_3$ satisfies the conclusions of Lemmas~\ref{Lemma_find_set_with_correct_sum},~\ref{Lemma_zero_sum_equipartition}, as well as Lemma~\ref{lem:greedyzerosum}  with high probability. Also, with high probability, by Chernoff's bound, we have $|R_i|=q_i pn\pm \sqrt n \log n$.
\par Let $X\subseteq G$  with $|X\triangle R|\leq  p^{10^{10}}n/\log(n)^{10^{24}}$, $\sum X=0$, and $|X|=3m_3+4m_4+5m_5$. 
Notice that since $|R|=|R_3|+|R_4|+|R_5|=(q_3+q_4+q_5)pn\pm 3\sqrt n\log n=pn\pm 3\sqrt n \log n$ and $|R|=|X|\pm p^{10^{10}}n/\log(n)^{10^{23}}$, we have $pn=3m_3+4m_4+5m_5\pm p^{10^{10}}n/\log(n)^{10^{22}}$. Using the definitions of $q_i$, this implies that $|R_i|=q_ipn\pm \sqrt n \log n=im_i\pm  p^{10^{10}}n/\log(n)^{10^{21}}$. In particular, these imply $|(R_4\cup R_3)\cap X|\geq 4m_4+3m_3- 2 p^{10^{10}}n/\log(n)^{10^{20}}\geq 4m_4+0.1pn$, and $|(R_5\cup R_3)\cap X|\geq 5m_5+3m_3-   2p^{10^{10}}n/\log(n)^{10^{20}}\geq 5m_5+0.1pn$ (using $q_3\geq q_4, q_5$).
\par If $q_4$ (or $q_5$) is less than $n^{-1/10^{100}}$, we will argue that we may assume that $m_4=0$ (or $m_5=0$), only at the expense of working with a set $X$ with a slightly larger symmetric difference with $X$, i.e. a set $X$ for which $|X\triangle R|\leq  p^{10^{10}}n/\log(n)^{10^{23}}$. To see this, note if $q_4\leq n^{-1/10^{100}}$, then we can find $4m_4\leq n^{1-\eps}$ many zero-sum sets of size $4$. We can achieve this by invoking the property from Lemma~\ref{lem:greedyzerosum}. As $n^{1-\eps}\ll p^{10^{10}}n/\log(n)^{10^{100}}$, and the set of vertices we delete are   zero-sum, we can continue the argument with $m_4=0$, and the remaining vertices of $X$ (which are zero-sum). If $q_5$ is also small, the same operation can be performed disjointly to find a small number of zero-sum sets of size $5$.  
\par If $m_i=0$ for some $i\in \{4,5\}$, set $X_i=\emptyset$. Otherwise, by the previous paragraph, the corresponding lemmas apply to $R_i$, and $X_4$ and $X_5$ are defined via the following operations. By Lemma~\ref{Lemma_find_set_with_correct_sum} applied with $R=R_4, Z=(R_4\cup R_3)\cap X$, pick a zero-sum set $X_4\subseteq (R_4\cup R_3)\cap X$ with $|X_4|=4m_4$ and having $|X_4\triangle R_4|\leq   p^{10^{10}}n/\log(n)^{10^{19}}$. 
By Lemma~\ref{Lemma_find_set_with_correct_sum} applied with $R=R_5, Z= (R_5\cup R_3)\cap X\setminus X_4$, pick a zero-sum set $X_5\subseteq (R_5\cup R_3)\cap X\setminus X_4$ with $|X_5|=5m_5$ and having $|X_5\triangle R_5|\leq p^{10^{10}}n/\log(n)^{10^{19}}$.
Let $X_3=X\setminus (X_4\cup X_5)$ and notice that $|X_3|=3m_3$, $|X_3\triangle R_3|\leq p^{10^{10}}n/\log(n)^{10^{18}}$, and $\sum X_3=\sum X-\sum X_4-\sum X_5=0$. From Lemma~\ref{Lemma_zero_sum_equipartition} we have that each $X_i$ has a zero-sum partition into $m_i$ sets of size $i$. Together, all of these sets give a  zero-sum  $M$-partition of $X$.
\end{proof}

\subsubsection{Tannenbaum's Problem}
In the remainder of this section we characterise the multisets $M$ for which it is possible to find a zero-sum $M$-partition of $G\setminus 0$ for an abelian group $G$, thereby resolving an old problem of Tannenbaum \cite{tannenbaumpartitions}. It is well-known that abelian groups with just one involution (an order $2$ element of $G$) have that $\sum G\neq 0$, so a zero-sum partition in this case is impossible. Abelian groups without involutions have odd order, and for such groups a characterisation of the multisets $M$ for which $G$ has a zero-sum $M$-partition  was given by Tannenbaum \cite{tannenbaumold} (the necessary and sufficient condition is that $M\subseteq \{2,3,\dots\}$ and $\sum M=n-1$). Hence, we will be concerned with abelian groups with at least $3$ involutions for the rest of the section. For groups with $3$ involutions, a characterization was found by Zeng~\cite{Zeng2015OnZP} (the necessary and sufficient condition is again that $M\subseteq \{2,3,\dots\}$ and $\sum M=n-1$). The characterization for groups with $>3$ involutions turns out to be substantially more involved, see Theorem~\ref{Theorem_zero_sum_partitions_main}.

We begin with the following technical lemma that allows us to work with a random partition which will be critical for the proof of Lemma~\ref{Lemma_zero_sum_large_t}.

\begin{lemma}\label{Lemma_random_paired_subset}
Let $G$ be a size $n$ set and suppose $G$ is partitioned as $G=\{g_1, h_1\}\cup \dots \{g_m, h_m\}\cup I$ (with $|I|+2m=n$). Let $Y$ be a $p$-random subset of $[m]$ and set $X=I\cup \bigcup_{i\in Y}\{g_i, h_i\}$.  Then, we can partition $X=Q\cup R\cup S$ where $Q, R$ are disjoint and  $p/2$-random subsets of $G$, and $S$ is a $(1-p)$-random subset of $I$.
\end{lemma}
\begin{proof}
For any $i$, notice that $P(\{g_i, h_i\}\cap X=\emptyset)=1-p$ and  $P(\{g_i,h_i\}\subseteq  X)=p$, and  $P(|\{g_i,h_i\}\cap X|=1)=0$, with these events being independent for different $i$. In order to make sure that elements end up in $Q$/$R$/$S$ independently, we need to define them carefully as follows: 
Define $Q, R, S$ conditional on the outcome of $Y$ as follows:
For each $\{g_i, h_i\}$ such that $i\in Y$, place both $g_i, h_i$ in $Q$ with probability $a=p/4$, place both $g_i, h_i$ in $R$ with probability $a=p/4$, place  $g_i$ in $Q$ and $h_i$ in $R$ with probability $b=1/2-p/4$, and place  $g_i$ in $R$ and $h_i$ in $Q$ with probability $b=1/2-p/4$ (noting that $2a+2b=1$). Additionally place each $k\in I$ into $Q$ with probability $p/2$, into $R$ with probability $p/2$, and into $S$ with probability $(1-p)$. Do these latter set of choices independently of each other, and independently of the former choices made while choosing $Y$. 

\par It remains to show that $Q$ and $R$ are disjoint $p/2$-random subsets of $G$, as that $Q,R,S$ partitions $X$ and that $S$ is a $(1-p)$-random subset of $I$ follows directly. First, note that each element $g$ of $G$ is included in $Q$ with probability $p(a+b)=p/2$, included in $R$ with probability $p(a+b)=p/2$, and included in both $Q$ and $R$ with probability $0$, regardless of whether $g\in I$. Now, we show that the collection of  such events for each $g\in G$ are independent. Since choices for different $i$ are already done independently, it is sufficient to show that for all $i$, $\{g_i, h_i\}\cap Q$ and $\{g_i, h_i\}\cap R$ are $p/2$-random subsets of $\{g_i, h_i\}$. 
Fix some $i$, and note that $P(g_i\in Q)=P(h_i\in Q)= p(a+b)=p/2$,  $P(g_i, h_i\in Q)= pa=(p/2)^2$, $P(g_i\in Q, h_i\not\in Q)=P(g_i\not\in Q, h_i\in Q) =pb=(p/2)(1-p/2)$, and $P(g_i, h_i\not\in Q)= 1-p+pa=(1-p/2)^2$.  The same holds for $R$ (since $Q$ and $R$ are symmetric), and so we get that  $\{g_i, h_i\}\cap Q$ and $\{g_i, h_i\}\cap R$ are $p/2$-random subsets of $\{g_i, h_i\}$ as required. 
\end{proof}

\par Let $I(G)$ be the set of order $2$ elements of $G$, noting that $I(G)\cup \{0\}$ is isomorphic to $\mathbb{Z}_2^k$ for some $k$. Define $f(G)=(|G|-|I(G)|-1)/2$ and note that $f(G)$ is the number of inverse pairs in $G$ (and so in particular an integer).
A trivial necessary condition to have a zero-sum $M$-partition of $G\setminus 0$ is that $f(G)\geq m_2(M)$ (otherwise some involution would have to be contained in a zero-sum 2-set of the partition, which is impossible).
The following lemma shows that for any multiset $M$ it is possible to find a zero-sum $M$-partition of a group  $G$ as long as $m_2(M)$ is significantly smaller than $f(G)$.

\begin{lemma}\label{Lemma_zero_sum_large_t}
Let $G$ be a sufficiently large abelian group with $|I(G)|\geq 3$, and $M\subseteq \{2,3,4,5, \dots\}$ a multiset with $\sum M=n-1$ and $m_2(M)\leq f(G)-0.0001n$.  Then $G\setminus\{0\}$ has a zero-sum $M$-partition.
\end{lemma} 
\begin{proof} 
As in Lemma~\ref{Lemma_zero_sum_345_partition}, without loss of generality, we can assume that $M\subseteq \{2,3,4,5\}$.
Define $m_i:=m_i(M)$ for $i=2,3,4,5$. The basic idea will be to partition use Lemma~\ref{Lemma_zero_sum_345_partition} to get $m_3/m_4/m_5$ sets of orders $3/4/5$ in such a way that the remaining elements form inverse pairs. We will apply Lemma~\ref{Lemma_zero_sum_345_partition} three times in order to achieve this, twice to random subsets $R,Q$ in the group $G$, and once to a random subset $S$ in the group $I(G)\cup e$. In the paragraphs that follow, we construct these random sets, and check the properties needed for Lemma~\ref{Lemma_zero_sum_345_partition}.

Note that $\sum_{i\geq 3}im_i=\sum M-2m_2=n-1-2m_2=2f(G)+|I(G)|-2m_2\geq |I(G)|+0.0002n$.
Let $p= \frac{3m_3+4m_4+5m_5-|I(G)|}{2f(G)}=\frac{\sum M-2m_2-|I(G)|}{2f(G)}=1-\frac{m_2}{f(G)}\geq \frac{f(G)-m_2(G)}{n}\geq 0.0001$.
If $|I(G)|\geq p^{10^{10}}n/\log(n)^{10^{26}}$, then set  $m_i^Q=\lceil \frac12\frac{pm_in}{n-1-2m_2}\rceil$, $m_i^R=\lceil \frac12\frac{pm_in}{n-1-2m_2}\rceil$, $m_i^S=m_i-2\lceil\frac12 \frac{pm_in}{n-1-2m_2}\rceil$ for $i=3,4,5$. If $|I(G)|< p^{10^{10}}n/\log(n)^{10^{26}}$, then set $m_i^Q=\lfloor m_i/2\rfloor$, $m_i^R=m_i-\lfloor m_i/2 \rfloor$ and $m_i^S=0$ for $i=3,4,5$. Note that in both cases, we have  $m_i=m_i^Q+m_i^R+m_i^S$ for each $i=3,4,5$. Also note that when $|I(G)|\geq p^{10^{10}}n/\log(n)^{10^{26}}$, we have $\sum_{i=3}^5m_i^S\pm 12=\sum_{i=3}^5i(m_i^S\pm 1)= \sum_{i=3}^5(im_i - \frac{pim_in}{n-1-2m_2})=|I(G)|+2f(G)-2m_2-pn=|I(G)|(1-p)-p\leq|I(G)|-3$, and that in all cases we have $\sum_{i=3}^5im_i^Q \leq \sum_{i=3}^5im_i-\sum_{i=3}^5im_i^S-3$.
For $\ast=Q,R,S$, define $M^\ast=\{m_3^{\ast}\times 3, m_4^{\ast}\times 4, m_5^{\ast}\times 5\}$ and note that $M^Q\cup M^R\cup M^S\cup\{m_2\times 2\}=M$.  

Recall $2f(G) =n-|I(G)|-1$ and enumerate $G\setminus (I(G)\cup \{\id\})$ as $g_1, g_1^{-1}, \dots, g_{f(G)}, g_{f(G)}^{-1}$. 
Pick a $p$-random set $Y$ of $\{1, \dots, f(G)\}$, and set $X=I(G)\cup\bigcup_{i\in Y}\{g_i, g_i^{-1}\}$, noting that $\mathbb E(|X|)=2pf(G)+|I(G)|=3m_3+4m_4+5m_5$. 
Use Lemma~\ref{Lemma_random_paired_subset} to partition $X=Q\cup R\cup S$ with $Q, R$ $p/2$-random and $S$ a $(1-p)$-random subset of $I(G)$.
\begin{claim}
With positive probability, the following all hold.
\begin{enumerate}
\item $f(G)-m_2-\sqrt n \log n\leq |Y|\leq f(G)-m_2$.
\item $3m_3+4m_4+5m_5-2\sqrt n\log n\leq |X|\leq 3m_3+4m_4+5m_5$ 
\item $|Q|, |R|, |S|= 3m_3^{\ast}+4m_4^{\ast}+5m_5^{\ast}\pm p^{10^{10}}n/\log(n)^{10^{25}}$ (for ${\ast}=Q,R,S$).  
\item  $Q,R$ satisfy Lemmas~\ref{Lemma_find_set_with_correct_sum} and~\ref{Lemma_zero_sum_345_partition} with respect to $G$.
\item If $|I(G)|\geq p^{10^{10}}n/\log(n)^{10^{26}}$, then  $S$ satisfies Lemmas~\ref{Lemma_find_set_with_correct_sum} and~\ref{Lemma_zero_sum_345_partition} with respect to the subgroup $I(G)\cup \id$.
\end{enumerate}
\end{claim}
\begin{proof}
We'll show that with probability $\geq1/2$ the upper bounds of (1) and (2) both hold, whilst all other parts of the claim hold with high probability. 
For the upper bound of (1), note that $|Y|$ is a binomial random variable with expectation $pf(G)=(3m_3+4m_4+5m_5-|I(G)|)/2=f(G)-m_2\in \mathbb{Z}$. 
For binomial random variables, if the expectation is an integer, then it is also the median. Thus we have that the median of $|Y|$ is  $f(G)-m_2$  which shows that the upper bound (1) holds with probability $\geq 1/2$.  The lower bound of (1) comes from Chernoff's bound. Note that $|X|=2|Y|+|I(G)|$ and $2m_2+3m_3+4m_4+5m_5=n-1=2f(G)+|I(G)|$ show that (1) implies (2). 
For  (3), note that Chernoff's bound gives $|Q|,|R|=pn/2\pm \sqrt n \log n$ and $|S|=(1-p)|I(G)|\pm \sqrt n \log n$. When $|I(G)|\geq p^{10^{10}}n/\log(n)^{10^{26}}$, then $pn/2= \frac12\sum_{i=3}^5i\frac{pm_in}{n-1-2m_2}=\sum_{i=3}^5im_i^Q\pm 3=\sum_{i=3}^5im_i^R\pm 3$ and $(1-p)|I(G)|=\sum_{i=3}^5im_i^S\pm 13$ imply (3). When $|I(G)|< p^{10^{10}}n/\log(n)^{10^{26}}$, then $|S|\leq |I(G)|$ gives (3) for ``$S$'', while the result for $Q,R$ then follows from (2) and $|Q|=|R|\pm 2\sqrt n \log n $. 
  Properties (4), (5) are immediate from Lemmas~\ref{Lemma_find_set_with_correct_sum}  and~\ref{Lemma_zero_sum_345_partition}.
\end{proof}

Partition $G\setminus (X\cup \{0\})$ into $f(G)-|Y|$ zero-sum sets of size $2$ as $Z^2_{|Y|+1}, \dots, Z^2_{f(G)}$. Notice that we have $f(G)\geq f(G)-m_2\geq |Y|$, and so we can set $J=Z^2_{|Y|+1}\cup \dots \cup Z^2_{f(G)-m_2}$ to get a zero-sum set of size exactly  $2f(G)-2m_2-2|Y|=n-1-2m_2-|X|=3m_3+4m_4+5m_5-|Q|-|R|-|S|\leq 3\sqrt n \log n$.
When $|I(G)|\geq p^{10^{10}}n/\log(n)^{10^{26}}$, use Lemma~\ref{Lemma_find_set_with_correct_sum} (with $Z=I(G)$) to pick a subset $S'\subseteq I(G)$ with $|S\triangle S'|\leq p^{10^{10}}n/\log(n)^{10^{25}}$, $|S'|=3m_3^S+4m_4^S+5m_5^S$ and $\sum S'=0$. Otherwise set $S'=\emptyset$ (In this application we use $Z=I(G)$ and $m=\sum_{i=3}^5 m_i^S$ which satisfy $m\leq |Z|-3$ from the first paragraph). 
Use Lemma~\ref{Lemma_find_set_with_correct_sum}  to find a set $Q'\subseteq (X\cup J)\setminus S'$ with $|Q\triangle Q'|\leq p^{10^{10}}n/\log(n)^{10^{24}}$, $|Q'|=3m_3^Q+4m_4^Q+5m_5^Q$ and $\sum Q'=0$ (In this application we use $Z=(X\cup J)\setminus S'$ and $m=\sum_{i=3}^5 m_i^Q$ which satisfy $m\leq |Z|-3$ from the first paragraph).    Set $R'=(X\cup J)\setminus (Q'\cup S')$ and note that $|R\triangle R'|\leq p^{10^{10}}n/\log(n)^{10^{23}}$, $|R'|=3m_3^R+4m_4^R+5m_5^R$, and $\sum R'=0$.
Apply Lemma~\ref{Lemma_zero_sum_345_partition} to $\ast=Q', R', S'$ to get $M^{\ast}$-partitions of these sets respectively. Putting the partitions of $Q', R', S'$ together with the sets $Z^2_{m_2}, \dots, Z^2_{f(G)}$, we get a zero-sum $M$-partition of $G$.
\end{proof}

In the rest of the section we deal with groups and multisets having $f(G)-0.0001\leq m_2(M)\leq f(G)$.
We say that a subset of an abelian group $S$ is $\Sigma$-generic if for every proper non-empty subset $A\subseteq S$, we have that $\sum A$ is generic.

\begin{lemma}\label{Lemma_find_one_generic_subset}
Let $3\leq k\leq 10$ and let $G$ be a sufficiently large abelian group with $|I(G)|\geq  n/10^{9000}$. Then  $I(G)$ contains a zero-sum set $S$ of size $k$ which is  $\Sigma$-generic (in $G$).
\end{lemma}
\begin{proof}

Let $H=I(G)\cup \id$ and recall that this is a subgroup.
Let $v_1, \dots, v_k$ be free variables in $H\ast F_{k-1}$. Let $y= v_{k-1}^{-1}\dots v_1^{-1}$ and set $S=\{v_1, \dots, v_{k-1}, y\}$, noting that pairs $w,w'\in S$ are linear and separable (using part (a) of the definition of ``separable''). 
For all nonempty $A\subseteq \{v_1, \dots, v_{k-1}\}$, define $g_{A}=\prod A$ and set $T=\{g_{A}: \{v_1, \dots, v_{k-1}\} \supseteq A\neq \emptyset\}$. Notice that all elements of $T$ are linear. 
By Lemma~\ref{Lemma_separated_set_random}, there is a projection $\pi:H\ast F_{k-1}\to H$ which separates $\pi(S\cup T)$ and has $\pi(S\cup T)\subseteq H\setminus N(G)$ (for this application, we have $p=1$, $n'=|H|\geq n/10^{9000}$ and $U=N(G)\leq 10^{-9000}\leq p^{800}n/10^{4000}$ since $n$ is sufficiently large). We  have that $\pi(S)$ is a zero-sum  set of size $k$ (since $\prod S=\id$ and all pairs $w,w'\in S$ are separable). The fact that $\pi(S)$ is $\Sigma$-generic follows from $\pi(T)$ being generic (for $A\subseteq S\setminus v_k$, the definition of ``$\Sigma$-generic'' is immediate for $\pi(A)$ due to $\pi(g_A)$ being generic. For $A$ containing $v_k$, note that $\sum \pi(A)=-\sum \pi(S\setminus A)$ which is generic because $\pi(g_{S\setminus A})$ is generic). 
\end{proof}

The following theorem was poved independently by Caccetta-Jia, Engawa, and Tannenbaum and classifies what sorts of zero-sum partitions the groups $\mathbb Z_2^m$ has. We use it as a black box.
\begin{theorem}[Cacceta-Jia \cite{caccetta1997binary}; Engawa \cite{egawa1997graph}; Tannenbaum \cite{tannenbaumpartitions}]\label{Theorem_zero_sum_Z2m}
For every multiset $M\subseteq \{3,4,\dots\}$ with $\sum M=2^m-1$, the group $\mathbb{Z}_2^m\setminus\{0\}$ has a zero-sum $M$-partition.
\end{theorem}

The following lemma is basically a version of this theorem, but additionally guarantees that one set in the partition is $\Sigma$-generic.
\begin{lemma}\label{Lemma_zero_sum_partition_with_generic_subset}
Let $m\geq 2$ and $G$ be a sufficiently large abelian group of order $n$ with $|I(G)|\geq 3$. Let $M\subseteq \{2, \dots, 10\}$ be a multiset with $\sum M=n-1$ and $m_2(M)=f(G)$. Fix some $x\in M$. Then, there is a zero-sum $M$-partition of $G\setminus \{0\}$. Additionally we can assume that a size $x$ set in this partition is $\Sigma$-generic (in $G$).
\end{lemma}
\begin{proof}
Let $M'\subseteq M$ be the sub-multiset consisting of all elements of size at least $3$ and let $M''=\{m_2(G)\times 2\}$. Note that $\sum M'=\sum M-2m_2(M)=n-1-2f(G)=|I(G)|-1$ and that $G\setminus (I(G)\cup \{0\})$ has a zero-sum $M''$-partition $\mathcal P$.

If $|I(G)|\leq n/10^{9000}$, then note that for every $g$, the number of solutions to $x^2=g$ in $G$ is either $|I(G)|+1$ or $0$ (for any two such solutions $x,x'$, we have $y=x^{-1}x'$ is a solution to $y^2=\id$ and the set of such solutions is exactly $I(G)\cup \id$). This means that when $|I(G)|\leq n/10^{9000}$, everything is generic, and so all sets are $\Sigma$-generic. To get the lemma,
use Theorem~\ref{Theorem_zero_sum_Z2m} to get a $M'$-partition of $I(G)$. Extend this to a $M$-partition satisfying the lemma by adding $\mathcal P$.   The same argument works if $|I(G)|\geq n/10^{9000}$ and $x=2$ (since there are at most $10^{9000}\leq f(G)$ non-generic elements, one of the sets in $\mathcal P$ always has only generic elements).

\par So, suppose that  $|I(G)|\geq n/10^{9000}$ and $x\geq 3$. Let $H=I(G)\cup \id$, recalling that this is a subgroup.
Use Lemma~\ref{Lemma_find_one_generic_subset} to find a $\Sigma$-generic set $S_1$ of size $x$. 
Let $X=I(G)\setminus (S_1\cup \id)$ and note that $|X\triangle H|\leq |H|/\log^{10^{24}} |H|$ and $\sum X=0$.
By Lemma~\ref{Lemma_zero_sum_345_partition} applied to $H$ with $p=1$,  $X$ has a zero-sum $(M'\setminus \{x\})$-partition. Together with $S_1$ and $\mathcal P$, this gives a zero-sum $M$-partition of $G$. 
\end{proof}

The following two lemma deals with the case when $m_2$ and $f(G)$ are within a small additive constant of each other.
\begin{lemma}\label{Lemma_zero_sum_small_t}
Let $G$ be a sufficiently large abelian group with $|I(G)|\geq 3$. 
Let $M\subseteq \{2, \dots, 10\}$ be a multiset with $\sum M=n-1$,  $t:=f(G)-m_2(M)\in\{0,\dots, 10\}$. Suppose that the $t$ largest elements of $M$ are all $\geq 3$ and add up to $\geq 2t+3$. Then $G\setminus\{0\}$ has a zero-sum $M$-partition.
\end{lemma}

\begin{proof}
If $f(G)=0$, then  $G\cong \mathbb Z_2^k$ and $m_2(M)=0$, and so the lemma follows from Theorem~\ref{Theorem_zero_sum_Z2m}. So, suppose $f(G)\geq 1$.
Since $f(G)=(n-|I(G)\cup \{0\}|)/2>0$ and $I(G)\cup \{\id\}$ is a subgroup, Lagrange's Theorem tells us that  $|I(G)\cup \{0\}|\leq n/2$.

Let $y_1, \dots, y_t$ be the $t$ largest elements of $M$, noting $y_1, \dots, y_t\geq 3$ and $y_1+ \dots+ y_t\geq 2t+3$.
Let $y_0'=y_1+\dots+y_t-2t\geq 3$, $y_1', \dots, y_t'=2$, and set $M'=(M\setminus\{y_1, \dots, y_t\})\cup\{y_0', y_1', \dots, y_t'\}$.  Note that $m_2(M')= m_2(M)+t=f(G)$ and $\sum M'=\sum M=n-1$.
Use Lemma~\ref{Lemma_zero_sum_partition_with_generic_subset} to find a  zero-sum $M$-partition of $G$ having a $\Sigma$ generic size $y'_0$ set  $A$. Note that $y'_0=y_1+ \dots+ y_t-2t$ means that we can partition $A=A_1\cup \dots \cup A_t$ into sets of size $|A_i|=y_i-2$. Label $A_i=\{a_i^1, \dots, a_i^{y_i-2}\}$ and set $a_i=\prod A_i=a_i^1 \dots a_i^{y_i-2}$. Thinking of  $x$ as the free variable in $G\ast F_1$,
define $S=\{x, a_1x, a_2a_1x, \dots, a_{t-1}\dots a_2a_1x\}$.
Also define the partition $S\cup S^{-1} \cup A=T_1\cup T_2\cup\dots \cup T_t$ with  $T_1=\{x, x^{-1}a_1^{-1}, a_1^1, \dots, a_1^{y_1-2}\}$, $T_2=\{a_1x, x^{-1}a_1^{-1}a_2^{-1}, a_2^1, \dots, a_2^{y_2-2}\}, \dots,  T_t=\{a_{t-1}\dots a_2a_1x, x^{-1}, a_t^1, \dots, a_t^{y_t-2}\}$ (more formally, if we denote $x_i:= a_i\dots a_1x$, then we have $S=\{x_0, \dots, x_{t-1}\}$ and $T_i=A_i\cup \{x_{i-1}, x^{-1}_{i\pmod{t-1}}\}$).

Note that for all $w\in S\cup S^{-1}$ are linear and all non-inverse $w,w'\in S\cup S^{-1}$ are separable (for this, first notice that by $\Sigma$-genericness,  all partial product $a_ia_{i-1}\dots a_j$ are generic unless $i=t,j=1$. Now if $w, w'\in S$ or $w,w'\in S^{-1}$ then the pair is separable by (c), while if $w\in S, w'\in S^{-1}$ the pair is separable by (b)).
Using Lemmas~\ref{Lemma_upper_bound_set_of_linear_words} and~\ref{Lemma_lower_bound_separated_set}, find a projection $\pi$ which separates $S$ and has $\pi(x)\not\in I(G)\cup \{\id\}$ (Lemma~\ref{Lemma_lower_bound_separated_set} gives us $0.9n$ projections which separate $S$, while Lemma~\ref{Lemma_upper_bound_set_of_linear_words} tells us that there are $|I(G)\cup \{0\}|\leq n/2$ projections with $\pi(x)\in I(G)\cup \{0\}$). Note that $\pi(x)\not\in I(G)\cup \{0\}$ implies that $\pi(S)$ is disjoint from $I(G)\cup \{0\}$ (using that $a_1, \dots, a_k\in I(G)\cup \{0\}$ and $I(G)\cup \{0\}$ is a subgroup, we have that all $y\in \pi(S)$ are in one of the cosets $\pi(x)(I(G)\cup \{0\})$ or $\pi(x^{-1})(I(G)\cup \{0\})$).
Note that this implies that all the elements in $\pi(S)$ are distinct (all non-inverse pairs $w,w'\in S$ are separable giving $\pi(w)\neq \pi(w')$. For $w,w'\in S$ with $w'=w^{-1}$, we cannot have $\pi(w)=\pi(w')$ since $\pi(S)$ is disjoint from $I(G)\cup \{\id\}$). 
Now, in each case replace $\pi(S)\cup \pi(S^{-1})\cup A$ in the original  $M'$-partition, by  $\pi(T_1), \dots, \pi(T_t)$. This gives a $M$-partition  of $G$.
\end{proof}

The next lemma is a stronger version of the previous one with the ``$t$ largest elements of $M$ are all $\geq 3$'' condition dropped.
\begin{lemma}\label{Lemma_zero_sum_small_t_better}
Let $G$ be a sufficiently large abelian group with $|I(G)|\geq 3$. 
Let $M\subseteq \{2, \dots, 10\}$ be a multiset with $\sum M=n-1$,  $t:=f(G)-m_2(M)\in\{0,\dots, 5\}$. Suppose that the $t$ largest elements of $M$  add up to $\geq 2t+3$. Then $G\setminus\{0\}$ has a zero-sum $M$-partition.
\end{lemma}
\begin{proof}
The proof is by induction on $t$. The initial cases are $t=0$ and $t=1$ which follow trivially from Lemma~\ref{Lemma_zero_sum_small_t}. Suppose that $t\geq 2$ and that the lemma is true for smaller $t$. Let $m_1\geq  \dots\geq m_t$ be the $t$ largest elements of $M$ and note that if $m_t\geq 3$, then the lemma follows from Lemma~\ref{Lemma_zero_sum_small_t}. Thus we can assume that $m_t=2$. Note that $\sum_{m_i\geq 3}im_i=|I(G)|+2t\geq 2t+3$ which gives $m_1+\dots+m_t\geq 2t+5$. Note that we cannot have $m_1\leq 3$ since then $m_1+ \dots+ m_{t-1}\leq 3(t-1)$, giving $2t+5\leq m_1+\dots+m_t\leq 3t-1$ which is impossible for $t\leq 5$. 

We have established that $m_1\geq 4$ and $m_1+\dots+m_t\geq 2t+5$. Consider the multiset $M'=(M\setminus \{m_1\})\cup \{m_1-2, 2\}$. If $m_1\geq 5$, then we have $f(G)-m_2(M')=t-1$ and the $t-1$ largest elements of $M'$ add up to $\geq (m_1-2) +m_2+\dots+m_{t-1}=(m_1 +m_2+\dots+m_{t})-4\geq (2t+5)-4=2(t-1)+3$. 
If $m_1= 4$, then we have $f(G)-m_2(M')=t-2$ and the $t-2$ largest elements of $M'$ add up to $\geq m_2+\dots+m_{t-1}=(m_1 +m_2+\dots+m_{t})-6\geq (2t+5)-6=2(t-2)+3$. In either case,   by induction, there is a zero-sum $M'$-partition of $G$. Combining the size $m_1-2$ and $2$ sets in this partition into a size $m_1$ set produces a $M$-partition.
\end{proof}

We'll need the following result about zero-sum 3-sets.
\begin{lemma}\label{Lemma_triple_noninvolutions}
Let $G\neq \mathbb{Z}_2^m, \mathbb{Z}_2^m\times \mathbb{Z}_4$ be a sufficiently large abelian group. Then there are $n/100$ disjoint triples of distinct non-involutions $x,y,z$ with $x+y+z=0$.
\end{lemma}
\begin{proof}
In any large abelian group $G$, Lemma~\ref{Lemma_zero_sum_345_partition} applied with $p=1$, and $X\subseteq G\setminus\{0\}$ a zero-sum set of cardinality at least $n-4$ and divisible by $3$ gives $0.33332n$ disjoint zero-sum 3-sets (that such a subset $X$ exists follows from Lemma~\ref{Lemma_find_set_with_correct_sum}).
If $G\neq \mathbb{Z}_2^m, \mathbb{Z}_2^m\times \mathbb{Z}_4, \mathbb{Z}_2^m\times \mathbb{Z}_3$,  then $|I(G)\cup \{0\}|\leq 0.25n$. At least $0.33332n-0.25n\geq 0.01n$ of these zero-sum $3$-sets must also be disjoint from $I(G)\cup \{0\}$. All of these sets satisfy the lemma, so it remains to just look at $\mathbb{Z}_2^m\times \mathbb{Z}_3$.  
In this case, we know from Theorem~\ref{Theorem_zero_sum_Z2m} that $\mathbb{Z}_2^m$ has $0.1n$ disjoint zero-sum 3-sets $\{x_1, y_1, z_1\}, \dots, \{x_{0.1n}, y_{0.1n}, z_{0.1n}\}$. Let $(0,1)$ be the generator of the $\mathbb{Z}_3$ subgroup of  $\mathbb{Z}_2^m\times \mathbb{Z}_3$.  
Now $\{(x_1,1), (y_1,1), (z_1,1)\}, \dots, \{(x_{0.1n},1), (y_{0.1n},1), (z_{0.1n},1)\}$ are zero-sum 3-sets of non-involutions.
\end{proof}

Somewhat annoyingly, our proof doesn't quite work with the group $\mathbb{Z}_2^m\times \mathbb{Z}_4$, so we deal with this group by hand in the following lemma.
\begin{lemma}\label{Lemma_zero_sum_Z2XZ4}
For sufficiently large $m$, let $G=\mathbb{Z}_2^m\times \mathbb{Z}_4$. Let $t\in [3, 0.001 f(G)]$ with $|I(G)|+2t\equiv 0 \pmod{3}$ and consider the multiset $M=\{(f(G)-t)\times 2, \frac{1}{3}(|I(G)|+2t)\times 3\}$. Then $(\mathbb{Z}_2^m\times \mathbb{Z}_4)\setminus \{0\}$ has a zero-sum $M$-partition.
\end{lemma}
\begin{proof}
Note $f(G)=2^{m}$ and $|I(G)|=2^{m+1}-1$.
Write $t=k+3s$ for $k\in \{3,4,5\}$ and $s\in \mathbb{Z}$. Note that $|I(G)|+2k\equiv |I(G)|+2(k+3s)\equiv |I(G)|+2t\pmod 3$, and so we can define a multiset $M'=\{(f(G)-k)\times 2, \frac13(|I(G)|+2k)\times 3\}$. Note that $\sum M'=2(f(G)-k)+\frac33(|I(G)|+2k)=n-1$, and that the $k$ largest   elements of $M'$ add up to $\geq \min(3k, |I(G)|+2k)\geq 2k+3$. Thus Lemma~\ref{Lemma_zero_sum_small_t_better} gives us a zero-sum $M'$-partition of $G$. Let $\mathcal S$ be the family of 3-sets in this partition.  Let $K=\bigcup \mathcal S\setminus I(G)$, noting  that $|K|=2k$. 

Let $(0,1)$ be the generator of the  $\mathbb{Z}_4$ subgroup of  $\mathbb{Z}_2^m\times \mathbb{Z}_4$. Note that $I(G)=\{(x,0), (x,2): x\in \mathbb{Z}_2^m\}\setminus (0,0)$. Let $K'=K\cup (K+(0,1))\cup(K+(0,2))\cup (K+(0,3))$ noting that $|K'|\leq 8k$. Let $\mathcal S'\subseteq \mathcal S$ be the subfamily of 3-sets which are disjoint from $K'$ and of the form ``$\{(x,0), (y,0), (z,0)\}$''. We claim that $|\mathcal S'|\geq 0.01n$.
To see this, first not that there are at most $|K'|\leq 24$ 3-sets intersecting $K'$. All the remaining  $\geq \frac13(|I(G)|+2k)-24=\frac132^{m+1}-24$ sets must be of the form $\{(x,0), (y,0), (z,0)\}$ or $\{(x,2), (y,2), (z,0)\}$ for some $x,y,z\in \mathbb Z_2^m$. There are at most $2^m/2$ 3-sets of the form $\{(x,2), (y,2), (z,0)\}$ in $\mathcal S$ (since $|\{(x,2):x\in \mathbb Z_2^m\}|=2^m$ and every triple of this form used two elements of $\{(x,2):x\in \mathbb Z_2^m\}$), leaving us with $\frac132^{m+1}-24-2^m/2\geq 0.01n$ $3$-sets in $\mathcal S'$.

 For each $S\in \mathcal S'$, let $S=\{a_S,b_S,c_S\}$ and note that $A_S=\{a_S+(0,1), a_S+(0,3)\}$, $B_S=\{b_S+(0,1), b_S+(0,3)\}$, $C_S=\{c_S+(0,1), c_S+(0,3)\}$ must be distinct 2-sets in the partition (using that $S\in \mathcal S'$).
Define three zero-sum sets $T^a_S=\{a_S, b_S+(0,1), c_S+(0,3)\}$, $T^b_S=\{b_S, c_S+(0,1), a_S+(0,3)\}$, $T^a_S=\{c_S, a_S+(0,1), b_S+(0,3)\}$.  Now for each $S\subseteq \mathcal S'$, replace $S, A_S, B_S, C_S$ in the original partition by  $T^a_S, T^b_S, T^c_S$ to get a $M$-partition of $G$.
\end{proof}

The following lemma finds zero-sum $M$-partitions for multisets $M$ with $f(G)-m_2(G)$ outside of the ranges considered by Lemmas~\ref{Lemma_zero_sum_large_t} and~\ref{Lemma_zero_sum_small_t_better}.
\begin{lemma}\label{Lemma_zero_sum_medium_t}
Let $G$ be a sufficiently large abelian group with $|I(G)|\geq 3$. 
Let $M\subseteq \{2, \dots, 10\}$ be a multiset with  $f(G)-m_2(M)\in[3,   0.001n]$.
Then $G\setminus \id$ has a zero-sum $M$-partition.
\end{lemma}
\begin{proof}
As in Lemma~\ref{Lemma_zero_sum_small_t}, we can assume that $|I(G)\setminus 0|\leq n/2$. The proof is by induction on $t$.
The initial cases are when  $t=3,4,5$, in which case the lemma follows from Lemma~\ref{Lemma_zero_sum_small_t_better} (to see this, we need to know that the condition ``the $t$ largest elements of $M$ add up to at least $2t+3$'' holds for these values of $t$. Letting $y_1\geq \dots\geq y_t$ be the t largest elements, note that $\sum_{i=1}^ty_t\geq ty_t$ and $\sum_{i=1}^ty_t\geq \sum_{i>y_t}im_i(M)$ both hold. Now, using $\sum_{i\geq 3}im_i(M)=n-1-2m_2(M)=|I(G)|+2t\geq 2t+3$, we get $\sum_{i=1}^ty_i\geq \min(3t, \sum_{i\geq 3}im_i(M))\geq \min(3t, 2t+3)=2t+3$).
Now suppose that $t\geq 6$ and that the lemma holds for multisets $M'$ with $f(G)-m_2(M') <t$ and let  $M\subseteq \{2, \dots, 10\}$ be a multiset with  $f(G)-m_2(M)=t$, and $\sum M=n-1$.  

Suppose that $\max M\geq 4$. Let $M'=(M\setminus \{\max M\})\cup\{\max M-2,2\}$. Note that  $f(G)-m_2(M')=t-1$ or $t-2$ (depending on whether $\max M=4$ or not).  By induction, there is a zero-sum $M'$-partition of $G$. Combining the size $\max M-2$ and $2$ sets in this partition into a size $\max M$ set produces a $M$-partition.

Suppose that $\max M\leq 3$, which implies that $M\subseteq \{2,3\}$. If $G=\mathbb{Z}_2^m\times \mathbb{Z}_4$, then the lemma follows from Lemma~\ref{Lemma_zero_sum_Z2XZ4}, so suppose this doesn't happen. Let $M'=(M\setminus \{3,3\})\cup\{2,2,2\}$. Note that  $f(G)-m_2(M')=t-3$, so by induction, we get an $M'$-partition of $G$. In this partition there are  $2(t-3)\leq 0.01n$ non-involutions in sets of size $\geq 3$. Therefore  
Lemma~\ref{Lemma_triple_noninvolutions} gives us a zero sum set $\{x,y,z\}$ of non-involutions so that $\{x,-x\}, \{y,-y\}, \{z,-z\}$ are sets in the $M'$-partition. Now replace $\{x,-x\}, \{y,-y\}, \{z,-z\}$  by $\{x,y,z\}, \{-x,-y,-z\}$ to get an $M$-partition of $G$.
\end{proof}

The following is the main result of this section. Together with earlier results of Tannenbaum~\cite{tannenbaumold} and Zeng~\cite{Zeng2015OnZP}, it answers a problem of Tannenbaum from 1983 \cite{tannenbaumpartitions}. 
\begin{theorem}\label{Theorem_zero_sum_partitions_main}
Let $G$ be a sufficiently large abelian group with $|I(G)|\geq 3$ and let $M\subseteq \{2,3,4,\dots\}$ be a multiset. Then there is a zero-sum   $M$-partition of $G \setminus \{0\}$ if, and only if, all of the following are true.
\begin{enumerate}
\item   $\sum M=n-1$ and $f(G)\geq m_2(M)$.
\item If $f(G)= m_2(M)+1$, then $\max M\geq 5$.
%\item If $f(G)= m_2(M)+2$, then $\max M\geq 4$.
\end{enumerate}
\end{theorem}
\begin{proof}
``If'' direction: Let $t=f(G)-m_2(M)$ and note that this is $\geq 0$ by (1). As in Lemma~\ref{Lemma_zero_sum_345_partition}, without loss of generality, we can suppose that $\max M\leq 10$. 
Note that if $t\geq 3$, then we get a zero-sum $M$-partition by Lemma~\ref{Lemma_zero_sum_large_t} or~Lemma~\ref{Lemma_zero_sum_medium_t}. Otherwise, we apply Lemma~\ref{Lemma_zero_sum_small_t_better}, where we additionally need the condition that ``the $t$ largest elements of $M$ add up to at least $2t+3$''. For $t=1$,  we have $\max M\geq 5\geq 2\times1 +3$. For $t=2$, let $x_1, x_2\in M$ be the two largest elements of $M$.  We have $x_1, x_2\geq 3$ (since $\sum_{x_i\geq 3}x_i=|I(G)|+2t\geq 3+2\times 2$ which implies  that $x_1, x_2\neq 2$). We also have $x_1\geq 4$  (otherwise we'd have $M\subseteq \{2,3\}$ giving  $3m_3(M)=\sum M-2m_2(M)=n-1-2m_2=|I(G)|+2f(G)-2m_2=|I(G)|+4$, which gives a contradiction since on one hand $3m_3(M)\equiv 0 \pmod 3$, but on another hand $|I(G)|+1=2^j$ for some $j$ and so $|I(G)|+4= 2^j+3\equiv 1$ or  $2\pmod 3$). These give $x_1+x_2\geq 4+3\geq 2\times 2+3$.

``Only  if'' direction: Consider some zero-sum $M$-partition   $G\setminus \id=S_1\cup \dots S_k$. Then, by the definition of ``$M$-partition'', we have the equality between multisets $\{|S_i|: i=1, \dots, k\}=M$, which gives $\sum M=|G\setminus \id|=n-1$. Suppose that $S_1, \dots, S_{m_2(M)}$ are the size $2$ sets of this partition. Then for each $i$, we must have $S_i=\{g_i, g_i^{-1}\}$, where $g_i$ are distinct non-involutions of $G$. This means that $S_1\cup \dots\cup S_{m_2(M)}\subseteq G\setminus (I(G)\cup \id)$ giving $2m_2=|S_1\cup \dots\cup S_{m_2(M)}|\leq n-1-|I(G)|=2f(G)$. Thus we have established (1). 

If $f(G)= m_2(M)+1$, then we learn that there is some unique non-involution $g$ with $g,-g\not\in S_1\cup \dots\cup S_{m_2(M)}$. Note that it's impossible for a zero-sum set to have precisely one non-involution (since $I(G)\cup \id$  is a subgroup of $G$), and so we get that $g,-g\subseteq S_j$ for some zero-sum set $S_j$ in the partition. Letting $S_j=\{g, -g, a_1, \dots, a_{|S_j|-2}\}$, we get that $a_1+\dots +a_{|S_j|-2}=0$. This implies that $|S_j|\geq 5$, since it's impossible for $1$ or $2$ distinct non-identity involutions to sum to $0$ in an abelian group.

%If $f(G)= m_2(M)+2$, then we learn that there are two non-involutions $g,h$ with $g,-g, h, -h\not\in S_1\cup \dots\cup S_{m_2(M)}$. Suppose for contradiction that $\max M\leq 3$. This means that $g,-g, h, -h$ appear in some size $3$ sets in the partition. As in the previous case, for any $S_j$, it's impossible that $|\{g,-g, h, -h\}\cap S_j|=1$ or  $\{g,-g\}\subseteq S_j$ or $\{h,-h\}\subseteq S_j$. The only remaining case is that for distinct $S_j, S_i$ we have $g,h\in S_j$, and $-g, -h\in S_i$ (possibly after renaming $g$ and $-g$). Let $S_j=\{g,h, a\}$ and $S_i=\{-g,-h,b\}$ for involutions $a,b$. Then since these are zero-sum we get $a=-g-h$ and $b=g+h$. But then $a=-b$, and so since these are involutions we have $a=b$. This contradicts $S_i$ and $S_j$ being disjoint.
\end{proof}

\section{Concluding remarks}\label{sec:concluding}
\subsubsection*{Further applications}

Our methods here have implications for several other conjectures/problems in the area. Bors and Wang~\cite{bors2023compositions} used some of our results to study the group generated by all complete mappings of a finite group $G$. The first author~\cite{muyesser2023cycle} used our results to prove the Friedlander-Gordon-Tannenbaum Conjecture about possible cycle types of orthomorphisms in groups. Another relevant problem is the conjecture of Graham and Sloane on harmonious labellings of trees \cite{graham1980additive}. As this application requires novel ideas, we defer exploring it to a future paper. For the interested reader wishing to get involved in the area, we now list several directions of further research that we believe to be of interest.
\subsubsection*{Counting aspect}
Using more refined arguments in place of the nibble-type arguments we use in the current paper (see for example \cite{bennett2019natural}), the authors anticipate that Theorem~\ref{thm:mainintro} could be strengthened to give the existence of exponentially many complete mappings, as opposed to just one. Counting spanning structures with this approach has been quite fruitful recently to address the famous $n$-queens problem \cite{bowtell2021n, simkin2022lower}. We did not pursue this direction in this paper, as the proof of the Hall-Paige conjecture by Eberhard, Manners, and Mrazovi\'c \cite{asymptotichallpaige} gives a much more precise count on the number of complete mappings than we could hope to accomplish with our methods. 
\par However, we are curious if Fourier-analytic methods used in \cite{asymptotichallpaige} could be used to count sequencings in sequenceable groups, or transversals in subsquares of multiplication tables. This would potentially give alternative resolutions to conjectures of Ringel \cite{ringeloldproblem} and Snevily \cite{snevily}.

\subsubsection*{Small $n$}
Our methods only work for large $n$, and it seems hopeless to extend them to work for all $n$. However some of our theorems likely do extend to all $n$ --- and it would be interesting to reprove them for all $n$ using different methods. One such theorem is our classification of subsquares without transversals. 
\begin{conjecture}
The following holds for all $n$. Let $G$ be a group, and let $A,B\subseteq G$ with $|A|=|B|=n$. Then, $A\times B$ has a transversal, unless there exists some $k\geq 1$, $g_1,g_2\in G$ and a subgroup $H\subseteq G$ such that one of the following holds.
\begin{enumerate}
    \item $H$ is a group that does not satisfy the Hall-Paige condition, and $A\cong g_1H$ and $B\cong Hg_2$. 
    \item $H\cong (\mathbb{Z}_{2})^k$ and $g_1 A= H\setminus\{a_1,a_2\}$, $g_2 B=H\setminus\{b_1,b_2\}$ for some distinct $a_1,a_2\in H$ and distinct $b_1,b_2\in H$ such that $a_1+a_2+b_1+b_2=0$
\end{enumerate}
\end{conjecture}

It wouldn't make sense to extend Theorem~\ref{thm:maintheoremnondisjoint} to all $n$ (due to the probabilistic nature of the statement of that theorem). However if we set $p=1$ and $|(R^1_A\cup R^2_B\cup R^3_C) \Delta (X\cup Y\cup Z)|$ to be very small, we expect that the following should hold.

\begin{conjecture}
Let $G$ be a group of order $n$.
 Let $X,Y,Z$ be equal-sized subsets of $G_A$, $G_B$, and $G_C$ of size $n-1$  having  $\sum X+\sum Y + \sum Z = 0$  (in $G^{\mathrm{ab}}$).
Then, $H_G[X,Y,Z]$ contains a perfect matching. 
\end{conjecture}
Note that the above implies the Hall-Paige conjecture for all $n$, by setting $X,Y,Z$ to be $G\setminus \{\id\}$.

\subsubsection*{Bounds in the main theorem}
It would be interesting to sharpen the bounds in e.g. Theorem~\ref{thm:maintheoremnondisjoint} in various ways. One interesting parameter to optimize is the size of the symmetric difference $|(R^1_A\cup R^2_B\cup R^3_C) \Delta (X\cup Y\cup Z)|$. We proved this with the upper bound $|(R^1_A\cup R^2_B\cup R^3_C) \Delta (X\cup Y\cup Z) |\leq p^{10^{10}}n/\log(n)^{10^{11}}$, and it would be interesting to figure out what the best possible bound one could impose here is. 

One improvement that can be made with essentially no modifications to the proof is to use the upper bound $|(R^1_A\cup R^2_B\cup R^3_C) \Delta (X\cup Y\cup Z) |\leq p^{10^{10}}n/\log(|G'|)^{10^{11}}$. To obtain this, simply replace instances of $\log n$ by $\log|G'|$ throughout the proof (and observe that the ultimate source of all logarithms is Theorem~\ref{Theorem_write_commutators_as_short_products}). 

From the other side, it is easy to see that Theorem~\ref{thm:maintheoremnondisjoint} isn't true with $|(R^1_A\cup R^2_B\cup R^3_C) \Delta (X\cup Y\cup Z)|\gg p^2n$ --- this is because under the assumptions of that theorem, vertices have $\approx p^2n$ edges going into $R^1_A\cup R^2_B\cup R^3_C$. If $|(R^1_A\cup R^2_B\cup R^3_C) \Delta (X\cup Y\cup Z)|\gg p^2n$, then it would be possible to choose $X,Y,Z$ satisfying this where some vertices have no edges going into $X\cup Y\cup Z$ (and so there will be no perfect matching).
It would be interesting to sharpen these bounds.

It would also be interesting to improve the bounds on Theorem~\ref{thm:maintheoremnondisjoint} in the special case when $p=1$. In this case, the theorem reduces to a non-probabilistic statement.
\begin{problem}
For each $n$, what is the largest number $f(n)$ such that the following is true?

 Let $G$ be a group of order $n$. 
 Let $X,Y,Z$ be  subsets of $G_A$, $G_B$, and $G_C$ respectively, each of size $n-f(n)$ with $\sum X+\sum Y + \sum Z = 0$  (in $G^{\mathrm{ab}}$).
Then, $H_G[X,Y,Z]$ contains a perfect matching. 
\end{problem}
From  Theorem~\ref{thm:maintheoremnondisjoint}, we get that for large $n$, we have $f(n)\geq n/\log(n)^{10^{11}}$. As discussed above, the logarithmic factors are redundant for abelian groups, so in this case, one could hope for a much more precise understanding of the corresponding function $f(n)$. Concretely, we propose the following.

\begin{problem}\label{prob:cyclicgroups}
For each $n$, what is the largest number $g(n)$ such that the following is true?

Let $G$ be the cyclic group of order $n$. 
 Let $X,Y,Z$ be  subsets of $G_A$, $G_B$, and $G_C$ respectively, each of size $n-g(n)$ with $\sum X+\sum Y + \sum Z = 0$.
Then, $H_G[X,Y,Z]$ contains a perfect matching. 
\end{problem}
From our results, it follows that $g(n)=\Omega(n)$, and it is not hard to see that $g(n)\leq n/2$. It would already be interesting to determine $g(n)$ asymptotically. One can also consider the non-partite version of the problem. That is, what is the smallest subset $S\subseteq \mathbb{Z}_n$ with $|S|$ divisible by $3$, $\sum S = 0$, and $S$ cannot be partitioned into triples with zero-sum?

\subsubsection*{Strong complete mappings}
 An \textbf{orthomorphism} of a group $G$ is a bijection $\psi \colon G\to G$ such that $x\to x^{-1}\psi(x)$ is also bijective. A \textbf{strong complete mapping} of a group $G$ is a a complete mapping which is also an orthomorphism. Evans raised the fascinating problem of characterising the groups $G$ which contain strong complete mappings (see \cite{evans2013existence}). We remark that a strong complete mapping of cyclic groups corresponds to the placement of non-attacking queens on a toroidal chessboard. Using this correspondence, a recent result of Bowtell and Keevash \cite{bowtell2021n} can be interpreted as an estimation on the number of strong complete mappings of cyclic groups. It would be very interesting to see if methods we develop in this paper for general groups can be combined with the strategies in \cite{bowtell2021n} to make progress on Evans' problem.

\subsubsection*{Mappings of groups with other properties}
Let $S$ be a multiset with elements coming from a group $G$. When is there a bijection $\phi\colon G\to G$ such that the multiset $\{x\phi(x)\colon x\in G\}$ is equal to $S$? The Hall-Paige conjecture corresponds to the case when $S=G$. In  \cite{hall1952combinatorial}, Hall answers this question for abelian groups. It would be interesting to generalise these results to non-abelian groups. In \cite{anastos2022splitting}, a conjecture in this direction is given in the setting of Latin squares (quasi-groups). This seems like an exciting direction to generalise the Ryser-Brualdi-Stein conjecture \cite{keevash2020new}.

\subsubsection*{The K\'ezdy-Snevily conjecture}
We end with another problem similar in spirit to the previous one, but this time concerned only with cyclic groups. It was proposed initially by Snevily in 2000 \cite{snevily}, and reiterated by K\'ezdy and Snevily \cite{kezdy2002distinct}.

\begin{problem}
    For any positive $k$ and $n$ with $k<n$, show that any sequence $a_1,a_2,\ldots, a_k$ of not necessarily distinct elements of $\mathbb{Z}_n$ admits a permutation $\pi$ such that the sequence $a_{\pi(1)}+1,a_{\pi(2)}+2,\ldots, a_{\pi(k)}+k$ are all distinct (in $\mathbb{Z}_n$).
\end{problem}

\par The results in the current paper can be used to address the above problem for large $k$, whenever the sequence $a_1,a_2,\ldots, a_k$ does not contain repetitions. Alon addressed the above problem whenever $n$ is prime \cite{alonsnevily}. The previously stated result of Hall \cite{hall1952combinatorial} can be used to address the problem when $k=n-1$. K\'ezdy and Snevily \cite{kezdy2002distinct} solved the problem when $2k\leq n+1$. Otherwise, the problem seems to be wide open.
\par We refer the reader to a 2013 survey by Ron Graham that includes an amusing interpretation of the above problem \cite{graham2013juggling}. The survey by Ullman and Velleman \cite{ullman2019differences} also contains a nice exposition for results of this flavour, including the aforementioned result of Hall. 

\section*{Acknowledgements}
We would like to thank an anonymous referee for a careful reading of this manuscript and helpful suggestions that simplified several of our proofs.

\bibliographystyle{abbrv}
\bibliography{bib}

\end{document}